\documentclass[11pt]{aims}
\usepackage{amsmath, amssymb, mathrsfs}
  \usepackage{paralist}
  \usepackage{graphics, color} %% add this and next lines if pictures should be in esp format
  \usepackage{epsfig} %For pictures: screened artwork should be set up with an 85 or 100 line screen
\usepackage{graphicx}  \usepackage{epstopdf}%This is to transfer .eps figure to .pdf figure; please compile your paper using PDFLeTex or PDFTeXify.
 \usepackage[colorlinks=true]{hyperref}
 \hypersetup{urlcolor=blue, citecolor=red}
 \usepackage[toc,page]{appendix}
\usepackage{chngcntr}
\usepackage{listings}% http://ctan.org/pkg/listings
\usepackage[left=2cm,bottom=2cm,top = 2cm, right =2cm]{geometry}

\usepackage{titlesec}
\titleformat{\section}{\Large\bfseries}{\thesection.}{4pt}{}
\titleformat{\subsection}{\large\bfseries}{\thesection.\arabic{subsection}.}{4pt}{}
\titleformat{\subsubsection}{\bfseries}{\thesection.\arabic{subsection}.\arabic{subsubsection}.}{4pt}{}
\titleformat*{\paragraph}{\bfseries}
\titleformat*{\subparagraph}{\bfseries}
     \def\DOI{}

\newtheorem{theorem}{Theorem}[section]

\newtheorem{lemma}[theorem]{Lemma}
\newtheorem{proposition}[theorem]{Proposition}
\theoremstyle{definition}
\newtheorem{definition}[theorem]{Definition}
\newtheorem{remark}[theorem]{Remark}
\newtheorem{comment}[theorem]{Comment}

\newcommand{\Rb}{\mathbb{R}}
\newcommand{\Ac}{\mathcal{A}}

\newcommand{\Cc}{\mathcal{C}}
\newcommand{\Dc}{\mathcal{D}}
\newcommand{\Ec}{\mathcal{E}}
\newcommand{\Fc}{\mathcal{F}}
\newcommand{\Gc}{\mathcal{G}}

\newcommand{\Oc}{\mathcal{O}}

\newcommand{\Mc}{\mathcal{M}}

\newcommand{\Vc}{\mathcal{V}}

\newcommand{\Pc}{\mathcal{P}}

\newcommand{\inn}{\textup{in}}
\newcommand{\out}{\textup{ex}}

\newcommand{\Ls}{\mathscr{L}}
\newcommand{\Ms}{\mathscr{M}}
\newcommand{\Hs}{\mathscr{H}}
\newcommand{\As}{\mathscr{A}}

\newcommand{\rj}{\left< r \right>}

\newcommand{\pa}{\partial}
\newcommand{\ba}{\begin{array}}
\newcommand{\ea}{\end{array}}
\newcommand{\la}{\langle}
\newcommand{\ra}{\rangle}
\newcommand{\e}{\varepsilon}
\newcommand{\tq}{\tilde q}
\newcommand{\tnu}{\tilde \nu}

\newcommand{\tr}{\tilde r}

\newcommand{\tm}{\tilde m}

\begin{document}                        %% Standard LaTeX command

%%      -----------------------------------------------------------------------
%%      -------------------------------- TITLE -----------------------------
%%      -----------------------------------------------------------------------

\title{Refined description and stability for singular solutions of the 2D Keller-Segel system}

%%      -----------------------------------------------------------------------------
%%      ------------------------------- AUTHORS -----------------------------
\author[C. Collot]{Charles Collot}
\address{Courant Institute of Mathematical Sciences, New York University, 251 Mercer Street, New York, NY 10003, United States of America.}
\email{(Expired email)  cc5786@nyu.edu
(Current email) charles.collot@cyu.fr}
\author[T. Ghoul]{Tej-Eddine Ghoul}
\address{Department of Mathematics, New York University in Abu Dhabi, Saadiyat Island, P.O. Box 129188, Abu Dhabi, United Arab Emirates.}
\email{teg6@nyu.edu}
\author[N. Masmoudi]{Nader Masmoudi}
\address{Department of Mathematics, New York University in Abu Dhabi, Saadiyat Island, P.O. Box 129188, Abu Dhabi, United Arab Emirates.}
\email{masmoudi@cims.nyu.edu}
\author[V. T. Nguyen]{Van Tien Nguyen}
\address{Department of Mathematics, New York University in Abu Dhabi, Saadiyat Island, P.O. Box 129188, Abu Dhabi, United Arab Emirates.}
\email{Tien.Nguyen@nyu.edu}

\begin{abstract}
We construct solutions to the two dimensional parabolic-elliptic Keller-Segel model for chemotaxis that blow up in finite time $T$. The solution is decomposed as the sum of a stationary state concentrated at scale $\lambda$ and of a perturbation. We rely on a detailed spectral analysis for the linearised dynamics in the parabolic neighbourhood of the singularity performed by the authors in \cite{CGNNarx19a}, providing a refined expansion of the perturbation. Our main result is the construction of a stable dynamics in the full nonradial setting for which the stationary state collapses with the universal law $\lambda \sim 2e^{-\frac{2+\gamma}{2}}\sqrt{T-t}e^{-\sqrt{\frac{|\ln (T-t)|}{2}}}$ where $\gamma$ is the Euler constant. This improves on the earlier result by Raphael and Schweyer \cite{RSma14} and gives a new robust approach to so-called type II singularities for critical parabolic problems. A by-product of the spectral analysis we developed is the existence of unstable blowup dynamics with speed $\lambda_\ell \sim C_0(T-t)^{\frac{\ell}{2}} |\ln(T-t)|^{-\frac{\ell }{2(\ell - 1)}}$ for $\ell \geq 2$ integer. 

\end{abstract}

% With AMS-LaTeX, \maketitle follows the abstract
\maketitle   

%%      ---------------------------------------------------------------------
%%      ------------------- TABLE OF CONTENTS (OPTIONAL) --------------------
%%      ---------------------------------------------------------------------

%% ***** IF YOUR PAPER IS OVER 40 PAGES AND YOU WISH TO HAVE A TABLE
%% ***** OF CONTENTS, PLEASE UNCOMMENT THE FOLLOWING LINE

% \tableofcontents

%%      ---------------------------------------------------------------------
%%      ---------------------------- BODY OF PAPER --------------------------
%%      ---------------------------------------------------------------------

%%      Please input or insert the body of your paper here.

\section{Introduction}
\subsection{The Keller-Segel system}
This paper is concerned with the Keller-Segel system modelling chemotaxis. \textit{Chemotaxis} is a biological phenomenon
describing the change of motion of a population density (of certain cells, animals, and of particles as well) in response (\textit{taxis}) to an external chemical stimulus spread in the environment where they reside. The chemical signal can be secreted by the species itself or supplied to it by an external source. As a consequence, the species changes its movement toward (\textit{positive chemotaxis}) or away from (\textit{negative chemotaxis}) a higher concentration of the chemical substance. A possible fascinating issue of a positive chemotactical movement is the aggregation of the organisms involved to form a more complex organism or body. The first mathematical model for chemotaxis was proposed by Keller-Segel \cite{KSjtb70} to describe the aggregation of the slime mold amoebae Dictyostelium discoideum (see also Patlak \cite{PATbmb53} for an earlier model and \cite{KSjtb71a}, \cite{KSjtb71b}, \cite{KELbook80} for various assessments). Since the publication of \cite{KSjtb70}, a large literature has addressed the mathematical, biological and medical aspects of \textit{chemotaxis}, showing the importance of the problem and the great interest that different kinds of scientists carry on it. We recommend the reference \cite{Hjdmv03} for a survey of the mathematical problems encountered in the study of the Keller-Segel model and also a wide bibliography, including references on other types of models describing \textit{chemotaxis}.

The present paper deals with a simplified version of the Keller-Segel model introduced by Nanjundiah \cite{Njtb73}, which reads as follows:
\begin{equation}\label{eq:KS}
\arraycolsep=1.4pt\def\arraystretch{1.6}
\left\{ \begin{array}{rl}
\partial_t u &= \Delta u - \nabla \cdot (u \nabla \Phi_u),\\
0 & =\Delta \Phi_u + u,
\end{array}
\right. \quad (x,t)\in \mathbb{R}^2\times \mathbb{R}_+.
\end{equation}
Here $u$ represents the cell density, and $ \Phi_u$ is the
concentration of chemoattractant which can be defined directly by
\begin{equation}
\Phi_u(x,t) = -\frac{1}{2\pi}\int_{\mathbb{R}^2}\log|x - y|u(y,t)dy.
\end{equation}
The nonlinear term $\nabla . (u \nabla \Phi_u)$ models the cell movement towards higher concentrations of the chemical signal. The more the cells aggregate, the more is the attracting chemical substance produced by the cells. This process is counterbalanced by cell diffusion, but if the cell density is sufficiently large, the nonlocal chemical interaction counterbalances the diffusion and results in a blowup of the cell density. The solution blows up in finite time $T$ in the sense that
$$\limsup_{t \to T}\|u(t)\|_{L^\infty(\mathbb{R}^2)} = +\infty,$$
and the blowup set $\mathcal{S}$ is then defined by
\begin{equation}
\mathcal{S} = \{\hat a \in \mathbb{R}^2 \mid \exists (x_k, t_k) \to (\hat a,T) \; \text{such that}\; |u(x_k, t_k)| \to +\infty\}.
\end{equation}

Solutions of system \eqref{eq:KS} satisfy the conservation of the total mass as well as the positivity of the cell density,
\begin{equation}
M:= \int_{\mathbb{R}^2}u(x,t)dx = \int_{\mathbb{R}^2}u_0(x)dx, \quad  \mbox{and} \quad \mbox{if }u_0\geq 0 \ \mbox{ then } \ u(t) \geq 0.
\end{equation}
There is also a scaling invariance: if $u$ is a solution then so is the rescaled function 
$$u_\lambda(x,t) =\lambda^{-2} u(\lambda^{-1} x,\lambda^{-2} t), \quad \forall\lambda > 0.$$
As the mass which is a conserved quantity is invariant under the above renormalisation, the problem is called critical. In two space dimensions, it was first proved (see J\"ager-Luckhaus \cite{JLtams92}, Corrias-Perthame-Zaag \cite{CPZmjm04}) that there is global existence for solutions with small initial mass, while blowup occurs for a large initial mass. The existence of a mass threshold was then conjectured in \cite{Njtb73}, \cite{CPmb81}, \cite{Clec84}, namely that the chemotactic collapse (blowup) should occur if and only if $M$ is greater than $8\pi$. This $8\pi$ mass threshold conjecture is later proven in \cite{DNRjde98}, \cite{BDPjde06}, \cite{BCCjfa12} (see also \cite{Namsa95}, \cite{Njia01} for related results in the bounded domain case). Following \cite{BCCjfa12}, the criticality of the mass value can be seen by computing the second moment 
\begin{equation}
\frac{d}{dt}\int_{\mathbb{R}^2}|x|^2u(x,t)dx = 4M\left(1 - \frac{M}{8\pi}\right).
\end{equation}
Thus, if $M > 8\pi$, the right hand side is strictly negative, and positive solutions with finite second moment cannot be globally defined, or this second moment would reach zero in finite time.

Below the threshold $M < 8\pi$, Dolbeault-Perthame announced in \cite{DPcrasp04} that there is global existence of a solution for system \eqref{eq:KS} in a weak sense. This result is further completed and improved in \cite{BDPjde06}, \cite{BCCjfa12} through the existence of \textit{free-energy solutions}. Furthermore, the asymptotic behavior is given by a unique self-similar profile of the system (see also \cite{NSproc04} for radially symmetric results concerning self-similar behavior).

At the threshold $M = 8\pi$, the authors of \cite{BKGLtmna06, BKLNmmas06} show the existence of global radially symmetric solutions to system \eqref{eq:KS} for initial data with finite or infinite second moment. In \cite{BCMcpam08}, Blanchet-Carrillo-Masmoudi proved the existence of solutions to \eqref{eq:KS} concentrating in infinite time through the \textit{free energy functional} introduced by Nagai-Senbai-Yoshida in \cite{NSYfe97}. Furthermore, they showed that the solution converges to a delta Dirac distribution at the center of mass.

The system \eqref{eq:KS} has a family of explicit stationary solutions of the form
\begin{equation}\label{def:Qmua}
\forall\lambda > 0, \; a \in \mathbb{R}^2, \quad U_{\lambda, a}(x) =\frac{1}{\lambda^2}U\left(\frac{x - a}{\lambda}\right) \quad \text{with}\quad U(x)=  \frac{8}{(1 + |x|^2)^2}.
\end{equation}
These solutions have the threshold mass $M=8\pi$ and infinite second moment. They play an important role in the description of concentration both in finite and infinite time. Ghoul-Masmoudi \cite{GMcpam18} construct concrete infinite time blowup solutions to \eqref{eq:KS} with threshold mass $M=8\pi$ admitting the asymptotic dynamic as $t \to +\infty$,
$$u(x,t) \sim U_{\lambda(t)}(x)e^{-\frac{|x|^2}{2t}} \quad \textup{with} \quad\lambda^2(t) \sim \frac{I}{\ln t} \;\; \textup{and} \;\; I = \int_{\mathbb{R}^2} |x|^2 u_0(x)dx,$$
see also Davila-del Pino-Dolbeault-Musso-Wei \cite{DPDMWarx19} for an entirely different approach from that of \cite{GMcpam18} that leads to the same blowup rate.

It is worth to mention that the study of the positive steady states of the problem \eqref{eq:KS}, namely the solutions of the elliptic system 
\begin{equation}\label{eq:KSst}
\arraycolsep=1.4pt\def\arraystretch{1.6}
\left\{ \begin{array}{ll}
0 &= \Delta  u - \nabla .( u \nabla \Phi_u),\\
0 & =\Delta \Phi_u +  u,
\end{array}
\right. \quad x \in \mathbb{R}^2, \;\; u > 0,
\end{equation}
is equivalent to the study of the ground states of the equation
\begin{equation}\label{eq:ev}
\Delta v +\lambda_0 e^v = 0, \quad x \in \mathbb{R}^2 \;\;\text{and} \;\;\lambda_0 > 0. 
\end{equation}
This basic feature observed in \cite{MWnon06} follows from the fact that the solution of \eqref{eq:KSst} satisfies the relation
$$\int_{\mathbb{R}^2}u|\nabla (\log u - \Phi_u)|^2 dx = 0,$$
so that $u =\lambda e^{\Phi_u}$ for some positive constant $\lambda$, resulting in equation \eqref{eq:ev}. Note that $ V_{\lambda,a}(x) = \log U_{\lambda,a}(x),$ where $U_{\lambda,a}$ is defined by \eqref{def:Qmua}, is a solution to \eqref{eq:ev} with $\lambda_0 = 1$. The asymptotic behaviour of solutions of \eqref{eq:ev}, in a bounded domain $\Omega$ of $\mathbb{R}^2$ or in the whole space for which $\lambda_0 \int_{\Omega}e^{v}$ remains uniformly bounded, is well understood after the works by Brezis -Merle \cite{BMcpde91}, Nagasaki-Suzuki \cite{NSaa90}, Li-Shafrir \cite{LSiumj94}, Manuel-Wei \cite{MWnon06} and references therein. Their results read as follows: $\lambda_0 e^v$ approaches a superposition of Dirac deltas in the interior of $\Omega$. 
More precisely, the authors in \cite{MWnon06} show that for all $\lambda_0$ sufficiently small, there exists a solution $v_{\lambda_0}$ to equation \eqref{eq:ev} such that
$$v_{\lambda_0}(x) = \sum_{i = 1}^m V_{\mu_i\sqrt{\lambda_0}, a_i}(x) + \mathcal{O}(1) \quad \text{and}\quad \lambda \int_{\Omega} e^{v_{\lambda_0}(x)}dx \to  8\pi m \quad \text{as} \quad\lambda_0 \to 0,$$
where $V_{\mu_i\sqrt {\lambda_0},a_i}$ is defined above, and the $a_i$'s are the local maxima of $v_{\lambda_0}$ in the interior of $\Omega$ and the $\mu_i$'s are the positive constants.\\

\noindent Above the threshold $M>8\pi$, concrete examples of finite time blowup solutions are constructed by Herrero-Vel\'azquez in \cite{HVma96} (the scaling law found there is false but after correcting it the rest of the proof remains valid), with a further stability study in \cite{Vsiam02, Vsiam04b, Vsiam04a} and by Raph\"ael-Schweyer \cite{RSma14}. Regarding the temporal blowup rate, the central issue is to distinguish type I from type II blowup. We say that a solution $u(t)$ of \eqref{eq:KS} exhibits type I blowup at $t = T$ if there exist a constant $C > 0$ such that
\begin{equation}\label{def:TypeI}
\limsup_{t\to T}(T-t)\|u(t)\|_{L^\infty(\mathbb{R}^2)} \leq C,
\end{equation}
otherwise, the blowup is of type II. This notion is motivated by the ODE $u_t=u^2$ obtained by discarding diffusion and transport in the equation. The lower blowup rate estimate
$$ \|u(t)\|_{L^\infty(\mathbb{R}^2)} \geq C(T-t)^{-1}$$
is obtained for any blowup solutions of \eqref{eq:KS} by Kozono-Sugiyama \cite{KSjee08}. Importantly, it is known that in the two dimensional case any blowup solution of \eqref{eq:KS} is of type II (see Theorem 8.19 in \cite{SSbook11} and Theorem 10 in \cite{NScm08} for such a statement). In \cite{Sjmpa13}, Suzuki studies the problem \eqref{eq:KS} in a bounded domain, with Dirichlet condition for the Poisson part, i.e. $\Phi_u\mid_{\partial \Omega} = 0$, so that the blowup is excluded on the boundary. More precisely, he proves that 
\begin{equation}\label{eq:converu}
u(x,t) \rightharpoonup \sum_{\hat a \in \mathcal{S}} m(\hat a) \delta_{\hat a}(dx) + f(x)dx \quad \text{in}\;\; \Mc(\bar \Omega) = \mathcal{C}(\bar \Omega)'
\end{equation}
as $t \to T$, where $0 < f(x) \in L^1(\Omega)\cap \mathcal{C}(\Omega \setminus \mathcal{S})$. Furthermore, the author also asserts that $m(\hat a) = 8\pi$ holds for each $\hat a \in \Omega \cap \mathcal{S}$.\\

\subsection{Statement of the result}

Singularity formation for critical problems has attracted a great amount of work since the seminal results for dispersive equations by Merle-Raphael \cite{MRam05}, Krieger-Schlag-Tataru \cite{KSTim08}, Rodnianski-Sterbenz \cite{RSam10}, Raphael-Rodnianski \cite{RRmihes12} and references therein. The approach of \cite{MRam05}, \cite{RRmihes12}, \cite{MRRim13}, relying on a careful understanding of the solution near the stationary state, and on modulation laws computed via so-called tail dynamics, has been carried on to parabolic problems \cite{RScpam13,RSapde2014,RSma14}, \cite{Sjfa12} and \cite{Capde17}. Type II singularities for the semilinear heat equation had been previously studied by means of matched asymptotic expansions in an unpublished paper by Herrero-Velazquez (see \cite{HVcras94} for an announcement of their result and \cite{FHVslps00} for a formal analysis). This result was later confirmed by Mizoguchi \cite{Made04, Mma07}, and a new inner-outer gluing technique developed recently by Davila-del Pino-Wei \cite{DPWim19} (see also \cite{PMWats19} and references therein for recent results). A new approach for the construction of singular solutions of parabolic problems was initiated in \cite{HRjems19}, \cite{CRSmams19}, \cite{CMRjams19}, \cite{MRSimrn18} and the present paper fits into this series of works. The aim is to study type II blow-up as well as self-similar singularities, for supercritical and critical equations, in a unified and more natural approach (see first comment below). The present paper aims at applying for the first time this new approach to the delicate degenerate problem of the critical collapse for the parabolic-elliptic Keller-Segel problem \eqref{eq:KS}. In comparison with \cite{RSma14}, we obtain a refined expansion for the scale (proving the precise universal law \eqref{id:lawlambdastable}), the nonradial stability of the dynamics, the existence of unstable blow-up laws and remove the slightly supercritical mass restriction ($M$ close to $8\pi$). The solutions we construct are in the following function space
\begin{equation}\label{def:spacemathcalE}
\mathcal E:=\left\{u:\mathbb R^2\rightarrow \mathbb R, \quad \| u \|_{\mathcal E}^2:=\sum_{k=0}^2 \int_{\mathbb R^2}\langle x \rangle^{\frac 32 +2k}|\nabla^k u|^2<\infty  \right\}.
\end{equation}

\begin{theorem}[Stable blowup solutions]\label{theo:Stab}

There exists a set $\mathcal O\subset \mathcal E \cap L^1(\mathbb R^2)$ of initial data $u_0$ such that the following holds for the associated solution to \eqref{eq:KS}. It blows up in finite time $T=T(u_0)>0$ according to the dynamic
$$
u(x,t)=\frac{1}{\lambda^2(t)}\left(U+\tilde u \right)\left(\frac{x-x^*(t)}{\lambda (t)}\right),
$$
where
\begin{itemize}
\item \emph{(Precise law for the scale)}
\begin{equation}\label{id:lawlambdastable}
\lambda (t)=2e^{-\frac{2+\gamma}{2}}\sqrt{T-t}e^{-\sqrt{\frac{|\ln T-t|}{2}}}(1+o_{t\uparrow T}(1)).
\end{equation}
\item \emph{(Convergence of the blow-up point)} There exists $X=X(u_0)\in \mathbb R^2$ such that $x^*(t)\rightarrow X$ as $t\uparrow T$.
\item \emph{(Convergence to the stationary state profile)}
\begin{equation}\label{est:theo1conver}
\int_{\mathbb R^2} \left(\tilde u^2 (t,y)+\langle y \rangle^2|\nabla \tilde u (t,y)|^2\right)dy\rightarrow 0 \quad \mbox{as }t\uparrow T.
\end{equation}
\item \emph{(Stability)} For any $u_0\in \mathcal O$, there exists $\delta(u_0)>0$ such that if $v_0\in \mathcal E \cap L^1(\mathbb R^2)$ satisfies $\| v_0-u_0\|_{\mathcal E}\leq \delta(u_0)$ then $v_0\in \mathcal O$ and the same conclusions hold true for the corresponding solution $v$.
\item \emph{(Continuity)} For any fixed $u_0 \in \mathcal O$, one has $(T(v_0),X(v_0))\rightarrow (T(u_0),X(u_0))$ as $\| v_0-u_0\|_{\mathcal E}\rightarrow 0$.
\end{itemize}

\end{theorem}

\begin{remark}

In Theorem \ref{theo:Stab}, the initial datum $u_0$ can possibly be non-radial, and possibly sign-changing. The exponent $3/2$ in the definition \eqref{def:spacemathcalE} of the function space $\mathcal E$ allows for initial data that are arbitrarily large in $L^1$ (but sufficiently spread out away from the singularity). Additionally, our proof involves a detailed understanding of the perturbation $\tilde u$, see Definition \ref{def:bootstrap}.
\end{remark}

We are also able to construct for problem \eqref{eq:KSst} blowup solutions having other unstable blowup speeds by the same analysis. This corresponds to the case where the leading order part of the perturbation is located on an eigenmode with faster decay, while the eigenmodes with slower decay are not excited. This is only obtained here in the radial case. The corresponding solutions are sign-changing.

\begin{theorem}[Unstable blowup solutions]\label{theo:UnStab}

For any $\ell\in \mathbb N$ with $\ell \geq 2$, there exists an initial datum $u_0\in \mathcal E\cap L^1$ with spherical symmetry, such that the corresponding solution to  \eqref{eq:KS} blows up in finite time $T>0$ according to the dynamic
$$
u(x,t)=\frac{1}{\lambda^2(t)}\left(U+\tilde u \right)\left(\frac{x}{\lambda (t)}\right),
$$
where
\begin{equation}\label{id:lawlambdaunstable}
\lambda(t) \sim C(u_0) (T-t)^\frac{\ell}{2} |\ln (T-t)|^{-\frac{\ell}{2(\ell - 1)}},
\end{equation}
and
$$
\int_{\mathbb R^2} \left(\tilde u^2 (t,y)+\langle y \rangle^2|\nabla \tilde u (t,y)|^2\right)dy\rightarrow 0 \quad \mbox{as }t\uparrow T.
$$
\end{theorem}

\begin{remark}

We only give a complete proof of Theorem \ref{theo:Stab}, and sketch how it can be adapted to derive the conclusion of Theorem \ref{theo:UnStab}. Only one major issue arises. These unstable blow-ups are related with eigenmodes of the linearized dynamics that decay faster. Since we do not control the constant (giving the decay rate) in the nonradial coercivity estimate of Proposition \ref{pr:coercivitenonradial}, we are hence only able to construct such unstable blow-ups in the radial sector.
\end{remark}

\begin{comment} [A robust approach]   Parabolic type II singularities involve two-scale problems. To leading order the solution is given by a stationary state that is concentrated at a scale that is much smaller than the parabolic scale (i.e., distance $\sqrt{T-t}$ from the singularity). The core of our approach, following \cite{HRjems19}, \cite{CRSmams19}, \cite{CMRjams19} and \cite{MRSimrn18}, is to obtain the local expansion of the perturbation in the parabolic neighbourhood of the singularity, relying on a precise spectral analysis. The interactions between the stationary state and its perturbation are understood via projecting the dynamics on the corresponding eigenmodes, yielding the obtention of the scaling law. This allows for the construction of a detailed approximate solution, and gives a functional framework to control the remainder, \emph{simultaneously at the scale of the stationary state and at the parabolic scale}. Previous works such as \cite{RSma14} controlled very accurately the solution near the stationary state, but less in the parabolic zone which implied for example the use of what they call radiations and modified modulation equations.

Let us now comment on other available techniques. The inner-outer gluing method, developed in \cite{DPWim19} for the study of singularity formation for the two-dimensional harmonic map flow, provides an interesting alternative framework to ours; the perturbation is controlled separately close to the stationary state and in the parabolic zone, relying on parabolic estimates and on fixed point arguments. The control of the perturbation is obtained under suitable orthogonality conditions, which yields the desired scaling law. Another approach in \cite{HVcras94} and \cite{Made04, Mma07} for the semilinear heat equation used intersection numbers argument, which are unfortunately not available in the nonradial case. We believe the spectral argument in our technique is the starting point for further studies, such as the possible classification of blowup rates in the nonradial case, and our framework would be adapted to the case of semi-linear hyperbolic equations as well.
\end{comment}

\begin{comment} [Main novelties] There are important novelties in the present work, and it is enlightening to compare with the one of \cite{CMRjams19}. The common approach is the analysis of the linearised operator around a stationary state that is concentrated at a scale smaller than the parabolic scale, here \eqref{def:Lsz}. This spectral problem is solved in Proposition \ref{prop:SpecRad} (obtained in \cite{CGNNarx19a}). The perturbation of the stationary state is then expanded along the eigenmodes and the precise knowledge of the eigenvalues then allows for the computation of the scaling laws, here \eqref{id:lawlambdastable} and \eqref{id:lawlambdaunstable}. However, in the present critical case the aforementioned spectral problem is degenerate. Indeed, the concentrated stationary state at scale $\nu\rightarrow 0$ gives a singular limit (the operator does not converge to its pointwise limit). The stable blow-up law corresponds moreover to a degenerate eigenvalue that is $0$ to leading order, and one needs to go to next order corrections that are of order $1/|\log \nu|$ and $1/|\log \nu|^2$ instead of being polynomial.

The nonlinear analysis is more involved too, as we cannot rely on signed quantities or parameters restriction (respectively equation (1.8) and Lemma A.1 in \cite{CMRjams19}). In the part of the analysis that establishes upper bounds on the perturbation, certain nonlinear terms cannot be treated as lower order in comparison with linear terms (i.e. considered as forcing terms for the linearised evolution). We then put them as part of the linearised operator, and show that the spectral structure remains true to leading order (Lemma \ref{lemm:SpecAbar}). This uses as a key feature that the nonlinear terms, due to their very algebraical form in \eqref{eq:KS}, are orthogonal to the resonance of the linear operator near the stationary state (see \cite{CGNNarx19a}). Moreover, to deal with the nonradial part of the solution we use a new coercivity estimate \eqref{est:coerLnuy}, which includes the scaling term in the linearised operator, what greatly simplifies our analysis.
\end{comment}

\begin{comment} [Extension of the stability analysis] Our analysis relies on the description and control of the solution \emph{near} the singularity, which is almost decoupled from what happens \emph{away} from it. Indeed, first, the key part of the analysis takes place in a space with exponentially decaying weight in the zone $|x-x^*|\lesssim \sqrt{T-t}$, second, the treatment of the exterior part allows for cutting of the \emph{full} solution at distance $|x-x^*|\gtrsim 1$ from the singularity. The present analysis could be easily adapted to prove the following two results. First, the analogue of Theorem \ref{theo:Stab} for the existence of this stable blowup dynamics at a point inside a smooth bounded domain, second the existence of multiple stationary states blowing up simultaneously as described in Theorem \ref{theo:Stab} at $n$ distinct points for any $n\geq 2$.
\end{comment}

\paragraph{Notations.} Throughout this paper, we use the notation $A \lesssim B$ to say that there exists a constant $C > 0$ such that $ 0 \leq A \leq CB$. Similarly, $A \sim  B$ means that there exist constants $0 < c < C$ such that $cA \leq B \leq CA$. We denote by
$$\left< r \right> = \sqrt{1 + r^2}$$
the Japanese bracket. Let $\chi \in \mathcal{C}_c^\infty(\mathbb{R}^2)$ be a cut-off function with $0 \leq \chi \leq 1$, $\chi(x) = 1$ for $|x| \leq 1$ and $\chi(x) = 0$ for $|x| \geq 2$. We define for all $M > 0$, 
\begin{equation}\label{def:chiM}
\chi_{_M}(x) = \chi\Big(\frac{x}{M}\Big).
\end{equation}
Given $\nu >0$ and a function $f$, we introduce:
$$f_\nu(z) = \frac{1}{\nu^2}f\Big(\frac{z}{\nu} \Big), \quad y = \frac{z}{\nu}, \quad \zeta = |z|, \quad r = |y|.$$
We introduce the differential operator
\begin{equation}\label{def:Lambda}
\Lambda f(z) = \frac{d}{d \nu}\big[f_\nu(z)\big]_{\nu = 1} = \nabla \cdot(zf) = 2f  + z\cdot \nabla f,
\end{equation}
and the linearized operator around the scaled stationary solution $U_\nu$,
\begin{align}\label{def:Lsz}
&\mathscr{L}^z f(z) = \mathscr{L}^z_0 f - \beta\Lambda f,\\
&\label{def:L0Mz}
\mathscr{L}_0^z f = \nabla \cdot(U_\nu \nabla \mathscr{M}^z f) \quad \textup{with} \quad \mathscr{M}^z f = \frac{f}{U_\nu} - \Phi_f.
\end{align}
In the partial mass setting, namely for
$$m_f(\zeta) = \frac{1}{2\pi}\int_{|z| \leq \zeta} f(z)dz,$$
the operator $\mathscr{L}^z$ acting on radially symmetric functions is transformed to 
\begin{align}\label{def:AsAs0zeta}
&\mathscr{A}^\zeta = \mathscr{A}^\zeta_0 -\beta \zeta\partial_\zeta ,\\
&\mathscr{A}^\zeta_0 = \partial_\zeta^2  -\frac{1}{\zeta}\partial_\zeta + \frac{\partial_\zeta(Q_\nu \cdot)}{\zeta} \quad \textup{and} \quad Q_\nu(\zeta) = \frac{4\zeta^2}{\zeta^2 + \nu^2}.
\end{align}
In terms of the $y$-variable, we work with its rescaled versions
\begin{align}\label{def:LsLs0Ms}
&\mathscr{L} f(y)  = \mathscr{L}_0 f - \beta\nu^2\Lambda f,\\
&\label{def:Ls0Ms}
\mathscr{L}_0 f = \nabla \cdot(U \nabla \mathscr{M} f) \quad \textup{and} \quad \mathscr{M} f = \frac{f}{U} - \Phi_f.
\end{align}
In the partial mass setting the linear operator $\Ls$ becomes
\begin{align}\label{def:AsAs0}
&\mathscr{A}  = \mathscr{A}_0  -\beta \nu^2 r\partial_r ,\\
&\mathscr{A}_0 = \partial_r^2 - \frac{1}{r}\partial_r + \frac{\partial_r(Q \cdot)}{r} \quad \textup{and} \quad Q(r) = \frac{4r^2}{1 + r^2}.
\end{align}
We also introduce the weight functions
\begin{align}\label{def:omeganu}
&\omega_\nu(z) = \frac{\nu^2}{U_\nu(z)}e^{-\frac{\beta |z|^2}{2}}, \quad \rho_0(z) =e^{-\frac{\beta |z|^2}{2}}, \\
&\label{def:omegarho}
\omega(y) = \frac{1}{U(y)}e^{-\frac{\beta\nu^2 |y|^2}{2}}, \quad \rho(y) =e^{-\frac{\beta \nu^2 |y|^2}{2}}. 
\end{align}
The partial mass of the solution is formally a radial solution in dimension $0$, so that to take $k$ adapted derivatives we will use the notation $D^k$ for $k\in \mathbb N$, where
$$
D^{2k}=\left(\zeta \partial_\zeta (\frac{\partial_\zeta}{\zeta}) \right)^{2k}, \ \ D^{2k+1}=\partial_\zeta D^{2k}
$$
and the notation for integers modulo $2$,
$$
k\wedge 2 = k\mod 2.
$$
For a function $f$ of a variable $\xi$ representing any variable in the problem, the radial and non-radial parts of $f$ are defined as:
$$
f(\xi)=f^0(\xi)+f^{\perp}(\xi), \qquad f^{0}(\xi)=(2\pi|\xi|)^{-1} \int_{S(0,|\xi|)} f(\tilde \xi)dS.
$$

\subsection{Strategy of the proof} 
We briefly explain the main steps of the proof of Theorem \ref{theo:Stab} and sketch the different points in the proof of Theorem \ref{theo:UnStab}.\\ 

\noindent \underline{\textit{Renormalization and linearization of the problem:}} The essential part of the analysis lies in the parabolic zone $|x-X|\lesssim \sqrt{T-t}$. Since neither $X$ nor $T$ are known a priori, our method will compute them \emph{dynamically}. In view of the scaling invariance of the problem \eqref{eq:KS}, we introduce the change of variables
\begin{equation}\label{def:wztauvarsIntro}
u(x,t) = \frac{1}{\mu^2}w(z,\tau), \quad \Phi_u(x,t) = \Phi_w(z, \tau), \quad z = \frac{x - x^*}{\mu}, \quad \frac{d\tau}{dt} = \frac{1}{\mu^2},
\end{equation}
where $\mu(t)$ and $x^*(t)$ are time dependent parameters to be fixed later. They will in fine satisfy $ -\mu_\tau/\mu \to\beta_\infty >0$, $\mu(t) \sim \sqrt{2\beta_\infty (T-t)}$ (see \eqref{reinte:exprmu}) and $x^*\to X$, so that $z$ is indeed the parabolic variable. The equation satisfied by $w$ is:
\begin{equation}\label{eq:wztauIntro}
\partial_\tau w  = \nabla \cdot(\nabla w - w\nabla \Phi_w) -\beta \nabla \cdot(z w) + \frac{x^*_\tau}{\mu} \cdot \nabla w \quad \textup{with} \quad\beta = -\frac{\mu_\tau}{\mu}. 
\end{equation}
There is no Type I blowup solutions for the problem \eqref{eq:KS} in the sense of \eqref{def:TypeI}. Thus, our goal is to construct an unbounded global-in-time solution $w(z, \tau)$ for equation \eqref{eq:wztauIntro}. In particular, we construct a solution of the form 
$$w(z,\tau) = U_\nu(z) + \eta(z, \tau),$$
where $\nu(\tau)$ is the main parameter function in our analysis which drives the law of blowup, and $\eta$ solves the \emph{linearized equation in the parabolic zone}
\begin{equation}\label{eq:etaztauIntro}
\partial_\tau \eta = \mathscr{L}^z \eta +\left(\frac{\nu_\tau}{\nu}-\beta\right)\Lambda U_\nu- \nabla \cdot \Big(\eta \Phi_\eta  \Big) +\frac{x^*_\tau}{\mu}.\nabla \left(U_\nu+\eta\right)
\end{equation}
Here, $\mathscr{L}^z$ is the linearized operator defined by \eqref{def:Lsz} . Our aim is then reduced to construct for equation \eqref{eq:etaztauIntro} a global-in-time solution $\eta(z, \tau)$ satisfying \eqref{est:theo1conver}. \\

\noindent \underline{\textit{Properties of the linearized operator:}} We expect that only the first two terms contribute to leading order in the right hand side of \eqref{eq:etaztauIntro}. In the radial setting, studying the operator $\mathscr{L}^z$ is equivalent to studying $\mathscr{A}^\zeta$, the linearised operator around $Q_\nu = m_{U_\nu}$ in the partial mass setting  defined by \eqref{def:AsAs0zeta}. Indeed, we have the relation 
\begin{equation}\label{eq:relLszAszeta}
\mathscr{L}^z f(\zeta) = \frac{1}{\zeta} \partial_{\zeta} \Big(\mathscr{A}^\zeta m_f(\zeta) \Big).
\end{equation}
In the regime $0 < \beta \nu^2 \ll 1$, we proved in \cite{CGNNarx19a} that $\mathscr{A}^\zeta$ is self-adjoint in $L^2_{\frac{\omega_\nu}{\zeta}}$ with compact resolvant (see Proposition \ref{prop:SpecRad} for a precise statement), its spectrum being:
\begin{equation}\label{eq:spectrumAIntro}
\textup{spec}(\mathscr{A}^\zeta) = \Big\{\alpha_n = 2\beta\big(1 - n + \frac{1}{2\ln \nu} + \bar \alpha_n\big), \;  \bar \alpha_n  = \mathcal{O}\left( \frac{1}{|\ln \nu|^2}\right), \; n \in \mathbb{N} \Big\}.
\end{equation}  
(a refinement of $\bar \alpha_0$ and $\bar \alpha_1$ up to an accuracy of order $1/|\ln \nu|^2$ is needed to derive a precise blowup rate). The eigenfunction $\phi_{n,\nu}$ of $\mathscr{A}^\zeta$ corresponding to the eigenvalue $\alpha_n$ is explicit to leading order, of the form 
\begin{equation}\label{def:eigenfunctionIntro}
\phi_{n, \nu}(\zeta) = \sum_{j = 0}^n c_{n,j}\beta^j \nu^{2j-2}T_{j}\big(\frac{\zeta}{\nu} \big) + \tilde{\phi}_{n, \nu}, \quad c_{n,j} = 2^j \frac{n!}{(n - j)!}.
\end{equation}
Above, $\tilde{\phi}_{n, \nu}$ is of smaller order, and $(T_j)_{j\in \mathbb N}$ are defined by $T_{j + 1} = -\mathscr{A}_0^{-1}T_{j}$ and $T_{0}(r) = \frac{r^2}{(1 + r^2)^2}=m_{\Lambda U}$. The resonance of $\mathscr{A}_0$ is $T_0$: $\mathscr{A}_0 T_0 = 0$, so they generate the generalised kernel of $\mathscr{A}_0$. They admit the following asymptotic at infinity
\begin{equation}\label{est:Tj}
\textup{for} \;\; j \geq 1, \quad T_j(r) \sim \hat d_j r^{2 j - 2} \ln r \quad \textup{with} \quad \hat d_{j+1} = -\frac{\hat d_j}{4j(j+1)}, \quad \hat d_1 = -\frac{1}{2}.
\end{equation}
Moreover, the following spectral gap estimate holds: for $g\in L^2_\frac{\omega_\nu}{\zeta}$ in the domain of $\mathcal A^\zeta$ with $g\perp \phi_{n,\nu}$ in $L^2_\frac{\omega_\nu}{\zeta}$ for $0 \leq j \leq N$, 
\begin{equation}\label{est:spectralgapIntro}
\big\langle g, \mathscr{A}^\zeta g \big\rangle_{L^2_\frac{\omega_\nu}{\zeta}} \leq \alpha_{N+1}\big \|g\big\|_{L^2_\frac{\omega_\nu}{\zeta}}. 
\end{equation}
On the non-radial sector, we also prove in \cite{CGNNarx19a} that the slightly modified linear operator $\tilde {\mathscr{L}}$ defined by 
$$
\tilde {\mathscr{L}} u =\Delta u-\nabla  \cdot (u \nabla \Phi_{U})-\nabla \cdot (U\nabla \tilde \Phi_u)- b \nabla \cdot (yu), \quad \tilde \Phi_u= \frac{1}{\sqrt{\rho}}(-\Delta)^{-1}(u \sqrt{\rho}),
$$
is coercive for the following well-adapted scalar product
\begin{equation}\label{def:scatsIntro}
\langle u,v\rangle _\ast = \int_{\mathbb{R}^2} u\sqrt{\rho} \mathscr{M}(v \sqrt{\rho}) dy \quad \textup{with} \quad \mathscr{M} u = \frac{u}{U} - \Phi_u
\end{equation}
(equivalent in norm to $L^2_{\omega_0}$ under suitable orthogonality conditions). We show that for $u$ without radial component with $\int_{\mathbb{R}^2} u \partial_i U \sqrt{\rho} dy = 0$ for $i = 1,2$:
\begin{equation}\label{est:coerLsIntro}
\big\langle u,\tilde {\mathscr{L}} u\big\rangle_\ast \leq -\delta_0 \|\nabla u\|^2_{L^2_\omega}
\end{equation}
for a constant $\delta_0 > 0$. The advantage of this coercivity is that the scaling term $b\nabla \cdot(y \cdot)$ is taken into account, which greatly simplifies our analysis for the nonradial part. Note that controling the scaling term is one of the difficulties in the analysis performed in \cite{RSma14} where the renormalized operators of $\mathscr{L}_0, \mathscr{M}$ and the dissipation structure of the problem together with a sharp control of tails at infinity play a crucial role in their analysis. A similar situation happens in many other critical blowup problems, see for example \cite{MMRam14}, \cite{MMRjems15}, \cite{MRRim13}, \cite{RScpam13}, \cite{RSapde2014}, \cite{RRmihes12}, etc. \\

\noindent \underline{\textit{Approximate solution and a formal derivation of the blowup rate:}} Let $N \in \mathbb{N}$ with $N \gg 1$, and consider the approximate solution to \eqref{eq:wztauIntro} of the form 
\begin{align*}
W=W[\nu, \mathbf{a},\beta ](z,\tau) &= U_\nu(\zeta) +a_1(\tau)\big(\varphi_{1,\nu}(\zeta)-\varphi_{0, \nu}(\zeta)\big)+ \sum_{n = 2}^N a_n(\tau)\varphi_{n, \nu}(\zeta),
\end{align*}
where  $\nu(\tau)$, $\mathbf{a}(\tau) = \big(a_1, \cdots, a_n\big)(\tau)$ and $\beta(\tau)$ are parameters to be determined, and the $\varphi_{n, \nu}$'s are the radial eigenfunctions of $\mathscr{L}^z$. In the partial mass setting this gives
\begin{equation}\label{exp:mWIntro}
m_W[\nu, \mathbf{a},\beta ](\zeta, \tau) = Q_\nu(\zeta) + a_1(\tau)(\phi_{1, \nu}(\zeta)-\phi_{0, \nu}(\zeta)\big)+ \sum_{n = 2}^N a_n(\tau)  \phi_{n, \nu}(\zeta).
\end{equation}
Here, the term $a_1(\phi_{1,\nu} - \phi_{0,\nu})$ is the main perturbation term driving the law of blowup speed; and the term $\sum_{n = 2}^N a_n \phi_{n, \nu}$ is a higher order perturbation added to produce a big constant in the spectral gap \eqref{est:spectralgapIntro}, which is used to close the $L^2_\frac{\omega_\nu}{\zeta}$ estimate for the radial part of the remainder. The generated error  $m_E$ (defined by \eqref{exp:mE}) from the approximate solution \eqref{exp:mWIntro} is of size 
\begin{equation}\label{est:errorIntro}
\|m_E\|_{L^2_\frac{\omega_\nu}{\zeta}} \lesssim \frac{\nu^2}{|\ln \nu|}.
\end{equation} 
Assuming temporarily that $W$ is an exact solution to \eqref{eq:KS}, after an appropriate projection of \eqref{eq:KS} onto $\phi_{n, \nu}$ for $n=0,...,N$, we end up with the dynamical system (see Lemma \ref{lemm:Mod} for more details):
$$
\left\{ \begin{array}{l l l }
\textup{Mod}_0&:=& \nu^2\Big(\frac{\nu_\tau}{\nu} -\beta \Big)+\beta \frac{a_1}{4}(\tilde \alpha_1-1-\tilde \alpha_0)= \mathcal{O}\left(\frac{\nu^2}{|\ln \nu|^3} \right),\\
\textup{Mod}_1&:=&a_{1,\tau}-\beta a_1 \left(\frac{1}{\ln \nu}+\frac{\ln 2 -\gamma -2-\ln\beta}{2|\ln \nu|^2}\right)+a_1\frac{\beta_\tau}{\beta}= \mathcal{O}\left(\frac{\nu^2}{|\ln \nu|^3} \right),\\
 \textup{Mod}_n&:=& a_{n, \tau} - 2\beta a_n \alpha_n = \mathcal{O}\left(\frac{\nu^2}{|\ln \nu|^3} \right) \quad \textup{for} \quad 2 \leq n \leq N.
\end{array}
\right.
$$
The first equation describes how the leading term in the perturbation forces the stationary state to shrink, and the second one how the leading term in the perturbation evolves. We now fix the parameter $\beta$ by setting
\begin{equation} \label{defbetaintro}
\frac{a_1}{4\nu^2} = -1 + \frac{1}{2\ln \nu} + \frac{\ln 2 - \gamma - 1 - \ln\beta}{4|\ln \nu|^2}.
\end{equation}
In this case, the first equation reduces to:
\begin{equation}\label{eq:ODEnuIntro}
\frac{\nu_\tau}{\nu} =\beta \Big[-\frac{1}{2|\ln \nu|} + \frac{\ln 2 - \gamma - 1 - \ln\beta}{4|\ln \nu|^2} \Big] + \mathcal{O}\left(\frac{1}{|\ln \nu|^3} \right), 
\end{equation} 
where $\gamma$ is the Euler constant appearing in the refinement of $\bar \alpha_0, \bar \alpha_1$ (see Proposition \ref{prop:SpecRad}). Solving this equation gives  
$$\nu(\tau) \sim A e^{-\sqrt{ \frac{\tau}{2}}} \quad \textup{with} \quad A = \sqrt{2/\beta_\infty}e^{-\frac{\gamma + 2}{2}},$$ 
from which $\tau \sim |\ln (T-t)|$ and $\mu(t) \sim \sqrt{2\beta_\infty(T-t)}$ where $T=\lim_{\tau \to \infty}t(\tau)$, and we derive the blowup rate as stated in Theorem \ref{theo:Stab}. The identity \eqref{defbetaintro} ensured $\nu \rightarrow 0$ as $t\rightarrow T$, hence is dynamical way to compute the blow-up time. Our derivation of the blowup rate is consistent with the formal analysis in \cite{HVma96} by means of matched asymptotic expansions. We would lile to emphasise the fact that the refinement of the first two eigenvalues up to an accuracy of order $\frac{1}{|\ln \nu|^2}$ is crucial in deriving the precise value of $A$ here. Note that the rigorous analysis in \cite{RSma14} could not give the value of $A$. \\

\noindent \underline{\textit{Decomposition of the solution and modulation equations:}} To produce a solution of the full nonlinear problem we decompose the solution as 
\begin{equation}\label{eq:decomwWeIntro}
w(z, \tau) = W[\nu, \mathbf{a},\beta](\zeta) + \varepsilon(z, \tau), \; m_w(\zeta, \tau) = m_W[\nu, \mathbf{a},\beta](\zeta) + m_\varepsilon(\zeta, \tau).
\end{equation}
The uniqueness of this decomposition is ensured by the orthogonality conditions
\begin{equation}\label{eq:orthoIntro}
\big\langle m_\varepsilon , \phi_{n, \nu} \big\rangle _{L^2_\frac{\omega_\nu}{\zeta}} = 0 \quad \textup{for} \;\; 0 \leq n \leq N, \quad \int_{\mathbb{R}^2} \varepsilon^\perp \nabla U_\nu \sqrt{\rho_0} dz = 0.
\end{equation}
The control of the radial part of $\varepsilon$ is done via the partial mass setting, i.e. $m_\varepsilon$, based on the spectral properties of the linear operator $\mathscr{A}^\zeta$, and the control of the nonradial part $\varepsilon^\perp$ of $\varepsilon$ is based on the coercivity estimate \eqref{est:coerLsIntro}. Here, $m_\varepsilon$ and $\varepsilon^\perp$ solve the equations (where $P_\nu=a_1 (\varphi_{1,\nu}-\varphi_{0, \nu})+ \sum_{n = 2}^N a_n \varphi_{n, \nu}$):
\begin{align}
\partial_\tau m_\varepsilon = \mathscr{A}^\zeta m_\varepsilon + \frac{\partial_\zeta ((2P_\nu + m_\varepsilon ) m_\varepsilon )}{2\zeta}  + m_E + N_0(\varepsilon^\perp),\label{eq:meIntro}\\
\partial_\tau \varepsilon^\perp = \mathscr{L}^z \varepsilon^\perp - \nabla \cdot \mathcal{G}(\varepsilon^\perp) + \frac{x_\tau^*}{\mu}\cdot \nabla (W + \varepsilon^0) + N^\perp(\varepsilon^\perp),\label{eq:eperpIntro}
\end{align}
where $m_E$ is the generated error estimated in \eqref{est:errorIntro}, $\nabla \cdot \mathcal{G}(\varepsilon^\perp)$ contains small linear terms, and $N_0$ and $N^\perp$ stand for higher order quadratic nonlinear terms corresponding to the projections on radial and nonradial modes. After projecting the above equations on suitable directions, we arrive at the full modulation equations (see Lemma \ref{lemm:Mod} for complete expressions)
\begin{align}\label{est:ModIntro}
|\textup{Mod}_0| + |\textup{Mod}_1| &= \mathcal{O} \Big( \frac{1}{|\ln \nu|^2} \|m_\varepsilon \|_{L^2_\frac{\omega_\nu}{\zeta}} + \frac{\nu^2}{|\ln \nu|^3} \Big), \qquad |\textup{Mod}_n| & = \mathcal{O} \Big( \frac{1}{|\ln \nu|} \|m_\varepsilon \|_{L^2_\frac{\omega_\nu}{\zeta}} + \frac{\nu^2}{|\ln \nu|^2} \Big),
\end{align}
and (see \eqref{eq:equive0} and Lemma \ref{lemm:xstar} for more details)
\begin{equation}\label{est:xtauIntro}
 \left| \frac{x_\tau^*}{\mu}\right| \lesssim  \|\varepsilon^\perp\|_{L^2_{\omega_\nu}}.
\end{equation}

\noindent \underline{\textit{Control of the remainder:}} In view of \eqref{est:ModIntro}, the main quantity we need to control is the $L^2_\frac{\omega_\nu}{\zeta}$-norm of $m_\varepsilon $, which is of size (based on the error generated by the approximate solution, see \eqref{est:errorIntro})
\begin{equation}\label{est:L2omeIntro}
\|m_\varepsilon \|_{L^2_\frac{\omega_\nu}{\zeta}} \lesssim \frac{\nu^2}{|\ln \nu|},
\end{equation}
so that the leading order dynamical system \eqref{eq:ODEnuIntro} driving the law of blowup still holds up to an accuracy of order $\frac{1}{|\ln \nu|^3}$. At the linear lever, i.e. without taking into account the nonlinear term in \eqref{eq:meIntro}, it's simple to achieve \eqref{est:L2omeIntro} thanks to the spectral gap estimate \eqref{est:spectralgapIntro}. However, the only spectral gap \eqref{est:spectralgapIntro} is not enough to control nonlinear terms directly and to close the estimate \eqref{est:L2omeIntro}. Indeed, the perturbation $\varepsilon$ can be large near the origin, and the sole $L^2_\omega$ orthogonality conditions for it do not allow for dissipation type estimates there. Our idea is to put back certain nonlinear terms in the linearised operator and to prove that the spectral gap \eqref{est:spectralgapIntro} still holds true for this perturbation. We slightly modify the decomposition \eqref{eq:decomwWeIntro} and extract the leading order part of $\varepsilon$ near the origin:
\begin{equation}\label{eq:decom2Intro}
m_w(\zeta, \tau) = Q_{\tilde \nu}(\zeta) + a_1\big(\phi_{1, \tilde \nu}(\zeta) - \phi_{0, \tilde \nu}(\zeta)\big) + \tilde{m}_w,
\end{equation}
where we introduce the new parameter function $\tilde \nu \sim \nu$ (see Lemma \ref{lemm:nutil}) to impose the orthogonality condition localised at the scale of the stationary state
\begin{equation}\label{eq:OrtIntro}
\int_0^{+\infty}\tilde m_v \chi_{_M}T_0 \frac{\omega_0}{r} dr = 0 \quad \textup{with}\quad\tilde m_w(\zeta) =\tilde m_v(\zeta/\tilde \nu),
\end{equation}
where $M \gg 1$ is a fixed constant and the $\phi_{n,\tilde \nu}$'s are the eigenfunctions of the linearized operator $\tilde{\mathscr{A}}^\zeta$ around $Q_{\tilde \nu}$, defined as in \eqref{def:eigenfunctionIntro} with $\nu$ replaced by $\tilde \nu$. The orthogonality condition \eqref{eq:OrtIntro} allows to derive the coercivity of $\mathscr{A}_0$ (see Lemma \ref{lemm:coerA0}). This coercivity together with the dissipation structure of the problem yield the control of $\tilde m_v$ and its derivatives near the origin in the parabolic zone $|z|\lesssim 1$ (or $r \ll \frac{1}{\nu}$) (see \eqref{def:norm mvin} and Lemma \ref{lemm:Innercontrol}), and we obtain a pointwise bound for $\tilde m _v$ (see \eqref{pointwisemerefined}). When using the decomposition \eqref{eq:decom2Intro}, the linear operator $\mathscr{A}^\zeta$ is changed into (see the beginning of Section \ref{sec:mainEneEst})
$$\bar {\mathscr{A}}^\zeta = \frac{\mathscr{A}^\zeta + \tilde{\mathscr{A}}^\zeta}{2}.$$
and the nonlinear terms can now be estimated directly. A remarkable fact is that $\bar {\mathscr{A}}^\zeta$ adds a perturbation to $\mathscr{A}^\zeta$ that avoids the resonance near the origin. As a consequence the spectral structure of $\bar{ \mathscr{A}}^\zeta$ remains the same and the spectral gap still holds true, see Lemma \ref{lemm:SpecAbar}. We finally arrive at
$$\frac{d}{d\tau}\|m_\varepsilon \|^2_{L^2_\frac{\omega_\nu}{\zeta}} \leq -\|m_\varepsilon \|^2_{L^2_\frac{\omega_\nu}{\zeta}} + C\frac{\nu^2}{|\ln \nu|^2},$$
from which  \eqref{est:L2omeIntro} then directly follows after an integration in time. 

The control of the nonradial part is greatly simplified thanks to the coercivity estimate \eqref{est:coerLsIntro}. To measure the size of $\varepsilon^\perp$, we use the well-adapted norm  related to \eqref{def:scatsIntro},
$$\|\varepsilon^\perp\|_0^2 = \nu^2 \int_{\mathbb{R}^2} \varepsilon^\perp \sqrt{\rho_0} \mathscr{M}^z (\varepsilon^\perp \sqrt{\rho_0}) dz \sim \| \varepsilon^\perp\|^2_{L^2_{\omega_\nu}}.$$
In particular, we establish the following monotonicity formula (see Lemma \ref{lemm:ebot0}) 
$$\frac{d}{d\tau}\| \varepsilon^\perp\|^2_0 \leq -\delta'\| \varepsilon^\perp\|^2_0 + Ce^{-2\kappa \tau} \quad \textup{for some $0 < \kappa \ll 1$}, $$
for some constant $\delta'>0$, which gives $\| \varepsilon^\perp\|_{L^2_{\omega_\nu}} \lesssim e^{-\kappa \tau}$ after an integration in time. \\

\noindent A control on additional higher order regularity norms on the solution is also required to close the remaining nonlinear terms. We use parabolic regularity to obtain from our key decay in $L^2_\omega$ decays for higher order derivatives. This is done outside the blow-up zone, where the exponentially decaying weight $\omega$ cannot control the solution. In this zone however the renormalised solution is close to zero and the analysis boils down to the stability of the zero solution subject to small boundary terms. This is also done near the origin as explained previously, where the weight $\omega$ does not control the solution at scale $\nu$. In this zone the renormalised solution is close to the stationary state $U_\nu$, and the scaling term in the dynamics is negligible. We then control the perturbation via suitable coercivity estimates, as the boundary terms coming from the parabolic zone are already controlled.\\

\noindent The rest of paper is organized as follows: In Section \ref{sec:Reno}, we formulate the problem and recall key properties of the linearized operator. Section \ref{sec:StableTheo1} is devoted to the proof of Theorem \ref{theo:Stab} assuming technical details which are left to Section \ref{sec:ControlStab}. In Section \ref{sec:UnstabTheo2}, we sketch the proof of Theorem \ref{theo:UnStab}.

\section{Linear analysis in the parabolic zone} \label{sec:Reno}

\subsection{Parabolic variables and renormalisation}
Given $(\mu,x^*) \in C^1([0,T),(0,\infty)\times \mathbb R^2)$, we introduce the parabolic variables
\begin{align} \label{selfsimvar}
\quad z = \frac{x - x^*(t)}{\mu}, \quad \frac{d\tau }{dt} = \frac{1}{\mu^2}, \ \ \tau(0)=\tau_0,
\end{align}
and the corresponding renormalisation
$$
u(x,t) = \frac{1}{\mu^2(t)}w(z,\tau), \quad \Phi_u(x,t) = \Phi_w(z,\tau).
$$
The renormalisation rate is encoded by the following parameter
\begin{equation} \label{def:beta}
\beta =-\frac{\mu_\tau}{\mu}.
\end{equation}
The variables \eqref{selfsimvar} are indeed parabolic ones as we will have, once translating back to original variables, that $\beta\rightarrow \beta_{\infty}>0$ and $\mu(\tau)\sim \sqrt{T-t}$ from \eqref{reinte:exprmu} for some blow-up time $T>0$. The problem \eqref{eq:KSst} is transformed into the new equation
\begin{align}\label{eq:wztau}
\partial_\tau w  = \nabla \cdot \big(\nabla w - w \nabla \Phi_w\big) -\beta \nabla \cdot(zw) +\frac{x^*_\tau}{\mu} \cdot \nabla w.
\end{align}
In the partial mass setting, and in the parabolic variables \eqref{selfsimvar}, that is, introducing
\begin{equation}\label{def:mwPM}
m_w(\zeta)=\frac{1}{2\pi} \int_{B(0,\zeta)} w(z)dz, \ \ \zeta=|z|,
\end{equation}
the corresponding equation reads as 
\begin{equation} \label{partialmassselfsim}
\pa_\tau m_{w}=\pa_{\zeta}^2m_w-\frac{1}{\zeta}\pa_\zeta m_w+\frac{\pa_\zeta (m_w^2)}{2\zeta}-\beta \zeta\pa_\zeta m_w + N_0(w^\perp),
\end{equation}
where for $S(0,\zeta)$ the sphere at the origin with radius $\zeta$,
\begin{equation} \label{def:N0perp}
N_0(w^\perp) = -\frac{1}{2\pi}\int_{S(0,\zeta)}w^{\perp} \Big(\nabla \Phi_{w^{\perp}} -  \frac{x^*_{\tau}}{\mu} \Big) \cdot \vec n dS,
\end{equation}
with $w^\bot (z)=w(z)-(2\pi \zeta)^{-1}\int_{S(0,\zeta)}wdS $  the nonradial part of $w$.

\subsection{Spectral analysis and coercivity for the linearised operator}

Linearising equation \eqref{partialmassselfsim} around the rescaled soliton $Q_\nu$ (see \eqref{def:AsAs0zeta}) leads to the study of the linearized operator $\As^\zeta$ whose spectrum has been studied in details in \cite{CGNNarx19a} via matched asymptotic expansions. Note that in the radial setting, studying the linear operator $\Ls^z$  is equivalent to studying $\As^\zeta$ through the relation \eqref{eq:relLszAszeta}. In particular, if $\phi_{n, \nu}$ is an eigenfunction of $\As^\zeta$, then $\pa_\zeta \phi_{n, \nu}/\zeta$ is a radial eigenfunction of $\Ls^z$. For the reader's convenience, we recall from Proposition 1.1 of \cite{CGNNarx19a} the spectral properties of the operator $\As^\zeta$.

\begin{proposition}[Spectral properties of $\As^\zeta$, \cite{CGNNarx19a}] \label{prop:SpecRad} The linear operator $\As^\zeta : H^2_{\frac{\omega_\nu}{\zeta}} \to L^2_{\frac{\omega_\nu}{\zeta}}$ is essentially self-adjoint with compact resolvant. Moreover, given any $N\in \mathbb N$, $0 < \beta_* < \beta^*$ and $0<\delta \ll 1$, then there exists $\nu^*>0$ such that the following holds for all $0<\nu \leq \nu^*$ and $\beta_* \leq \beta \leq \beta^*$. \\
\noindent $(i)$ \textup{(Eigenvalues)} The first $N + 1$ eigenvalues $\alpha_{0},...,\alpha_N $ are given by
\begin{equation}\label{def:specAsb}
\alpha_{n} = 2\beta \Big(1 - n   + \frac{1}{2\ln \nu} + \tilde{\alpha}_{n}\Big),\;\; \tilde \alpha_{n} =   O\left(\frac{1}{|\ln \nu|^2}\right).
\end{equation}
Moreover, we have the refined estimate with $\gamma$ the Euler constant:
\begin{equation}\label{est:nu0ntil1}
\tilde{\alpha}_n = \frac{\ln 2 - \gamma - n - \ln \beta}{4|\ln \nu|^2} + \Oc\left(\frac{1}{|\ln \nu|^3} \right) \; \textup{for}\; n = 0,1.
\end{equation}
\noindent $(ii)$ \textup{(Eigenfunctions)} There exist eigenfunctions $\phi_{n, \nu}$ given by \eqref{def:hatphi} with:
\begin{equation}\label{est:PhinL2norm1}
 \langle \phi_{n, \nu}, \phi_{m, \nu} \rangle^2_{L^2_{\frac{\omega_\nu}{\zeta}}} = c_{n}\delta_{m,n}, \qquad \qquad \mbox{for }n,m\leq N.
\end{equation}
where, for some $c>0$,
$$ c_0\sim \frac{|\ln \nu|}{8}, \quad c_1\sim \frac{|\ln \nu|^2}{4}, \quad c|\ln \nu|^2\leq c_n \leq  \frac 1c |\ln \nu|^2,$$
We also have the pointwise estimates for $ k = 0,1,2$,
\begin{align}\label{est:PhinPointEst}
\left|D^k \phi_{n, \nu}(\zeta)\right|& +\left|D^k \beta\partial_\beta \phi_{n, \nu}(\zeta) \right| +\left|D^k \nu\partial_\nu \phi_{n, \nu}(\zeta) \right|  \lesssim \left(\frac{\zeta}{\nu+\zeta}\right)^{2-k\wedge 2}\frac{ \langle\zeta\rangle^{2n - 2+\delta} \big(1 +  \zeta^2\ln \big\la\frac{\zeta}{\nu} \big\ra\, \mathbf{1}_{\{n \geq 1\}} \big)}{(\zeta + \nu)^{2 + k}}.
\end{align}

\noindent $(iii)$ \textup{(Spectral gap estimate)} For any $g \in L^2_{\frac{\omega_\nu}{\zeta}}$ belonging to the domain of $\As^{\zeta}$ with $\langle g, \phi_{j, \nu}\rangle_{L^2_{\frac{\omega_\nu}{\zeta}}} = 0$ for $0 \leq j \leq N$, one has
\begin{equation}\label{est:SpecGap1}
\int_0^{\infty} g(\zeta) \As^\zeta g(\zeta) \frac{\omega_\nu (\zeta)}{\zeta}d\zeta   \leq  \alpha_{N+1}\int_0^{\infty} g^2(\zeta)  \frac{\omega_\nu (\zeta)}{\zeta}d\zeta  .
\end{equation}
\end{proposition}

\medskip
 
Due to the degeneracy of the problem, one has to track precise information of the eigenfunctions, especially for the first two ones. By construction, the eigenfunctions are obtained through the following approximation
\begin{equation}\label{def:hatphi}
 \phi_{n, \nu}(\zeta)= \sum_{j = 0}^n c_{n,j} \beta^j \nu^{2j-2} T_j\big(\frac{\zeta}{\nu}\big)+ \tilde{ \phi}_{n,\nu}(\zeta) ,
\end{equation}
where for $\As_0$ the linearised operator around the stationary state $Q$ introduced in \eqref{def:AsAs0},
\begin{equation} \label{def:ProTj}
T_j=(-1)^j (\As_0)^{-j}T_0, \ \ T_0(r)=\frac 18 r\pa_r Q(r), \ \ c_{n,j} = 2^j \frac{n!}{(n - j)!},
\end{equation}
and for $\psi_0(r) = \frac{r^2}{\rj^4}$ and $\tilde{\psi}_0(r) = \frac{r^4 + 4r^2\ln r - 1}{\rj^4}$ the inverse $\As_0^{-1} $ is given by
\begin{align}\label{def:invAs}
\As_0^{-1} f(r) := \frac{1}{2} \psi_0(r) \int_r^{1} \frac{\zeta^4 + 4\zeta^2 \ln \zeta - 1}{\zeta} f(\zeta) d\zeta + \frac{1}{2}\tilde\psi_0(r)\int_0^r \zeta f(\zeta) d\zeta,
\end{align}
and $\tilde{\phi}_{n,\nu}$ is a smaller order remainder described in the following.
\begin{lemma}[Eigenfunctions of $\As^\zeta$, \cite{CGNNarx19a}] \label{lemm:EigAzeta} Under the hypotheses of Proposition \ref{prop:SpecRad}, one has the following identities and upper bounds for $k = 0, 1,2$.\\
(i) \textup{(Estimates for $\phi_{0,\nu}$)}
\begin{align} \label{est:poinwisephi02}
|D^k \nu \pa_\nu \phi_{0,\nu} &|\lesssim \left(\frac{\zeta}{\nu+\zeta}\right)^{2-k\wedge 2} \frac{\langle \zeta\rangle^2}{(\nu+\zeta)^{2+k}}, \\
|D^k \beta \pa_\beta \phi_{0,\nu} | &\lesssim \left(\frac{\zeta}{\nu+\zeta}\right)^{2-k\wedge 2} \frac{\langle \zeta\rangle^2}{(\nu+\zeta)^{k}},\label{est:poinwisephi021}
\end{align}
and 
\begin{equation} \label{pointwisephi0}
|D^k\tilde{ \phi}_{0,\nu}(\zeta)|\lesssim  \frac{1}{(\nu+\zeta)^k} \left( \mathbf{1}_{\{ \zeta \lesssim \nu\}} +  \frac{|\ln \zeta|}{|\ln \nu|} \mathbf{1}_{\{\zeta \gtrsim \nu\}}\right).
\end{equation}
(ii) \textup{(Estimates for $\phi_{1,\nu}$)} Firstly,
\begin{equation}\label{est:phi1til}
|D^k\tilde{ \phi}_{1,\nu}(\zeta)|\lesssim  \left(\frac{\zeta}{\nu+\zeta}\right)^{2-k\wedge 2} \frac{ \langle\zeta\rangle^{4} }{|\ln \nu|(\zeta + \nu)^{k}},
\end{equation}
and, secondly,
$$\nu \pa_\nu  \phi_{1,\nu}=-2\beta (r\pa_rT_1)\left(\frac{\zeta}{\nu}\right)+R_1, \quad  \beta \pa_\beta \phi_{1,\nu}=\phi_{1,\nu}-\phi_{0,\nu}+R_2,$$
where
\begin{equation} \label{est:poinwisephi1}
|D^k R_1(\zeta)|\lesssim  \left(\frac{\zeta}{\nu+\zeta}\right)^{2-k\wedge 2} \frac{ \langle\zeta\rangle^{4} }{|\ln \nu|(\zeta + \nu)^{k}},
\end{equation}
\begin{equation} \label{est:poinwisephi12}
|D^k R_2(\zeta)|\lesssim  \left(\frac{\zeta}{\nu+\zeta}\right)^{2-k\wedge 2} \frac{ \langle\zeta\rangle^{4} }{(\zeta + \nu)^{k}}.
\end{equation}
(iii) \textup{(Cancellation near the origin)} For all $0\leq n\leq N$:
\begin{equation} \label{eq:pointwisephin-phi0}
|D^k \left(\phi_{n,\nu}- \phi_{0,\nu}\right)|  \lesssim  \min \left(\nu^2 \left\la \frac \zeta \nu\right\ra^2, \frac{1}{\ln \nu} \right) \left(\frac{\zeta}{\nu+\zeta}\right)^{2-k\wedge 2} \frac{ \langle\zeta\rangle^{2n + 2} }{(\zeta + \nu)^{2+ k}}.
\end{equation}
\end{lemma}

\bigskip

We also prove in \cite{CGNNarx19a} a coercivity property of the linearized operator $\Ls$ acting on purely nonradial functions avoiding the direction $\nabla U$. Before stating it, note that on the one hand, in the first decomposition for $\Ls$ \eqref{def:LsLs0Ms}, the operator $\Ls_0$ is self-adjoint in $L^2(\Rb^2)$ endowed with the inner product 
\begin{equation} \label{def:premierscalar}
\la u,v \ra_{_\Ms} = \int_{\Rb^2} u \Ms v\, dy, \quad  \|u\|_{_\Ms}^2 =  \la u, u \ra_{_\Ms}.
\end{equation}
The positivity of the norm $ \|\cdot\|_{_\Ms}$ is subject to some suitable orthogonality conditions as showed in the following lemma.
\begin{lemma}[Coercivity of $\Ms$, \cite{RSma14}, \cite{CGNNarx19a}] \label{lemm:coerML0} Let $u$ be such that $\displaystyle\int_{\Rb^2} \frac{|u|^2}{U}dy<+\infty$ and $\int_{|x|=r} u= 0$ for almost every $r>0$. Then, we have
\begin{equation}\label{eq:posM}
\int_{\Rb^2} u \Ms u dy \geq 0,
\end{equation}
and there exist $\delta_1, \delta_2, C > 0$ such that
\begin{equation} \label{bd:coercivite L2}
\int_{\Rb^2} u\Ms u dy \geq \delta_1\int_{\Rb^2} \frac{u^2}{U}dy - C\Big[ \la u, \Lambda U\ra_{L^2}^2  + \la u, \pa_1 U\ra_{L^2}^2 + \la u, \pa_2 U\ra_{L^2}^2 \Big].
\end{equation}
If $u$ additionally satisfies  $\displaystyle\int_{\Rb^2} \frac{|\nabla u|^2}{U}dy<+\infty$ , then one has
\begin{equation} \label{bd:coercivite H1}
\int_{\Rb^2} U |\nabla (\Ms u)|^2 dy \geq  \delta_2\int_{\Rb^2} \frac{|\nabla u|^2}{U} dy - C\Big[\la u, \pa_1 U\ra_{L^2}^2 + \la u, \pa_2 U\ra_{L^2}^2 \Big].
\end{equation}
\end{lemma}

On the other hand, $\Ls$ can be written the following way:
%\Ls u &= \Ls_0 u -b\nabla.(yu) \quad \textup{with} \quad \Ls_0 u = \nabla \cdot(U \nabla \Ms u) \;\; \textup{and} \;\; \Ms u = \frac{u}{U} - \Phi_u, \label{def:Lsform01}\\
\begin{align}
\Ls u &= \Hs u - \nabla U \cdot \nabla \Phi_u \quad \textup{with} \quad \Hs u = \frac{1}{\omega} \nabla \cdot \Big(\omega \nabla u\Big) + 2(U - b)u,\label{def:LsformH1}
\end{align}
with the weight function $\omega$ defined in \eqref{def:omegarho}. Not containing the nonlocal part, the operator $\Hs$ is self-adjoint in the weighted $L^2_{\omega}(\Rb^2)$ with inner product
\begin{equation}\label{def:L2rhonorm}
\la u, v\ra_{L^2_{\omega}} = \int_{\Rb^2} u v \omega(y) 
\, dy, \quad  \| u \|_{L^2_\omega}^2 = \la u, u \ra_{L^2_\omega}.
\end{equation}
For our problem at hand, we will be able to neglect the part of the nonlocal term in $\Ls$ located away from the origin. We thus introduce the following "mixed" scalar product, with a localised Poisson field, which matches the two previous ones \eqref{def:premierscalar} and \eqref{def:L2rhonorm} to leading order close and away from the origin respectively:
\begin{equation}\label{def:quadform2}
\langle u,v\rangle_\ast:= \int_{\Rb^2} u\sqrt{\rho}\Ms (v \sqrt{\rho})dy.
\end{equation}
To avoid the far away contribution for the Poisson field which is not under control in $L^2(\omega)$, we localize the Poisson field in the linearized operator accordingly:
$$\tilde \Phi_u= \frac{1}{\sqrt{\rho}}(-\Delta)^{-1}(u \sqrt{\rho}),$$
and consider the slightly modified operator
$$
\tilde \Ls u =\Delta u-\nabla  \cdot (u \nabla \Phi_{U})-\nabla \cdot (U\nabla \tilde \Phi_u)- b \nabla \cdot (yu).
$$
We claim that in the non-radial sector, the localised operator $\tilde \Ls$ is coercive for the mixed scalar product \eqref{def:L2rhonorm} under the natural orthogonality assumption to $\nabla U$.
\begin{proposition}[Coercivity of $\tilde \Ls$, \cite{CGNNarx19a}] \label{pr:coercivitenonradial}
There exists $\delta _0 ,C>0$ and $b^*>0$ such that for all $0<b \leq b^*$, if $u \in \dot{H}^1_{\omega}$ and satisfies $\int_{|y|=r}u(y)dy=0$ for almost every $r\in (0,\infty)$, then we have
\begin{equation}\label{est:coerLnuy}
\langle -\tilde \Ls u,u\rangle_\ast \geq \delta_0 \| \nabla u \|_{L^2_\omega}^2 - C \sum_{i = 1}^2 \left(\int_{\Rb^2} u\partial_{i}U\sqrt{\rho}dy\right)^2.
\end{equation}
\end{proposition}

\begin{remark} Proposition \ref{pr:coercivitenonradial} concerns $\tilde{\Ls}$ instead of $\Ls$. However, the difference
$$(\Ls - \tilde{\Ls}) u = \nabla U \cdot \nabla (\tilde{\Phi}_u - \Phi_u),$$
will be controlled by dissipative effects and the fast decay of the stationary solution $U$ together with an appropriate outer norm defined in Definition \ref{def:bootstrap}. 
\end{remark}

\section{Stable blowup dynamics} \label{sec:StableTheo1}
\subsection{Inner variables and renormalisations}

The most important part of the analysis is done in parabolic variables \eqref{selfsimvar}. We will also use the inner variables
\begin{align} \label{blowupvar}
\quad y = \frac{z}{\nu}, \quad \frac{ds }{d\tau} = \frac{1}{\nu^2}, \ \ s(0)=s_0,
\end{align}
and the corresponding renormalisation:
\begin{align} \label{firstrenormalisationinner}
w(z,\tau) = \frac{1}{\nu^2(\tau)}v(y,s), \quad \Phi_w(z,\tau) = \Phi_v(y,s),
\end{align}
so that the problem \eqref{eq:wztau} is transformed into the further renormalised equation
\begin{align}\label{eq:vys}
\partial_s v  = \nabla \cdot \big(\nabla v - v \nabla \Phi_v\big) +\left(\frac{\nu_s}{\nu}-\nu^2\beta \right) \nabla \cdot(zv) +\frac{\nu x^*_\tau}{\mu} \cdot \nabla v .
\end{align}
\noindent  The parameter $\nu$ in \eqref{blowupvar} is fixed via the first orthogonality condition in \eqref{orthogonality}. This is the orthogonal projection in $L^2_{\frac {\omega_\nu} {\zeta}}$ of the solution $w$ onto the set of stationary states $(U_{\lambda})_{\lambda >0}$. Roughly speaking, this says that $w$ is close to $U_\nu$ in the parabolic zone. However, it will not be enough to show that in inner variables \eqref{firstrenormalisationinner}, $v$ is close to $U$, with a sufficient estimate in order to control the nonlinear terms.

To cope with this issue, we use a second decomposition involving a slightly rescaled stationary state at scale $\tilde \nu$, fixed by another orthogonality condition \eqref{eq:orthmwT0}. In particular, we introduce a modified parameter $\tnu$ with the associated modified inner variables for the partial mass \eqref{def:mwPM}
\begin{equation} \label{blowupvartilde}
m_v(\tr, s)= m_w(\zeta, \tau), \quad \tr = \frac{\zeta}{\tnu}, \quad \frac{ds}{d\tau} = \frac{1}{\tnu^2},
\end{equation}
where $m_v$ solves the new equation
\begin{equation} \label{eq:pamblowup}
\pa_s m_{v}=\pa_{\tr}^2m_v-\frac{1}{\tr}\pa_{\tr} m_v+\frac{\pa_{\tr} (m_v^2)}{2\tr}+\left(\frac{\tnu_s}{\tnu}-\tnu^2\beta \right)\tr\pa_{\tr} m_v + \tilde{N}_0(v^\perp),
\end{equation}
where 
$$\tilde{N}_0(v^\perp)= -\frac{1}{2\pi}\int_{S(0,\tr)}v^{\perp} \Big( \nabla \Phi_{v^{\perp}} - \frac{ x_s^*}{\mu \tilde{\nu}} \Big).\vec n dS.$$

\subsection{Ansatz}
\subsubsection{First decomposition in parabolic variables}
In parabolic variables \eqref{selfsimvar} we take an approximate solution to \eqref{eq:wztau} of the form
\begin{align}\label{def:WnuaPro}
W=W[\nu, \mathbf{a}, \beta ](z,\tau) &= U_\nu(\zeta) +a_1(\tau)(\varphi_{1,\nu}-\varphi_{0, \nu})+ \sum_{n = 2}^N a_n(\tau)\varphi_{n, \nu}(\zeta) \\
\nonumber & \equiv U_\nu(\zeta) + \Psi_{1,\nu} + \Psi_{2,\nu}.
\end{align}
where above, we recall that the rescaled stationary state is
$$U_\nu(\zeta)=\frac{1}{\nu^{2}}U\Big(\frac{\zeta}{\nu}\Big) = \frac{8\nu^2}{(\nu^2 + \zeta^2)^2},$$
$\nu(\tau)$, $\mathbf{a}(\tau) = \big(a_1, \cdots, a_n\big)(\tau)$ and $\beta(\tau)$ are unknown functions to be determined, $(\varphi_{n, \nu})_{0\leq n \leq N}$ are the radial eigenfunctions of $\Ls^z$ defined in \eqref{def:Lsz}, and the approximate perturbation is
\begin{equation}\label{def:Psi1Psi2N}
 \Psi_{1, \nu}(\tau, \zeta) = a_1(\tau)(\varphi_{1,\nu}-\varphi_{0, \nu}), \quad \Psi_{2,\nu}(\tau, \zeta) =  \sum_{n = 2}^N a_n(\tau)\varphi_{n, \nu}(\zeta).
\end{equation}
The approximate solution to \eqref{partialmassselfsim} in partial mass \eqref{def:mwPM} corresponding to \eqref{def:WnuaPro} is
\begin{equation}\label{exp:mW}
m_W[\nu, \mathbf{a}, \beta ](\zeta, \tau) = Q_\nu(\zeta) + a_1(\phi_{1, \nu}-\phi_{0, \nu})+ \sum_{n = 2}^N a_n(\tau)  \phi_{n, \nu}(\zeta) \equiv Q_\nu(\zeta) + P_\nu, 
\end{equation}
where the partial mass of the stationary state $Q_\nu(\zeta) = m_{U_\nu}(\zeta)$ is given by 
$$Q_\nu(\zeta) = Q\Big(\frac \zeta \nu\Big) = \frac{4\zeta^2}{\nu^2 + \zeta^2},$$
$(\phi_{n, \nu})_{0\leq n \leq N}$ are the renormalised eigenfunctions of $\As^\zeta$ defined as in \eqref{def:eigenfunctionIntro}, and the partial mass of the approximate perturbation is
\begin{equation}\label{def:P1P2N}
P_\nu = P_{1, \nu} + P_{2,\nu}, \quad P_{1, \nu} = a_1(\phi_{1, \nu}-\phi_{0, \nu}), \quad  P_{2,\nu} = \sum_{n = 2}^N a_n(\tau)  \phi_{n, \nu}(\zeta).
\end{equation}
The full solutions to \eqref{eq:wztau} and \eqref{partialmassselfsim} are then decomposed as
\begin{equation} \label{decomposition}
w=W+\e,  \qquad m_w = m_W + m_\e = Q_\nu + P_{1, \nu} + \bar m_\e \quad \textup{with} \quad \bar m_\e = P_{2, \nu} + m_\e. 
\end{equation}
Decomposing the remainder $\e$ between radial part and nonradial parts (with $0$ or $\perp$ superscript respectively):
$$
\e=\e^{0}+\e^{\perp},  \ \ q=q^{0}+q^{\perp},
$$
the decomposition \eqref{decomposition} is ensured by the orthogonality conditions
\begin{equation} \label{orthogonality}
m_\e \perp_{L^2_{\frac {\omega_\nu} {\zeta}}} \phi_{n, \nu} \mbox{ for } 0\leq n \leq N, \ \ \int_{\Rb^2} q^\perp \partial_{1}U \sqrt{\rho} dy=\int_{\Rb^2} q^\perp \partial_{2}U \sqrt{\rho} dy=0.
\end{equation}
In the blow-up variables \eqref{blowupvar}, we will use the notation $r = |y|$, and
\begin{equation}\label{def:vVq}
v=V+q, \ \ V[\nu,\beta,\bold a](s,y)=U(y)+ \frac{a_1}{\nu^2}\big[\varphi_1(r)-\varphi_0(r)\big] + \sum_{n=2}^N \frac{a_n}{\nu^2}\varphi_n(r),
\end{equation}
where we have the relations for the remainders and eigenfunctions:
$$q(s,y)=\nu^2 \e (\tau,z), \quad m_q(s,r) = m_\e(\tau, \zeta), \quad \varphi_n(r) = \frac{\pa_r \phi_n(r)}{r}, \quad \pa_r\Phi_{\varphi_n}(r) = -\frac{\phi_n(r)}{r},$$
and:
$$
\phi_n(r)=\nu^2\phi_{n, \nu}(\zeta), \ \ \varphi_{n,\nu} (\zeta)=\frac{\pa_\zeta \phi_{n,\nu}(\zeta)}{\zeta}, \ \ \varphi_n(r)=\nu^4\varphi_{n, \nu}(z).
$$

\subsubsection{Second decomposition in parabolic variables and inner variables}

We then consider the decomposition in modified inner variables \eqref{blowupvartilde}
\begin{equation}\label{dec:vtr}
m_v(\tr, s) = Q(\tr) + \tilde P_1(\tr, s) - \tilde N_1(\tr, s) + \tm_v(\tr,s).
\end{equation}
Above, $\tilde P_1(\tr, s) - \tilde N_1(\tr, s)$ is the modified approximate perturbation where
\begin{equation}\label{def:P1nutil}
\tilde P_1(\tr, s) = P_{1, \tnu}(\zeta, \tau) = a_1(\tau)\big(\phi_{1, \tnu}(\zeta) - \phi_{0,\tnu}(\zeta) \big),
\end{equation}
with $\phi_{n, \tnu}$ being the eigenfunction, given by \eqref{def:hatphi} with $\nu$ replaced by $\tilde \nu$, of
$$
\tilde{\As}^\zeta = \tilde{\As}_0^\zeta - \beta \zeta \pa_\zeta, \qquad  \tilde{\As}_0^\zeta = \partial_\zeta^2  -\frac{1}{\zeta}\partial_\zeta + \frac{\partial_\zeta(Q_{\tilde \nu} \cdot)}{\zeta},
$$
which is the linearized operator around $Q_{\tnu}$, and, with $\As_0^{-1}$ given by \eqref{def:invAs},
\begin{align}\label{def:N1til}
\tilde{N}_1(\tr, s) &= \As_0^{-1}\left(\frac{\pa_{\tr}\tilde{P}_1^2}{2\tr} + 8\beta \tnu^2 \tilde \phi_0 \right) = \tnu^4\big(\tilde \As_0^\zeta \big)^{-1}\Big(\frac{\pa_\zeta P_{1,\tnu}^2}{2\zeta} + 8\beta \tilde \phi_{0, \tnu}(\zeta) \Big) := N_{1, \tnu}(\zeta, \tau), 
\end{align}
where $\tilde{\phi}_{0, \tnu}(\zeta) = \phi_{0, \tnu}(\zeta) - \tnu^{-2}T_0(\zeta/ \tnu)$ satisfies the pointwise bound \eqref{pointwisephi0}. The introduction of $\tilde{N}_1$ is just a technical issue for the control of the inner norm \eqref{bootstrap:in}. Roughly speaking, we want the source error term to be of size $\Oc\left(\frac{\nu^4}{|\ln \nu|}\right)$ for the norm \eqref{def:norm mvin}, however the error terms created by $\tilde P_1$ are only of size $\Oc(\nu^4)$ on compact sets, but with strong decay at infinity. The correction $\tilde{N}_1$ then precisely cancels these terms. 

In terms of $(\zeta, \tau)$-variables, the decomposition \eqref{dec:vtr} is equivalent to
\begin{equation}\label{eq:decmwt}
m_w = Q_{\tnu} + P_{1, \tnu} - N_{1, \tnu} + \tm_w, \quad \tm_w(\zeta, \tau) = \tm_v(\tr, s),
\end{equation}
where the modified remainder is
\begin{equation}\label{def:mwtil}
\tm_w = Q_\nu - Q_{\tnu} + P_{1, \nu} - P_{1, \tnu}  + N_{1, \tnu} +  \bar m_\e.
\end{equation} 
The parameter $\tilde \nu$ is fixed by the orthogonality condition in modified inner variables:
\begin{equation}\label{eq:orthmwT0}
\int_0^\infty \tilde m_v(\tr) \chi_{_M}(\tr) T_0(\tr) \frac{\omega_0(\tr)}{\tr} d\tilde r=0 \quad \textup{for some fixed constant} \; M \gg 1,
\end{equation}
with $\chi_{_M}$ being defined as in \eqref{def:chiM}. In particular, this orthogonality condition and the rough bound \eqref{bootstrap:L2omeofme} for $m_\e$ (see Lemma \ref{lemm:nutil} below) ensure $\nu$ and $\tilde \nu$ are close:
$$\big|\nu - \tnu\big| \lesssim \frac{\nu}{|\ln \nu|}.$$

\subsection{The error generated by the approximate solution}
We claim that  $W[\nu,\mathbf{a}, \beta ](z, \tau)$ introduced in \eqref{def:WnuaPro} is a good approximate profile to \eqref{eq:wztau} in the following sense. 

\begin{lemma}[Approximate profile]  \label{lemm:appProf}   Assume that $(\nu, \mathbf{a}, \beta, x^*)$ are $\Cc^1$ maps with
$$(\nu, \mathbf{a}, \beta , x^*): [\tau_0, \tau_1) \mapsto (0,\nu^*] \times (0, a^*]^N \times \Big( \frac 12 - \beta ^*, \frac 12 + \beta ^*\Big) \times \Rb^2,$$
for  $0 < \nu^*, a^*, \beta^* \ll 1$ and $1 \ll \tau_0 <\tau_1 \leq +\infty$, with a priori bounds: 
$$\left|\frac{\nu_\tau}{\nu} \right| \lesssim \frac{1}{|\ln \nu|}, \quad  |a_1| \lesssim \nu^2, \quad |a_n| \lesssim \frac{\nu^2}{|\ln \nu|^2} \;\; \textup{for}\;\; n \in \{2,\cdots,N\}, \quad \big|\beta_\tau \big| \lesssim \frac{1}{|\ln \nu|^3}.$$
Then the error generated by \eqref{def:WnuaPro} for \eqref{eq:wztau} is given by 
\begin{align*}
E(z, \tau) = -\partial_\tau W  + \nabla \cdot(\nabla W - W \nabla \Phi_W) - \beta  \nabla \cdot(z W) + \frac{x^*_{\tau}}{\mu} \cdot \nabla W \equiv \frac{\partial_\zeta m_E}{\zeta} +  \frac{x^*_{ \tau}}{\mu} \cdot \nabla W,
\end{align*}
where 
\begin{equation}\label{eq:mW}
m_E(\zeta, \tau) = -\partial_\tau m_W  - \beta  \zeta \partial_\zeta m_W + \zeta\partial_\zeta \big(\zeta^{-1} \partial_\zeta m_W\big) + \frac{\partial_\zeta m_W^2}{2\zeta},
\end{equation}
can be decomposed as
\begin{equation}\label{exp:mE}
m_E(\zeta, \tau) =  \sum_{n = 0}^N \textup{Mod}_n \;  \phi_{n, \nu}(\zeta) + \tilde{m}_E(\zeta, \tau)+\frac{ \pa_{\zeta}P_\nu ^2}{2\zeta},
\end{equation}
with 
\begin{align}
&\textup{Mod}_0 = \left(\frac{\nu_\tau}{\nu} - \beta \right)8\nu^2 + a_{1, \tau}  - a_1 2\beta(1 + \tilde{\alpha}_0), \label{def:Mod0} \\
&\textup{Mod}_n =  -  \big[a_{n,\tau} - 2\beta( 1 - n + \tilde{\alpha}_n) a_n\big],\label{def:Modn}
\end{align}
and
\begin{align}
\label{expr:mEphi0} &\langle \tilde{m}_E, \phi_{0, \nu} \rangle_{L^2_\frac{\omega_\nu}{\zeta}} = -\frac{a_1}{8}\left( \frac{\nu_\tau}{\nu} - \frac{\beta_\tau}{\beta} |\ln \nu| \right) + \Oc\left( \frac{\nu^2}{|\ln \nu|^2}\right),\\
\label{expr:mEphi1} &\langle  \tilde{m}_E, \phi_{1, \nu} \rangle_{L^2_\frac{\omega_\nu}{\zeta}} = \frac{a_1}{4} |\ln \nu|\left( \frac{\nu_\tau}{\nu} - \frac{\beta_\tau}{\beta} |\ln \nu| \right) + \Oc\left( \frac{\nu^2}{|\ln \nu|}\right),
\end{align}
\begin{align}
\label{expr:mEphi2} \|\tilde{m}_E\|^2_{L^2_{\frac{\omega_\nu}{\zeta}}} = \Oc\left(\frac{\nu^4}{|\ln \nu|^2}\right), \; \langle \tilde{m}_E, \phi_{n, \nu} \rangle_{L^2_\frac{\omega_\nu}{\zeta}} = \Oc(\nu^2) \;\; \textup{for}\;\; n \in \{2, \cdots, N\}.
\end{align}
\end{lemma}
\begin{proof} Plugging the expansion \eqref{exp:mW} into \eqref{eq:mW} yields 
$$m_E(\zeta, \tau) = \left(\frac{\nu_\tau}{\nu} - \beta\right) \zeta \pa_\zeta Q_\nu - \pa_\tau P_\nu + \As^\zeta P_\nu + \frac{\partial_\zeta P_\nu^2}{2\zeta}.$$
By noting that $\zeta \partial_\zeta Q_\nu \sim 8\nu^2 \phi_{0,\nu}$ and recalling from Proposition \ref{prop:SpecRad} that $\As^\zeta \phi_{n, \nu} = 2\beta(1 - n +\tilde{\alpha}_n) \phi_{n, \nu}$, we write 
\begin{align}\label{exp:mEpr}
m_E(\zeta, \tau)& = \left[\left(\frac{\nu_\tau}{\nu} - \beta \right)8\nu^2 + a_{1, \tau}  - a_1 2\beta(1 + \tilde{\alpha}_0)\right] \phi_{0, \nu}  - \sum_{n = 1}^N \big[ a_{n,\tau} - a_n 2\beta(1 - n +\tilde{\alpha}_n) \big] \; \phi_{n, \nu}  + \tilde{m}_E  + \frac{\partial_\zeta P_\nu^2}{2\zeta}, 
\end{align}
where the remainder is given by
\begin{eqnarray*}
\tilde{m}_E &  = &  -  a_1\Big[ \frac{\nu_\tau}{\nu} \nu \pa_\nu \big(\phi_{1, \nu} - \phi_{0, \nu}\big) + \frac{\beta_\tau}{\beta} \beta \pa_\beta \big(\phi_{1, \nu} - \phi_{0, \nu}\big) \Big]  -\sum_{n = 2}^N  \Big[ a_n \frac{\nu_\tau}{\nu} \nu \pa_\nu \phi_{n, \nu}  + \sum_{n= 2}^N a_n \frac{\beta_\tau}{\beta} \beta \pa_\beta   \phi_{n, \nu} \Big] \\
&&+ \left( \frac{\nu_\tau}{\nu} - \beta\right)\big(\zeta \partial_\zeta Q_\nu - 8\nu^2 \phi_{0, \nu}\big). 
\end{eqnarray*}
From the a priori assumptions, we remark that the leading order term in $\tilde{m}_E$ is $a_1 \frac{\nu_\tau}{\nu} \big[\nu \pa_\nu \big(\phi_{1, \nu} -  \phi_{0, \nu}\big) \big]$. In particular, we have from Lemma \ref{lemm:EigAzeta} the identity
$$\nu \pa_\nu \big(\phi_{1, \nu} -  \phi_{0, \nu}\big) = -2\beta (r\pa_r T_1)\big(\frac{\zeta}{\nu}\big) + \Oc\big( \frac{\langle \zeta \rangle}{|\ln \nu|}\big) = \beta + \Oc\big( \frac{\langle \zeta \ra}{|\ln \nu|}\big),$$
from which and the asymptotic behavior $T_1(r) = -\frac{1}{2}\ln r + \frac{1}{4} + \Oc\big( \frac{\ln r}{r^2} \big)$ for $r \gg 1$,  we compute asymptotically
\begin{align*}
& \la  \nu \pa_\nu \big(\phi_{1, \nu} -  \phi_{0, \nu}\big), \phi_{0, \nu} \ra _{L^2_{\frac{\omega_\nu}{\zeta}}} = \beta \int_0^{+\infty} \phi_{0, \nu}(\zeta) \zeta^{-1} \omega_\nu(\zeta) d\zeta + \Oc\left( \frac{1}{|\ln \nu|}\right)\\
& \qquad  \qquad  \qquad = \frac{\beta \nu^2}{8} \int_0^{+\infty}r e^{-\frac{\beta \nu^2 r^2}{2}}dr + \Oc\left( \frac{1}{|\ln \nu|}\right) = \frac{1}{8} + \Oc\left( \frac{1}{|\ln \nu|}\right),\\
& \la  \nu \pa_\nu \big(\phi_{1, \nu} - \phi_{0, \nu}\big), \phi_{1, \nu} \ra _{L^2_{\frac{\omega_\nu}{\zeta}}} = \beta \int_0^{+\infty}\phi_{1, \nu}(\zeta) \zeta^{-1} \omega_\nu(\zeta) d\zeta + \Oc(1)\\
 &\qquad  \qquad  \qquad  = -\frac{\beta^2 \nu^4 |\ln \nu|}{8} \int_0^{+\infty}r^3 e^{-\frac{\beta \nu^2 r^2}{2}}dr + \Oc(1) = -\frac{|\ln \nu|}{4} + \Oc(1),
\end{align*}
and for $n \geq 2$,
\begin{align*}
\la  \nu \pa_\nu \big(\phi_{1, \nu} -  \phi_{0, \nu}\big), \phi_{n, \nu} \ra _{L^2_{\frac{\omega_\nu}{\zeta}}} = \Oc(|\ln \nu|).
\end{align*}
From Lemma \ref{lemm:EigAzeta}, we have the relation
$$\beta\partial_\beta \big( \phi_{1, \nu} - \phi_{0, \nu}\big) = 2\beta T_1 \Big( \frac{\zeta}{\nu} \Big) + \Oc(\la \zeta \ra ^4),$$
from which and from the behavior of $T_1$, we compute 
\begin{align*}
& \la  \beta \pa_\beta \big(\phi_{1, \nu} -  \phi_{0, \nu}\big), \phi_{0, \nu} \ra _{L^2_{\frac{\omega_\nu}{\zeta}}} = 2\beta \int_0^{+\infty} T_1 \Big( \frac{\zeta}{\nu} \Big) \phi_{0, \nu}(\zeta) \zeta^{-1} \omega_\nu(\zeta) d\zeta + \Oc\left(1\right)\\
&\qquad \qquad = -\frac{\beta \nu^2}{8} \int_0^{+\infty} (\ln r)\, e^{-\frac{\beta \nu^2 r^2}{2}}rdr + \Oc\left(1\right) \ = \ -\frac{|\ln \nu|}{8} + \Oc\left(1\right),\\
&\la  \beta \pa_\beta \big(\phi_{1, \nu} -  \phi_{0, \nu}\big), \phi_{1, \nu} \ra _{L^2_{\frac{\omega_\nu}{\zeta}}} = 2\beta \int_0^{+\infty} T_1 \Big( \frac{\zeta}{\nu} \Big) \phi_{1, \nu}(\zeta) \zeta^{-1} \omega_\nu(\zeta) d\zeta + \Oc\left(1\right)\\
& \qquad \qquad  = \frac{\beta^2 \nu^4}{8} \int_0^{+\infty} (\ln r)^2 r^2 e^{-\frac{\beta \nu^2 r^2}{2}}rdr + \Oc\left(|\ln \nu|\right) \ = \ \frac{|\ln \nu|^2}{4} + \Oc\left(|\ln \nu|\right),
\end{align*}
and for $n \geq 2$,
$$\la  \beta \pa_\beta \big(\phi_{1, \nu} -  \phi_{0, \nu}\big),  \phi_{n, \nu} \ra _{L^2_{\frac{\omega_\nu}{\zeta}}} = \Oc(|\ln \nu|^2).$$
Using the a priori assumptions, \eqref{eq:pointwisephin-phi0} and \eqref{est:PhinPointEst}, we obtain the following rough estimates:
\begin{align*}
\left\|\sum_{n = 2}^N   a_n \frac{\nu_\tau}{\nu} \nu \pa_\nu \phi_{n, \nu} \right\|_{L^2_{\frac{\omega_\nu}{\zeta}}} \lesssim \sum_{n=2}^{N}\left|a_n \frac{\nu_\tau}{\nu} \right| |\ln \nu| \lesssim \frac{\nu^2}{|\ln \nu|^2},  \\
\left\|\sum_{n= 2}^N a_n \frac{\beta_\tau}{\beta} \beta \pa_\beta \phi_{n, \nu}  \right\|_{L^2_{\frac{ \omega_\nu}{\zeta}}} \lesssim \sum_{n= 2}^N \left|a_n \frac{\beta_\tau}{\beta}\right| |\ln \nu| \lesssim \frac{\nu^2}{|\ln \nu|^4},
\end{align*}
Using $\zeta\pa_\zeta Q_\nu(\zeta) = 8T_0(r)$ yields 
$$\left\|\big(\zeta \partial_\zeta Q_\nu - 8\nu^2 \phi_{0, \nu}\big) \right\|_{L^2_{\frac{\omega_\nu}{\zeta}}}  = \nu^2\| \tilde{\phi}_{0, \nu}\|_{L^2_{\frac{\omega_\nu}{\zeta}}} \lesssim  \frac{\nu^2}{|\ln \nu|}.$$
The collection of the above estimates yields the estimate for $\tilde{m}_E$ and closes the proof of Lemma \ref{lemm:appProf}.
\end{proof}

\subsection{Bootstrap regime}

We describe a regime in which the solution $w$ is close to the approximate solution $W$. The most important quantity is $\| m_\e \|_{L^2(\frac{\omega_\nu}{\zeta})}$ giving radial $L^2$ control in the parabolic zone $\zeta\lesssim 1$, from which we are able to close the leading dynamical system driving the law of blowup solutions in Lemma \ref{lemm:Mod}.

To close estimates at the nonlinear level, we use higher order regularity norms in the parabolic zone, that we decompose as the union of the inner zone $\zeta \leq \zeta_*$ and of the middle range zone $\zeta_*\leq \zeta \leq \zeta^*$, where we fix two numbers
$$
0<\zeta_*\ll 1 \ \ \mbox{ and } \ \ \zeta^*\gg 1.
$$
First, we use an inner norm for the radial part in modified blow-up variables \eqref{blowupvar}
\begin{equation}\label{def:norm mvin}
\| \tilde m_{v} \|_{\inn}^2 :=   \int_0^{\infty} \Big[ -\tnu^2\big(\chi_{\frac{\zeta_*}{\tnu}} \tilde m_{v}\big) + \As_0 \big(\chi_{\frac{\zeta_*}{\tnu}} \tilde m_{v}\big) \Big]  \As_0 \big(\chi_{\frac{\zeta_*}{\tnu}} \tilde m_{v}\big)\frac{\omega_0(\tr)}{\tr} d\tr 
\end{equation}
where  $\chi_{\frac{\zeta_*}{\tnu}}$ is defined as in \eqref{def:chiM} and 
$\omega_0(\tr) = ( U(\tr))^{-1}$. Second, we use standard Sobolev norms for the full remainder $\| \e \|_{H^2(\zeta_*\lesssim \zeta \lesssim \zeta^*)}$ in the zone $\zeta_*\leq |z|\leq \zeta^*$.

The nonradial part is controlled in the parabolic zone by the norm
$$
\| \e^\perp \|_0^2= \nu^2 \int_{\Rb^2} \e^\perp\sqrt{\rho_0}\Ms^z \big(\e^{\perp} \sqrt{\rho_0} \big)dz,
$$
$\Ms^z$ being defined by \eqref{def:L0Mz} and at a higher order regularity level in the inner zone
\begin{equation}\label{def:qperpnorminn}
\| q^\perp \|_{\inn}^2= -\int_{\Rb^2} \Ls_0 (\chi_{\frac{\zeta_*}{\nu}} q^\perp) \Ms  (\chi_{\frac{\zeta_*}{\nu}} q^\perp)dy = \int U| \nabla \big(\Ms  (\chi_{\frac{\zeta_*}{\nu}} q^\perp) \big)|^2 dy,
\end{equation}
where $\rho$, $\Ls_0$ and $\Ms$ are defined in \eqref{def:omegarho} and \eqref{def:Ls0Ms}  respectively. Note the equivalence thanks to the orthogonality \eqref{orthogonality} and the coercivity \eqref{bd:coercivite L2}:
\begin{equation}\label{eq:equive0}
\|\e^\perp\|_0 \sim \|\e^\perp\|_{L^2_{\omega_\nu}}.
\end{equation}
Finally, in the outer zone $\zeta \geq \zeta^*$, the full lower order perturbation is decomposed in radial and nonradial parts,
\begin{equation}\label{def:what}
\hat w = w - U_\nu - \Psi_{1,\nu} = \Psi_{2,\nu} + \e =: \hat w^0 + \hat w^\perp,
\end{equation}
and is controlled outside by an outer norm in the parabolic variables:
$$
\| \hat w \|_{\out}^{2p}=\int \left(1-\chi_{\frac{\zeta^*}{4}}\right) \left(||z|^{2-\frac 14} \hat w |+||z|^{2+\frac 34}\nabla \hat w | \right)^{2p}\frac{dz}{|z|^2} \quad \textup{for some}\;\; p\gg 1.
$$

\begin{definition}[Bootstrap Initiation] \label{def:ini}

Let $N\in \mathbb N$, $\kappa>0$, $\tau_0\gg 1$, $\mu_0=e^{-\tau_0/2}$. We say that $w_0$ satisfies the initial bootstrap conditions if there exists $x^*_0\in \mathbb R^2$, $\beta_0>0$ and $\nu_0>0$ such that the following holds true. In the variables \eqref{selfsimvar} one has the decomposition \eqref{decomposition} with the orthogonality conditions \eqref{orthogonality} and the following estimates:
\begin{itemize}
\item[(i)] \emph{(Compatibility condition for the initial renormalisation rate $\beta_0$)}
$$
\frac{a_1}{4\nu^2} = -1 + \frac{1}{2\ln \nu_0} + \frac{\ln 2 - \gamma - 1 - \ln \beta_0}{4|\ln \nu_0|^2}.
$$
\item[(ii)] \emph{(Initial modulation parameters)} For $ \bar c_0 =  \sqrt{\frac{2}{\beta_0}} e^{-\frac{2+\gamma}{2}}$:
\begin{align} \label{parametersinit}
&\bar c_0 e^{-\sqrt{\beta_0 \tau_0}}\left(1-\frac{1}{|\ln \nu_0|}\right) \leq  \nu_0\leq \bar c_0 e^{-\sqrt{\beta_0 \tau_0}}\left(1+\frac{1}{|\ln \nu_0|}\right), \\
\label{parametersinitan}
&\frac 12 -\frac{1}{|\ln \nu_0|}< \beta_0<\frac 12 +\frac{1}{|\ln \nu_0|}, \quad |a_n|< \frac{\nu^2_0}{|\ln \nu_0|^2} \ \mbox{ for } 2\leq n \leq N.
\end{align}
\item[(iii)] \emph{(Initial remainder)}
\begin{align} \label{mepsiloninitL2omega}
&\nu_0^2\|m_q(s_0) \|_{L^2_{\frac{\omega}{r}}} = \|m_\varepsilon(\tau_0) \|_{L^2_{\frac{\omega_\nu}{\zeta}}}< \frac{\nu^2_0}{|\ln \nu_0|},\\
&\| m_\e \|_{H^2(\frac{\zeta_*}{4}\leq \zeta \leq 4\zeta^*)} <\frac{\nu_0^2}{|\ln \nu_0|},\\
& \label{bd:initint}
\| \tilde m_v \|_{\inn}  < \frac{\nu_0^2}{|\ln \nu_0|},\\
\end{align}
and
\begin{align}
 &\label{bd:initout}
\| \hat w^{0} \|_{\out} < \frac{\nu_0^2}{|\ln \nu_0|},\\
&\| \e^\perp \|_0 + \| \e^\perp \|_{H^2(\frac{\zeta_*}{4}\leq |z|\leq 4\zeta^*)} < e^{-\kappa \tau_0},\\
 \label{bd:initintperp}
&\| q^\perp \|_{\inn}<\frac{e^{-\kappa \tau_0}}{\nu_0},\\
&\label{bd:initoutperp}
\| \hat w^\perp \|_{\out} < e^{-\kappa \tau_0}.
\end{align}
\end{itemize}
\end{definition}

Our goal is to prove that solutions satisfying the initial bootstrap conditions defined by the previous definition will stay close to the approximate solution forward in time, in the following sense.

\begin{definition}[Bootstrap]  \label{def:bootstrap}

Let $\kappa>0$, $\tau_0\gg 1$, and $K''\gg K'\gg K\gg 1$. A solution $w$ to \eqref{eq:KS} is said to be trapped on $[\tau_0,\tau^*]$ if it satisfies the initial bootstrap conditions in the sense of Definition \ref{def:ini} at time $\tau_0$ and the following conditions on $(\tau_0,\tau^*]$. There exists $\mu \in C^1([0,t^*],(0,\infty))$ such that the solution can be decomposed according to \eqref{decomposition}, \eqref{orthogonality} on $[\tau_0,\tau^*]$ with:
\begin{itemize}
\item[(i)] \emph{(Compatibility condition for the renormalisation rate $\beta$)}
\begin{equation} \label{bootstrap:condition}
\frac{a_1}{4\nu^2} = -1 + \frac{1}{2\ln \nu} + \frac{\ln 2 - \gamma - 1 - \ln \beta}{4|\ln \nu|^2}.
\end{equation}
\item[(ii)] \emph{(Modulation parameters)} For $\bar c =  \sqrt{\frac{2}{\beta}}e^{-\frac{2+\gamma}{2}}$:
\begin{align} \label{bootstrap:param1}
\bar c e^{-\sqrt{\beta_0 \tau+\int_{\tau_0}^\tau \beta}}\left(1-\frac{K'\ln| \ln \nu|}{|\ln \nu|}\right) \leq  \nu\leq \bar c e^{-\sqrt{\beta_0 \tau+\int_{\tau_0}^\tau \beta}}\left(1+\frac{K'\ln |\ln \nu|}{|\ln \nu|}\right),
\end{align}
\begin{align}
&\frac 12 -\frac{K'}{|\ln \nu|}< \beta <\frac 12 +\frac{K'}{|\ln \nu|},&\\
\label{bootstrap:param2}
&|a_n|< \frac{K\nu^2}{|\ln \nu|^2} \ \mbox{ for } 2\leq n \leq N.&
\end{align}
\item[(iii)] \emph{(Remainder)}
\begin{align} \label{bootstrap:L2omeofme}
&\nu^2\|m_q(s)\|_{L^2_{\frac{\omega}{r}}} = \|m_\varepsilon(\tau) \|_{L^2_{\frac{\omega_\nu}{\zeta}}}< \frac{K \nu^2}{|\ln \nu|},\\
 \label{bootstrap:mid}
&\| m_\e(\tau) \|_{H^2(\zeta_*\leq \zeta\leq \zeta^*)} <\frac{K' \nu^2}{|\ln \nu|},\\
 \label{bootstrap:in}
&\| \tilde m_{v} \|_{\inn}  < \frac{K''\nu^2}{|\ln \nu|},\\
 \label{bootstrap:out} 
&\| \hat w^{0} \|_{\out} < K'' \frac{\nu^2}{|\ln \nu|},
\end{align}
and
\begin{align}
 \label{bootstrap:perpenergy}
&\| \e^\perp \|_0 < K e^{-\kappa \tau},\\
 \label{bootstrap:perpmid}
&\| \e^\perp \|_{H^2(\zeta_*\leq |z|\leq \zeta^*)} < K' e^{-\kappa \tau},\\
 \label{bootstrap:inperp}
&\| q^\perp \|_{\inn}<K''\frac{e^{-\kappa \tau}}{\nu},\\
 \label{bootstrap:outnonradial} 
&\| \hat w^\perp \|_{\out} < K'' e^{-\kappa \tau}.
\end{align}

\end{itemize}

\end{definition}

The initial bootstrap conditions define an open set in which trajectories are trapped, in the sense of the following Proposition.

\begin{proposition}[Exitence of a solution trapped in the bootstrap regime] \label{pr:bootstrap}
There exists a choice of the constants $N\in \mathbb N$, $0<\kappa \ll 1$, $K\gg 1$, $K'\gg 1$, $K''\gg 1$ and $\tau_0\gg 1$ such that any solution which is initially in the bootstrap in the sense of Definition \ref{def:ini} will be trapped on $[\tau_0,\infty)$ in the sense of Definition \ref{def:bootstrap}.
\end{proposition}
\begin{proof} This is the heart of the present paper. The next Lemmas and Propositions prepare for its proof finally done in Subsection \ref{sec:ProofofProExist}.
\end{proof}

\begin{remark} \label{re:orderconstants}
The constants are determined in the following order. First, $N$ is chosen large enough and $\kappa$ small enough so that constants in time derivatives of Lyapunov functionals have correct signs. Then $K$ is chosen first, followed by $K'$ depending on $K$, and $K''$ depending on $K,K'$. Finally, $\tau_0$ is chosen last, so that this parameter is always adjusted throughout the proof to obtain various smallness.
\end{remark}

\subsection{Properties of the bootstrap regime}
The following lemma gives some properties of a solution trapped in the bootstrap regime.
\begin{lemma}[A priori control in the bootstrap]
Let $w$ be a solution in the bootstrap regime in the sense of Definition \ref{def:bootstrap}. Then for $\tau_0$ large enough depending on $\kappa,K,K',K''$ the following estimates hold on $[\tau_0,\tau^*]$:
\begin{itemize}
\item (Estimate on $\nu$)
\begin{equation} \label{bd:lnnu}
\frac{\sqrt{\tau}}{2}\leq |\ln \nu|\leq \sqrt{\tau}.
\end{equation}
\item (Estimate on the approximate perturbation) For $C$ independent on the bootstrap constants $\kappa,K,K',K''$ for $k=0,1,2$:
\begin{align} 
&|D^k P_{1,\nu} (\zeta)|\lesssim \nu^2  \left(\frac{\zeta}{\nu+\zeta}\right)^{2-k\wedge 2}  \frac{ \la \ln \la r\ra \ra}{(\nu+\zeta)^k}, \label{pointwisehatpsi}\\
& \quad |D^k P_{2,\nu} (\zeta)|\lesssim \frac{1}{|\ln \nu|^2}   \left(\frac{\zeta}{\nu+\zeta}\right)^{2-k\wedge 2}  \frac{\langle \zeta \rangle^{C(N)}\la \ln \la r\ra \ra}{(\nu+\zeta)^k}. \label{pointwisehatpsi2}
\end{align}
\item (Refined pointwise estimate in the parabolic zone) For $\zeta\leq \zeta^*/2$:
\begin{equation} \label{pointwisemerefined}
\big|\tilde m_w(\zeta)\big| + \big|\zeta \pa_\zeta\tilde m_w(\zeta)\big| \lesssim \frac{\zeta^2}{(\zeta+\nu)^2} \frac{\nu^2}{|\ln \nu|}  \sqrt{\la \ln \langle r \rangle \ra}.
\end{equation}

\item (Pointwise estimates in the outer zone) For $\zeta=|z|\geq \zeta^*/2$, for the full perturbation:
\begin{equation} \label{pointwiseaway}
 \big| \Psi_\nu + \e^0\big| \lesssim \nu^2|\ln \nu| |z|^{-2+\frac 14}, \quad \big|P_\nu+m_\e\big|+ \big|\zeta\pa_\zeta (P_\nu + m_\e)\big|\lesssim |\ln \nu|\nu^2\zeta^{\frac 14},
\end{equation}
and for the higher order part of the perturbation:
\begin{equation} \label{pointwiseaway2}
 \big| \Psi_{2,\nu} + \e^0\big| \lesssim \frac{\nu^2}{|\ln \nu|} |z|^{-2+\frac 14}, \; \big|P_{2,\nu} + m_\e\big|+ \big|\zeta\pa_\zeta (P_{2,\nu} + m_\e)\big|\lesssim \frac{\nu^2}{|\ln \nu|}\zeta^{\frac 14},
\end{equation}
and for the nonradial part:
\begin{equation} \label{pointwiseawaynonradial}
 |\e^\perp(z)| \lesssim e^{-\kappa \tau}(1+|z|)^{-2+\frac 14}.
\end{equation}
\item (Pointwise estimate on the nonradial Poisson field) For all $z \in \Rb^2$,
\begin{equation} \label{poissonLinfty}
|\Phi_{\e^\perp} (z)|+(1+|z|)|\nabla \Phi_{\e^\perp} (z)| \lesssim (1+|z|)^{\frac 12} \frac{e^{-\kappa \tau}}{\nu^3}.
\end{equation}
\end{itemize}

\end{lemma}

\begin{proof}

Note that \eqref{bd:lnnu} is a direct consequence of \eqref{bootstrap:param1}.\\

\noindent \textbf{Step 1} \emph{The approximate perturbation}. From \eqref{def:P1P2N}, the pointwise estimate \eqref{eq:pointwisephin-phi0} and the bootstrap bounds \eqref{bootstrap:condition}, \eqref{bootstrap:param1} and \eqref{bootstrap:param2}, there holds for a constant $C$ depending on $N$ only, 
\begin{eqnarray*}
\left| D^k P_{1, \nu}(\zeta)\right|  & \leq & |a_1||D^k(\phi_{1, \nu}(\zeta)-\phi_{0,\nu}(\zeta)) \leq  C\nu^2  \left(\frac{\zeta}{\nu+\zeta}\right)^{2-k\wedge 2}  \frac{1+\ln \langle \frac{|z|}{\nu}\rangle}{(\nu+\zeta)^k} ,\\
\left| D^k P_{2,\nu}(\zeta)\right|  & \leq & \sum_{n=2}^N |a_n| \big|D^k\phi_{n, \nu}(\zeta)\big| \leq  \frac{C}{|\ln \nu|^2} \left(\frac{\zeta}{\nu+\zeta}\right)^{2-k\wedge 2} \langle \zeta \rangle^C \frac{1+\ln \langle \frac{|z|}{\nu}\rangle}{(\nu+\zeta)^k}, 
\end{eqnarray*}
which are the pointwise estimates \eqref{pointwisehatpsi} and \eqref{pointwisehatpsi2}.

\noindent \textbf{Step 2} \emph{Pointwise estimates for $\tilde m_w$}. Let $f(r)=\As_0 \tilde m_v$ and $\tilde m_w=\As_0^{-1}f$ be defined by the formula \eqref{def:invAs}. Consider the zone $r\leq \zeta^*/2\nu$. Note that one has
$$
f(r)=\As_0 (\chi_{\frac{\zeta_*}{\nu}} \tilde m_v) \quad  \mbox{ for } \;\; r\leq \frac{2\zeta_*}{\nu},
$$
and for  $2\zeta_* \leq r \nu \leq \zeta^*$, 
$$
|f(r)|=|\As_0 \tilde m_v(r)|\lesssim |\pa_{rr} \tilde m_v|+\nu|\pa_r \tilde m_v|+\nu^2| \tilde m_v|=\nu^2(|\pa_{\zeta\zeta}\tilde m_w|+|\pa_{\zeta} \tilde m_w|+|\tilde m_w|).
$$
Hence, from the bootstrap bounds \eqref{bootstrap:mid}, \eqref{bootstrap:in} and the relation $\tilde m_w = Q_\nu - Q_{\tilde{\nu}} + P_{1, \tnu} - P_{1,\nu} +  P_{2,\nu} + m_\e$, we estimate
\begin{align*}
& \| f \|_{L^2_{\omega_0}(r\leq \frac{\zeta^*}{2\nu})}  \lesssim  \| \As_0 (\chi_{\frac{\zeta_*}{\nu}} \tilde m_v) \|_{L^2_{\omega_0}}+\nu^2\| |\pa_{\zeta\zeta} \tilde m_w|+|\pa_{\zeta} \tilde m_w|+|\tilde m_w| \|_{L^2_{\omega_0} ( \frac{2\zeta_*}{\nu} \leq r \leq \frac{\zeta^*}{2\nu})}\\
& \quad \lesssim \nu^2 \|\tilde m_v \|_{\inn}+ \|m_\e\|_{H^2(\zeta_*\leq \zeta \leq \zeta^*)} + \left\|Q_{\tnu} - Q_{\nu} + P_{1, \tnu} - P_{1,\nu} + P_{2,\nu}  - N_{1, \tnu}\right\|_{H^2(\zeta_*\leq \zeta \leq \zeta^*)}  \lesssim \frac{\nu^2}{|\ln \nu|}.
\end{align*}
Using the explicit inversion formula of $\As_0^{-1}$, we write
$$
\chi_{\frac{\zeta_*}{\nu}} \tilde m_v = \frac{1}{2} \psi_0(r) \int_r^{1} \frac{\zeta^4 + 4\zeta^2 \ln \zeta - 1}{\zeta} f(\zeta) d\zeta + \frac{1}{2}\tilde\psi_0(r)\int_0^r \zeta f(\zeta) d\zeta + c_0 \psi_0,
$$
where $\psi_0$ and $\tilde{\psi}_0$ are the two linearly independent solutions to $\As_0 \psi = 0$ given by
\begin{equation}\label{def:psi01}
\psi_0(r) = \frac{r^2}{\rj^4} \quad \textup{and} \quad \tilde{\psi}_0(r) = \frac{r^4 + 4r^2\ln r - 1}{\rj^4}.
\end{equation}
From the orthogonality condition \eqref{eq:orthmwT0}, we use the coercivity of $\As_0$ given in Lemma \ref{lemm:coerA0} to estimate the constant $|c_0| \lesssim \|\tilde{m}_v\|_\inn \lesssim \frac{\nu^2}{|\ln \nu|}$. We then estimate by Cauchy-Schwarz inequality and the decay of $\psi_0$,
\begin{align*}
k = 0, 1, \quad |(r \pa_r)^k\tilde m_v(r)| & \lesssim  \| f \|_{L^2_{\omega_0}(r\leq \frac{\zeta^*}{2\nu})}\frac{r^2}{1+r^2}   \sqrt{ \ln \langle r \rangle } + \frac{|c_0|}{1 + r^2}   \lesssim \; \frac{\nu^2}{|\ln \nu|} \frac{r^2}{1+r^2} \sqrt{ \ln \langle r \rangle } ,
\end{align*}
for $r\leq \zeta^*/2\nu$. This concludes the proof of the pointwise estimate \eqref{pointwisemerefined}.\\

\noindent \textbf{Step 3} Radial far away pointwise estimates. We first prove \eqref{pointwiseaway2}. We recall the Sobolev embedding:
\begin{equation} \label{weightedsobo}
\| |z|^{2-\frac 14}|\hat w^{0}|\|_{L^{\infty}(|z|\geq \frac{\zeta^*}{2})}\lesssim \| \hat w^{0}\|_{\out}.
\end{equation}
Due to the the bootstrap bound \eqref{bootstrap:out}, the above inequality proves \eqref{pointwiseaway2}. From the relation $\pa_\zeta m_{\hat w}=\zeta \hat w^0$ this in turn implies the inequality for $\zeta\geq \zeta^*/2$:
$$
|\pa_\zeta  m_{\hat w}|\lesssim \frac{\nu^2}{|\ln \nu|} \zeta^{\frac 14-1}.
$$
From \eqref{bootstrap:mid} and Sobolev embedding, we get 
$$
| m_{\hat w}(\frac{\zeta^*}{2})|\lesssim \frac{\nu^2}{|\ln \nu|}.
$$
The two above inequalities and the fundamental Theorem of Calculus then imply the second inequality in \eqref{pointwiseaway2}. Recall that
$$
 w-U_\nu=\Psi_\nu + \e, \ \  m_w-Q_\nu=P_\nu+m_{\e}.
$$
from which and the relation \eqref{bootstrap:condition}, the pointwise estimate \eqref{eq:pointwisephin-phi0} and the bound \eqref{pointwiseaway2}, we deduce \eqref{pointwiseaway}.\\

\noindent \textbf{Step 4} Estimate of the Poisson field. We first note that from \eqref{bootstrap:outnonradial} and \eqref{bootstrap:perpmid}, and the Sobolev embedding \eqref{weightedsobo} there holds for $|z|\geq \zeta_*/2$:
\begin{equation} \label{intersetp4pointiwse}
|\e^\perp(z)| \lesssim e^{-\kappa \tau} |z|^{-2+\frac 14}.
\end{equation}
From a change of variables, the coercivity given in Lemma \ref{lemm:coerML0} and the bootstrap bound \eqref{bootstrap:perpenergy}, we have the localized estimate
$$
\int_{\mathbb R^2} |\chi_{\frac{\zeta_*}{\nu}}q^\perp|^2 \frac{dy}{U(y)} \lesssim \frac{e^{-2\kappa \tau}}{\nu^2}.
$$
Using this estimate and \eqref{pointwiseawaynonradial} yield
$$
\int_{\mathbb R^2} |\e^{\perp}(z)|^2(1+|z|)^{2\alpha}dz\lesssim \frac{e^{-2\kappa \tau }}{\nu^6},
$$
for any $\alpha<3/4$. Applying \eqref{bd:poisson1} for $\alpha=1/2$ then implies \eqref{poissonLinfty}.
\end{proof}

\section{Control of the solution in the bootstrap regime} \label{sec:ControlStab}

This section is devoted to the proof of Proposition \ref{pr:bootstrap}. It will be showed through a series of Lemmas in which the dynamics of the parameters is controlled and the a priori estimates for the remainder are bootstrapped.

\subsection{Modulation equations}

Injecting the decomposition \eqref{decomposition} in the equation \eqref{partialmassselfsim}, we obtain the following equation of the remainder in the partial mass setting:
\begin{equation} \label{eq:me}
\pa_\tau m_{\e}=\As^\zeta m_\e+  \frac{\pa_\zeta\big[ (2P_\nu + m_\e) m_\e\big]}{2\zeta} + m_E + N_0(\e^\perp),
\end{equation}
where 
$P_\nu$ and $m_E$ are introduced in \eqref{exp:mW} and \eqref{eq:mW} respectively and
\begin{equation}\label{def:NL0}
N_0(\e^\perp) =  -\frac{1}{2\pi}\int_{S(0,\zeta)}\e^{\perp} \Big(\nabla \Phi_{\e^{\perp}} - \frac{x_\tau^*}{\mu}  \Big) \cdot \vec n dS.
\end{equation}
We write from \eqref{eq:wztau} and the decomposition \eqref{decomposition} the equation satisfied by $\e^\bot$: 
\begin{align}\label{eq:epsbot}
\partial_\tau \e^\bot = \Ls^z \e^\bot - \nabla \cdot \Gc(\e^\perp)  + \frac{x_\tau^*}{\mu}\cdot \nabla W+\frac{x_\tau}{\mu}.\nabla \e^0 + N^\perp(\e^\perp),
\end{align}
where $\Ls^z$ is the linearized operator defined in \eqref{def:Lsz} and 
\begin{align}\label{def:Geps}
\Gc(\e^\perp) &= \e^\perp \nabla \Phi_{\Psi_\nu + \e^0} + (\Psi_\nu + \e^0)\nabla \Phi_{\e^\perp}, \\
N^\perp(\e^\perp) &= - \left[\nabla \cdot \left(\e^\perp \Big( \nabla \Phi_{\e^\perp} - \frac{x_\tau^*}{\mu}\Big)\right) \right]^\perp. \label{def:Neps}
\end{align}
Projecting \eqref{eq:me} on the directions generated the orthogonality conditions gives the time evolution of the parameters below. The solution of the following equations gives the desired blowup laws as explained in the strategy of the proof. 
\begin{lemma}[Modulation equations] \label{lemm:Mod}
Let $w$ be a solution in the bootstrap regime in the sense of Definition \ref{def:bootstrap}. Then, the following estimates hold on $[\tau_0,\tau^*]$:
\begin{align} \label{eq:mod0}
8 \nu^2\left(\dfrac{\nu_\tau}{\nu} - \beta \right) + a_{1,\tau} - 2\beta a_1(1 + &\tilde{\alpha}_0)  + \dfrac{a_1}{\ln \nu}\dfrac{\nu_\tau}{\nu} + a_1 \frac{\beta_\tau}{\beta} =  \Oc\left(\frac{\Dc(\tau)}{|\ln \nu|} \|m_\e\|_{L^2_{\frac{\omega_\nu}{\zeta}}} \right) + \Oc\left(\frac{\nu^2}{|\ln \nu|^3}\right),
\end{align}
\begin{equation} \label{eq:mod1}
a_{1,\tau} - 2\beta a_1 \tilde{\alpha}_1 + \dfrac{a_1}{\ln \nu} \dfrac{\nu_\tau}{\nu}  + a_1 \frac{\beta_\tau}{\beta} = \Oc\left( \frac{\Dc(\tau)}{|\ln \nu|^2} \|m_\e\|_{L^2_{\frac{\omega_\nu}{\zeta}}} \right) + \Oc\left(\frac{\nu^2}{|\ln \nu|^3}\right),
\end{equation}
and for $n \in {2, \cdots, N}$,
\begin{equation} \label{eq:mod2}
a_{n,\tau} - 2\beta a_n(1 - n + \tilde{\alpha}_n) = \Oc\left( \frac{\Dc(\tau)}{|\ln \nu|} \|m_\e\|_{L^2_{\frac{\omega_\nu}{\zeta}}} \right) + \Oc\left(\frac{\nu^2}{|\ln \nu|^2}\right),
\end{equation}
where $\tilde{\alpha}_n$ is given in Proposition \ref{prop:SpecRad} and 
$$\Dc(\tau) = \left|\frac{\nu_\tau}{\nu}\right| + \left|\frac{\beta_\tau}{\beta}\right|.$$
\end{lemma}
\begin{proof}

Taking the scalar product of \eqref{eq:me} with $\phi_{n, \nu}$ in $L^2_{\frac{\omega_\nu}{\zeta}}$ for $n=0,...,N$, using the orthogonality \eqref{orthogonality} and the self-adjointness of $\As^\zeta$ yields the identity
\begin{equation} \label{interm:modexpr}
\langle \pa_\tau m_{\e},\phi_{n,\nu}\rangle =\left \langle \frac{\pa_\zeta \big[(2P_\nu + m_\e ) m_\e\big]}{2\zeta} + m_E + N_0(\e^\perp),\phi_{n,\nu} \right\rangle,
\end{equation}
where we write $\la \cdot, \cdot \ra \equiv \la \cdot, \cdot \ra_{L^2_{\frac{\omega_\nu}{\zeta}}}$ for simplificity and $\omega_\nu$ is the weight introduced in \eqref{def:omeganu}. \\
We are going to estimate all terms appearing in the above equation. \\

\noindent \emph{- The time derivative term}. We start by differentiating the first orthogonality condition in \eqref{orthogonality} to get
\begin{eqnarray*}
0 &=&\frac{d}{d\tau}\langle m_\e,\phi_{n,\nu}\rangle = \langle \pa_\tau m_\e, \phi_{n,\nu} \rangle+\Big\langle m_\e, \frac{\nu_\tau}{\nu}(\nu \pa_\nu \phi_{n,\nu}+\frac{\nu\pa_\nu \omega_\nu}{\omega_\nu}\phi_{n,\nu})+\frac{\beta_\tau}{\beta}(\beta \pa_\beta \phi_{n,\nu}+ \frac{\beta \pa_\beta \omega_\nu}{\omega_\nu}\phi_{n,\nu})\Big\rangle,
\end{eqnarray*}
where we have the algebraic identities
$$
\frac{\nu \pa_\nu \omega_\nu}{\omega_\nu}= \frac{\frac{\nu^4}{2}(1+\frac{\zeta^2}{\nu^2})}{\nu^4(1+\frac{\zeta^2}{\nu^2})^2}=\frac{1}{(1+\frac{\zeta^2}{\nu^2})}, \qquad \frac{\beta \pa_\beta \omega_\nu}{\omega_\nu} = -\frac{\beta\zeta^2}{2}.
$$
Using Cauchy-Schwarz inequality, \eqref{est:poinwisephi1} or \eqref{est:poinwisephi02}, \eqref{est:poinwisephi021}  and \eqref{est:poinwisephi12} if $n=0,1$, the orthogonality condition \eqref{orthogonality}, and \eqref{est:PhinPointEst} for $n\geq 2$, one obtains
\begin{align*}
\left| \langle m_\e ,\frac{\beta \zeta^2}{2} \phi_{n,\nu}\rangle \right|+\left| \langle m_\e ,\beta \pa_\beta \phi_{n,\nu}\rangle \right|&+\left| \langle m_\e ,\nu \pa_\nu \phi_{n,\nu}\rangle \right|  \lesssim \| m_\e\|_{L^2_{\frac{\omega_\nu}\zeta}} \times \left\{ \begin{array}{cl} 1  \ \ &\mbox{ if }\;  n=0,1,\\ 
 |\ln \nu| \ \ &\mbox{ if } \; n\geq 2. \end{array} \right.
\end{align*}
From \eqref{est:poinwisephi1}, \eqref{est:poinwisephi02} and \eqref{est:PhinPointEst}, and Cauchy-Schwarz, we estimate for $n \in \{0, \cdots, N\}$,
$$
\left| \langle m_\e,\frac{\nu \pa_\nu \omega_\nu}{\omega_\nu}\phi_{n,\nu}\rangle \right| \lesssim \| m_\e\|_{L^2_{\frac{\omega_\nu}\zeta}}.
$$
Collecting the above estimates yields
\begin{equation} \label{interm:modpatau}
\left| \langle \pa_\tau m_\e, \phi_{n,\nu} \rangle \right| \lesssim \mathcal D(\tau)\| m_\e\|_{L^2_{\frac{\omega_\nu}\zeta}} \times \left\{ \begin{array}{cl} 1 \ \ &\mbox{ if }\; n=0,1,\\ 
|\ln \nu|  \ \ &\mbox{ if } \; n\geq 2. 
\end{array} \right.
\end{equation}

\noindent \emph{- The lower order linear term}. One first writes by using integration by parts 
\begin{align}\label{eq:idPme}
&\left\langle \frac{\pa_\zeta(P_\nu m_\e)}{\zeta}, \phi_{n,\nu} \right\rangle\\
& = -\Big\langle m_\e,\frac{\zeta}{\omega_\nu}P_\nu \pa_\zeta((\phi_{n,\nu} - \phi_{0, \nu})\frac{\omega_\nu}{\zeta^2}) \Big\rangle -\Big\langle m_\e,\frac{\zeta}{\omega_\nu}P_\nu \pa_\zeta(\phi_{0, \nu}\frac{\omega_\nu}{\zeta^2}) \Big\rangle.\nonumber
\end{align}
Using  \eqref{pointwisehatpsi}, \eqref{pointwisehatpsi2} and the degeneracy near the origin \eqref{eq:pointwisephin-phi0}, one obtains the rough bound
\begin{align*}
&\left|\frac{\zeta}{\omega_\nu}P_\nu \pa_\zeta\big((\phi_{n,\nu} - \phi_{0,\nu})\frac{\omega_\nu}{\zeta^2}\big)\right|\lesssim  \frac{1}{|\ln \nu|^2}  \langle \zeta \rangle^C \frac{ r^2 \ln \langle r \rangle}{\langle r \rangle^6},
\end{align*}
which yields the estimates
$$
\Big\| \frac{\zeta}{\omega_\nu}P_\nu \pa_\zeta\big((\phi_{n,\nu} - \phi_{0,\nu})\frac{\omega_\nu}{\zeta^2} \big) \Big\|_{L^2_\frac{\omega_\nu}{\zeta}}\lesssim \frac{1}{|\ln \nu|^2},
$$
$$
\left|\Big\langle m_\e,\frac{\zeta}{\omega_\nu}P_\nu \pa_\zeta((\phi_{n,\nu} - \phi_{0, \nu})\frac{\omega_\nu}{\zeta^2}) \Big\rangle  \right| \lesssim \frac{ \|m_\e\|_{L^2_{\frac{\omega_\nu}{\zeta}}}}{|\ln \nu|^2}.
$$
As for the last term in \eqref{eq:idPme}, we write $\phi_{0, \nu} = \frac{1}{\nu^2}T_0(r) + \tilde{\phi}_{0, \nu}$. Using the pointwise estimate \eqref{pointwisephi0}, we have the estimate
\begin{align*}
&\left|\frac{\zeta}{\omega_\nu}P_\nu \pa_\zeta\big(\tilde\phi_{0,\nu}\frac{\omega_\nu}{\zeta^2}\big)\right|\lesssim  \frac{1}{|\ln \nu|^2} \left(\frac{\zeta}{\nu+\zeta}\right)^2\frac{\ln \langle r \rangle}{(\nu+\zeta)^2} ,
\end{align*}
from which and the Cauchy-Schwarz inequality, we obtain
\begin{align}
\left|\Big\langle m_\e,\frac{\zeta}{\omega_\nu}P_\nu \pa_\zeta(\tilde \phi_{0, \nu}\frac{\omega_\nu}{\zeta^2}) \Big\rangle\right| \lesssim  \frac{1}{|\ln \nu|^2}  \|m_\e\|_{L^2_{\frac{\omega_\nu}{\zeta}}}.\label{est:proPhi0nu}
\end{align}
To estimate the contribution coming from $T_0$, we use the algebraic identity $8T_0 = r^2 U$ and write
\begin{align}
\frac{\zeta}{\omega_\nu}\pa_\zeta\left(\frac{1}{\nu^2}T_0(r)\frac{\omega_\nu}{\zeta^2}\right) = -\frac{\beta \zeta^2}{8 \omega_\nu} e^{-\frac{\beta \zeta^2}{2}},\label{eq:TOcom}
\end{align} 
and recall from  \eqref{pointwisehatpsi} and \eqref{pointwisehatpsi2} that $P_\nu = P_{1,\nu} + P_{2,\nu}$ satisfies
\begin{equation}\label{est:Prough}
|P_\nu(\zeta)|\lesssim \frac{1}{|\ln \nu|^2} \frac{\ln \langle r \rangle}{\langle r \rangle^{2}} \langle \zeta \rangle^C.
\end{equation}
Hence, we have
\begin{align*}
&\left|\Big\langle m_\e,\frac{\zeta}{\omega_\nu}P_\nu \pa_\zeta\left(\frac{1}{\nu^2}T_0(r)\frac{\omega_\nu}{\zeta^2}\right) \Big\rangle\right| = \left| \frac{\beta}{8} \int m_\e \zeta P_\nu e^{-\beta \zeta^2/2} d\zeta \right|  \lesssim \|m_\e\|_{L^2_{\frac{\omega_\nu}{\zeta}}} \left(\int \zeta^3 P_\nu^2 e^{-\beta \zeta^2} \omega_\nu^{-1} d\zeta \right)^\frac{1}{2}\\
& \quad \lesssim  \|m_\e\|_{L^2_{\frac{\omega_\nu}{\zeta}}}  \times \frac{K}{|\ln \nu|^2} \left(\int r^3  \frac{\langle \nu r \rangle^{2C} \ln^2 \langle r \rangle}{\langle r \rangle^4} e^{-\beta \nu^2 r^2/2} U(r) dr  \right)^\frac{1}{2} \lesssim \frac{1}{|\ln \nu|^2}  \times \|m_\e\|_{L^2_{\frac{\omega_\nu}{\zeta}}}. 
\end{align*}
Injecting these estimates into \eqref{eq:idPme} yields 
 \begin{equation}\label{est:SmallLinear}
 \left| \left\langle \frac{\pa_\zeta(P_\nu m_\e)}{\zeta}, \phi_{n,\nu} \right\rangle\right| \lesssim \frac{1}{|\ln \nu|^2} \|m_\e\|_{L^2_{\frac{\omega_\nu}{\zeta}}}. 
 \end{equation}

\noindent  \emph{ - The error term}. From the expression \eqref{exp:mE}, the orthogonality \eqref{est:PhinL2norm1}, we have 
\begin{equation} \label{interm:mod0}
\langle m_E,\phi_{n,\nu}\rangle = \| \phi_{n,\nu}\|_{L^2_{\frac{\hat\omega}{\zeta}}}^2 \times \textup{Mod}_n + \langle \tilde{m}_E, \phi_{n,\nu} \rangle + \left\langle \frac{\pa_\zeta P_\nu^2}{2\zeta}, \phi_{n,\nu} \right\rangle, 
\end{equation}
where the contribution of $\langle \tilde{m}_E, \phi_{n,\nu} \rangle$ is precisely given by \eqref{expr:mEphi0}, \eqref{expr:mEphi1} and \eqref{expr:mEphi2}. Applying \eqref{est:SmallLinear} with $m_\e = P_\nu$, we have the estimate 
\begin{align*}
\left|\left\langle \frac{\pa_\zeta (P_\nu^2)}{\zeta}, \phi_{n,\nu} \right\rangle \right| &\lesssim \frac{1}{|\ln \nu|^2} \big(\|P_{1, \nu}\|_{L^2_{\frac{\omega_\nu}{\zeta}}} + \|P_{2,\nu}\|_{L^2_{\frac{\omega_\nu}{\zeta}}} \big) \lesssim \frac{1}{|\ln \nu|^2}\left( \nu^2 |\ln \nu| + \frac{\nu^2}{|\ln \nu|} \right) \lesssim \frac{ \nu^2}{|\ln \nu|}. 
\end{align*}
As for the projection on $\phi_{0,\nu}$, we again use the relation $8T_0 = r^2 U$ to obtain a better estimate. By integration by parts and $\phi_{0, \nu} = \frac{1}{\nu^2}T_0(r) + \tilde{\phi}_{0,\nu}$, we write 
 $$\left\langle \frac{\pa_\zeta (P_\nu^2)}{2\zeta}, \phi_{0,\nu} \right\rangle = \left\langle P_\nu, \frac{\zeta}{\omega_\nu} P_\nu \pa_\zeta \Big((\frac{1}{\nu^2}T_0(r) + \tilde\phi_{0,\nu}) \frac{\omega_\nu}{2\zeta^2} \Big) \right\rangle,$$
Applying \eqref{est:proPhi0nu} with $m_\e = P_{2,\nu}$ yields 
$$\left|\left\langle P_{2,\nu}, \frac{\zeta}{\omega_\nu} P_{\nu} \pa_\zeta \Big(\tilde\phi_{0,\nu} \frac{\omega_\nu}{\zeta^2} \Big) \right\rangle \right| \lesssim \frac{1}{|\ln \nu|^2} \|P_{2,\nu}\|_{L^2_{\frac{\omega_\nu}{\zeta}}} \lesssim  \frac{\nu^2}{|\ln \nu|^3}.$$
A similar estimate for \eqref{est:proPhi0nu}, by using $|P_{1,\nu}(\zeta)| \lesssim \nu^2 \ln \la r\ra$, yields the estimate
$$\left|\left\langle P_{1,\nu}, \frac{\zeta}{\omega_\nu} P_{\nu} \pa_\zeta \Big(\tilde\phi_{0,\nu} \frac{\omega_\nu}{\zeta^2} \Big) \right\rangle \right| \lesssim \nu^4 |\ln \nu|^C \lesssim \frac{\nu^2}{|\ln \nu|^3}.$$
From \eqref{eq:TOcom} and \eqref{est:Prough}, we have 
\begin{align*}
\left|\left\langle P_\nu, \frac{\zeta}{\omega_\nu} P_\nu \pa_\zeta \Big(\frac{1}{\nu^2} T_0(r) \frac{\omega_\nu}{\zeta^2} \Big) \right\rangle \right|& = \left|\frac{\beta}{8}\int P_\nu^2 \zeta e^{-\beta \zeta^2/2}d\zeta \right|\\
&\lesssim \frac{\nu^2}{|\ln \nu|^4}\int \frac{\ln^2 \langle r \rangle}{\langle r \rangle^4}\langle \nu r\rangle^{2C}  r e^{-\beta \nu^2 r^2/ 2}dr \lesssim \frac{\nu^2}{|\ln \nu|^4}.
\end{align*}
We then conclude
\begin{equation}\label{est:Pnonlinear}
\left|\left\langle \frac{\pa_\zeta P_\nu^2}{2\zeta}, \phi_{n,\nu} \right\rangle\right| \lesssim \frac{\nu^2}{|\ln \nu|}, \quad \left|\left\langle \frac{\pa_\zeta P_\nu^2}{2\zeta}, \phi_{0,\nu} \right\rangle\right| \lesssim \frac{\nu^2}{|\ln \nu|^3}.
\end{equation}

\noindent \emph{- The nonlinear term}.  We write by integration by parts using $\phi_{0, \nu} = \frac{1}{\nu^2}T_0(r) + \tilde{\phi}_{0, \nu}$,
\begin{align*}
\Big \langle \frac{\pa_\zeta m_\e^2}{2\zeta},\phi_{n,\nu} \Big\rangle &= -\frac 12 \Big\langle m_\e^2, \frac{\zeta}{\omega_\nu} \pa_\zeta \Big(\frac{(\phi_{n,\nu} - \phi_{0, \nu})}{\zeta^2}\omega_\nu\Big)\Big\rangle  - \frac 12 \Big\langle m_\e^2, \frac{\zeta}{\omega_\nu} \pa_\zeta (\frac{\frac{1}{\nu^2}T_0(r) + \tilde{\phi}_{0, \nu}}{\zeta^2}\omega_\nu)\Big\rangle.
\end{align*}
From the degeneracy near the origin \eqref{eq:pointwisephin-phi0}, we have the rough bound
$$\left|\frac{\zeta}{\omega_\nu} \pa_\zeta \Big(\frac{(\phi_{n,\nu} - \phi_{0, \nu})}{\zeta^2}\omega_\nu\Big)\right| \lesssim \frac{1}{\nu^2} \frac{\ln \langle r \rangle}{\langle r \rangle^2} \la \zeta\ra^C,$$
and from the pointwise estimate \eqref{pointwisephi0}, 
$$\left|\frac{\zeta}{\omega_\nu} \pa_\zeta \Big(\frac{\tilde\phi_{0, \nu}}{\zeta^2}\omega_\nu\Big) \right| \lesssim \frac{1}{\nu^2}\frac{\ln \langle r \rangle}{\langle r \rangle^2}.$$
We directly get the bounds
$$\left|\Big\langle m_\e^2, \frac{\zeta}{\omega_\nu} \pa_\zeta \Big(\frac{(\phi_{n,\nu} - \phi_{0, \nu})}{\zeta^2}\omega_\nu\Big)\Big\rangle_{L^2_{\frac{\omega_\nu}{\zeta}}(\zeta \leq \zeta_*)} \right| + \left|\Big\langle m_\e^2, \frac{\zeta}{\omega_\nu} \pa_\zeta \Big(\frac{\tilde \phi_{0, \nu}}{\zeta^2}\omega_\nu\Big)\Big\rangle \right| \lesssim \frac{1}{\nu^2} \|m_\e\|^2_{L^2_{\frac{\omega_\nu}{\zeta}}}.$$
From the pointwise bound far away \eqref{pointwiseaway2} and the definition of $P_{2,\nu}$, we have $|m_\e(\zeta)| \lesssim \frac{\nu^2}{|\ln \nu|} \la \zeta \ra ^C$, from which we get 
$$\left|\Big\langle m_\e^2, \frac{\zeta}{\omega_\nu} \pa_\zeta \Big(\frac{(\phi_{n,\nu} - \phi_{0, \nu})}{\zeta^2}\omega_\nu\Big)\Big\rangle_{L^2_{\frac{\omega_\nu}{\zeta}}(\zeta \geq \zeta_*)} \right| \lesssim \frac{\nu^2}{|\ln \nu|^2}\int_{\zeta \geq \zeta_*} \zeta^C e^{-\beta \zeta^2/2} d\zeta \lesssim \frac{\nu^2}{|\ln \nu|^2}. $$
Using \eqref{eq:TOcom}, we compute
\begin{align*}
&\left|\Big\langle m_\e^2, \frac{\zeta}{\omega_\nu} \pa_\zeta (\frac{T_0(r)}{\nu^2 \zeta^2}\omega_\nu)\Big\rangle\right|= \left|\frac{\beta}{8\nu^4}\int m_\e^2 \zeta^2 U(r) \frac{\omega_\nu}{\zeta}d\zeta \right|= \left|\frac{\beta \nu^2}{8}\int m_q^2 r^2U(r) \frac{\omega}{r} dr \right| \lesssim \nu^2 \|m_q\|^2_{L^2_{\frac{\omega}{r}}} = \frac{1}{\nu^2}\|m_\e\|^2_{L^2_{\frac{\omega_\nu}{\zeta}}}.
\end{align*}
Collecting the above estimates yields the final bound for the nonlinear term
\begin{equation} \label{interm:modNL}
\left| \Big\langle \frac{\pa_\zeta m_\e^2}{2\zeta},\phi_{n,\nu}\Big\rangle \right| \lesssim \frac{1}{\nu^2}\|m_\e\|^2_{L^2_{\frac{\omega_\nu}{\zeta}}} + \frac{\nu^2}{|\ln \nu|^2} \lesssim \frac{K^2 \nu^2}{|\ln \nu|^2}.
\end{equation}

\noindent \emph{- The nonradial term}. We estimate from \eqref{poissonLinfty} and \eqref{est:PhinPointEst} and \eqref{est:xstar},
\begin{equation} \label{interm:modnonrad}
\left|\Big \langle N_0(\e^\perp),\phi_{n,\nu}\Big\rangle\right|\lesssim \frac{e^{- 2\kappa \tau}}{\nu^C} \lesssim \frac{\nu^2}{|\ln \nu|^2}.
\end{equation}
for some universal constant $C$ depending only on $N$, where we used in the second identity the bootstrap estimate \eqref{bootstrap:param1} and $\tau_0$ sufficiently large.\\

\noindent The conclusion follows by injecting \eqref{interm:modpatau}, \eqref{est:SmallLinear},  \eqref{interm:mod0}, \eqref{est:Pnonlinear}, \eqref{interm:modNL} and \eqref{interm:modnonrad} into the expression \eqref{interm:modexpr} and using \eqref{est:PhinL2norm1}. This ends the proof of Lemma \ref{lemm:Mod}.
\end{proof}

\subsection{Main energy estimate}\label{sec:mainEneEst}
This subsection is devoted to derive an energy estimate for the norm \eqref{bootstrap:L2omeofme} of $m_\e$ in $L^2_{\frac{\omega_\nu}{\zeta}}$. Taking into account the decomposition \eqref{eq:decmwt} and the smallness of the higher order part of the approximate perturbation $P_{2,\nu}$ in $L^2_{\frac{\omega_\nu}{\zeta}}$, i.e. $\|P_{2,\nu}\|_{L^2_{\frac{\omega_\nu}{\zeta}}} \lesssim \frac{\nu^2}{|\ln \nu|}$ (from \eqref{bootstrap:param2} and \eqref{est:PhinPointEst}), we will control instead of $m_\e$ the full higher order part of the perturbation:
$$\bar m_\e = m_\e + P_{2,\nu}.$$
Recall from \eqref{decomposition} and \eqref{eq:decmwt} the decomposition
$$m_w = Q_\nu + P_{1, \nu} + \bar m_\e \quad \textup{and} \quad  \bar m_\e = Q_{\tilde{\nu}} - Q_\nu +  P_{1, \tnu} - P_{1, \nu}  - N_{1, \tnu} + \tilde{m}_w,$$
and write from \eqref{partialmassselfsim} the equation satisfied by $\bar m_\e$,
\begin{equation}\label{eq:meMd}
\pa_\tau \bar m_\e = \bar \As^\zeta \bar m_\e + \frac{\pa_\zeta\big[(V + \tilde{m}_w) \bar m_\e\big]}{2\zeta} + \bar m_E + N_0(\e^\perp),
\end{equation}
where $N_0(\e^\perp)$ is defined as in \eqref{def:NL0} and 
$$V = P_{1, \nu} + P_{1, \tnu}  - N_{1, \tnu}.$$
Here, $\bar \As^\zeta$ is the modified operator  defined by
\begin{equation}\label{def:Asbar}
\bar \As^\zeta = \As^\zeta + \frac{\pa_\zeta \big((Q_{\tilde{\nu}} - Q_\nu) \cdot \big)}{2\zeta}  = \frac{1}{2} \big(\As^\zeta + \tilde{\As}^\zeta \big),
\end{equation}
where $\As^\zeta$ and $\tilde \As^\zeta$ are the linearized operator around $Q_\nu$ and $Q_{\tilde \nu}$ respectively, and the error $\bar m_E$ is:
\begin{equation}\label{dec:mEbar}
\bar m_E =  \sum_{n = 0}^1 \textup{Mod}_n \; \phi_{n, \nu}(\zeta) + \tilde{\bar m}_E + \frac{\pa_\zeta P_{1, \nu}^2}{2\zeta},
\end{equation}
(almost the same as \eqref{exp:mE} without taking into account the higher order approximate perturbation $P_{2,N}$) and the analogue of \eqref{expr:mEphi2} holds, namely
\begin{equation}\label{est:mEbarP1}
\| \tilde{\bar m}_E \|_{L^2_{\frac{\omega_\nu}{\zeta}}} + \left\|\frac{\pa_\zeta P_{1, \nu}^2}{2\zeta}\right\|_{L^2_{\frac{\omega_\nu}{\zeta}}} = \Oc\left( \frac{\nu^2}{|\ln \nu|}\right).
\end{equation}

The basic idea behind this modification is the ability of controlling the nonlinear term when performing the $L^2_{\frac{\omega_{\nu}}{\zeta}}$ energy estimate thanks to 
\begin{itemize}
\item[1.] The pointwise bound \eqref{pointwisemerefined} for $\tilde{m}_w$, which avoids the resonance $T_0$ of the operator $\Ls_0$ from the orthogonality condition \eqref{eq:orthmwT0}. Recall that $T_0$ is obtained by differentiating the rescaled stationary state $Q_\nu$ at $\nu = 1$, so it is natural to slightly modify the parameter function $\nu$ by $\tilde{\nu}$ to cancel out the component $T_0$. In particular, the orthogonality condition \eqref{eq:orthmwT0} allows us to derive the coercivity of $\As_0$ in Lemma \ref{lemm:coerA02}, which is a key ingredient in obtaining the control of $\tilde{m}_w$ as in \eqref{bootstrap:in}. 
\item[2.] The spectral gap of $\bar \As^\zeta$ still holds true under the orthogonality conditions \eqref{orthogonality} up to a sufficiently small error. The key feature is that this operator can be written as 
$$\bar \As^\zeta = \As_0^\zeta + \Pc - \beta \zeta \pa_\zeta \quad \textup{with} \quad \Pc = \frac{1}{2\zeta}\pa_\zeta \big(\Vc \cdot \big),\quad \Vc=(Q_\nu - Q_{\tilde{\nu}})$$
where the particular form of the perturbation $\Pc$ yields a cancellation as it is orthogonal to $T_0$ in $L^2(\omega_0/\zeta)$. Roughly speaking, the modified eigenvalues and eigenfunctions are the same up to some sufficiently small error, which is enough to obtain for $\bar \As^\zeta$ almost the same spectral gap as for $\As^\zeta$. 
\end{itemize}
We first claim the following spectral properties of $\bar \As^\zeta$.  

\begin{lemma}[Spectral gap for $\bar \As^\zeta$] \label{lemm:SpecAbar} There exists a universal $C'>0$ such that the following holds. Assume $1/2\leq \beta \leq 2$ and $|\nu-\tilde \nu|\leq C\nu/|\log \nu|$ for some $C>0$, and let $\bar \omega_\nu = \sqrt{\omega_\nu \omega_{\tilde{\nu}}}$. Fix any $N\in \mathbb N$ and assume the orthogonality condition \eqref{orthogonality}, i.e. $m_\e \perp \phi_{n,\nu}$ in $L^2_{\frac{\omega_\nu}{\zeta}}$ for $0 \leq n \leq N$. Then, for $\nu$ small enough:
\begin{equation} \label{bd:perturbativegap}
\int \bar m_\e \bar \As^\zeta \bar m_\e \frac{\bar \omega_\nu}{\zeta} d\zeta \leq -2\beta (N - C') \int \bar m_\e^2\frac{\bar \omega_\nu}{\zeta} d\zeta + C \sum_{n = 2}^N |a_n|^2 |\ln \nu|^2.
\end{equation}
\end{lemma}

Lemma \ref{lemm:SpecAbar} is a direct consequence of the following proposition whose proof is given in details in our previous work, and of the bound \eqref{est:nutildev} on $\nu-\tilde \nu$ whose proof is relegated at the end of this Subsection.
\begin{proposition}[Spectral properties of $\bar \As^\zeta$, \cite{CGNNarx19a}] \label{pr:spectralbarAzeta}
Assume the hypotheses of Proposition \ref{prop:SpecRad}, and that the function $\Vc$ in the operator $\Pc$ satisfies 
$$ |\Vc(\zeta)| + |\zeta\pa_\zeta \Vc(\zeta)| \lesssim \frac{\nu^2}{|\ln \nu|} \frac{\zeta^2 }{(\nu^2 + \zeta^2)^2},$$
Then, the operator $\bar \As^\zeta: H^2_{\frac{\bar \omega_\nu}{\zeta}} \rightarrow L^2_{\frac{\bar \omega_\nu}{\zeta}}$ is essentially self-adjoint with compact resolvant, where
$$\bar \omega_\nu(\zeta)=\omega_\nu(\zeta) \exp \left( \int_0^\zeta \frac{\Pc(\tilde \zeta)}{\tilde \zeta}d\tilde \zeta\right).$$ 
Its first $N+1$ eigenvalues $(\bar \alpha_n)_{0\leq n\leq N}$ satisfy
\begin{equation} \label{bd:stabilityalpha}
|\bar \alpha_n-\alpha_n|\leq \frac{C'}{|\log \nu|^2},
\end{equation}
and there exist associated renormalised eigenfunctions $(\bar \phi_{n,\nu})_{0\leq n\leq N}$ satisfying
\begin{equation} \label{bd:stabilityeigenmodes}
\frac{\| \bar \phi_{n,\nu}-\phi_{n,\nu} \|_{L^2(\frac{\omega_\nu}{\zeta})}}{\| \phi_{n,\nu} \|_{L^2(\frac{\omega_\nu}{\zeta})}}\leq C' \frac{\| \phi_{n,\nu}\|_{L^2(\frac{\omega_\nu}{\zeta})}}{\sqrt{|\log \nu|}}.
\end{equation}
\end{proposition}

\begin{remark} Using the fact that $\nu \sim \tilde{\nu}$ (see \eqref{est:nutildev}) and the explicit relation $\bar \omega_\nu = \left( \frac{\zeta^2 + \tilde{\nu}^2}{\zeta^2 + \nu^2}\right) \omega_\nu \sim \omega_\nu$, we have the following equivalence of weighted Lebesgue norms:
$$\| \cdot\|_{L^2_{\frac{\omega_\nu}{\zeta}}} \sim \| \cdot\|_{L^2_{\frac{\bar \omega_\nu}{\zeta}}}.$$
\end{remark}

\begin{proof}[Proof of Lemma \ref{lemm:SpecAbar}] The proof uses a standard orthogonal decomposition, and the spectral stability estimates provided by Proposition \ref{pr:spectralbarAzeta}. We recall that $\bar \As^\zeta$ is self-adjoint in $L^2_{\frac{\bar \omega_\nu}{\zeta}}$, with eigenvalues $\bar \alpha_n$ and eigenfunctions $\bar \phi_{n,\nu}$. We decompose $m_\e$ onto the first $N+1$ eigenmodes of $\bar \As^\zeta$:
\begin{equation}\label{dec:mebarphi}
\bar m_\e = \sum_{n = 0}^N b_n \bar \phi_{n,\nu} + \bar m_\e^\perp, \ \ \bar m_\e^\perp \perp \bar \phi_n \mbox{ in } L^2(\bar \omega_\nu/\zeta) \mbox{ for } 0\leq n\leq N,
\end{equation}
so that from Proposition \ref{pr:spectralbarAzeta} and the spectral Theorem, there holds the spectral gap:
$$\int \bar m_\e^\perp \bar \As^\zeta \bar m_\e^\perp \frac{\bar \omega_\nu}{\zeta}d\zeta \leq \bar \alpha_{N+1}\int \big|\bar m_\e^\perp\big|^2 \frac{\bar \omega_\nu}{\zeta}d\zeta \leq -2\beta (N-1) \int \big|\bar m_\e^\perp\big|^2 \frac{\bar \omega_\nu}{\zeta}d\zeta
 $$
where we used \eqref{bd:stabilityalpha} and took $\nu$ small enough in the last inequality. As the eigenfunctions $\bar \phi_{n, \nu}$ are mutually orthogonal in $L^2(\bar \omega_\nu/\zeta)$, we have from the above inequality:
\begin{align*}
\int \bar m_\e \bar \As^\zeta \bar m_\e \frac{\bar \omega_\nu}{\zeta} d\zeta &= \int \bar m_\e^\perp \bar \As^\zeta \bar m_\e^\perp \frac{\bar \omega_\nu}{\zeta}d\zeta + \sum_{n = 0}^N b_n^2 \bar \alpha_n \int |\bar \phi_{n, \nu}|^2 \frac{\bar \omega_\nu}{\zeta} d\zeta\\
&\quad \leq -2\beta (N-1) \int \bar m_\e^2 \frac{\bar \omega_\nu}{\zeta}d\zeta + \sum_{n = 0}^N b_n^2 | \alpha_n| \| \bar \phi_{n, \nu}\|^2_{L^2_{\frac{\bar \omega_\nu}{\zeta}}}.
\end{align*}
Above the parameters $b_n$ satisfy, from the orthogonality \eqref{orthogonality}, $|\omega_\nu-\bar \omega_\nu|\lesssim \omega_\nu/|\ln \nu|$ and \eqref{bd:stabilityeigenmodes}:
\begin{align*}
&b_n \|\bar \phi_{n, \nu} \|_{L^2_{\frac{\bar \omega_\nu}{\zeta}}}^{2}  = \int \bar m_\e \bar \phi_{n, \nu} \frac{\bar \omega_\nu}{\zeta} d\zeta\\
& \quad = \int \bar m_\e (\bar \phi_{n, \nu} - \phi_{n, \nu}) \frac{\bar \omega_\nu}{\zeta}d\zeta + \int \bar m_\e \phi_{n, \nu} \frac{\bar \omega_\nu - \omega_\nu}{\zeta} d\zeta + a_n \| \phi_{n, \nu}\|^2_{L^2_{\frac{\omega_\nu}{\zeta}}} \mathbf{1}_{\{2 \leq n \leq N\}}\\
& \quad \lesssim \|\bar m_\e\|_{L^2_{\frac{\omega_\nu}{\zeta}}} \left( \|\bar \phi_{n, \nu} - \phi_{n, \nu} \|_{L^2_{\frac{\omega_\nu}{\zeta}}} +  \left\| \phi_{n, \nu} \frac{\bar \omega_\nu - \omega_\nu}{\omega_\nu} \right\|_{L^2_{\frac{\omega_\nu}{\zeta}}}\right)+ |a_n| \| \phi_{n, \nu}\|^2_{L^2_{\frac{\omega_\nu}\zeta}} \mathbf{1}_{\{2 \leq n \leq N\}}\\
& \quad  \lesssim \frac{\|\bar m_\e\|_{L^2_{\frac{\omega_\nu}{\zeta}}}\| \phi_{n, \nu} \|_{L^2_{\frac{\omega_\nu}{\zeta}}}}{|\ln \nu|} + |a_n| \| \phi_{n, \nu}\|^2_{L^2_{\frac{\omega_\nu}\zeta}} \mathbf{1}_{\{2 \leq n \leq N\}}. 
\end{align*}
Estimate \eqref{bd:perturbativegap} for $\nu$ small enough then follows from the two above inequalities and \eqref{bd:stabilityeigenmodes}. This concludes the proof of Lemma \ref{lemm:SpecAbar} assuming Proposition \ref{pr:spectralbarAzeta}.
\end{proof}

We are now in the position to derive the main energy decay of $\bar m_\e$. 
\begin{lemma}[Monotonicity of $\bar m_\e$ in $L^2_{\frac{\bar \omega_\nu}{\zeta}}$] \label{lemm:L2omeControl} Let $w$ be a solution in the bootstrap regime in the sense of Definition \ref{def:bootstrap}. Then, for $\tau_0$ large enough the following estimate holds on $[\tau_0,\tau^*]$:
\begin{align}\label{eq:monoL2ome}
\frac{1}{2}\frac{d}{d\tau} \big\|\bar m_\e \big\|^2_{L^2_{\frac{\bar \omega_\nu}{\zeta}}} &\leq -2\beta\left(N - C \right)\big\|\bar m_\e \big\|^2_{L^2_{\frac{\bar \omega_\nu}{\zeta}}} + C\big\|\bar m_\e \big\|_{L^2_{\frac{\bar \omega_\nu}{\zeta}}}\left( \frac{|\textup{Mod}_0|}{\sqrt{|\ln \nu|}} +  |\textup{Mod}_1|\right) + C\left(\sum_{n = 2}^N |a_n|^2 |\ln \nu|^2 + \frac{\nu^4}{|\ln \nu|^2}\right),
\end{align}
where $C>0$ is a universal constant independent on the bootstrap constants $N$, $\kappa$, $K$, $K'$ and $K''$, and $\textup{Mod}_0$ and $\textup{Mod}_1$ are given as in \eqref{def:Mod0} and \eqref{def:Modn}.
\end{lemma}
\begin{proof} We multiply the equation \eqref{eq:meMd} with $\bar m_\e \frac{\bar \omega_\nu}{\zeta}$ and integrate over $[0,+\infty)$,
\begin{align*}
\frac{1}{2}\frac{d}{d\tau} \int \bar m_\e^2 \frac{\bar \omega_\nu}{\zeta}d\zeta & = \frac{1}{2}\int \bar m_\e^2 \frac{\pa_\tau \bar \omega_\nu}{\zeta} d\zeta + \int \bar m_\e\bar \As^\zeta \bar m_\e\frac{\bar \omega_\nu}{\zeta}d\zeta + \int \bar m_\e \bar m_E\frac{\bar \omega_\nu}{\zeta}d\zeta \\
&+ \int \frac{\pa_\zeta(V + \tm_w) \bar m_\e)}{2\zeta} \bar m_\e \frac{\bar \omega_\nu}{\zeta}d\zeta + \int N_0(\e^\perp) \bar m_\e \frac{\bar \omega_\nu}{\zeta}d\zeta.
\end{align*}
In the following, we shall write $\la \cdot, \cdot \ra $ for $\la \cdot, \cdot \ra_{L^2_{\frac{\bar \omega_\nu}{\zeta}}}$ for simplicity.\\
\noindent \textit{- The time derivative term.} We first compute 
\begin{align*}
\frac{\partial_\tau \bar \omega_\nu}{\bar \omega_\nu}  = \pa_\tau \ln \bar \omega_\nu =  \frac{1}{2}\frac{\nu_\tau}{\nu} \frac{1}{1 + r^2} + \frac{1}{2}\frac{\tilde\nu_\tau}{\tilde \nu} \frac{1}{1 + \tilde r^2} - \beta_\tau \frac{\zeta^2}{2}. 
\end{align*}
We obviously have the bound 
$$\left|\Big\la \bar m_\e^2, \frac{\nu_\tau}{\nu} \frac{1}{1 + r^2} + \frac{\tilde \nu_\tau}{\tilde \nu} \frac{1}{1 + \tilde r^2} \Big \ra\right| \leq \left(\left|\frac{\nu_\tau}{\nu}\right| + \left|\frac{\tilde \nu_\tau}{\tilde \nu}\right| \right) \|\bar m_\e \big\|^2_{L^2_{\frac{\bar \omega_\nu}{\zeta}}},$$
and 
$$ \quad \left|\Big\la \bar m_\e^2, -\beta_\tau \frac{\zeta^2}{2}\Big \ra_{L^2_{\frac{\bar \omega_\nu}{\zeta}}(\zeta \leq 1)}\right| \leq |\beta_\tau| \|\bar m_\e \big\|^2_{L^2_{\frac{\bar \omega_\nu}{\zeta}}}.$$
For $\zeta > 1$, we use \eqref{pointwiseaway2} to get 
\begin{equation} \label{est:pointwiseawayme}
|\bar m_\e(\zeta)| \lesssim \frac{\nu^2}{|\ln \nu|} \la \zeta \ra^C \quad \textup{for} \quad \zeta \geq 1,
\end{equation} 
from which we obtain
$$ \left|\Big\la \bar m_\e^2, \beta_\tau \frac{\zeta^2}{2}\Big \ra_{L^2_{\frac{\omega_\nu}{\zeta}}(\zeta \geq 1)}\right| \lesssim |\beta_\tau| \frac{\nu^4}{|\ln \nu|^2} \int_{0}^{+\infty} \zeta ^C e^{-\frac{\beta \zeta^2}{2}} d\zeta \lesssim \frac{\nu^4}{|\ln \nu|^2}.$$
Hence, 
\begin{equation}\label{est:L2ome1}
\left|\int \bar m_\e^2 \frac{\pa_\tau \bar \omega_\nu}{\zeta} d\zeta\right| \lesssim \left( \left| \frac{\nu_\tau}{\nu}\right| + \left| \frac{\tilde \nu_\tau}{\tilde \nu}\right| + |\beta_\tau| \right)\|\bar m_\e \big\|^2_{L^2_{\frac{\omega_\nu}{\zeta}}} + \frac{\nu^4}{|\ln \nu|^2}.
\end{equation}

\noindent \textit{- The linear and error terms.} We have by Lemma \ref{lemm:SpecAbar},
$$\big \la \bar \As^\zeta \bar m_\e, \bar m_\e \big \ra \leq -\beta(N - C')\|\bar m_\e \big\|^2_{L^2_{\frac{\bar \omega_\nu}{\zeta}}} + C \sum_{n = 2}^N |a_n|^2 |\ln \nu|^2.$$
From the decomposition \eqref{dec:mEbar} and the fact that $\bar m_\e \perp \phi_{n, \nu}$ for $n = 0, 1$ in $L^2_{\frac{\omega_\nu}{\zeta}}$, we write
\begin{align*}
\big \la \bar m_\e, \bar m_E \big \ra =  \sum_{n = 0}^1 \textup{Mod}_n \Big \la \bar m_\e, \phi_{n, \nu} \frac{\omega_{\tilde{\nu}} - \omega_\nu}{\omega_{\tilde{\nu}}} \Big \ra  +  \big \la \bar m_\e, \tilde {\bar m}_E + \frac{P_1 \pa_\zeta P_1}{\zeta} \big \ra.
\end{align*}
From Lemma \ref{lemm:Mod} and \eqref{est:mEbarP1}, we have by Cauchy-Schwarz inequality,
\begin{align*}
\left| \big \la \bar m_\e, \bar m_E \big \ra\right| &\lesssim \|\bar m_\e \|_{L^2_{\frac{\bar \omega_\nu}{\zeta}}}\left(\sum_{n = 0}^1 \left|\textup{Mod}_n\right| \left\|\phi_{n, \nu} \frac{\omega_{\tilde{\nu}} - \omega_\nu}{\omega_{\tilde{\nu}}} \right\| _{L^2_{\frac{\bar \omega_\nu}{\zeta}}} + \left\|\tilde {\bar m}_E + \frac{P_1 \pa_\zeta P_1}{\zeta} \right\| _{L^2_{\frac{\bar \omega_\nu}{\zeta}}}\right)\\
& \quad \lesssim \|\bar m_\e \|_{L^2_{\frac{\bar \omega_\nu}{\zeta}}} \left( \frac{|\textup{Mod}_0|}{\sqrt{|\ln \nu|}} +  |\textup{Mod}_1| + \frac{\nu^2}{|\ln \nu|} \right). 
\end{align*}

\noindent \textit{- The small linear term.} From \eqref{pointwisehatpsi}, we have the bound
\begin{align*}
&\left|\frac{\pa_\zeta V}{2\zeta}\right| + \left| \frac{\zeta}{2 \bar \omega_\nu}\pa_\zeta \Big(\frac{V}{2\zeta^2}\bar \omega_\nu \Big) \right|\leq  C\left(\left| \frac{\pa_\zeta V}{\zeta} \right| + \left| \frac{V}{\zeta^2} \right|\right) \leq C \frac{\ln^2 \la r \ra}{1 + r^2} \leq C,
\end{align*}
from which and from an integration by parts, we obtain 
\begin{align*}
&\left|\Big \la \frac{\pa_\zeta \big(V \bar m_\e\big)}{2\zeta}, \bar m_\e  \Big \ra\right| \leq C\int \bar m_\e^2 \left(\left|\frac{\pa_\zeta V}{\zeta} \right| +  \left|\frac{\zeta}{\bar \omega_\nu} \pa_\zeta \left(\frac{V \bar \omega_\nu}{\zeta} \right)\right| \right) \frac{\bar \omega_\nu}{\zeta} d\zeta \leq C \|\bar m_\e\|^2_{L^2_\frac{\omega_\nu}{\zeta}}.
\end{align*}

\noindent \textit{- The nonlinear term.} We use integration by part and the pointwise estimates \eqref{pointwisemerefined} and \eqref{pointwiseaway} to get
\begin{align*}
&\left|\Big \la \frac{\pa_\zeta \big(\tilde{m}_w \bar m_\e\big)}{2\zeta}, \bar m_\e  \Big \ra\right| = \int \bar m_\e^2 \left(\left|\frac{\pa_\zeta (\tilde{m}_w)}{2\zeta} \right| +  \left|\frac{\zeta}{\bar \omega_\nu} \pa_\zeta \left(\frac{(\tilde{m}_w) \bar \omega_\nu}{\zeta} \right)\right| \right) \frac{\bar \omega_\nu}{\zeta} d\zeta\\
& \quad \lesssim \frac{1}{|\ln \nu|} \int \bar m_\e^2 \frac{\nu^2}{\zeta^2 + \nu^2}  \sqrt{\ln \la r \ra} \la \zeta \ra ^C \frac{\bar \omega_\nu}{\zeta} d\zeta \lesssim \frac{1}{|\ln \nu|}\int_{\zeta \leq 1} \bar m_\e^2 \frac{\bar \omega_\nu}{\zeta} d\zeta + \frac{\nu^6}{|\ln \nu|^2}\int_{\zeta \geq 1} \la \zeta \ra^{3C} e^{-\beta \zeta^2/2} d\zeta\\
& \quad \lesssim \frac{1}{|\ln \nu|}\|\bar m_\e\|^2_{L^2_\frac{\omega_\nu}{\zeta}} + \frac{\nu^6}{|\ln \nu|^2}.
\end{align*}

\noindent \emph{- The nonradial term}. Since the contribution from the nonradial term is exponentially small in $\tau$, we just need a rough estimate by splitting the integration  into two parts $\zeta \leq \zeta_*/2$ and $\zeta \leq \zeta_*/2$, then using the pointwise estimates  \eqref{poissonLinfty}, \eqref{pointwisemerefined}, \eqref{pointwiseaway} and \eqref{est:xstar} to get
$$
\left|\Big \langle N_0(\e^\perp),m_\e\Big\rangle\right|\lesssim e^{-\kappa \tau} \lesssim \frac{\nu^4}{|\ln 
\nu|^2}.
$$
The collection of the above estimates and the fact that $ \Big|\frac{\nu_\tau}{\nu}\Big| + \Big|\frac{\tilde \nu_\tau}{\tilde\nu}\Big| + |\beta_\tau| \lesssim \frac{1}{|\ln \nu|}$ from \eqref{eq:mod0}, \eqref{eq:mod1}, \eqref{bootstrap:condition} and \eqref{est:nutildev} yield the conclusion of Lemma \ref{lemm:L2omeControl}.
\end{proof}

In view of the monotonicity formula \eqref{eq:monoL2ome}, we need to estimate $\tnu$ and $\tnu_\tau$ in term of $\nu$. It is a consequence of the orthogonality condition \eqref{orthogonality} and of the rough estimate \eqref{bootstrap:L2omeofme}. 
\begin{lemma}[Estimate for $\tnu$] \label{lemm:nutil} We have 
\begin{equation}\label{est:nutildev}
|\tnu^2 - \nu^2| + |(\tnu^2 - \nu^2)_\tau| \lesssim \frac{K}{\sqrt{\ln M}} \frac{\nu^2}{|\ln \nu|}.
\end{equation}
\end{lemma}
\begin{proof} We recall from the decomposition \eqref{eq:decmwt}, 
$$\bar m_\e =  Q_{\tnu} - Q_\nu + P_{1, \tnu} - P_{1, \nu} - N_{1, \tnu} + \tilde{m}_w,$$
subject to the orthogonality conditions from \eqref{orthogonality} and \eqref{eq:orthmwT0},
$$\bar m_\e \perp_{L^2_{\frac{\omega_\nu}{\zeta}}} \phi_{0, \nu}, \phi_{1, \nu}, \quad \int \tilde{m}_w( \tnu \tr) \chi_{_M}(\tr) T_0(\tr) \frac{\omega_0(\tr)}{\tr} d\tr = 0.$$
Recall that $8T_0(r) = r^2 U(r)$, so the last condition is then written as 
\begin{align*}
0 &= \int \tm_w(\nu r) \chi_{_M}(r) r dr= \int \Big[\bar m_q(r) + Q(r) - Q(rh) + P_{1}(r) - P_{1}(rh) + N_1(rh) \Big] \chi_{_M}(r) r dr,
\end{align*}
where $\bar m_q(r) = \bar m_\e(\nu r)$ satisfies the estimate $\|\bar m_q\|_{L^2_{\omega/r}} \lesssim \frac{1}{|\ln \nu|}$, $P_1(r) = P_{1,\nu}(\nu r)$, and we write for short
$$h = \frac{\nu}{\tnu}. $$
A direct calculation yields
\begin{align*}
\int\big(Q(r) - Q(rh)\big)\chi_{_M}(r) rdr &= (1 - h^2)\int \frac{r^2}{(1 + h^2 r^2)(1 + r^2)} \chi_{_M}(r) rdr =  (1 - h^2) |\ln M|\big(1 + o_{M \to +\infty}(1) \big). 
\end{align*}
From Lemma \ref{lemm:EigAzeta} and the asymptotic behavior of $T_1$ given in \eqref{est:Tj}, we estimate for $r \leq 2M$, 
\begin{align*}
\big|P_1(r) - P_1(rh)\big| &= \left|2\beta a_1 \big( T_1(r) - T_1(rh)\Big)\right| + \Oc\left(\frac{|a_1|}{|\ln \nu|} \right)  \lesssim \nu^2 \left(|\ln h|  + \Oc\left(\frac{\ln^2 r}{r^2} \right) + \frac{1}{|\ln \nu|} \right),
\end{align*}
which gives the estimate
$$ \left|\int\Big(P_{1}(r) - P_{1}(rh)\Big)\chi_{_M}(r) rdr  \right| \lesssim \left(|\ln h| + \nu^2 |\ln \nu|^3 + \frac{1}{|\ln \nu|}\right) \lesssim |1 - h^2| + \frac{1}{|\ln \nu|}.$$
We have by Cauchy-Schwarz, 
$$\left|\int \bar m_q \chi_{_M}(r) rdr  \right| \lesssim \|\bar m_q\|_{L^2_{\omega/r}} \left(\int_0^{2M} \frac{r^3 dr}{(1 + r^2)^2} dr \right)^{\frac{1}{2}} \lesssim \frac{K\sqrt{|\ln M|}}{|\ln \nu|}.$$
As for the correction term, we have the estimate
\begin{align*}
\left|\int N_1(rh)\chi_M r dr \right| &=  \left|\int\As_0^{-1}\left(\frac{\pa_r P_1^2(rh)}{r}  + 8\beta \tnu^2 \tilde{\phi}_0(rh)\right) \chi_M r dr\right| \lesssim \tnu^4 \int_{r\leq 2M} r \ln \la r\ra dr  \lesssim \tnu^4 M^2 \ln M. 
\end{align*}
Gathering these estimates yields $|1- h^2| \lesssim \frac{K}{\sqrt{|\ln M|} |\ln \nu| }$, which implies the first estimate in \eqref{est:nutildev}. 
The estimate for the time derivative is similar by using the identity
\begin{align*}
0 &= \pa_\tau \int \tm_w(\nu r) \chi_{_M}(r) r dr= \pa_\tau \int \Big[\bar m_q(r) + Q(r) - Q(rh) + P_{1}(r) - P_{1}(rh) + N_1(rh) \Big)  \Big] \chi_{_M}(r) r dr.
\end{align*}
and the equation satisfied by $\bar m_q$, so we omit it here. This concludes the proof of Lemma \ref{lemm:nutil}.
\end{proof}

\subsection{Estimates at higher order regularity in the middle range}

We are now using standard parabolic regularity techniques to derive the $H^2$ control of $m_\e$ in the middle range $\zeta_* \leq \zeta \leq \zeta^*$. In particular, we claim the following. 
\begin{lemma}[$H^2$ control of $m_\e$ in the middle range] \label{eq:H2controlme} Let $w$ be a solution in the bootstrap regime in the sense of Definition \ref{def:bootstrap}. Then, we have the following bounds for $ \tau \in [\tau_0,\tau^*]$:
\begin{equation}\label{est:H2me}
\|m_{\e}\|_{H^2(\zeta_* \leq \zeta \leq \zeta^*)} \leq C(K, \zeta_*, \zeta^*)\frac{\nu^2}{|\ln \nu|}.
\end{equation}
\end{lemma}
\begin{proof} From the $L^2_{\frac{\hat{\omega}}{\zeta}}$ control of $m_\e$, we already have the estimate 
\begin{equation*}
\|m_{\e}\|_{L^2_{\tilde{\omega}}( \frac{\zeta_*}4 \leq \zeta \leq 4\zeta^*)} \leq C \|m_\e\|_{L^2_{\frac{\omega_\nu}{\zeta}}} \leq CK \frac{\nu^2}{|\ln \nu|} \quad \textup{for some}\quad C = C(\zeta_*, \zeta^*) > 0,
\end{equation*}
where $\tilde{\omega}(\zeta) = \zeta^3 e^{-\beta\zeta^2/2}$.

We shall rely on this estimate to derive bounds for the higher derivatives. This regularity procedure is standard, but we give it for the sake of completeness. Let us consider $k= 0, 1, 2$ and $\chi_k(\zeta)$ be a smooth cut-off function defined as 
$$0 \leq \chi_k(\zeta) \leq 1, \quad \chi_k(\zeta) = \left\{\begin{array}{cl} 1 & \quad \textup{for} \quad   (k+2)\zeta^*/4 \leq \zeta \leq (6-k)\zeta^*/4,\\
0 & \quad \textup{for}\quad \zeta \in (0,(k+2)\zeta_*/8) \cap ((6-k)\zeta^*/2, +\infty).
\end{array} \right.
$$
We also write for simplicity,
$$m_{\e, k}(\zeta, \tau) = \pa_\zeta^k m_\e(\zeta, \tau) \chi_k(\zeta) \quad \textup{for} \quad k = 0,1,2.$$
From equation \eqref{eq:me}, we see that $m_{\e, k}$ satisfy the equations
\begin{align}
\pa_\tau m_{\e,0} &= \As^\zeta m_{\e,0} + [\chi_0, \As^\zeta] m_\e + \Fc \chi_0, \label{eq:me0}\\
\pa_\tau m_{\e,1} &= \As^\zeta m_{\e,1} + [\chi_1, \As^\zeta]\pa_\zeta m_\e  + \big[\pa_\zeta, \As^\zeta\big] m_\e \chi_1 + \pa_\zeta \Fc \chi_1, \label{eq:me1} 
\end{align}
\begin{align}\label{eq:me2}
\pa_\tau m_{\e,2} = \As^\zeta m_{\e,2}  &+ [\chi_2, \As^\zeta] \pa_\zeta^2 m_\e + \big[\pa_\zeta, \As^\zeta\big] \pa_\zeta m_\e \chi_2   +  \pa_\zeta \big[\pa_\zeta, \As^\zeta\big] m_\e \chi_2 +  \pa_\zeta^2 \Fc \chi_2, \nonumber
\end{align}
where $\As^\zeta$, introduced in \eqref{def:AsAs0zeta}, is rewritten as
$$\As^\zeta  = \frac{1}{\tilde{\omega}}\pa_\zeta\Big(\tilde{\omega}\pa_\zeta \Big) + \Pc_0  \quad \textup{with}  \quad \Pc_0 = -\frac{4\nu^2}{\zeta(\nu^2 + \zeta^2)}\pa_\zeta  + U_\nu, \quad  \tilde{\omega}(\zeta) = \zeta^3e^{-\frac{\beta \zeta^2}{2}}.$$
The commutators are defined as
\begin{align}
\big[ \chi_k, \As^\zeta \big] = -2\chi_k' \pa_\zeta - \chi_k'' + \left(\frac{1}{\zeta} - \frac{Q_\nu}{\zeta} +\beta \zeta \right)\chi_k',\\
\big[\pa_\zeta, \As^\zeta\big] = - \pa_\zeta\left(\frac{1 - Q_\nu + \beta \zeta^2}{\zeta} \right)\pa_\zeta + \pa_\zeta U_\nu ,
\end{align}
and the source term is given by
\begin{align*}
\Fc =  \frac{\pa_\zeta \big[ (2P_\nu + m_\e) m_\e\big]}{2\zeta} + m_E + N_0(\e^\perp).
\end{align*}
From the second estimate in \eqref{exp:mEpr} and the bootstrap estimates in Definition \ref{def:bootstrap}, we arrive at the following bound:
\begin{equation}
\big\| \Fc \chi_0\big\|_{L^2_{\tilde{\omega}}} \leq C\frac{\nu^2}{|\ln \nu|}.
\end{equation}
Integrating \eqref{eq:me0} against $m_{\e,0}\tilde{\omega}$ yields the energy identity
\begin{align*}
\frac{1}{2}\frac{d}{d\tau}\|m_{\e,0}\|^2_{L^2_{\tilde{\omega}}}& = -\|\pa_\zeta m_{\e,0}\| ^2_{L^2_{\tilde{\omega}}}+ \frac{\beta_\tau}{2}\|m_{\e,0} \zeta\|^2_{L^2_{\tilde{\omega}}} + \int_0^{+\infty} \Big(\Pc m_{\e,0} + [\chi_0, \As^\zeta] m_\e + \Fc \chi_0 \Big) m_{\e,0} \tilde{\omega} d\zeta.
\end{align*}
Using the fact that $|\beta_\tau \zeta| \lesssim \frac{\zeta^*}{|\ln \nu|^3}$, $\left|\frac{\nu^2}{\zeta(\zeta^2  +\nu^2)} \right| + |U_\nu(\zeta)| \lesssim \nu^2$ for $\zeta_*/4 \leq \zeta \leq 4\zeta^*$ and Young inequality yields
\begin{align*}
&\frac{|\beta_\tau|}{2}\|m_{\e,0} \zeta\|^2_{L^2_{\tilde{\omega}}} + \left|\int_0^{+\infty} \Pc_0 m_{\e,0} m_{\e,0} \tilde{\omega}d\zeta \right|+  \left|\int_0^{+\infty} \Fc \chi_0 m_{\e,0} \tilde{\omega} d\zeta\right|\\
& \qquad \lesssim \nu^2 \|\pa_\zeta m_{\e,0} \zeta\|^2_{L^2_{\tilde{\omega}}} + \|m_{\e} \chi_0\|_{L^2_{\tilde{\omega}}} + \|\Fc \chi_0\|^2_{L^2_{\tilde{\omega}} }\lesssim \nu^2 \|\pa_\zeta m_{\e,0}\|^2_{L^2_{\tilde{\omega}}} +  \frac{CK\nu^4}{|\ln \nu|^2}.
\end{align*}
Using integration by parts and Cauchy Schwarz inequality with $\epsilon$, we have 
$$\left|\int_0^{+\infty} [\chi_0, \As^\zeta] m_\e m_{\e,0} \tilde{\omega} d\zeta \right| \leq \frac{1}{4}\| \pa_\zeta m_{\e, 0}\|_{L^2_{\tilde{\omega}}}^2 + C\|m_\e \chi_0\|^2_{L^2_{\tilde{\omega}}}.$$
Gathering these above estimate yields 
\begin{equation}
\frac{d}{d\tau}\|m_{\e,0}\|^2_{L^2_{\tilde{\omega}}} \leq -\|\pa_\zeta m_{\e,0}\|^2_{L^2_{\tilde{\omega}}} + \frac{CK \nu^4}{|\ln \nu|^2}. \label{id:Em0}
\end{equation}
Similarly, we integrate equations \eqref{eq:me1} and \eqref{eq:me2} against $m_{\e,1} \tilde{\omega}$ and $m_{\e,2} \tilde{\omega}$ respectively, then use integration by parts and Cauchy-Schwarz inequality with $\epsilon$ and note that $\| m_{\e,1}\|^2_{L^2_{\tilde{\omega}}} \leq \|\pa_\zeta m_{\e,0}\|^2_{L^2_{\tilde{\omega}}} + \frac{CK \nu^4}{|\ln \nu|^2}$ and $\| m_{\e,2}\|^2_{L^2_{\tilde{\omega}}} \leq \|\pa_\zeta m_{\e,1}\|^2_{L^2_{\tilde{\omega}}} + \frac{CK \nu^4}{|\ln \nu|^2}$ by definition, to derive the following energy identities 
\begin{align}
\frac{d}{d\tau}\|m_{\e,1}\|^2_{L^2_{\tilde{\omega}}} & \leq -\|\pa_\zeta m_{\e,1}\|^2_{L^2_{\tilde{\omega}}} + C_1\|\pa_\zeta m_{\e,0}\|^2_{L^2_{\tilde{\omega}}} + \frac{CK \nu^4}{|\ln \nu|^2}, \label{id:Ee1}\\
\frac{d}{d\tau}\|m_{\e,2}\|^2_{L^2_{\tilde{\omega}}} & \leq -\|\pa_\zeta m_{\e,2}\|^2_{L^2_{\tilde{\omega}}} + C_2\|\pa_\zeta m_{\e,1}\|^2_{L^2_{\tilde{\omega}}} + \frac{CK \nu^4}{|\ln \nu|^2}. \label{id:Ee2}
\end{align}
By summing up \eqref{id:Em0}, \eqref{id:Ee1} and \eqref{id:Ee2}, we arrive at
\begin{align*}
&\frac{d}{d\tau} \left(\|m_{\e,0}\|^2_{L^2_{\tilde{\omega}}} + \frac{1}{2C_1}\|m_{\e,1}\|^2_{L^2_{\tilde{\omega}}} + \frac{1}{4C_1C_2}\|m_{\e,2}\|^2_{L^2_{\tilde{\omega}}} \right)\\
& \qquad  \leq -\frac{1}{2}\|\pa_\zeta m_{\e,0}\|^2_{L^2_{\tilde{\omega}}} -\frac{1}{4C_1}\|\pa_\zeta m_{\e,1}\|^2_{L^2_{\tilde{\omega}}} -\frac{1}{4C_1C_2}\|\pa_\zeta m_{\e,2}\|^2_{L^2_{\tilde{\omega}}}  + \frac{CK \nu^4}{|\ln \nu|^2}.
\end{align*}
Using the above differential inequality and Poincar\'e inequality, we integrate in time to obtain the desire conclusion.  This ends the proof of Lemma \ref{eq:H2controlme}.
\end{proof}

\begin{lemma}[$H^2$ control of $\e^\perp$ in the middle range] \label{eq:H2controleperp} Let $w$ be a solution in the bootstrap regime in the sense of Definition \ref{def:bootstrap}. Then, we have the following bounds for $ \tau \in [\tau_0,\tau^*]$:
\begin{equation}\label{est:H2eperp}
\|\e^\perp\|_{H^2(\zeta_* \leq \zeta \leq \zeta^*)} \leq C(K, \zeta_*, \zeta^*)e^{-\kappa \tau}.
\end{equation}
\end{lemma}
\begin{proof} From the bootstrap bound \eqref{bootstrap:perpenergy} of $\e^\perp$ and the equivalence of the norm \eqref{eq:equive0}, we already have the estimate 
\begin{equation*}
\|\e^\perp\|_{L^2( \frac{\zeta_*}4 \leq \zeta \leq 4\zeta^*)} \leq C \|\e^\perp\|_0 \leq CK e^{-\kappa \tau} \quad \textup{for some}\quad C = C(\zeta_*, \zeta^*) > 0.
\end{equation*}
from which and a standard parabolic regularity argument as for the proof of Lemma \ref{eq:H2controlme} yields the conclusion of Lemma \ref{eq:H2controleperp}.
\end{proof}

\subsection{Higher order regularity energy estimates in the inner zone}
This section is devoted to the control of $\|\tm_w\|_{\inn}$. The basic idea is that the scaling term $\beta \zeta \pa_\zeta \tm_w$ is regarded as a small perturbation of $\As^\zeta$ in the blowup zone $\zeta \leq \zeta_*$ for some fixed small $\zeta_* \ll 1$, namely that the dynamics resembles $\pa_\tau \tm_w = \As_0^\zeta \tm_w$. Since $T_0$ spans the kernel of $\As_0$, we need to rule it out by imposing the local orthogonality condition \eqref{eq:orthmwT0}. It allows us to use the key ingredient, the coercivity of $\As_0$ (see Lemma \ref{lemm:coerA0} below), for establishing the monotonicity formula of $\|\tm_w\|_{\inn}$.

From the equation \eqref{eq:pamblowup} and  the decomposition \eqref{dec:vtr}, $\tm_v$ satisfies the equation 
\begin{equation}\label{eq:mvtil}
\pa_s \tilde{m}_v =  \As_0 \tilde{m}_v + \tilde \eta \; \tr \pa_{\tr} \tm_v + \frac{\pa_{\tr} \big[(\tilde P_1 - \tilde{N}_1) \tilde{m}_v\big]}{\tr} + \frac{\pa_{\tr}\big(\tilde{m}_v^2\big)}{2\tr} + \Ec_1(\tr, s),
\end{equation}
where $\As_0$ is defined as in \eqref{def:AsAs0} but here acting on the variable $\tr$ instead of $r$ (the reader should bear in mind that $\tm_v$ is a function of the $\tr$ variable), and we write for simplicity
$$\tilde\eta = \frac{\tnu_s}{\tnu} - \beta \tnu^2 = \Oc(\nu^2),$$
and the error is given by 
\begin{align}
\Ec_1 &= \big(-\pa_s + \As_0\big) \big(\tilde P_1 - \tilde{N}_1\big) +  \tilde\eta \; \tr \pa_{\tr} \Big(Q + \tilde P_1 - \tilde{N}_1\Big) + \frac{\pa_{\tr} (\tilde{P}_1 - \tilde{N}_1)^2}{2\tr} + \tilde N_0(v^\perp).\label{def:errorEc1}
\end{align}
For a fixed small constant $\zeta_*$ with $0 < \zeta_* \ll 1$, we introduce
\begin{equation}
\tilde{m}_v^*(\tr, s) = \chi_{\frac{\zeta_*}{\tnu}}(\tr) \tilde{m}_v(\tr, s),
\end{equation}
where $\chi_{\frac{\zeta_*}{\tnu}}$ is defined as in \eqref{def:chiM}. We write from \eqref{eq:mvtil} the equation for  $\tilde{m}_v^*$,
\begin{equation} \label{eq:mwtilstar}
\pa_s \tilde{m}_v^* =  \As_0 \tilde{m}_v^* +  \tilde\eta \; \tr \pa_{\tr} \tm_v^*   + \frac{\pa_{\tr} \big((\tilde P_1 - \tilde{N}_1) \tilde{m}^*_v\big)}{\tr} + \frac{\pa_{\tr}\big(\tilde{m}^*_v \tilde{m}_v\big)}{2\tr} + \chi_{\frac{\zeta_*}{\tnu}}\Ec_1 + T(\tm_v),\nonumber
\end{equation}
where 
\begin{align} \label{def:Tmv}
T(\tm_v) &=  - \Big(\pa_{\tr}^2  + \pa_s \big)\chi_{\frac{\zeta_*}{\tnu}} \tilde{m}_v  + \pa_{\tr}\chi_{\frac{\zeta_*}{\tnu}} \left[ -2\pa_{\tr} + \frac{1}{\tr} -  \frac{\big[2Q + 2\tilde P_1 - 2\tilde{N}_1 + \tm_v\big]}{2\tr}  + \tilde \eta \; \tr   \right] \tilde{m}_v.
\end{align}
Thanks to the orthogonality condition \eqref{eq:orthmwT0}, we can use the coercivity estimate (see Lemma \ref{lemm:coerA0})
\begin{equation}\label{est:coer1}
-\int \tm_v^* \As_0 \tm_v^* \frac{\omega_0}{\tr}d\tr \geq \delta_1 \int \left(|\pa_{\tr}\tm_{v}^*|^2  + \frac{|\tm_{v}^*|^2}{1 + \tr^2}\right) \frac{\omega_0}{\tr} d\tr ,
\end{equation}
and (see Lemma \ref{lemm:coerA02})
\begin{equation}\label{est:coer2}
\int |\As_0 \tm_{v}^*|^2 \frac{\omega_0}{\tr}d\tr \geq \delta_0 \int \left(|\pa_{\tr}^2\tm_{v}^*|^2 +  \frac{|\pa_{\tr}\tm_{v}^*|^2}{ \langle \tr \rangle^2} + \frac{|\tm_v^*|^2}{ \langle \tr \rangle^4(1 + \ln^2 \la \tr \ra )} \right) \frac{\omega_0}{\tr} d\tr.
\end{equation}
Since the support of $\tm_v^*$ is $\tr \leq \frac{2\zeta_*}{\tnu}$, we have by \eqref{est:coer2} the control
\begin{equation}
\int |\As_0 \tm_{v}^*|^2 \frac{\omega_0}{\tr}d\tr \geq \frac{C\tnu^2}{\zeta_*^2}\left(-\int \tm_v^* \As_0 \tm_v^* \frac{\omega_0}{\tr}d\tr \right).
\end{equation}
Thanks to these coercivity estimates, we are able to establish the following monotonicity formula to control the inner norm. 

\begin{lemma}[Inner energy estimate] \label{lemm:Innercontrol} We have for $\tau_0$ large enough
\begin{equation}\label{eq:2ndEne}
\frac{d}{d\tau}\| \As_0 \tm_v^*\|_{L^2_\frac{\omega_0}{\tr}}^2 \leq C \left(\| \As_0 \tm_v^*\|_{L^2_\frac{\omega_0}{\tr}}^2  + \frac{1}{\zeta_*^2}\| m_\e\|^2_{H^2 (\zeta_* \leq \zeta \leq \zeta^*)} + \frac{\nu^4}{\zeta_*^2|\ln \nu|^2}\right),
\end{equation}
and 
\begin{equation}\label{eq:1stEne}
\frac{1}{2}\frac{d}{d\tau} \left[- \tnu^2\int \tm_v^* \As_0 \tm_v^* \frac{\omega_0}{\tr} \right]   \leq -\frac{1}{2} \|\As_0 \tm_v^*\|_{L^2_\frac{\omega_0}{\tr}} +   \frac{C'}{\zeta_*^2}\| m_\e\|^2_{H^2 (\zeta_* \leq \zeta \leq \zeta^*)} + \frac{C'\nu^4}{ \zeta_*^2|\ln \nu|^2},
\end{equation}
where $C$ and $C'$ are independent from the constants $\kappa,K, K', K''$ and $\zeta_*$. 
\end{lemma}
\begin{proof}

\noindent \textbf{Step 1} \emph{Energy estimate for the second order derivative}. We first prove \eqref{eq:2ndEne}. We integrate equation \eqref{eq:mvtil} against $\As_0^2 \tm_v^*$ in $L^2_{\frac{\omega_0}{\tr}}$,  and using the self-adjointness of $\As_0$ in $L^2_{\frac{\omega_0}{\tr}}$ and the fact that $\int f\As_0 f \frac{\omega_0}{\tr}d\tr  \leq 0$ (note that $T_0$ is in the kernel of $\As_0$ which is strictly positive on $(0, +\infty)$, so a standard Sturm-Liouville argument yields the non-negativity of $-\As_0$) to write the energy identity 
\begin{align}
&\frac{1}{2}\frac{d}{ds} \int |\As_0 \tm_v^*|^2 \frac{\omega_0}{\tr}d\tr \leq \int \As_0\big[\chi_{\frac{\zeta_*}{\tnu}} \Ec_1   + \Fc(\tm_v) + T(\tm_v)\big]\As_0 \tm_v^* \frac{\omega_0}{\tr} d\tr, \label{id:Ener1}
\end{align}
where 
$$\Fc(\tm_v) =  \tilde \eta \; \tr \pa_{\tr} \tm_v^*  + \frac{\pa_{\tr} \big((\tilde P_1 - \tilde{N}_1) \tilde{m}^*_v\big)}{\tr} + \frac{\pa_{\tr}\big(\tilde{m}^*_v \tilde{m}_v\big)}{2\tr}.$$

\noindent \underline{The error term:} We claim the following estimate for the error term
\begin{equation}\label{est:Ec1}
\int \left|\As_0 \big(  \Ec_1(\tr,s) \chi_{\frac{\zeta_*}{\tnu}}(\tr)\big) \right|^2\frac{\omega_0(\tr)}{\tr}d\tr\leq C \frac{\nu^8}{|\ln \nu|^2},
\end{equation}
for some positive constant $C$ which is independent from $K$, $K'$ and $K''$. Since the proof of \eqref{est:Ec1} is technical and a bit lengthy, we postpone it to the end of this section and continue the proof of the lemma. We have by Cauchy-Schwarz inequality, and 
\begin{align*}
\left|\int \As_0\big(\chi_{\frac{\zeta_*}{\tnu}} \Ec_1\big) \As_0 \tm_v^* \frac{\omega_0}{\tr} d\tr  \right| \lesssim \frac{\nu^4}{|\ln \nu|} \|\As_0 \tm_v^* \|_{L^2_\frac{\omega_0}{\tr}} \lesssim \nu^2\|\As_0 \tm_v^* \|_{L^2_\frac{\omega_0}{\tr}}^2 + \frac{\nu^6}{|\ln \nu|^2}. 
\end{align*}

\noindent \underline{The scaling term:} From the definition of $\As_0$, we compute 
\begin{equation}\label{eq:commutatorA0rdr}
\As_0 (\tr \pa_{\tr} f) = \tr \pa_{\tr} \As_0 f + [\As_0, \tr \pa_{\tr}] f, 
 \end{equation}
 with 
 $$
 \big[\As_0, r\pa_r\big] = 2\As_0  -r U\pa_r - \frac{2(r^2 - 1)}{r^2 + 1} U.
 $$
We then write by integration by parts and the coercivity estimate \eqref{est:coer2}, 
\begin{align*}
&\left|\tilde\eta\int \As_0\big(\tr \pa_{\tr} \tm_v^*\big) \As_0 \tm_v^* \frac{\omega_0}{\tr}d\tr\right|= \left|\tilde\eta \int \Big[\tr \pa_{\tr} \As_0 \tm_v^* + 2\As_0 \tm_v^* - rU \pa_{\tr} \tm_v^* - \frac{2(\tr^2 - 1)}{ \langle \tr \rangle ^2 } U \tm_v^* \Big]\As_0 \tm_v^* \frac{\omega_0}{\tr}d\tr \right|\\
& \lesssim \nu^2 \left[ \int |\As_0 \tm_v^*|^2 \pa_{\tr}(\omega_0) d\tr +  \|\As_0 \tm_v^* \|_{L^2_\frac{\omega_0}{\tr}} + \int \left(\frac{|\pa_{\tr} \tm_v^*|^2}{\langle \tr \rangle^2} + \frac{|\tm_v^*|^2}{\langle \tr \rangle^8}\right) \frac{\omega_0}{\tr}d\tr \right] \lesssim \nu^2\|\As_0 \tm_v^* \|_{L^2_\frac{\omega_0}{\tr}}.
\end{align*}

\noindent \underline{The small linear and nonlinear term:} By definition, we have
\begin{equation}\label{def:commAsfg}
\As_0 (fg) = g\As_0 f + [\As_0, g]f \quad \textup{with} \quad [\As_0, g] =  2\pa_{\tr}g \pa_{\tr} - \Big[1 - Q \Big]\frac{\pa_{\tr}g}{r}, 
\end{equation}
and
\begin{equation}\label{def:commAsdr}
\As_0(\pa_{\tr} f) = \pa_{\tr} \As_0 f + [\As_0, \pa_{\tr}] f, 
\end{equation}
with 
$$ [\As_0, \pa_{\tr}] = \left(\frac{3\tr^4 - 6\tr^2 - 1}{8\tr^2}\right)U \pa_{\tr} + \frac{4\tr}{(1  +\tr^2)}U.$$
This gives the formula
\begin{align*}
\As_0 \Big(\frac{\pa_{\tr} \big(\tm_v^* F \big)}{\tr}\Big) &= \As_0 \left( \frac{F}{\tr} \pa_{\tr} \tm_v^* + \frac{\pa_{\tr} F}{\tr} \tm_v^*  \right) = \As_0 (\pa_{\tr} \tm_v^*) \frac{F}{\tr} + \big[\As_0, \frac{F}{\tr} \big] \pa_{\tr} \tm_v^* + \As_0 \tm_v^* \frac{\pa_{\tr} F}{\tr} + \big[ \As_0, \frac{\pa_{\tr} F}{\tr} \big] \tm_v^*\nonumber \\
& = \Big[ \pa_{\tr} \As_0 \tm_v^* + \big[\As_0, \pa_{\tr}\big] \tm_v^* \Big] \frac{F}{\tr}+ \big[\As_0, \frac{F}{\tr} \big] \pa_{\tr} \tm_v^* + \As_0 \tm_v^* \frac{\pa_{\tr} F}{\tr} + \big[ \As_0, \frac{\pa_{\tr} F}{\tr} \big] \tm_v^*. \label{eq:comFmv}
\end{align*}
We then write by integration by parts
\begin{align*}
\int \As_0 \Big(\frac{\pa_{\tr} \big(\tm_v^* F \big)}{\tr}\Big) \As_0 \tm_v^* \frac{\omega_0}{\tr}d\tr &= \int |\As_0 \tm_v^*|^2 \left[-\frac{1}{2} \frac{\tr}{\omega_0} \pa_{\tr}\left(\frac{F \omega_0}{\tr^2} \right) + \frac{\pa_{\tr} F}{\tr}\right] \frac{\omega_0}{\tr}d\tr \\
& \quad + \int \left[\big[\As_0, \pa_{\tr}\big] \tm_v^* \frac{F}{\tr}+ \big[\As_0, \frac{F}{\tr} \big] \pa_{\tr} \tm_v^* + \big[ \As_0, \frac{\pa_{\tr} F}{\tr} \big] \tm_v^*\right] \As_0 \tm_v^* \frac{\omega_0}{\tr}d\tr.
\end{align*}
By the definition of the commutators, we observe 
$$\left|-\frac{1}{2} \frac{\tr}{\omega_0} \pa_{\tr}\left(\frac{F \omega_0}{\tr^2} \right) + \frac{\pa_{\tr} F}{\tr}\right| \lesssim \frac{|\pa_{\tr} F|}{\tr} + \frac{|F|}{\tr^2},$$
and 
\begin{align*}
&\left|\big[\As_0, \pa_{\tr}\big] \tm_v^* \frac{F}{\tr}+ \big[\As_0, \frac{F}{\tr} \big] \pa_{\tr} \tm_v^* + \big[ \As_0, \frac{\pa_{\tr} F}{\tr} \big] \tm_v^*\right|\\
&  \quad \lesssim \sum_{i = 0}^1 \frac{|\pa_{\tr}^i F|}{\tr^{2 - i}} |\pa_{\tr}^2 \tm_v^*| +\sum_{i = 0}^2 \frac{|\pa_{\tr}^i F|}{\tr^{3 - i}}  |\pa_{\tr} \tm_v^*|  + \left(\sum_{i = 1}^2 \frac{|\pa_{\tr}^i F|}{\tr^{4 - i}}  + \frac{|F|}{\tr \la\tr\ra^5}\right) | \tm_v^*|.
\end{align*}
Now, consider $F = 2\tilde{P}_1 - 2\tilde{N}_1 + \tm_v$, we estimate from the definition of $\tilde{P}_1$ and $\tilde{N}_1$ and the pointwise bound \eqref{pointwisemerefined} for $k = 0, 1$ and $\tr \leq \frac{2\zeta_*}{\tnu}$,
\begin{align*}
\frac{|\pa_{\tr} F|}{\tr} + \frac{|F|}{\tr^2} \lesssim \nu^2, \quad \big|(\tr \pa_{\tr})^k F\big| \lesssim \nu^2 |\ln \la \tr \ra |.
\end{align*}
Hence, we have by Cauchy-Schwarz and \eqref{est:coer2}, 
\begin{align*}
\left|\int \As_0 \Big(\frac{\pa_{\tr} \big(\tm_v^* F \big)}{\tr}\Big) \As_0 \tm_v^* \frac{\omega_0}{\tr}d\tr \right| \lesssim \nu^2\|\As_0 \tm_v^*\|^2_{L^2_\frac{\omega_0}{\tr}}.  
\end{align*}

\noindent \underline{The cut-off term:} From the expression \eqref{def:Tmv} of $T(\tm_v)$, we compute
\begin{align*}
\As_0 T(\tm_v) = -2\As_0\big( \pa_{\tr}\chi_{\frac{\zeta_*}{\tnu}} \pa_{\tr} \tm_v \big) + \As_0 \big(F \tm_v \big)&= -2\pa_{\tr}\chi_{\frac{\zeta_*}{\tnu}} \pa_{\tr} \As_0 \tm_v + F \As_0 \tm_v -2\pa_{\tr}\chi_{\frac{\zeta_*}{\tnu}} \big[\As_0, \pa_{\tr}\big] \tm_v \\
 & \qquad + \big[\As_0, \pa_{\tr}\chi_{\frac{\zeta_*}{\tnu}}\big] \pa_{\tr} \tm_v + \big[\As_0, F \big] \tm_v,
\end{align*}
where 
$$F =  - \Big(\pa_{\tr}^2  + \pa_s \big)\chi_{\frac{\zeta_*}{\tnu}}  + \pa_{\tr}\chi_{\frac{\zeta_*}{\tnu}} \left[ \frac{1}{\tr} -  \frac{\big[2Q + 2\tilde P_1 - 2\tilde{N}_1 + \tm_v\big]}{2\tr}  + \tilde \eta \; \tr   \right].$$
By definition, we have 
\begin{equation}\label{est:chisr}
\left|(\tr\pa_{\tr} + \tr^2 \pa_{\tr}^2)\chi_{\frac{\zeta_*}{\tnu}} \right| \lesssim \mathbf{1}_{\{\frac{\zeta_*}{\tnu} \leq \tr \leq \frac{2\zeta_*}{\tnu}\}}, \quad \left|\pa_s\chi_{\frac{\zeta_*}{\tnu}} \right| \lesssim \frac{\nu^2}{|\ln \nu|} \mathbf{1}_{\{\frac{\zeta_*}{\tnu} \leq \tr \leq \frac{2\zeta_*}{\tnu}\}},
\end{equation}
which gives the estimate 
\begin{equation}\label{est:Fcut}
|\pa_{\tr}^k F| \lesssim \left(\frac{1}{\tr^{2 + k}} + \left|\frac{\tnu_s}{\tnu} \right| \frac{1}{\tr^k}\right)\mathbf{1}_{\{\frac{\zeta_*}{\tnu} \leq \tr \leq \frac{2\zeta_*}{\tnu}\}} \lesssim \frac{1}{\tr^{2 + k}}\mathbf{1}_{\{\frac{\zeta_*}{\tnu} \leq \tr \leq \frac{2\zeta_*}{\tnu}\}}.
\end{equation}
Since the cut-off term is localized in the zone $\tr \sim \frac{\zeta_*}{\tnu}$, we can use the mid-range control \eqref{bootstrap:in} to obtain a better estimate. More precisely, we use the decomposition \ref{decomposition}, i.e.
$$\tm_v = Q(\tr) - Q(r) + \tilde{P}_1(\tr) - P_1(r) - \tilde{N}_1(\tr) - P_2(r) +  m_q.$$
A direct computation and $\nu \sim \tnu$ ( see Lemma \ref{lemm:nutil}) yield 
\begin{align*}
k = 0,1,2, \quad \big|(\tr \pa_{\tr})^k\big[Q(\tr) - Q(r)\big]\big| \lesssim \frac{1}{|\ln \nu|  \langle \tr \rangle^2 }.
\end{align*}
By the definition of $P_1$ and $\tilde{P}_1$ and $P_2$, we have the rough bound for $\frac{\zeta_*}{\tnu} \leq \tr \leq \frac{2\zeta_*}{\tnu}$,
\begin{align*}
k = 0,1,2, \quad \big|(\tr \pa_{\tr})^k\big[\tilde{P}_1(\tr) - P_1(r) - \tilde{N}_1(\tr) - P_2(r)\big]\big| \lesssim \frac{\nu^2}{|\ln \nu|}.
\end{align*}
Combining with  the mid-range estimate \eqref{bootstrap:mid} yields
\begin{equation}\label{est:As0mvmidrange}
\int |\As_0 \tm_v|^2 \mathbf{1}_{\{\frac{\zeta_*}{\tnu} \leq \tr \leq \frac{2\zeta_*}{\tnu}\}} \frac{\omega_0}{\tr}d\tr \lesssim \frac{\nu^4}{|\ln \nu|^2} + \|m_\e\|_{H^2 (\zeta_* \leq \zeta \leq \zeta^*)}^2.
\end{equation}
Using this estimate and integration by parts, we estimate 
\begin{align*}
&\left|\int \left(-2\pa_{\tr}\chi_{\frac{\zeta_*}{\tnu}} \pa_{\tr} \As_0 \tm_v + F \As_0 \tm_v \right) \As_0 \tm_v^* \frac{\omega_0}{\tr}d\tr \right|= \left| \int |\As_0 \tm_v|^2 \left[\frac{\tr}{\omega_0} \pa_{\tr} \Big(\pa_{\tr}\chi_{\frac{\zeta_*}{\tnu}} \frac{\omega_0}{\tr} \Big) +   F \right]  \frac{\omega_0}{\tr}d\tr \right|\\
& \qquad  \lesssim \int |\As_0 \tm_v|^2 \left(\frac{1}{\tr^2} + \left| \frac{\tnu_s}{\tnu}\right| \right) \mathbf{1}_{\{\frac{\zeta_*}{\tnu} \leq \tr \leq \frac{2\zeta_*}{\tnu}\}} \frac{\omega_0}{\tr}d\tr   \lesssim \frac{\nu^6}{ \zeta_*^2 |\ln \nu|^2} + \frac{\nu^2}{\zeta_*^2}\|m_\e\|_{H^2 (\zeta_* \leq \zeta \leq \zeta^*)}^2.
\end{align*}
We also have by the commutator formulas \eqref{def:commAsfg} and \eqref{def:commAsdr},
\begin{align*}
&\left|-2\pa_{\tr}\chi_{\frac{\zeta_*}{\tnu}} \big[\As_0, \pa_{\tr}\big] \tm_v + \big[\As_0, \pa_{\tr}\chi_{\frac{\zeta_*}{\tnu}}\big] \pa_{\tr} \tm_v + \big[\As_0, F \big] \tm_v \right| \lesssim  \sum_{i = 0}^2 \frac{|\pa_{\tr}^i \tm_v|}{\tr^{4 - i}} \mathbf{1}_{\{\frac{\zeta_*}{\tnu} \leq \tr \leq \frac{2\zeta_*}{\tnu}\}}. 
\end{align*}
from which and Cauchy-Schwarz and \eqref{est:As0mvmidrange}, we obtain 
\begin{align*}
&\left|\int \left( -2\pa_{\tr}\chi_{\frac{\zeta_*}{\tnu}} \big[\As_0, \pa_{\tr}\big] \tm_v + \big[\As_0, \pa_{\tr}\chi_{\frac{\zeta_*}{\tnu}}\big] \pa_{\tr} \tm_v + \big[\As_0, F \big] \tm_v  \right) \As_0 \tm_v \frac{\omega_0}{\tr}d\tr\right|\\
& \lesssim  \left(\int |\As_0 \tm_v|^2 \mathbf{1}_{\{\frac{\zeta_*}{\tnu} \leq \tr \leq \frac{2\zeta_*}{\tnu}\}} \frac{\omega_0}{\tr}d\tr \right)^\frac{1}{2} \left( \int \frac{1}{\tr^4} \Big[ |\pa_{\tr}^2 \tm_v|^2 + \frac{|\pa_{\tr} \tm_v|^2}{\tr^2} + \frac{|\tm_v|^2}{\tr^4}  \Big] \frac{\omega_0}{\tr} d\tr \right)^\frac{1}{2}\\
&\quad  \lesssim \frac{\nu^2}{\zeta_*^2} \left( \frac{\nu^4}{|\ln \nu|^2} +  \| m_\e\|_{H^2(\zeta_* \leq \zeta \leq \zeta^*)} \right).
\end{align*}
Injecting the above estimates into \eqref{id:Ener1} and using $\frac{ds}{d\tau} = \frac{1}{\tnu^2}$ and $\nu \sim \tnu$ yields \eqref{eq:2ndEne}.\\

\noindent \textbf{Step 2} \emph{Energy estimate for the first order derivative}. We now prove \eqref{eq:1stEne}. In order to handle the term $\|\As_0 \tm_v^*\|_{L^2_\frac{\omega_0}{\tr}}^2$ appearing in \eqref{eq:2ndEne}, we use the second energy identity by integrating equation \eqref{eq:mvtil} against $-\tnu^2 \As_0 \tm_v^*$ in $L^2_{\frac{\omega_0}{\tr}}$ to write
\begin{align*}
\frac{1}{2}\frac{d}{ds}\left[ -\tnu^2\int \tm_v^* \As_0 \tm_v^* \frac{\omega_0}{\tr} d\tr\right] & = -\tnu^2\int |\As_0 \tm_v^*|^2 \frac{\omega_0}{\tr}d\tr  - \tnu^2 \int \As_0 \big(\chi_{\frac{\zeta_*}{\tnu}} \Ec_1 \big) \tm_v^* \frac{\omega_0}{\tr}d\tr \\
& \quad - \tnu^2\int\Big[\Fc(\tm_v) + T(\tm_v) +\frac{\tnu_s}{\tnu} \tm_v^*  \Big]\As_0 \tm_v^* \frac{\omega_0}{\tr} d\tr.
\end{align*}
Since the support of $\tm_v^*$ is in the interval $0 \leq \tr \leq \frac{2\zeta_*}{\tnu}$, we have by Hardy inequality \eqref{coerciviteA0:hardy} and \eqref{est:coer2},
\begin{equation*}
\int \frac{|\tm_v^*|^2}{1 + \tr^2} \frac{\omega_0}{\tr}d\tr \lesssim \int|\pa_{\tr}\tm_v^*|^2 \frac{\omega_0}{\tr}d\tr \lesssim \frac{\zeta_*^2}{\tnu^2}\int \frac{|\pa_{\tr}\tm_v^*|^2}{1 + \tr^2} \frac{\omega_0}{\tr}d\tr \lesssim \frac{\zeta_*^2}{\tnu^2}\int |\As_0 \tm_v^*|^2 \frac{\omega_0}{\tr}d\tr,
\end{equation*}
from which and Cauchy-Schwarz and \eqref{est:Ec1}, we estimate
\begin{align*}
& \left|\tnu^2\int\As_0\big( \chi_{\frac{\zeta_*}{\tnu}} \Ec_1 \big)\tm_v^* \frac{\omega_0}{\tr} d\tr    \right| \lesssim \tnu^4\int \frac{|\tm_v^*|^2}{1 + \tr^2} \frac{\omega_0}{\tr} d\tr + \frac{\zeta_*^2}{\tnu^2} \int |\As_0 \big(\chi_{\frac{\zeta_*}{\tnu}} \Ec_1 \big)|^2\frac{\omega_0}{\tr} d\tr\\
& \quad  \lesssim  \zeta_*^2 \tnu^2  \int \frac{|\pa_{\tr} \tm_v^*|^2}{1 + \tr^2} \frac{\omega_0}{\tr} d\tr + \frac{\zeta_*^2}{\tnu^2} \frac{\nu^8}{|\ln \nu|^2}  \lesssim \zeta_*^2\tnu^2 \|\As_0 \tm_v^*\|^2_{L^2_\frac{\omega_0}{\tr}} + \frac{\zeta_*^2 \nu^6}{|\ln \nu|^2}.
\end{align*}
As for the scaling term, we simply estimate by using Cauchy-Schwarz and \eqref{est:coer2}, 
\begin{align*}
\left|\tnu^2 \tilde \eta \int \tr \pa_{\tr} \tm_v^* \As_0 \tm_v^*  \frac{\omega_0}{\tr} d\tr \right| &\lesssim \tnu^2 |\tilde \eta|\|\As_0 \tm_v^* \|^2_{L^2_\frac{\omega_0}{\tr}} \left( \int |\tr \pa_{\tr} \tm_v^*|^2  \frac{\omega_0}{\tr} d\tr \right)^{\frac{1}{2}} \lesssim \tnu^2 \zeta_*^2 \|\As_0 \tm_v^* \|^2_{L^2_\frac{\omega_0}{\tr}}. 
\end{align*}
By the definition of $\tilde{P}_1$, $\tilde{N}_1$ and \eqref{pointwisemerefined}, we have the rough bound for $\tr \leq \frac{2\zeta_*}{\tnu}$ and $k = 0, 1$,
$$ \big|(\tr \pa_{\tr})^k ((2\tilde{P}_1 - 2\tilde{N}_1 + \tm_v) \big| \lesssim \nu^2 |\ln \la \tr \ra|,$$
from which and \eqref{est:coer2}, we obtain
\begin{align*}
&\left|\tnu^2\int \frac{\pa_{\tr} \big[(2\tilde{P}_1 - 2\tilde{N}_1 + \tm_v) \tm_v^* \big]}{2\tr} \As_0 \tm_v^* \frac{\omega_0}{\tr} d\tr  \right| \lesssim \tnu^4 |\ln \nu|^C \|\As_0 \tm_v^* \|^2_{L^2_\frac{\omega_0}{\tr}}.
\end{align*}
To estimate for the cut-off term, we use \eqref{est:chisr} and \eqref{est:Fcut} to bound 
$$|T(\tm_v)| = \big|-2\pa_{\tr} \chi_{\frac{\zeta_*}{\tnu}} \pa_{\tr} \tm_v  + F \tm_v\big| \lesssim \left[\frac{1}{\tr} |\pa_{\tr} \tm_v| + \left(\frac{1}{\tr^2} + \frac{\tnu_s}{\tnu}\right)|\tm_v|  \right] \mathbf{1}_{\{ \frac{\zeta_*}{\tnu} \leq \tr \leq \frac{2\zeta_*}{\tnu}\}}.$$
We then use Cauchy-Schwarz and \eqref{est:As0mvmidrange} and $|\tnu_s/\tnu| \lesssim \nu^2/|\ln \nu|$ to estimate 
\begin{align*}
& \left|\tnu^2\int T(\tm_v) \As_0\tm_v^* \frac{\omega_0}{\tr}d\tr \right| \lesssim \tnu^2 \left( \int |\As_0  \tm_v|^2 \mathbf{1}_{\{ \frac{\zeta_8}{\tnu} \leq \tr\leq \frac{2 \zeta_*}{\tnu} \}} \frac{\omega_0}{\tr} d\tr \right)^\frac{1}{2} \left(\int \frac{1}{\tr^2} \Big[|\pa_{\tr} \tm_v|^2 + \frac{|\tm_v|^2}{\tr^2} \Big] \mathbf{1}_{\{ \frac{\zeta_*}{\tnu} \leq \tr \leq \frac{2\zeta_*}{\tnu}\}} \frac{\omega_0}{\tr}d\tr   \right)^\frac{1}{2} \\
&\lesssim \frac{\nu^2}{\zeta_*^2} \left( \frac{\nu^4}{|\ln \nu|^2} + \| m_\e\|_{H^2(\zeta_* \leq \zeta \leq \zeta^*)} \right). 
\end{align*}
Since $|\tnu_s/\tnu| \lesssim \nu^2/|\ln \nu|^2$, we use the bootstrap estimate \eqref{bootstrap:in} to bound the last term 
$$\left|\tnu \tnu_s \int \tm_v^* \As_0 \tm_v^* \frac{\omega_0}{\tr}d\tr  \right| \lesssim \frac{\nu^6}{|\ln \nu|^2}.$$
Summing up these estimates and taking $\zeta_*$ small enough and using $\frac{ds}{d\tau} = \frac{1}{\tnu^2}$ and $\tnu \sim \nu$ yield the desired formula \eqref{eq:1stEne}. This concludes  the proof of Lemma \ref{lemm:Innercontrol} assuming \eqref{est:Ec1}.\\

\noindent \textbf{Step 3} \emph{Bound for the error}. We finally prove \eqref{est:Ec1}. We recall from Proposition \ref{prop:SpecRad} the identity
$$(\As_0 - \beta \tnu^2 \tr \pa_{\tr}) \phi_{n}(\tr) = 2\beta \tnu^2\big(1 -n +\tilde{\alpha}_n(\tnu)\big) \phi_{n}(\tr). $$
We then rewrite the error term by using the relation $\tr\pa_{\tr} Q=8T_0 \sim 8\phi_0$ and $ds/d\tau=1/\tnu^2$, 
\begin{align*}
\Ec_1(\tr,s) & = \Big[-a_{1,\tau}+ 2\beta \tilde{\alpha}_1(\tnu) a_1 \Big]\phi_1+\Big[a_{1,\tau}-2\beta (1 + \tilde{\alpha}_0(\tnu))a_1+8\eta\Big] \phi_0\\
& \quad -  \frac{8 \tnu_s}{\tnu}\tilde \phi_0 +\frac{a_1}{\tnu^2}\left(\frac{\tnu_s}{\tnu} \tr\pa_{\tr} (\phi_1-\phi_0)-\pa_s(\phi_1-\phi_0)\right) + \tilde{\Ec}_1(\tr,s),
\end{align*}
where  we used the definition of $\tilde{N}_1$, i.e, $\As_0 \tilde{N}_1 = \frac{\pa_{\tr}\tilde{P}_1^2}{2\tr} + 8\tnu^2\tilde \phi_0(\tr)$ with $\tilde{\phi}_0 = \phi_0 - T_0$ to write 
$$\tilde{\Ec}_1(\tr,s) =  \big(\pa_s - \eta \tr \pa_{\tr}\big) \tilde{N}_1  + \frac{\pa_{\tr} (2\tilde{P}_1 \tilde N_1 - \tilde N_1^2 )}{2\tr}+ \tilde N_0(v^\perp).$$
We recall from Lemma \ref{lemm:Mod} the modulation equations
\begin{align*}
&a_{1, \tau} - 2\beta a_1 \tilde{\alpha}_1(\nu) = \Oc\left(\frac{\nu^2}{|\ln \nu|^2} \right),\qquad \left(a_{1,\tau}- 2\beta(1 + \tilde\alpha_0(\nu))a_1 + 8\nu^2(\frac{\nu_\tau}{\nu}-\beta)\right) = \Oc\left(\frac{\nu^2}{|\ln \nu|^2} \right),
\end{align*} 
from which and from the estimate \eqref{est:nutildev}, we obtain 
\begin{align*}
&|a_{1,\tau}-2\beta a_1\tilde \alpha_1(\tnu)| = \Oc\left(\frac{\nu^2}{|\ln \nu|^2} \right),\qquad \left| \left(a_{1,\tau}+ 2\beta(1 + \tilde\alpha_0(\tnu))a_1 - 8\tnu^2(\frac{\tnu_\tau}{\tnu}-\beta)\right) \right|= \Oc\left(\frac{\nu^2}{|\ln \nu|}\right). 
\end{align*}
From \eqref{pointwisephi0} and \eqref{est:poinwisephi1}, the identities $\As_0 T_0=0$ and $\As_0 T_1=-T_0$ and $\tnu \sim \nu$, one has the bound for $r\leq \zeta^*/\tnu$,
$$
|\As_0 \phi_0|+|\As_0 \phi_1|\lesssim \tnu^2 \frac{\tr^2}{ \langle \tr\rangle ^4},
$$
and hence, we have the estimate
\begin{align*}
&\int_0^{\frac{\zeta^*}{\tnu}} |\As_0 \left(\Big[-a_{1,\tau}+ 2\beta \tilde{\alpha}_1(\tnu) a_1 \Big]\phi_1+\Big[a_{1,\tau}-2\beta (1 + \tilde{\alpha}_0(\tnu))a_1+8\eta\Big] \phi_0\right) |^2 \frac{\omega_0}{\tr}d\tr  \lesssim \frac{\nu^8}{|\ln \nu|^4} \int_0^{\frac{\zeta^*}{\tnu}} \frac{1}{1+\tr}d\tr\lesssim \frac{\nu^8}{|\ln \nu|^3}.
\end{align*}
From \eqref{pointwisephi0} and $|\tnu_s/ \tnu| \lesssim \nu^2/|\ln \nu|$, we have the estimate
$$
\left|\frac{\tnu_s}{\tnu} \right|^2\int_0^{\frac{\zeta^*}{\tnu}} |\As_0 \tilde \phi_0|^2\frac{\omega_0}{\tr}d\tr\lesssim \frac{\nu^{8}}{|\ln \nu|^2}.
$$
By writing 
$$
\pa_s (\phi_1-\phi_0)=\frac{\tnu_s}{\tnu}(\tnu \pa_{\tnu} (\phi_1-\phi_0)+\frac{\beta_s}{\beta}(\beta\pa_\beta (\phi_1-\phi_0))
$$
and using \eqref{est:poinwisephi02} and \eqref{est:poinwisephi1},  we have the bound
$$
|\As_0 (\pa_s (\phi_1-\phi_0))|\lesssim \frac{\nu^4}{|\ln \nu|}\frac{\tr^2}{\langle \tr \rangle^4}.
$$
We also have from \eqref{pointwisephi0} and \eqref{est:poinwisephi1} and $|\tnu_s/\tnu|\lesssim 1/|\ln \nu|$ the bound
$$
\left| \frac{\tnu_s}{\tnu}\As_0 (\tr\pa_{\tr} (\phi_1-\phi_0))\right|\lesssim \frac{\nu^2}{|\ln \nu|} \frac{\tr^2}{\langle \tr \rangle^4}.
$$
Hence, there holds the estimate
\begin{align*}
& \int_0^{\frac{\zeta^*}{\tnu}} |\As_0 \frac{a_1}{\tnu^2}\left(\frac{\tnu_s}{\tnu}\tr\pa_{\tr} (\phi_1-\phi_0)-\pa_s(\phi_1-\phi_0)\right)  |^2\frac{\omega_0}{\tr}d\tr \lesssim \frac{\nu^{8}}{|\ln \nu|^2}\int_0^{\frac{\zeta^*}{\tnu}} \left(\frac{1}{(1+\tr)^3}+\frac{1}{|\ln \nu|^2(1+\tr)}\right)\lesssim \frac{\nu^8}{|\ln \nu|^2}.
\end{align*}
As for the term $\tilde{\Ec}_1$, we use the definitions of $\tilde{N}_1$ and $\tilde P_1$ and $\tnu \sim \nu$ and $|\tnu_s/\tnu| \lesssim \nu^2 /|\ln \nu|$ to bound for $\tr \leq \frac{2\zeta_*}{\tnu}$,
$$|\As_0 \tilde{\Ec}_1(\tr, s)| \lesssim \tnu^4 \left(\left|\frac{\tnu_s}{\tnu} \right| + \tnu^2  \right)\frac{\ln \la \tr \ra}{\la \tr \ra^2} \lesssim \nu^6 \frac{\ln \la \tr \ra}{\la \tr \ra^2},$$
which gives 
$$\int_0^{\frac{2\zeta_*}{\tnu}} |\As_0 \tilde{\Ec}_1|^2 \frac{\omega_0}{\tr}d\tr \lesssim \nu^{12} |\ln \nu|^3 \lesssim \frac{\nu^8}{|\ln \nu|^2}.$$
The collection of the above estimates concludes the proof of \eqref{est:Ec1}.
\end{proof}

\subsection{Nonradial energy estimates}
We begin with estimating the parameter $x^*$. We claim the following. 
\begin{lemma}[Estimate on $x^*$] \label{lemm:xstar} We have for $\tau_0$ large enough
\begin{equation} \label{est:xstarrefined}
\left|\frac{x^*_\tau}{\mu} \right|^2 \leq \frac{C}{\nu^2} \int_{|y| \leq \frac{2\zeta_*}{\nu}} \frac{|\nabla q^\perp|^2}{U} dy  + Ce^{-2\kappa\tau},
\end{equation}
where $C > 0$ is independent from the bootstrap constants $\kappa,N,K, K', K''$. In particular, we have the estimate
\begin{equation}\label{est:xstar}
\left|\frac{x^*_\tau}{\mu} \right| \lesssim \frac{e^{-\kappa \tau}}{\nu^2}.
\end{equation}
\end{lemma}
\begin{proof} We multiply equation \eqref{eq:epsbot} by $\pa_{z_i} U_\nu \rho_0$ with $i =1,2$ and $\rho_0 = e^{-\beta |z|^2/2}$, and integrate over $\Rb^2$ and use the orthogonality condition \eqref{orthogonality} to write
\begin{align*}
&\frac{x^*_{i,\tau}}{\mu} \int\pa_{z_i} \big( U_\nu + \Psi_\nu \big) \pa_{z_i} U_\nu \rho_0 dz   = \int \left[- \Ls^z \e^\perp +\nabla \cdot \Big(\Gc(\e)   - \frac{x^*_\tau}{\mu} \e^0\Big) - N^\perp(\e^\perp)\right] \pa_{z_i} U_\nu \rho_0 dz.
\end{align*}
Using the definition of $U_\nu$ and $\Psi_\nu$, we have 
$$\int|\pa_{z_i} U_\nu|^2 e^{-\beta\frac{|z|^2}{4}} dz = \frac{c_0}{\nu^4}\int \frac{r^3}{(1 + r^2)^6}e^{-\beta \nu^2 \frac{r^2}4} dr \sim \frac{c_0'}{\nu^4},$$
and 
$$\left|\int \pa_{z_i}\Psi_\nu \pa_{z_i} U_\nu e^{-\beta\frac{|z|^2}{4}} dz\right| = \Oc(\nu^{-2}).$$
We write the linear operator as $\Ls_0 q^\perp = \nabla \cdot (U \nabla \Ms q^\perp)$, then by integration by parts and Cauchy-Schwarz and the decay of $U$ we estimate 
\begin{align*}
&\left|\int \Ls^z \e^\perp \pa_{z_i} U_\nu \rho_0 dz \right| = \left|\frac{1}{\nu^5}\int \nabla \cdot \big(U \nabla \Ms q^\perp \big) \pa_{y_i} U \rho dy - \frac{\beta}{\nu} \int \nabla \cdot(y q^\perp) \pa_{y_i}U \rho dy \right|\\
& =\left| \frac{1}{\nu^5}\int U \nabla \Ms q^\perp \cdot \nabla (\pa_{y_i} U \rho) dy + \frac{\beta}{\nu}\int q^\perp y \cdot \nabla (\pa_{y_i} U \rho) dy  \right| \\
& \quad \lesssim \frac{1}{\nu^5} \left(\int_{|y| \leq \frac{\zeta_*}{2\nu}} U |\nabla \Ms q^\perp|^2 \rho dy   \right)^\frac{1}{2} \left(\int_{|y| \leq \frac{\zeta_*}{2\nu}} U |\nabla (\pa_{y_i} U \rho)|^2 \rho^{-1}dy \right)^\frac{1}{2} \\
& \qquad + \frac{1}{\nu^6}\left(\nu^2 \int_{|y| \geq \frac{\zeta_*}{2\nu}} \frac{|q^\perp|^2}{U} \rho dy \right)^\frac{1}{2} \left(\int_{|y| \geq \frac{\zeta_*}{2\nu}} \left|\Ms \Big[\nabla\cdot (U \nabla (\pa_{y_i} U \rho)) \Big]\right|^2 \rho^{-1} dy \right)^\frac{1}{2}  + e^{-\kappa \tau} \\
& \qquad  + \frac{1}{\nu^2}\left(\nu^2 \int \frac{|q^\perp|^2}{U} \rho dy \right)^\frac{1}{2} \left(\int \big|y \cdot \nabla (\pa_{y_i} U \rho)\big|^2 \rho^{-1} dy \right)^\frac{1}{2} \\
& \qquad \lesssim \frac{1}{\nu^5}\left(\int_{|y| \leq \frac{2\zeta_*}{\nu}} U|\nabla \Ms q^\perp|^2 \rho \right)^\frac{1}{2} + \frac{1}{\nu^6}\|\e^\perp\|_0 \, \nu^{10} + \frac{1}{\nu^2}\|\e^\perp\|_0  \lesssim \frac{1}{\nu^5}\left(\int_{|y| \leq \frac{2\zeta_*}{\nu}} \frac{|\nabla q^\perp|^2}{U} \rho \right)^\frac{1}{2} + \frac{e^{-\kappa \tau}}{\nu^2}.
\end{align*}
We write by the definition $\e^0 = \frac{\pa_\zeta m_\e}{\zeta}$ and use \eqref{pointwiseaway2},  \eqref{pointwiseawaynonradial} and the decay of $U$, 
\begin{align*}
&\left|\frac{x_{i,\tau}^*}{\mu} \int \pa_{z_i} \e \pa_{z_i}U_\nu \rho_0 dz \right| = \left|\frac{x_{i,\tau}^*}{\mu} \int \left(\frac{\pa_\zeta m_\e}{\zeta} + \e^\perp \right) \pa_{z_i}\big( \pa_{z_i}U_\nu \rho_0\big) dz \right| \\
& \quad \lesssim \left|\frac{x_{i,\tau}^*}{\mu} \right| \frac{1}{\nu^4} \left(\int |\pa_\zeta m_\e|^2 \frac{1 }{\zeta U_\nu} d\zeta \right)^\frac{1}{2} +  \left|\frac{x_{i,\tau}^*}{\mu} \right|\frac{e^{-\kappa\tau}}{\nu^3}\\
& \quad \quad \lesssim \left|\frac{x_{i,\tau}^*}{\mu} \right| \left(\frac{1}{\nu^4} \|\tilde{m}_v\|_\inn + \frac{1}{\nu^2|\ln \nu|}\int_{\zeta \geq \frac{\zeta_*}{2}} \zeta^C \rho_0 d\zeta + \frac{e^{-\kappa\tau}}{\nu^3} \right) \lesssim \left|\frac{x_{i,\tau}^*}{\mu} \right| \frac{1}{\nu^2 |\ln \nu|}. 
\end{align*}
We have by Cauchy-Schwarz,
\begin{align*}
&\left |\int \nabla \cdot \Gc(\e^\perp) \pa_{z_i}U_\nu \rho_0 dz\right| = \frac{1}{\nu^5} \left|\int \nabla \cdot \Gc(q^\perp) \pa_{y_i}U \rho dy \right|\\
  & \lesssim \left(\int \Big| \nabla \cdot \Gc(q^\perp)\Big|^2 \rho dy \right)^\frac{1}{2}\left(\int |\pa_{y_i}U|^2 \rho dy  \right)^\frac{1}{2} \lesssim \frac{1}{\nu^5}\left(\int \Big|\nabla \cdot \Gc(q^\perp) \Big|^2 \rho_0 dz \right)^\frac{1}{2},
\end{align*}
where we recall from the definition of $\Gc(\e^\perp)$, 
$$\nabla \cdot \Gc(q^\perp) = \nabla q^\perp \cdot \nabla \Phi_{\Phi + q^0}  - 2q^\perp (\Psi + q^0) + \nabla (\Psi + q^0)\cdot \nabla \Phi_{q^\perp},$$
and $\Psi$ is given by
$$\Psi = \frac{a_1}{\nu^2} \frac{\pa_r}{r}(\phi_1 - \phi_0) + \sum_{n = 2}^N \frac{a_n}{\nu^2} \frac{\pa_r}{r}\phi_n.$$
By the definition of $\phi_n$, we have the estimate 
\begin{equation}\label{est:Psir}
|\Psi(r)| + |r \pa_r \Psi(r)| \lesssim \nu^2 \la r\ra^{-2} + \frac{\la r \ra ^{-4}}{|\ln \nu|^2}\mathbf{1}_{\{r \leq \frac{\zeta_*}{2\nu}\}} + \frac{\nu^2}{|\ln \nu|^2} \la \nu r \ra^C \mathbf{1}_{\{r \geq \frac{\zeta_*}{2\nu}\}},
\end{equation}
from which and \eqref{pointwiseawaynonradial} and \eqref{poissonLinfty}, we obtain the bound 
\begin{equation*}
\int \Big|\nabla \cdot \Gc(q^\perp)\Big|^2 \rho dy \lesssim \int_{|y| \leq \frac{2\zeta_*}{\nu}} \left(\frac{|\nabla q^\perp|^2}{U} + \frac{|q^\perp|^2}{\sqrt{U}} \right) dy + e^{-2\kappa \tau}.
\end{equation*}
By using \eqref{bootstrap:perpenergy}, \eqref{pointwiseawaynonradial} and \eqref{poissonLinfty}, we estimate the contribution from $N^\perp(\e^\perp)$, 
\begin{equation*}
\int \Big|N^\perp (\e^\perp)\Big|^2 \rho_0 dz \lesssim e^{-2\kappa \tau}.
\end{equation*}
The collection of the above estimates yields \eqref{est:xstarrefined} and concludes the proof of Lemma \ref{lemm:xstar}.
\end{proof}

We are now in the position to derive the energy estimate for the nonradial part. Let us begin with $\|\e^\perp\|_0$ by using the coercivity given in Proposition \ref{pr:coercivitenonradial}.  
\begin{lemma}[Control of $\|\e^\perp\|_0$]\label{lemm:ebot0}  We have
\begin{equation} \label{ene:ebot0}
\frac{1}{2}\frac{d}{d\tau} \| \e^\perp\|_0^2 \leq - \delta' \| \e^\perp\|_0^2   + Ce^{-2\kappa \tau},
\end{equation}
where $\delta', C > 0$ are independent from the bootstrap constants $K, K', K''$. 
\end{lemma}
\begin{proof} We multiply equation \eqref{eq:epsbot} by $\nu^2 \sqrt{\rho_0}\Ms^z (\e^\perp \sqrt{\rho_0})$ with $\rho_0(z) = e^{-\beta |z|^2/2}$ and integrate over $\Rb^2$ to write 
\begin{align*}
\frac{1}{2} \frac{d}{d\tau} \| \e^\perp(\tau)\|_0^2  &=\nu^2\int \Ls^z \e^\perp \sqrt{\rho_0} \Ms^z (\e^\perp \sqrt{\rho_0}) dz  + \frac{1}{2}\int_{\Rb^2} \frac{d}{d\tau}(\nu^2U_\nu^{-1}) |\e^\perp|^2 \rho_0  dz \\
&  - \nu^2 \int \big[\nabla \cdot \Gc(\e^\perp) - \frac{x^*_\tau}{\mu} \cdot \nabla (W + \e^0) - N^\perp(\e^\perp)  \big] \sqrt{\rho_0}\Ms^z (\e^\perp \sqrt{\rho_0}) dz.
\end{align*}
We have the coercivity estimate \eqref{est:coerLnuy} for the linear part, which using \eqref{est:HaPo1} with $\alpha=1$ gives:
\begin{align}
\label{est:dissp} \nu^2\int \tilde \Ls^z \e^\perp \sqrt{\rho_0} \Ms^z (\e^\perp \sqrt{\rho_0}) dz &= \int \tilde \Ls q^\perp \sqrt{\rho} \Ms (q^\perp \sqrt{\rho}) dy \leq - \delta_0\int \frac{|\nabla q^\perp|^2}{U} \rho dy
\end{align}
and 
\begin{align} 
\label{eq:controle0bynq} -\int \frac{|\nabla q^\perp|^2}{U} \rho dy\leq -C_0 \nu^2\int \frac{|q^\perp|^2}{U}\rho dy=-C_0\nu^2\int \frac{|\e^\perp|^2}{U_\nu}\rho_0 dz =-C_0 \|\e\|^2_0 .
\end{align}

\noindent \underline{\textit{The time derivative term:}} We write by the definition of $U_\nu$, 
$$\frac{d}{d\tau}(\nu^2U_\nu^{-1}) = \frac{\nu_\tau}{\nu} \nu \pa_\nu[( |z^2| + \nu^2)^2] = \frac{\nu_\tau}{\nu} \frac{4\nu^2}{U_\nu} \frac{\nu^2}{|z|^2 + \nu^2},$$
from which and $\left|\frac{\nu_\tau}{\nu} \right| \lesssim \frac{1}{|\ln \nu|}$  and \eqref{eq:controle0bynq}, we have the estimate
\begin{align*}
\left|\int_{\Rb^2} \frac{d}{d\tau}(\nu^2U_\nu^{-1}) |\e^\perp|^2 \rho_0  dz\right| & \lesssim \frac{\nu^2}{|\ln \nu|}\int \frac{|\e^\perp|^2}{U_\nu}\rho_0 dz \lesssim \frac{1}{|\ln \nu|}\int \frac{|\nabla q^\perp|^2}{U}\rho dy.
\end{align*}

\noindent \underline{\textit{The small linear and nonlinear terms}:} We write by integration by parts and  Cauchy-Schwarz,
\begin{align*}
&\left|\nu^2\int \nabla \cdot \Gc(\e^\perp) \sqrt{\rho_0} \Ms^z (\e^\perp \sqrt{\rho_0}) dz\right| = \left|\nu^2 \int \Gc(\e) \cdot \Big[ \nabla \Ms^z(\e^\perp \sqrt{\rho_0}) - \frac \beta 2 z  \Ms^z(\e^\perp \sqrt{\rho_0}) \Big]\sqrt{\rho_0}dz \right|\\
&\quad \lesssim \frac{\nu^2}{\sqrt{|\ln \nu|}}\int \left(U_\nu |\nabla \Ms^z(\e^\perp \sqrt{\rho_0})|^2 + U_\nu|\Ms^z(\e^\perp \sqrt{\rho_0})|^2 \right) dz  + \nu^2 \sqrt{|\ln \nu|} \int \frac{(1 + |z|^2)|\Gc(\e)|^2}{U_\nu} \rho_0 dz.
\end{align*}
Making a change of variables and using Hardy inequality \eqref{bd:hardyL2rho} yields
\begin{align*}
\nu^2\int U_\nu |\nabla \Ms^z(\e^\perp \sqrt{\rho_0})|^2 dz = \int U|\nabla \Ms(q^\perp \sqrt{\rho})|^2 dy \lesssim \int \frac{|\nabla q^\perp|^2}{U} \rho dy.
\end{align*}
We also have by \eqref{eq:controle0bynq},
\begin{align}\label{est:UMeperp}
\nu^2\int U_\nu |\Ms^z(\e^\perp \sqrt{\rho_0})|^2 dz \lesssim \nu^2 \int \frac{|\e^\perp|^2}{U_\nu} \rho_0 dz \lesssim \int \frac{|\nabla q^\perp|^2}{U}\rho dy.
\end{align}
We recall
$$\Psi_\nu + \e^0 = \frac{\pa_\zeta}{\zeta}(P_\nu + m_\e), \quad \nabla \Phi_{\Psi_\nu + \e^0} = -z \frac{(P_\nu + m_\e)}{\zeta^2},$$
where we have from the definition of $P_\nu$ and \eqref{pointwisemerefined},  \eqref{pointwiseaway} and \eqref{pointwiseaway2} the bound
$$|\zeta \pa_\zeta(P_\nu + m_\e)| \lesssim  \frac{1}{\sqrt{|\ln \nu|}} \mathbf{1}_{\{\zeta \leq \frac{\zeta_*}{2}\}} +  \nu^2 |\ln \nu| \zeta^{\frac{1}{4}} \mathbf{1}_{\{\zeta \geq \frac{\zeta_*}{2}\}}. $$
We then estimate by using \eqref{est:HaPo1},
\begin{align*}
& \nu^2 \sqrt{|\ln \nu|}\int \frac{(1 + |z|^2)|\e^\perp \nabla \Phi_{\Psi_\nu + \e^0} |^2}{U_\nu} \rho_0 dz\lesssim  \frac{\nu^2}{\sqrt{|\ln \nu|}} \int_{|z| \leq \frac{\zeta_*}{2}} \frac{|\e^\perp|^2}{U_\nu} \rho_0 dz + \nu^6|\ln \nu|^\frac{5}{2} \int_{|z| \geq \frac{\zeta_*}{2}} \frac{|\e^\perp|^2}{U_\nu}|z|^\frac{5}{2} \rho_0 dz\\
& \quad \lesssim \frac{\nu^2}{\sqrt{|\ln \nu|}} \int  \frac{|q^\perp|^2}{U} \rho dy + \nu^6|\ln \nu|^\frac{5}{2} \nu^{\frac{5}{2}}\int \frac{|q^\perp|^2}{U} (1 + |y|^\frac{5}{2}) \rho dy  \lesssim \frac{1}{\sqrt{|\ln \nu|}} \int \frac{|\nabla q^\perp|^2}{U}\rho dy. 
\end{align*} 
We recall $P_{1, \nu} = \frac{a_1}{\nu^2}(\phi_1(r) - \phi_0(r))$ and the bound 
$$\left|\Psi_{1, \nu} \right| = \left|\frac{\zeta \pa_\zeta P_{1, \nu}}{\zeta^2} \right| \lesssim \frac{|y|^2}{1 + |y| ^4}, \quad \forall y \in \Rb^2,$$
so that we write
\begin{align*}
&\nu^2 \sqrt{|\ln \nu|}\int \frac{(1 + |z|^2)|\Psi_{1,\nu}\nabla \Phi_{\e^\perp}|^2}{U_\nu} \rho_0 dz \lesssim 
\nu^4 \sqrt{|\ln \nu|} \int (1 + \nu^2 |y|^2) \big( | \nabla \Phi_{q^\perp \chi_{\frac{\zeta_*}{2\nu}}}|^2 + | \nabla \Phi_{q^\perp ( 1 - \chi_{\frac{\zeta_*}{2\nu}} )}|^2 \big)\rho  dy.
\end{align*}
We estimate by using \eqref{bd:poisson1} and \eqref{est:HaPo1}, 
\begin{align*}
&\nu^4 \sqrt{|\ln \nu|} \int (1 + \nu^2 |y|^2) | \nabla \Phi_{q^\perp \chi_{\frac{\zeta_*}{2\nu}}}|^2 \rho dy\\
& \quad  \lesssim \nu^4 \sqrt{|\ln \nu|} \left[\int_{|y| \leq \frac{\zeta_*}{\nu}} |q^\perp|^2 \langle y \rangle^4 dy + \nu^2\int_{|y| \leq \frac{\zeta_*}{\nu}} |q^\perp|^2 \langle y \rangle ^6dy   \right]  \int \frac{1}{1 + |y|^4} \rho dy \lesssim \nu^2 \sqrt{|\ln \nu|} \int \frac{|\nabla q^\perp|^2}{U} \rho dy.
\end{align*}
Using \eqref{estimationpoissonexte} and noting that $w^\perp = \e^\perp$, we estimate 
\begin{align*}
&\nu^4 \sqrt{|\ln \nu|} \int (1 + \nu^2 |y|^2) | \nabla \Phi_{q^\perp ( 1 - \chi_{\frac{\zeta_*}{2\nu}})}(y)|^2 \rho  dy  = \nu^4 \sqrt{|\ln \nu|} \int (1 + |z|^2) |\nabla \Phi_{\e^\perp (1 - \chi_{\frac{\zeta_*}{2}}) }(z)|^2 \rho_0 dz\\
& \quad \lesssim \nu^4 \sqrt{|\ln \nu|} \int_{|z| \leq 1}|\nabla \Phi_{\e^\perp (1 - \chi_{\frac{\zeta_*}{2}}) }(z)|^2 dz + \nu^4 \sqrt{|\ln \nu|} \int_{|z| \geq 1} \| \e^\perp\|_\out^2 |z|^\frac{3}{2} \rho_0 dz \lesssim e^{-2\kappa \tau}.
\end{align*}
We have by the definition of $P_{2, \nu}$ and \eqref{pointwisemerefined} and \eqref{pointwiseaway},  
\begin{align*}
|\Psi_{2, \nu} + \e^0| &= \left| \frac{\zeta\pa_\zeta (P_{2, \nu} + m_\e)}{\zeta} \right| \lesssim \frac{1}{\nu^2 |\ln \nu|} \frac{|y|^2 \ln \la |y| \ra}{1 + |y|^4} \mathbf{1}_{\{|z| \leq \frac{\zeta_*}{2}\}} + \frac{1}{|\ln \nu|} |z|^{-2 + \frac{1}{4}}\mathbf{1}_{\{|z| \geq \frac{\zeta_*}{2}\}}, 
\end{align*}
from which and a similar estimate as given above, we obtain
\begin{align*}
\nu^2 \sqrt{|\ln \nu|}\int \frac{(1 + |z|^2)|(\Psi_{2,\nu} + \e^0)\nabla \Phi_{\e^\perp}|^2}{U_\nu} \rho_0 dz \lesssim \frac{1}{\sqrt{|\ln \nu|}} \int \frac{|\nabla q^\perp|^2}{U}\rho dy  + e^{-2\kappa \tau}. 
\end{align*}
We also have by \eqref{pointwiseawaynonradial} and \eqref{poissonLinfty} and \eqref{est:xstar} the rough estimate 
\begin{align*}
&\nu^2 \sqrt{|\ln \nu|}\int \frac{\left|\e^\perp \nabla \Phi_{\e^\perp}\right|^2 +\left| \frac{x^*_\tau}{\mu} \e^\perp \right|^2 }{U_\nu} \rho_0 dz  \lesssim \frac{e^{-4\kappa \tau}}{\nu^{12}} \sqrt{|\ln \nu|} \int (1 + |z|^C)\rho_0 dz + \left|\frac{x^*_\tau}{\mu} \right|^2\sqrt{|\ln \nu|} \| \e^\perp\|_0^2\lesssim e^{-2\kappa \tau},
\end{align*}
and from \eqref{pointwiseaway2} and \eqref{bootstrap:in}, 
\begin{align*}
&\nu^2 \sqrt{|\ln \nu|} \left|\frac{x^*_\tau}{\mu} \right|^2 \int \frac{|\e^0|^2 }{U_\nu} \rho_0 dz = \sqrt{|\ln \nu|} \left|\frac{x^*_\tau}{\mu} \right|^2 \int_0^{+\infty} |\pa_\zeta m_\e|^2 \frac{\omega_\nu}{\zeta} d\zeta\\
& \quad \lesssim \sqrt{|\ln \nu|} \left|\frac{x^*_\tau}{\mu}\right| \left[\int_{\zeta \leq \frac{\zeta_*}{2}} |\pa_\zeta m_\e|^2 \frac{\omega_\nu}{\zeta} d\zeta +  \int_{\zeta \geq \frac{\zeta_*}{2}} |\pa_\zeta m_\e|^2 \frac{\omega_\nu}{\zeta} d\zeta \right]\\
& \qquad \lesssim \sqrt{|\ln \nu|} \left|\frac{x^*_\tau}{\mu}\right| \left[ \| \tilde m_v\|^2_\inn + \frac{\nu^2}{|\ln \nu|}\int_{\zeta \geq \frac{\zeta_*}{2}} \zeta^C e^{-\beta \zeta^2/2}d\zeta  \right] \lesssim e^{-2\kappa \tau}.
\end{align*}

\noindent \underline{\textit{The difference $(\Ls^z - \tilde{\Ls}^z)\e^\perp$}:} By the definition $\tilde{\Phi}_{q^\perp} = \frac{1}{\sqrt{\rho}} \Phi_{q^\perp \sqrt{\rho}}$, we have
\begin{align*}
&\Delta (\Phi_{q^\perp} - \tilde{\Phi}_{q^\perp}) = \beta \nu^2\left(1 + \frac{\beta \nu^2 |y|^2}{4} \right)\tilde{\Phi}_{q^\perp} -b\nu^2 y \cdot \nabla \tilde{\Phi}_{q^\perp},
\end{align*}
from which and an integration by parts and Cauchy-Schwarz inequality, we write
\begin{align*}
&\left|\nu^2 \int_{\Rb^2} \big(\Ls^z - \tilde \Ls^z \big)\e^\perp \sqrt{\rho_0} \Ms^z(\e^\perp \sqrt{\rho_0}) dz\right| \ = \ \left|\nu^2 \int_{\Rb^2} \nabla U_\nu \cdot \nabla (\Phi_{\e^\perp} - \tilde\Phi_{\e^\perp}) \sqrt{\rho_0} \Ms^z(\e^\perp \sqrt{\rho_0}) dz\right| \\
& \quad = \left|\int_{\Rb^2} \nabla U \cdot \nabla (\Phi_{q^\perp} - \tilde\Phi_{q^\perp}) \sqrt{\rho} \Ms(q^\perp \sqrt{\rho}) dy\right| \ \leq \ \left(\int U|\nabla \Ms(q^\perp \sqrt{\rho})|^2 dy\right)^\frac{1}{2} \left(\int U|\nabla (\Phi_{q^\perp} - \tilde{\Phi}_{q^\perp})|^2 \rho dy \right)^\frac{1}{2}\\
& \qquad + \nu^2\left(\int U|\Ms(q^\perp \sqrt{\rho})|^2 dy\right)^\frac{1}{2} \left(\int U \Big(|y \cdot \nabla (\Phi_{q^\perp} - \tilde{\Phi}_{q^\perp})|^2 + |\langle \nu y \rangle^ 2 \Phi_{q^\perp \sqrt{\rho}}|^2 \Big) \rho dy \right)^\frac{1}{2}\\
&\quad \lesssim \frac{1}{50}\int \frac{|\nabla q^\perp|^2}{U}\rho dy + C \int U \langle \nu y \rangle^ 2 |\nabla (\Phi_{q^\perp} - \tilde{\Phi}_{q^\perp})|^2 \rho dy + C\nu^2 \int U\langle \nu y \rangle^4|\Phi_{q^\perp \sqrt{\rho}}|^2 dy.
\end{align*}
We bound the last term by using \eqref{bd:poisson1} with $\alpha = 1, 3$ and \eqref{est:HaPo1},
\begin{equation*}
\nu^2 \int U\langle \nu y \rangle^ 4 |\Phi_{q^\perp \sqrt{\rho}}|^2 dy \lesssim \nu^2 \int |q^\perp|^2 \langle y \rangle^ 2\rho dy + \nu^6 \int |q^\perp|^2 \langle y \rangle^ 6\rho dy \lesssim \nu^2 \int \frac{|\nabla q^\perp|^2}{U}\rho dy.
\end{equation*}
For $|y| \geq \frac{\zeta_*}{2\nu}$, we use \eqref{bd:poisson1}, \eqref{est:HaPo1} and  \eqref{estimationpoissonexte} to estimate 
\begin{align*}
&\int_{|y| \geq \frac{\zeta_*}{2\nu}} U\langle \nu y \rangle^ 2|\nabla(\Phi_{q^\perp} - \tilde\Phi_{q^\perp})|^2 \rho dy\\
& \lesssim \int_{|y| \geq \frac{\zeta_*}{2\nu}} U\langle \nu y \rangle^ 2 \left(| \nabla \Phi_{q^\perp \chi_{\frac{\zeta_*}{2\nu}}}|^2\rho  + | \nabla \Phi_{q^\perp ( 1 - \chi_{\frac{\zeta_*}{2\nu}})}|^2\rho + \nu^4|y|^2 |\Phi_{q^\perp \sqrt{\rho}}|^2  + |\nabla \Phi_{q^\perp \sqrt{\rho}}|^2 \right) dy\\
& \qquad \lesssim \nu^2 \int \frac{|\nabla q^\perp|^2}{U} \rho dy + \nu^2\|\e^\perp\|_\out^2. 
\end{align*}
For $|y| \leq \frac{\zeta_*}{2\nu}$, we write by the definition $\tilde{\Phi}_{q^\perp} = \frac{1}{\sqrt{\rho}} \Phi_{q^\perp \sqrt{\rho}}$,
\begin{align*}
&\int_{|y| \leq \frac{\zeta_*}{2\nu}} U(1 + \nu^2 |y|^2)|\nabla(\Phi_{q^\perp} - \tilde\Phi_{q^\perp})|^2 \rho dy  \lesssim \int_{|y| \leq \frac{\zeta_*}{2\nu}}\frac{| \nabla \Phi_{q^\perp (1 - \sqrt{\rho})}|^2 + \nu^4 ( |y\Phi_{q^\perp \sqrt{\rho}}|^2 + ||y|^2 \nabla \Phi_{q^\perp \sqrt{\rho}}|^2}{1 + |y|^4} dy.
\end{align*}
The last two terms are estimated by using \eqref{bd:poisson1} with $\alpha = 1$ and \eqref{est:HaPo1}, 
\begin{align*}
&\int_{|y| \leq \frac{\zeta_*}{2\nu}} \frac{\nu^4(| y \Phi_{q^\perp \sqrt{\rho}}|^2 + ||y|^2\nabla \Phi_{q^\perp \sqrt{\rho}})|^2}{\langle y \rangle^4} dy \lesssim \nu^2 \left(\int |q^\perp|^2\langle y \rangle^2 \rho dy\right) \int_{|y| \leq \frac{\zeta_*}{2\nu}} \frac{1}{\langle y \rangle^4} dy   \lesssim \nu^2 \int \frac{|\nabla q^\perp|^2}{U} \rho dy. 
\end{align*}
For the remaining term, we split into two parts 
\begin{align*}
&\int_{|y| \leq \frac{\zeta_*}{2\nu}} U(1 + \nu^2|y|^2)|\nabla \Phi_{q^\perp (1 - \sqrt{\rho})}|^2 dy \lesssim \int_{|y| \leq \frac{\zeta_*}{2\nu}}  \frac{|\nabla \Phi_{q^\perp (1 - \sqrt{\rho}) \chi_{\frac{\zeta_*}{\nu}}}|^2 + |\nabla \Phi_{q^\perp (1 - \sqrt{\rho}) ( 1 - \chi_{\frac{\zeta_*}{\nu}})}|^2}{1 + |y|^4} dy.
\end{align*}
Since $(1 - \sqrt{\rho}) \chi_{\frac{\zeta_*}{\nu}}(x) \lesssim \zeta_*^2\mathbf{1}_{|x| \leq \frac{2\zeta_*}{\nu}} $, we estimate by using \eqref{bd:poisson1} with $\alpha = 1$ and \eqref{est:HaPo1},
\begin{align*}
&\int_{|y| \leq \frac{\zeta_*}{2\nu}}  \frac{|\nabla \Phi_{q^\perp (1 - \sqrt{\rho}) \chi_{\frac{\zeta_*}{\nu}}}(y)|^2}{1 + |y|^4} dy  \lesssim \zeta^2_* \int_{|y|\leq \frac{\zeta_*}{\nu}} |q^\perp|^2 (1 + |y|^2) dy \lesssim \zeta^2_* \int \frac{|\nabla q^\perp|^2}{U}\rho dy.
\end{align*}
Since $(1 - \sqrt{\rho})(1 - \chi_{\frac{\zeta_*}{\nu}}(x)) \leq \mathbf{1}_{|x| \geq \frac{\zeta_*}{\nu}}$,  and $\frac{1}{|x - y|} \lesssim \frac{\nu}{\zeta_*}$ for $|x| \geq \frac{\zeta^*}{\nu}$ and $|y| \leq \frac{\zeta_*}{2\nu}$, we estimate by using the outer norm \eqref{bootstrap:outnonradial}, 
\begin{align*}
&\int_{|y| \leq \frac{\zeta_*}{2\nu}}  \frac{|\nabla \Phi_{q^\perp (1 - \sqrt{\rho})(1 - \chi_{\frac{\zeta_*}{\nu}})}(y)|^2}{1 + |y|^4} dy \lesssim \int_{|y| \leq \frac{\zeta_*}{2\nu}} \frac{dy}{1 + |y|^4} \left(\int_{|x| \geq \frac{\zeta_*}{\nu}} \frac{|q^\perp|}{|x - y|}dx \right)^2\\
& \qquad  \lesssim \nu^2 \left(\int_{|z| \geq \zeta_*} \frac{|\e^\perp|}{|z|} dz  \right)^2 \lesssim \nu^2 \|\e^\perp\|^2_\out \left(\int_{|z| \geq \zeta_*} \frac{dz}{|z|^{(3 - \frac{1}{4} + \frac{1}{p})p'}} \right)^\frac{2}{p'} \lesssim \frac{\nu^2}{\sqrt{\zeta_*}} \|\e^\perp\|^2_\out.
\end{align*}
Gathering these obtained estimates and using the bootstrap bound \eqref{bootstrap:outnonradial} yields 
\begin{align*}
\left|\nu^2 \int_{\Rb^2} \big(\Ls^z - \tilde \Ls^z \big)\e^\perp \sqrt{\rho_0} \Ms^z(\e^\perp \sqrt{\rho_0}) dz\right| \lesssim \zeta_*^2 \int \frac{|\nabla q^\perp|^2}{U}\rho dy + e^{-2\kappa \tau}.
\end{align*}

\noindent \underline{\textit{The error transport term}:} It remains to estimate the error term by writing 
\begin{align*}
\frac{x_\tau^*}{\mu}\nu^2 \cdot \int \nabla W \sqrt{\rho_0} \Ms^z(\e^\perp \sqrt{\rho_0}) dz = \frac{x_\tau^*}{\mu} \nu \cdot  \int \nabla V  \sqrt{\rho} \Ms(q^\perp \sqrt{\rho}) dy,
\end{align*}
where $V(r, \tau) = \frac{1}{\nu^2} W(\zeta, \tau)$ is the approximate profile in the blowup variables defined from \eqref{def:WnuaPro}, i.e,
$$V = U + \Psi, \quad \Psi = \frac{a_1}{\nu^2} \frac{\pa_r}{r}(\phi_1 - \phi_0) + \sum_{n = 2}^N \frac{a_n}{\nu^2} \frac{\pa_r}{r}\phi_n.$$
By the definition of $\phi_n$, we have the estimate 
\begin{equation*}
|\Psi(r)| + |r \pa_r \Psi(r)| \lesssim \nu^2 \la r\ra^{-2} + \frac{\la r \ra ^{-4}}{|\ln \nu|^2}\mathbf{1}_{\{r \leq \frac{2\zeta_*}{\nu}\}} + \frac{\nu^2}{|\ln \nu|^2} \la \nu r \ra^C \mathbf{1}_{\{r \geq \frac{2\zeta_*}{\nu}\}}.
\end{equation*}
We use $\Ms \nabla U = 0$ and \eqref{est:xstarrefined} to estimate the leading term
\begin{align*}
&\left|\frac{x_\tau^*}{\mu} \nu \int \nabla U \sqrt{\rho}\Ms (q^\perp \sqrt{\rho}) dy\right| = \left|\int \Ms \big(\nabla U (\sqrt{\rho} -1) \big) q^\perp \sqrt{\rho} dy\right|\\
&  \lesssim \left|\frac{x_\tau^*}{\mu}\right| \nu \left[ \int_{|y| \leq \frac{\zeta_*}{2\nu}} \nu^2 |y| |q^\perp| \sqrt{\rho} dy + \int_{|y| \geq \frac{\zeta_*}{2\nu}} |\Ms \big(\nabla U (\sqrt{\rho} -1) \big) q^\perp \sqrt{\rho}| dy   \right]\\
& \lesssim \left|\frac{x_\tau^*}{\mu}\right| \nu^2 \left(\int_{|y| \leq \frac{\zeta_*}{2\nu}} |q^\perp|^2 (1 + |y|^2)\rho dy \right)^\frac{1}{2} \\
& \quad +  \left|\frac{x_\tau^*}{\mu} \right|\nu \left(\int_{|y| \geq \frac{\zeta_*}{2\nu}} |q^\perp|^2 (1 + |y|^2)\rho dy \right)^\frac{1}{2} \left(\int_{|y| \geq \frac{\zeta_*}{2\nu}} \frac{|\Ms (\nabla U (\sqrt{\rho} - 1))|^2}{1 + |y|^2} dy \right)^\frac{1}{2}\\
& \lesssim \left|\frac{x_\tau^*}{\mu}\right| \nu^2 \left(\int |q^\perp|^2(1 + |y|^2)\rho dy\right)^\frac{1}{2} \lesssim \nu \int \frac{|\nabla q^\perp|^2}{U}\rho dy + Ce^{-2\kappa \tau}.
\end{align*}
Using \eqref{est:Psir} and \eqref{est:xstarrefined} and \eqref{pointwiseawaynonradial}, we estimate 
\begin{align*}
&\left|\frac{x_\tau^*}{\mu} \nu \int \nabla \Psi \sqrt{\rho}\Ms (q^\perp \sqrt{\rho}) dy  \right|\\
& \lesssim \left|\frac{x_\tau^*}{\mu} \nu \right|\left(\int_{|y| \leq \frac{\zeta_*}{2}} |\nabla \Psi|^2 (1 + |y|^6) \rho \right)^\frac{1}{2} \left(\int_{|y|\leq \frac{\zeta_*}{2\nu}}|q^\perp|^2 (1  +|y|^2) \rho \right)^\frac{1}{2} + \frac{e^{-2\kappa \tau}}{|\ln \nu|^2}\\
& \lesssim \frac{1}{|\ln \nu|^2}\left(\int\frac{|\nabla q^\perp|^2}{U} \rho dy + \int |q^\perp|^2(1 +|y|^2) \rho dy \right) + \frac{e^{-2\kappa \tau}}{|\ln \nu|^2} \leq \frac{1}{|\ln \nu|^2} \int \frac{|\nabla q^\perp|^2}{U}\rho dy + Ce^{-2\kappa \tau}. 
\end{align*}
The collection of all the above estimates, using the control \eqref{eq:controle0bynq} and the coercivity \eqref{est:dissp}, yields \eqref{ene:ebot0}. This concludes the proof of Lemma \ref{lemm:ebot0}.
\end{proof}

\bigskip

We are now going to establish the monotonicity formula to control the inner norm \eqref{bootstrap:inperp}. The basic idea is that inside the blowup zone $|y| \leq \frac{\zeta_*}{2\nu}$ for $\zeta_* \ll 1$, the dynamic of \eqref{eq:qbotstar} resembles $\pa_s q^\perp = \Ls_0 q^\perp$ which allows us to use the special structure of $\Ls_0$, namely that $\Ls_0 f = \nabla \cdot(U \nabla \Ms f)$. We recall from \eqref{eq:epsbot}, the equation satisfied by $q^\perp$, 
\begin{equation}\label{eq:qperp}
\pa_s q^\perp = \Ls_0 q^\perp + \eta \Lambda q^\perp - \nabla \cdot \Gc(q^\perp) + \frac{\nu x^*_\tau}{\mu} \cdot \nabla (V + q^0) + N^\perp(q^\perp),
\end{equation}
where $\Gc(q^\perp)$ and $N^\perp(q^\perp)$ are defined as in \eqref{def:Geps}, and $V$ is the approximate solution in the blowup variables given by \eqref{def:WnuaPro}, namely that
\begin{equation*}
V = U +  \Psi, 
\end{equation*}
where we recall
$$\Psi =  \Psi_1 + \Psi_2, \quad \Psi_1 = \frac{a_1}{\nu^2}(\varphi_1(r) - \varphi_0(r)), \quad \Psi_2 = \sum_{n = 2}^N \frac{a_n}{\nu^2}\varphi_n(r),$$
and $\varphi_n$ is the radial eigenfunction of $\Ls_0$, namely that 
$$\varphi_n(r) = - \frac{\pa_r \phi_n}{r}, \quad \pa_r \varphi_n = -\frac{\phi_n}{r},$$
with $\phi_n$'s are the eigenfunctions of $\As$, and we write for short
$$\eta = \frac{\nu_s}{\nu} - \beta \nu^2 = \Oc(\nu^2), \quad \Lambda f = \nabla \cdot(yf).$$
We define
$$q^\perp_* = \chi_{\frac{\zeta_*}{\nu}}q^\perp,$$
where $\chi_{\frac{\zeta_*}{\nu}}$ is defined as in \eqref{def:chiM} and $q^\perp_*$ solves the equation 
\begin{align}\label{eq:qbotstar}
\pa_s q^\perp_* &= \Ls_0 q^\perp_* + \chi_{\frac{\zeta_*}{\nu}} \Big[\eta \nabla \cdot(y q^\perp)  - \nabla \cdot \Gc(q) + \frac{\nu x^*_\tau}{\mu} \cdot \nabla (V + q^0) + N^\perp(q^\perp) \Big] - [\Ls_0, \chi_{\frac{\zeta_*}{\nu}}] q^\perp + \pa_s \chi_{\frac{\zeta_*}{\nu}} q^\perp,
\end{align}
Consider the decomposition
\begin{equation}\label{dec:qstarci}
q_*^\bot = c_1 \pa_1 U_1 + c_2 \pa_2 U_2 + \tilde{q}_*^\bot \quad \textup{with} \quad c_i = \frac{\int q_*^\bot \pa_i U dy}{\int |\pa_i U|^2  dy } = \Oc(\|\e^\perp\|_0),
\end{equation}
which produces the orthogonality condition 
$$\int \tilde{q}_*^\perp \pa_1 U dy = \int \tilde{q}_*^\perp \pa_0 U dy = 0.$$
From Lemma \ref{lemm:coerML0}, we have the coercivity
\begin{equation}\label{est:coerMqtil}
-\int \Ls_0 \tilde{q}_*^\perp \Ms \tilde{q}_*^\perp  dy = \int U|\nabla \Ms \tilde{q}_*^\perp|^2 dy \geq \delta_2 \int \frac{|\nabla \tilde{q}_*^\perp|}{U} dy,
\end{equation}
and the positivity since $\int \Ls_0 \tq_*^\perp dy = 0$, 
\begin{equation}\label{eq:posMLqtil}
\int \Ls_0 \tq_*^\perp \Ms \Ls_0 \tq_*^\perp dy \geq 0.
\end{equation}
Recall that $\Ms \pa_1 U = \Ms \pa_2 U = 0$, hence, we have
\begin{equation}\label{id:MsLsqqtil}
\Ms q_*^\perp = \Ms \tilde{q}_*^\perp, \quad  \Ls_0 q_*^\perp = \Ls_0 \tilde{q}_*^\perp.
\end{equation}
Thanks to this coercivity, we are able to establish the monontonicity formula to control $\| q^\perp\|_\inn$.
\begin{lemma}[Control of $\|q^\perp\|_\inn$] \label{lemm:contolqbotinn}
\begin{align*}
\frac{d}{d\tau} \left[\int \tilde{q}^\perp_* \Ms \tilde{q}_*^\perp  dy - \int \Ls_0 \tilde q^\perp_* \Ms \tilde{q}_*^\perp dy  \right] & \leq  -\delta_2'\left[ \frac{1}{\zeta_*^2}\int \tilde{q}^\perp_* \Ms \tilde{q}_*^\perp  dy - \int \Ls_0 \tilde q^\perp_* \Ms \tilde{q}_*^\perp dy  \right] \nonumber\\
 & \qquad \qquad + \frac{C}{\nu^2} \Big(\|\e^\perp\|^2_0 +  \|\e^\perp\|^2_{H^1(\zeta_* \leq |z| \leq \zeta^*)} +  e^{-2\kappa \tau} \Big).
\end{align*}
\end{lemma}
\begin{proof} We multiply equation \eqref{eq:qbotstar} by $\Ms \tilde{q}_*^\perp$ and $-\Ms \Ls \tilde{q}_*^\perp$ and integrate over $\Rb^2$ and  use the relation \eqref{id:MsLsqqtil} and \eqref{est:coerMqtil} and \eqref{eq:posMLqtil} to write the energy identity
\begin{align*}
&\frac{1}{2}\frac{d}{ds}\left[\int \tilde{q}^\perp_* \Ms \tilde{q}_*^\perp  dy  - \int \Ms \tilde{q}_*^\perp \Ls_0\tilde{q}_*^\perp dy \right] \\
&\leq  - \delta_2 \int U |\nabla \Ms \tilde{q}_*^\perp|^2 dy  + \int \left(\chi_{\frac{\zeta_*}{\nu}}\Big[\eta \nabla \cdot(y q^\perp)  - \nabla \cdot \Gc(q^\perp) + \frac{\nu x^*_\tau}{\mu} \cdot \nabla (V+q^0) + N^\perp(q^\perp) \Big] \right) \\
& \qquad  \times  \Ms\Big(\tilde{q}_*^\perp +\Ls_0 \tilde{q}_*^\perp\Big)dy  - \int \left( [\Ls_0, \chi_{\frac{\zeta_*}{\nu}}] q^\perp - \pa_s \chi_{\frac{\zeta_*}{\nu}} q^\perp  \right) \Ms\Big(\tilde{q}_*^\perp +\Ls_0 \tilde{q}_*^\perp\Big)dy.
\end{align*}
Since we only use the positivity of $\Ms$, all the terms involving $\Ms \Ls_0 \tq_*^\perp$ are treated by using commutator formulas with $\Ms$ in order to reduce the order of derivatives so that they can be controlled by $-\delta_2\int U |\nabla \Ms \tilde{q}_*^\perp|^2 dy$. In particular, we shall use the following formula: for any well localized function $f$ without radial component,
\begin{equation}\label{def:comMLam}
[\Ms, \Lambda] f =  \frac{y \cdot\nabla U}{U^2}f + 2 \Phi_f \quad \textup{with} \quad \Lambda f = \nabla \cdot (y f),
\end{equation}
where we used the identity $y\cdot \nabla \Phi_f = \Phi_{y \cdot \nabla f} + 2\Phi_f$, and 
\begin{equation}\label{eq:comMfgdev}
[\Ms, \nabla] f = \frac{\nabla U}{U^2}f, \quad [\Ms, g] f = g\Phi_f - \Phi_{gf}.
\end{equation}

\noindent \underline{The scaling term:} We write $\chi_{\frac{\zeta_*}{\nu}} \nabla \cdot (y q^\perp) = \nabla \cdot (y q^\perp_*) - y\cdot \nabla \chi_{\frac{\zeta_*}{\nu}} q^\perp$, and use \eqref{def:comMLam} and the relation $\Ms q_*^\perp = \Ms \tq_*^\perp$ and the structure $\Ls_0 f = \nabla \cdot(U \nabla \Ms f)$ and integration by parts, to estimate
\begin{align*}
&\left|\eta\int\nabla \cdot(y q_*^\perp)\Ms\Ls_0\tq_*^\perp  dy\right|  \lesssim  \nu^2 \left|\int \nabla \cdot(y \Ms \tq_*^\perp) \Ls_0\tq_*^\perp  dy \right| + \nu^2 \left|\int \Big(\frac{y \cdot\nabla U}{U^2}q_*^\perp + 2 \Phi_{q_*^\perp}\Big) \Ls_0 \tq_*^\perp dy  \right|\\
& \lesssim \nu^2\int U|\nabla \Ms \tq_*^\perp|^2 \left(1 + \frac{|\nabla \cdot (y U)|}{U} + \frac{|\nabla U|}{U} \right)dy + \nu^2 \int U \left|\nabla\Big(\frac{y \cdot\nabla U}{U^2}q_*^\perp + 2 \Phi_{q_*^\perp}\Big)  \right|^2 dy\\
& \lesssim \nu^2\int U|\nabla \Ms \tq_*^\perp|^2 dy + \nu^2 \int \left(\frac{|\nabla q_*^\perp|^2}{U} + \frac{|\nabla U|^2}{U} |q_*^\perp|^2 + U|\nabla \Phi_{q_*^\perp}|^2 \right)dy\\
& \lesssim \nu^2 \int U|\nabla \Ms \tq_*^\perp|^2 dy + \nu^2  \int \frac{|\nabla \tq_*^\perp|^2}{U}dy + \nu^2 (|c_1|^2 + |c_2|^2) \lesssim \nu^2 \int U|\nabla \Ms \tq_*^\perp|^2 dy + \nu^2 \|\e^\perp\|_0^2.
\end{align*}
where we used in the last line the decomposition \eqref{dec:qstarci} and the coercivity \eqref{est:coerMqtil}.\\
Since $y\cdot \nabla \chi_{\frac{\zeta_*}{\nu}} \lesssim \mathbf{1}_{\{ \frac{\zeta_*}{\nu} \leq |y| \leq \frac{2\zeta_*}{\nu} \}}$, we use the midrange bootstrap \eqref{bootstrap:perpmid} and integration by parts to estimate
\begin{align*}
\left|\eta \int y\cdot \nabla \chi_{\frac{\zeta_*}{\nu}} q^\perp \Ms \Ls_0 \tq_*^\perp dy  \right| &\lesssim \nu^2 \int U|\nabla \Ms \tq_*^\perp|^2 dy + \nu^2 \int |\nabla \Ms \big( y\cdot \nabla \chi_{\frac{\zeta_*}{\nu}} q^\perp \big) | ^2 \\
& \quad \lesssim \nu^2 \int U|\nabla \Ms \tq_*^\perp|^2 dy + \nu^2 \|\e^\perp\|_{H^1(\zeta_* \leq |z| \leq \zeta^*)}^2.
\end{align*}
We also estimate by integration by parts and Cauchy-Schwarz and $\big|y\cdot \nabla \chi_{\frac{\zeta_*}{\nu} }\big| \leq \mathbf{1}_{\{\frac{\zeta_*}{\nu} \leq |y| \leq \frac{2\zeta_*}{\nu}\}}$,
\begin{align*}
&\left|\eta\int\chi_{\frac{\zeta_*}{\nu}} \nabla \cdot(y q^\perp)\Ms\tq_*^\perp  dy\right| =  \left|\eta\int \Big(\nabla \cdot(y q_*^\perp) - y\cdot \nabla \chi_{\frac{\zeta_*}{\nu} }q^\perp \Big)\Ms\tq_*^\perp  dy\right| \\
&\quad \leq \frac{\delta_2}{100}\int U|\nabla \Ms \tq_*^\perp|^2dy + \nu^4\int \frac{|y q_*^\perp|^2}{U}dy +  \nu^2\int_{\frac{\zeta_8}{\nu} \leq |y| \leq \frac{2\zeta_*}{\nu}} \big|q^\perp\Ms\tq_*^\perp\big| dy\\
& \quad \lesssim \frac{\delta_2}{100}\int U|\nabla \Ms \tq_*^\perp|^2dy  + \| \e^\perp\|_0^2 +  \|\e^\perp\|_{H^1(\zeta_* \leq |z| \leq \zeta^*)}^2. 
\end{align*}
Summing up these estimates yields 
\begin{align*}
&\left|\eta\int\chi_{\frac{\zeta_*}{\nu}} \nabla \cdot(y q^\perp)\Ms \big( \tq_*^\perp + \Ls_0\tq_*^\perp  \big)dy\right|   \leq \frac{\delta_2}{50} \int U|\nabla \Ms \tq_*^\perp|^2 dy + \|\e^\perp\|_{H^1(\zeta_* \leq |z| \leq \zeta^*)}^2 + \|\e^\perp\|_0^2.
\end{align*}
\noindent \underline{The $\Gc(q^\perp)$ and $N^\perp$ terms:} We write by integration by parts and Cauchy-Schwarz
\begin{align*}
&\left|\int \Big( \nabla \cdot \Gc(q^\perp) - N^\perp(q^\perp) \Big) \Ms \tq_*^\perp dy \right| \leq \frac{\delta_2}{100}\int U|\Ms \nabla \tq_*^\perp|^2 dy + C\int_{|y| \leq \frac{2\zeta_*}{\nu}} \frac{|\Gc(q^\perp)|^2 + |N^\perp(q^\perp)|^2}{U} dy\\
& \leq \left(\frac{\delta_2}{100} + \frac{1}{|\ln \nu|^2}\right)\int U|\Ms \nabla \tq_*^\perp|^2 dy + C e^{-2\kappa \tau},
\end{align*} 
where we use \eqref{est:Psir} and the definition \eqref{def:Geps} of $\Gc(q^\perp)$ and $N^\perp(q^\perp)$ and \eqref{pointwiseawaynonradial}, \eqref{poissonLinfty}  to estimate 
\begin{align*}
&\int_{|y| \leq \frac{2\zeta_*}{\nu}} \frac{|\Gc(q^\perp)|^2 + |N^\perp(q^\perp)|^2}{U} dy \lesssim \int_{|y| \leq \frac{2\zeta_*}{\nu}} \left(\frac{|q^\perp|^2}{U} \big|\Phi_{\Psi + q^0}\big|^2 + |\nabla \Phi_{q^\perp}|^2 \frac{|\Psi + q^0|^2}{U}  + |q^\perp \nabla q^\perp|^2\right) dy\\
 & \quad \lesssim \frac{1}{|\ln \nu|^2} \int_{|y| \leq \frac{2\zeta_*}{\nu}} |q^\perp|^2 (1 + |y|^2) dy + e^{-2\kappa \tau}  \lesssim  \frac{1}{|\ln \nu|^2}\int \frac{|\nabla \tq_*^\perp|^2}{U}dy + \frac{1}{|\ln \nu|^2}\|\e^\perp\|_{H^1(\zeta_* \leq |z| \leq \zeta^*)} + e^{-2\kappa \tau}\\
 & \quad  \lesssim \frac{1}{|\ln \nu|^2}\int U|\Ms \nabla \tq_*^\perp|^2 dy + C e^{-2\kappa \tau}.
\end{align*}
Similarly, we write by integration by parts and the definition of $\Gc$ and $N^\perp$ and Cauchy-Schwarz inequality,
\begin{align*}
&\left|\int \Ms \Big(\nabla \cdot \Gc(q^\perp) - N^\perp(q^\perp) \Big)\Ls_0 \tq_*^\perp dy\right| = \left|\int U \nabla \Big(\Ms \nabla \cdot \Gc(q^\perp) - \Ms N^\perp(q^\perp)\Big) \cdot \nabla  \Ms \tq_*^\perp dy  \right|\\
&\leq \frac{\delta_2}{100}\int U |\nabla \Ms \tq_*^\perp|^2 dy + \left|\int U \nabla \Ms \Big(\nabla q^\perp \cdot \nabla \Phi_{\Psi + q^0}  - 2 q^\perp(\Psi + q^0)\Big) \cdot \nabla \Ms \tq_*^\perp dy\right|\\
&  \quad  + C\int_{|y| \leq \frac{2\zeta_*}{\nu}} U |\nabla \Ms \nabla(\Psi + q^0) \cdot \nabla \Phi_{q^\perp})|^2 + C \int_{|y| \leq \frac{2\zeta_*}{\nu}} U | \nabla \Ms \nabla \cdot(q^\perp \nabla \Phi_{q^\perp})|^2 dy.  
\end{align*}
Using the commutator formula \eqref{eq:comMfgdev} and \eqref{est:Psir} and the relation \eqref{id:MsLsqqtil}, we obtain the estimate 
\begin{align*}
&\left|\int U \nabla \Ms \Big(\nabla q^\perp \cdot \nabla \Phi_{\Psi + q^0}  - 2 q^\perp(\Psi + q^0)\Big) \cdot \nabla \Ms \tq_*^\perp dy\right|\\
& \quad \leq  \int U |\nabla \Ms \tq_*^\perp|^2 \left( |\Psi + q_0| + \frac{1}{2}\left|\frac{\nabla \cdot (U \nabla \Phi_{\Psi + q^0})}{U} \right| \right)dy \\
& \qquad + \left|\int U\nabla \left( [\Ms, \Psi + q^0] q^\perp + [\Ms, \nabla \Phi_{\Psi + q^0}]\cdot \nabla q^\perp + [\Ms, \nabla]q^\perp \cdot \nabla \Phi_{\Psi + q^0} \right) \cdot \nabla \Ms \tq_*^\perp dy \right|\\
&  \leq \left( \frac{1}{|\ln \nu|^2} + \frac{\delta_2}{100}\right) \int U|\nabla \Ms \tq_*^\perp|^2\\
& \qquad  + C\int_{|y| \leq \frac{2\zeta_*}{\nu}} U \Big|\nabla\left( [\Ms, \Psi + q^0] q^\perp + [\Ms, \nabla \Phi_{\Psi + q^0}]\cdot \nabla q^\perp + [\Ms, \nabla]q^\perp \cdot \nabla \Phi_{\Psi + q^0} \right)  \Big|^2 dy \\
&\qquad \qquad \quad  \leq  \frac{\delta_2}{50}\int U|\nabla \Ms \tq_*^\perp|^2 + \frac{C}{|\ln \nu^2|} \left(\int|\nabla \tq_*^\perp|^2(1 + |y|^4)  + \|\e^\perp\|^2_{H^1 (\zeta_* \leq |z| \leq \zeta^*)}\right).
\end{align*}
From \eqref{pointwiseawaynonradial}, \eqref{poissonLinfty} and \eqref{est:Psir} and \eqref{bd:poisson1} with $\alpha = 1/2$,  we estimate 
\begin{align*}
&\int_{|y| \leq \frac{2\zeta_*}{\nu}} U |\nabla \Ms \nabla(\Psi + q^0) \cdot \nabla \Phi_{q^\perp})|^2 + \int_{|y| \leq \frac{2\zeta_*}{\nu}} U | \nabla \Ms \nabla \cdot(q^\perp \nabla \Phi_{q^\perp})|^2 dy\\
& \quad  \lesssim \frac{1}{|\ln \nu|^2} \int |\tq_*^\perp|^2(1 + |y|^2) dy + \frac{1}{|\ln \nu|^2} \|\e^\perp\|^2_{H^1 (\zeta_* \leq |z| \leq \zeta^*)} +   e^{-2\kappa \tau}.
\end{align*}
Summing up these estimates yields 
\begin{align*}
&\left|\int \Big[\nabla \cdot \Gc(q^\perp) - N^\perp(q^\perp) \Big]\Ms \big( \tq_*^\perp + \Ls_0\tq_*^\perp  \big)dy\right|  \leq \frac{\delta_2}{50} \int U|\nabla \Ms \tq_*^\perp|^2 dy + \|\e^\perp\|_{H^1(\zeta_* \leq |z| \leq \zeta^*)}^2 + e^{-2\kappa \tau}.
\end{align*}
\noindent \underline{The error terms:} From \eqref{est:Psir} and \eqref{pointwisemerefined}, we note 
$$|\Psi + q^0| + |r\pa_r (\Psi + q^0)| \lesssim \frac{1}{|\ln \nu|^2} \frac{1}{1 + r^4}, \quad \textup{for} \quad  r \leq \frac{2\zeta_*}{\nu},$$
from which and the fact that $\Ms \nabla U = 0$ and \eqref{est:xstarrefined}, we obtain the estimate
\begin{align*}
& \left| \frac{\nu x_\tau^*}{\mu} \cdot \int \nabla (V + q^0) \Ms \big( \tq_*^\perp + \Ls_0\tq_*^\perp  \big)dy\right| = \left| \frac{\nu x_\tau^*}{\mu} \cdot \int \Ms \nabla (\Psi + q^0) \big( \tq_*^\perp + \Ls_0\tq_*^\perp  \big)dy\right|\\
& \qquad \quad \lesssim \frac{1}{|\ln \nu|^2}\int \frac{|\nabla \tq_*^\perp|^2}{U}dy + \frac{C}{|\ln \nu|^2} \|\e^\perp\|^2_{H^1(\zeta_* \leq |z| \leq \zeta^*)} + Ce^{-2\kappa \tau}.
\end{align*}
We also have the estimate from \eqref{est:xstar} and Cauchy-Schwarz, 
\begin{align*}
&\left| \frac{\nu x_\tau^*}{\mu} \cdot \int \nabla q^\perp \Ms\tq_*^\perp dy\right| \leq \left| \frac{\nu x_\tau^*}{\mu} \right|\left(\int_{|y| \leq \frac{2\zeta_*}{\nu}} |q^\perp|^2 U dy\right)^\frac{1}{2}\left(\int U| \nabla \Ms \tq_*^\perp|^2 dy \right)^\frac{1}{2}\\
& \quad \leq \frac{\delta_2}{50} \int U| \nabla \Ms \tq_*^\perp|^2 dy + C\left| \frac{\nu x_\tau^*}{\mu} \right|^2 \frac{\|\e^\perp\|_0^2}{\nu^2} \leq \frac{\delta_2}{50} \int U| \nabla \Ms \tq_*^\perp|^2 dy + e^{-2\kappa\tau}.
\end{align*}
Similarly, we use the commutator formula \eqref{eq:comMfgdev} to obtain the estimate
\begin{align*}
&\left| \frac{\nu x_\tau^*}{\mu} \cdot \int \nabla q^\perp \Ms\Ls_0\tq_*^\perp dy\right|\leq \frac{\delta_2}{50} \int U| \nabla \Ms \tq_*^\perp|^2 dy + e^{-2\kappa\tau}.
\end{align*}

\noindent \underline{The commutator terms:} Since the support of $[\Ls_0, \chi_\frac{\zeta_*}{\nu}] q^\perp$ and $\pa_s \chi_\frac{\zeta_*}{\nu}$ is on $\frac{\zeta_*}{\nu} \leq |y| \leq \frac{2\zeta_*}{\nu}$, we use the midrange bootstrap \eqref{bootstrap:perpmid} and Cauchy-Schwarz to obtain the estimate
\begin{align*}
&\left|\int \left([\Ls_0, \chi_\frac{\zeta_*}{\nu}] q^\perp -  \pa_s \chi_\frac{\zeta_*}{\nu} q^\perp\right)\Ms \big( \tq_*^\perp + \Ls_0\tq_*^\perp  \big)dy\right|  \leq \frac{\delta_2}{50} \int U|\nabla \Ms \tq_*^\perp|^2 dy + \|\e^\perp\|_{H^1(\zeta_* \leq |z| \leq \zeta^*)}^2.
\end{align*}
A collection of all the estimates and using the fact that 
$$-\int \frac{|\nabla \tq_*^\perp|^2}{U}dy \lesssim - \int |\tq_*^\perp|^2(1  + |y|^2) dy \lesssim -\frac{\nu^2}{\zeta_*^2}\int \tq_*^\perp \Ms \tq_*^\perp dy,$$
and $\frac{ds}{d\tau} = \frac{1}{\nu^2}$ yield the conclusion of Lemma \ref{lemm:contolqbotinn}.
\end{proof}

\subsection{Analysis in the exterior zone} \label{sec:ext}

In the exterior zone, we derive an energy estimate for the $\| \hat w \|_{\out}$ norm of the full higher order perturbation \eqref{def:what}. From \eqref{def:what} and \eqref{eq:wztau}, the radial and non radial parts $\hat w^0$ and $\hat w^\perp$ of $\hat w$ solve the following equations:
\begin{equation} \label{id:eqhatw}
\partial_\tau \hat w^0+\beta \nabla.(z\hat w^0)-\Delta \hat w^0=e.\nabla \hat w^0 +f\hat w^0+g.\nabla \Phi_{(1-\chi_{\frac{\zeta^*}{2}})\hat w^0} +h^0+ N^0(\hat w^\perp),
\end{equation}
\begin{equation} \label{id:eqhatwperp}
\partial_\tau \hat w^\perp+\beta \nabla.(z\hat w^\perp)-\Delta \hat w^\perp=e.\nabla \hat w^\perp +f\hat w^\perp+g.\nabla \Phi_{(1-\chi_{\frac{\zeta^*}{2}})\hat w^\perp} +h^\perp+ N^\perp(\hat w^\perp),
\end{equation}
where
$$
e=-\nabla\left(\Phi_{U_\nu}+\Phi_{\Psi_{\nu,1}}\right), \quad f=2(U_\nu+\Psi_{\nu,1}), \quad g=-\nabla (U_\nu+\Psi_{\nu,1}),
$$
and from \eqref{def:Qmua}, \eqref{pointwisephi0}, \eqref{est:phi1til} and \eqref{est:nu0ntil1},
\begin{align*}
h^0&= \frac{1}{2}\Big(-\pa_\tau f + \Delta f - \nabla \cdot (f \nabla \Phi_f) - \beta \nabla \cdot(z f)\Big) \\
& = \frac{\nu_\tau}{\nu}\nabla .(zU_\nu)-\frac{\nu_\tau}{\nu}\nu\partial_\nu \Psi_{\nu,1}-\frac{\beta_\tau}{\beta}\beta \pa_\beta \Psi_{\nu,1}  +(\alpha_1 a_1-a_{1,\tau})(\varphi_1-\varphi_0) \\
& \quad +2\beta\left((\tilde \alpha_1-\tilde \alpha_0)a_1-a_1-4\nu^2\right)\varphi_0 +\beta\left(8\nu^2\varphi_0- \nabla .(zU_\nu)\right) -\nabla (\Psi_{\nu,1}).\nabla \Phi_{\Psi_{\nu,1}}+\Psi_{\nu,1}^2+g.\nabla \Phi_{\chi_{\frac{\zeta^*}{2}}\hat w^0},
\end{align*}
and 
\begin{eqnarray*}
h^\perp &=&\frac{x^*_\tau}{\mu}.\nabla \left(U_\nu+\Psi_{\nu,1}\right)+g.\nabla \Phi_{\chi_{\frac{\zeta^*}{2}}\hat w^\perp}.
\end{eqnarray*}
From the pointwise estimates \eqref{def:Qmua}, \eqref{pointwisephi0}, \eqref{est:phi1til} on $U$, $\varphi_0$ and $\varphi_1$, the Poisson formula for radial functions, we have the following estimates for $|z|\geq 1$ and $k=0,1,2$:
\begin{equation} \label{bd:estimationefgexterior}
|\partial_\zeta^k e(\zeta)|\lesssim \zeta^{-1+\delta-k}, \quad|\partial_\zeta^k f(\zeta)|\lesssim \nu^2|\log \nu|\zeta^{-2+\delta-k}, \quad |\partial_\zeta^k g(\zeta)|\lesssim \nu^2\zeta^{-3+\delta-k}.
\end{equation}
We recall that $|\nu_\tau /\nu|\lesssim |\log \nu |^{-1}$ and $|\beta_\tau/\beta|\lesssim |\log \nu|^{-3}$ from \eqref{eq:mod0}, \eqref{eq:mod1} and \eqref{bootstrap:condition}. Hence from \eqref{est:phi1til}, \eqref{est:nu0ntil1} and \eqref{est:xstar} we infer that for the forcing term for $|z| \geq 1$ and $k=0,1,2$:
\begin{equation} \label{bd:estimationhexterior}
|\partial_\zeta^k h^0(\zeta)|\lesssim \frac{\nu^2}{|\log \nu|} \zeta^{-2+\delta-k}, \quad |\nabla^k h^\perp(\zeta)|\lesssim  e^{-\kappa \tau} \nu \zeta^{-2+\delta-k}.
\end{equation}

\begin{lemma}
There exists a universal constant $C>0$ independent of the constants $N,\kappa,K,K',K'',\zeta^*$ such that for $\tau_0$ large enough:
\begin{equation} \label{bd:lyapunovext0}
\frac{d}{d\tau} \| \hat w^0 \|_{\out}^{2p} \leq 2p\left(-\frac{\beta}{4}+\frac{C}{\zeta^*}\right) \| \hat w^0\|_{\out}^{2p}+C\frac{K^{'2p}\nu^{4p}}{|\log \nu|^{2p}}+ C\| \hat w^0 \|_{\out}^{2p-1}\frac{K\nu^2}{|\log \nu|}.
\end{equation}
\begin{equation} \label{bd:lyapunovextperp}
\frac{d}{d\tau} \| \hat w^\perp \|_{\out}^{2p} \leq 2p\left(-\frac{\beta}{4}+\frac{C}{\zeta^*}\right) \| \hat w^\perp \|_{\out}^{2p}+CK^{'2p}e^{-2p\kappa \tau}+ C\| \hat w^\perp \|_{\out}^{2p-1}e^{-\kappa \tau}.
\end{equation}

\end{lemma}

\begin{proof}

\textbf{Step 1} \emph{(Linear energy estimate)}. We claim that for any function $\bar w$, for any vector field $\bar e$ and potential $\bar f$ satisfying $|\bar f|\lesssim \zeta^{*-1}$ and $|\bar e|+\zeta|\nabla \bar e|\lesssim 1$ a constant $C>0$ exists such that
\begin{align}\label{id:linearenergyout} 
 & \int \left(1-\chi_{\frac{\zeta^*}{4}} \right) \zeta^{(2-\frac 14)2p}\bar w^{2p-1}\left(-\beta \nabla.(z\bar w)+\Delta \bar w+\bar e.\nabla \bar w +\bar f\bar w \right)\frac{dz}{\zeta^2}\\
\nonumber  & \quad \leq -(2p-1)\int \left(1-\chi_{\frac{\zeta^*}{4}} \right) \zeta^{(2-\frac 14)2p}\bar w^{2p-2} |\nabla \bar w|^2 \frac{dz}{\zeta^2} \\
& \quad +\left(-\frac{\beta}{4}+\frac{C}{\zeta^*}\right)\int \left(1-\chi_{\frac{\zeta^*}{4}} \right) \zeta^{(2-\frac 14)2p}\bar w^{2p}\frac{dz}{\zeta^2}+C \| \bar w\|_{L^{2p}(\frac{\zeta^*}{4}\leq |z|\leq \zeta^*)}^{2p}.\nonumber
\end{align}
We now prove this estimate. Integrating by parts, as derivatives of $\chi_{\frac{\zeta^*}{4}}$ have support in $\{\frac{\zeta^*}{4}\leq |z|\leq \zeta^*\}$, we obtain
\begin{align*}
\int \left(1-\chi_{\frac{\zeta^*}{4}} \right) &\zeta^{(2-\frac 14)2p}\bar w^{2p-1}(-\beta \nabla.(z\bar w))\frac{dz}{\zeta^2}  =-\frac{\beta}{4}\int \left(1-\chi_{\frac{\zeta^*}{4}} \right) \zeta^{(2-\frac 14)2p}\bar w^{2p}\frac{dz}{\zeta^2}-\beta \int \partial_\zeta \chi_{\frac{\zeta^*}{4}} \zeta^{(2-\frac 14)2p}\bar w^{2p}\frac{dz}{\zeta^2},
\end{align*}
and
\begin{align*}
&\int \left(1-\chi_{\frac{\zeta^*}{4}} \right) \zeta^{(2-\frac 14)2p}\bar w^{2p-1} \Delta \bar w\frac{dz}{\zeta^2} \\
&= -(2p-1)\int \left(1-\chi_{\frac{\zeta^*}{4}} \right) \zeta^{(2-\frac 14)2p}\bar w^{2p-2} |\nabla \bar w|^2 \frac{dz}{\zeta^2}   +\frac{1}{2p} \int \bar w^{2p} \Delta \left(\left(1-\chi_{\frac{\zeta^*}{4}} \right) \zeta^{(2-\frac 14)2p-2 } \right)dz\\
& \quad\leq -(2p-1)\int \left(1-\chi_{\frac{\zeta^*}{4}} \right) \zeta^{(2-\frac 14)2p}\bar w^{2p-2} |\nabla \bar w|^2 \frac{dz}{\zeta^2} \\
& \qquad +\frac{1}{2p} \int \bar w^{2p} \left(\left(1-\chi_{\frac{\zeta^*}{4}} \right) \Delta \zeta^{(2-\frac 14)2p-2 }-2\nabla \chi_{\frac{\zeta^*}{4}}.\nabla \zeta^{(2-\frac 14)2p-2 }-\Delta \chi_{\frac{\zeta^*}{4}}\zeta^{(2-\frac 14)2p-2 } \right)dz\\
& \quad \leq  -(2p-1)\int \left(1-\chi_{\frac{\zeta^*}{4}} \right) \zeta^{(2-\frac 14)2p}\bar w^{2p-2} |\nabla \bar w|^2 \frac{dz}{\zeta^2} +  \frac{C}{|\zeta^*|^2} \int \left(1-\chi_{\frac{\zeta^*}{4}} \right) \zeta^{(2-\frac 14)2p}\bar w^{2p}\frac{dz}{\zeta^2}+\Oc\left(\| \bar w\|_{L^{2p}(\frac{\zeta^*}{4}\leq |z|\leq \zeta^*)}^{2p}\right).
\end{align*}
Using the bounds $|\bar f|\lesssim \zeta^{*-1}$ and $|\bar e|+\zeta|\nabla \bar e|\lesssim 1$ yield
\begin{align*}
&\int \left(1-\chi_{\frac{\zeta^*}{4}} \right) \zeta^{(2-\frac 14)2p}\bar w^{2p-1} e.\nabla \bar w\frac{dz}{\zeta^2} = -\frac{1}{2p}\int \bar w^{2p} \left(\left(1-\chi_{\frac{\zeta^*}{4}} \right) \nabla.\left(\zeta^{(2-\frac 14)2p-2} e\right)-\nabla \chi_{\frac{\zeta^*}{4}}.e \zeta^{(2-\frac 14)2p-2} \right)dz\\
& \qquad  \leq \frac{C}{\zeta^*}\int \left(1-\chi_{\frac{\zeta^*}{4}} \right) \zeta^{(2-\frac 14)2p}\bar w^{2p}\frac{dz}{\zeta^2}+\Oc\left(\| \bar w\|_{L^{2p}(\frac{\zeta^*}{4}\leq |z|\leq \zeta^*)}^{2p}\right),
\end{align*}
and 
$$
\int \left(1-\chi_{\frac{\zeta^*}{4}} \right) \zeta^{(2-\frac 14)2p}\bar w^{2p-1} f\bar w\frac{dz}{\zeta^2}\leq  \frac{C}{\zeta^*}\int \left(1-\chi_{\frac{\zeta^*}{4}} \right) \zeta^{(2-\frac 14)2p}\bar w^{2p} \frac{dz}{\zeta^2}.
$$
Summing the four identities above proves the linear estimate \eqref{id:linearenergyout}.\\

\noindent \textbf{Step 2} \emph{(Preliminary estimates)}. From the bound \eqref{bd:estimationhexterior}, we have the estimate
\begin{align} 
&\left( \int \left(1-\chi_{\frac{\zeta^*}{4}} \right) \zeta^{(2-\frac 14)2p} \left(|h^0|+\zeta|\nabla h^0|\right)^{2p}\frac{dz}{\zeta^2}\right)^{\frac{1}{2p}}   \lesssim \frac{\nu^2}{|\log \nu|}\left( \int \left(1-\chi_{\frac{\zeta^*}{4}} \right) \zeta^{(\delta-\frac 14)2p}\frac{dz}{\zeta^2}\right)^{\frac{1}{2p}}\lesssim \frac{\nu^2}{|\log \nu|},\label{bd:interexth0}
\end{align}
for $\delta $ small enough. Similarly, using the bound \eqref{bd:estimationefgexterior}, the estimate \eqref{estimationpoissonexte}, and the bootstrap bounds \eqref{bootstrap:out}, we obtain
\begin{align} 
&\left( \int \left(1-\chi_{\frac{\zeta^*}{4}} \right) \zeta^{(2-\frac 14)2p}\left(|g.\nabla \Phi_{(1-\chi_{\frac{\zeta^*}{2}})(\hat w^0)}|+\zeta |\nabla (g.\nabla \Phi_{(1-\chi_{\frac{\zeta^*}{2}})(\hat w^0)})|\right)^{2p} \frac{dz}{\zeta^2}\right)^{\frac{1}{2p}}  \leq C\| \hat w^0 \|_{\out} \nu^2 \leq \frac{\nu^2}{|\log \nu|}, \label{bd:interextpoisson0}
\end{align}
for $\tau_0$ large enough. Similarly, we estimate  from \eqref{bd:estimationhexterior} for the nonradial part,
\begin{equation} \label{bd:interexthperp}
\left( \int \left(1-\chi_{\frac{\zeta^*}{4}} \right) \zeta^{(2-\frac 14)2p} \left(|h^\perp|+\zeta|\nabla h^\perp|\right)^{2p}\frac{dz}{\zeta^2}\right)^{\frac{1}{2p}}\lesssim \nu^2 \left| \frac{x^*_\tau}{\mu} \right|\lesssim e^{-\kappa \tau},
\end{equation}
and using the bound \eqref{bd:estimationefgexterior}, the estimate \eqref{estimationpoissonexte}, and the bootstrap bound \eqref{bootstrap:outnonradial}:
\begin{align} 
&\left( \int \left(1-\chi_{\frac{\zeta^*}{4}} \right) \zeta^{(2-\frac 14)2p}\left(|g.\nabla \Phi_{(1-\chi_{\frac{\zeta^*}{2}})(\hat w^\perp)}|+\zeta |\nabla (g.\nabla \Phi_{(1-\chi_{\frac{\zeta^*}{2}})(\hat w^\perp)})|\right)^{2p} \frac{dz}{\zeta^2}\right)^{\frac{1}{2p}} \lesssim \| \hat w^\perp \|_{\out} \nu^2\leq e^{-\kappa \tau},\label{bd:interextpoissonperp}
\end{align}
for $\tau_0$ large enough.\\

\noindent \textbf{Step 3} \emph{(The energy estimate for $\hat w^0$)}. We claim that there holds the energy estimate:
\begin{align}\label{id:lyapunovout1}
 \frac{d}{d\tau} \left(\frac{1}{2p}\int \left(1-\chi_{\frac{\zeta^*}{4}} \right) \zeta^{(2-\frac 14)2p}(\hat w^0)^{2p}\frac{dz}{\zeta^2}\right)  &\leq \left(-\frac{\beta}{4}+\frac{C}{\zeta^*}\right) \int \left(1-\chi_{\frac{\zeta^*}{4}} \right) \zeta^{(2-\frac 14)2p}(\hat w^0)^{2p}\frac{dz}{\zeta^2}\\
&   +C\frac{K^{'2p}\nu^{4p}}{|\log \nu|^{2p}}+ C\left(\int \left(1-\chi_{\frac{\zeta^*}{4}} \right) \zeta^{(2-\frac 14)2p}(\hat w^0)^{2p} \frac{dz}{\zeta^2}\right)^{\frac{2p-1}{2p}} \frac{\nu^2}{|\log \nu|}.\nonumber
\end{align}
We compute from \eqref{id:eqhatw},
\begin{align}\label{id:energyestimatehatwout}
 &\frac{d}{d\tau} \left(\frac{1}{2p}\int \left(1-\chi_{\frac{\zeta^*}{4}} \right) \zeta^{(2-\frac 14)2p}(\hat w^0)^{2p}\frac{dz}{\zeta^2}\right)\\
 & = \int \left(1-\chi_{\frac{\zeta^*}{4}} \right) \zeta^{(2-\frac 14)2p}(\hat w^0)^{2p-1}  \left(-\beta \nabla.(z \hat w^0)+\Delta \hat w^0+e.\nabla \hat w^0 +f\hat w^0+g.\nabla \Phi_{(1-\chi_{\frac{\zeta^*}{4}})\hat w^0} +h^0+ N^0(\hat w^\perp) \right)\frac{dz}{\zeta^2}. \nonumber
\end{align}
For the last term there holds from \eqref{est:xstar} and \eqref{bootstrap:outnonradial},
$$
\left( \int \left(1-\chi_{\frac{\zeta^*}{4}} \right) \zeta^{(2-\frac 14)2p}\left|N^0(\hat w^\perp)\right|^{2p} \frac{dz}{\zeta^2}\right)^{\frac{1}{2p}}\lesssim \frac{e^{-\kappa\tau}}{\nu^C} \| \hat w^\perp \|_{\out}\leq \frac{\nu^2}{|\log \nu|},
$$
for $\tau_0$ large enough. Hence, from H\"older, the above bound and the estimates \eqref{bd:interextpoisson0}, \eqref{bd:interexth0}, we get
\begin{align*}
&\int \left(1-\chi_{\frac{\zeta^*}{4}} \right) \zeta^{(2-\frac 14)2p}(\hat w^0)^{2p-1}\left(g.\nabla \Phi_{(1-\chi_{\frac{\zeta^*}{4}})\hat w^0} +h^0+ N^0(\hat w^\perp) \right)\frac{dz}{\zeta^2}\\
& \quad \leq \left(\int \left(1-\chi_{\frac{\zeta^*}{4}} \right) \zeta^{(2-\frac 14)2p}(\hat w^0)^{2p} \frac{dz}{\zeta^2}\right)^{\frac{2p-1}{2p}}  \times \left( \int \left(1-\chi_{\frac{\zeta^*}{4}} \right) \zeta^{(2-\frac 14)2p} \left|g.\nabla \Phi_{(1-\chi_{\frac{\zeta^*}{4}})\hat w^0} +h^0+N^0(\hat w^\perp)\right|^{2p}\frac{dz}{\zeta^2}\right)^{\frac{1}{2p}} \\
&\quad\lesssim  \left(\int \left(1-\chi_{\frac{\zeta^*}{4}} \right) \zeta^{(2-\frac 14)2p}(\hat w^0)^{2p} \frac{dz}{\zeta^2}\right)^{\frac{2p-1}{2p}} \frac{\nu^2}{|\log \nu|},
\end{align*}
for $\tau_0$ large enough. We inject the above identity in \eqref{id:energyestimatehatwout}, and use the linear estimate \eqref{id:linearenergyout} with the bounds \eqref{bootstrap:mid} and \eqref{bootstrap:perpmid} for the boundary terms and \eqref{bd:estimationefgexterior} for the potential and vector field, which yields \eqref{id:lyapunovout1}.\\

\noindent \textbf{Step 4} \emph{(The energy estimate for $\nabla (\hat w^0)$)}. Let $\hat w^i=\zeta \partial_{z_i}\hat w$ for $i=1,2$. We claim the energy estimate for $i=1,2$:
\begin{align}\label{id:lyapunovout2}
 &\frac{d}{d\tau} \left(\frac{1}{2p}\int \left(1-\chi_{\frac{\zeta^*}{4}} \right) \zeta^{(2-\frac 14)2p}(\hat w^i)^{2p}\frac{dz}{\zeta^2}\right)  \leq  \left(-\frac{\beta}{4}+\frac{C}{\zeta^*}\right) \int \left(1-\chi_{\frac{\zeta^*}{4}} \right) \zeta^{(2-\frac 14)2p}(\hat w^i)^{2p}\frac{dz}{\zeta^2}+\frac{C}{\zeta^*}\| \hat w^0 \|_{\out}^{2p}\\
& \qquad \qquad   \qquad  \qquad   \qquad  \qquad   \qquad   \qquad  \quad +C\frac{K^{'2p}\nu^{4p}}{|\log \nu|^{2p}}+ C\left(\int \left(1-\chi_{\frac{\zeta^*}{4}} \right) \zeta^{(2-\frac 14)2p}(\hat w^i)^{2p} \frac{dz}{\zeta^2}\right)^{\frac{2p-1}{2p}} \frac{K\nu^2}{|\log \nu|},\nonumber
\end{align}
which we now prove. From \eqref{id:eqhatw} and the commutator relations $[\zeta \partial_{z_i},z.\nabla]=0$ and $[\zeta \partial_{z_i},\Delta]=\zeta^{-2}\zeta\partial_{z_i}-2\zeta^{-2}z.\nabla (\zeta \partial_{z_i})$, we infer the evolution equation for $\hat w^i$:
\begin{align*}
& \partial_{\tau} \hat w^i = -\beta \nabla .(z\hat w^i)+\Delta \hat w^i +\left( e-\frac{2z}{\zeta^2}\right).\nabla \hat w^i +\left(f+\zeta^{-2}-\frac{z.e}{\zeta^2}\right) \hat w^i  + \Fc,
\end{align*}
where
\begin{align*}
\Fc = \zeta \partial_{z_i}e.\nabla \hat w^0 +\zeta\partial_{z_i}f\hat w^0+\zeta\partial_{z_i}(g.\nabla \Phi_{(1-\chi_{\frac{\zeta^*}{2}})\hat w^0})+\zeta\partial_{z_i}h^0+\zeta \partial_{z_i}N^0(\hat w^\perp),
\end{align*}
giving the energy estimate with $e'= e- 2\zeta^{-2}z$ and $f'=f+\zeta^{-2}-\zeta^{-2} z.e$:
\begin{align} \label{id:energyestimatehatwiout} 
 &\frac{d}{d\tau} \left(\frac{1}{2p}\int \left(1-\chi_{\frac{\zeta^*}{4}} \right) \zeta^{(2-\frac 14)2p}(\hat w^i)^{2p}\frac{dz}{\zeta^2}\right)\\
 &= \int \left(1-\chi_{\frac{\zeta^*}{4}} \right) \zeta^{(2-\frac 14)2p}(\hat w^i)^{2p-1}
\Big(-\beta \nabla.(z\hat w^i)+\Delta \hat w^i+e'.\nabla \hat w^i +f'\hat w^i+\Fc \Big)\frac{dz}{\zeta^2}.\nonumber
\end{align}
We apply for the linear part the estimate \eqref{id:linearenergyout} with the bounds \eqref{bootstrap:mid} and \eqref{bootstrap:perpmid} for the boundary terms and \eqref{bd:estimationefgexterior}  for the vector field $e'$ and the potential $f'$:
\begin{align*}
&\int \left(1-\chi_{\frac{\zeta^*}{4}} \right) \zeta^{(2-\frac 14)2p}(\hat w^i)^{2p-1} \left(-\beta \nabla.(z\hat w^i)+\Delta \hat w^i+e'.\nabla \hat w^i +f'\hat w^i\right)\frac{dz}{\zeta^2}\\
& \qquad \leq  -(2p-1)\int \left(1-\chi_{\frac{\zeta^*}{4}} \right) \zeta^{(2-\frac 14)2p}\bar w^{2p-2} |\nabla \bar w|^2 \frac{dz}{\zeta^2} \\
& \qquad \quad + \left(-\frac{\beta}{4}+\frac{C}{\zeta^*}\right)\int \left(1-\chi_{\frac{\zeta^*}{4}} \right) \zeta^{(2-\frac 14)2p}(\hat w^i)^{2p}\frac{dz}{\zeta^2}+C\frac{K^{'2p}\nu^{4p}}{|\log \nu|^{2p}}.
\end{align*}
We compute now the remaining terms. From H\"older, the bounds $|\zeta \nabla e|\lesssim 1$ and $|\nabla f|\lesssim \zeta^{-2}$ from \eqref{bd:estimationefgexterior}, for any function $\bar w$:
\begin{equation} \label{bd:estimationcommutatorinterext}
\left( \int \left(1-\chi_{\frac{\zeta^*}{4}} \right) \zeta^{(2-\frac 14)2p} \left(| \partial_{z_i}e||\zeta \nabla \bar w|+| \zeta \bar w \partial_{z_i}f|\right)^{2p} \frac{dz}{\zeta^2}\right)^{\frac{1}{2p}}\lesssim \frac{1}{\zeta^*}\| \bar w \|_{\out}.
\end{equation}
Hence, using the above bound for $\bar w=\hat w^0$ and H\"older:
$$
 \int \left(1-\chi_{\frac{\zeta^*}{4}} \right) \zeta^{(2-\frac 14)2p}(\hat w^i)^{2p-1} \left( \zeta \partial_{z_i}e.\nabla \hat w^0+\zeta\partial_{z_i}f\hat w^0 \right)\frac{dz}{\zeta^2}\leq \frac{C}{\zeta^*} \| \hat w^0 \|_{\out}. 
$$
Again, from H\"older, the bounds \eqref{bd:interextpoisson0} and \eqref{bd:interexth0}
\begin{align*}
& \int \left(1-\chi_{\frac{\zeta^*}{4}} \right) \zeta^{(2-\frac 14)2p}(\hat w^i)^{2p-1} \left(\zeta\partial_{z_i}(g.\nabla \Phi_{(1-\chi_{\frac{\zeta^*}{4}})\hat w^0})+\zeta\partial_{z_i}h^0 \right)\frac{dz}{\zeta^2}\\
& \leq  \left( \int \left(1-\chi_{\frac{\zeta^*}{4}} \right) \zeta^{(2-\frac 14)2p}(\hat w^i)^{2p} \frac{dz}{\zeta^2}\right)^{\frac{2p-1}{2p}} \left( \int \left(1-\chi_{\frac{\zeta^*}{4}} \right) \zeta^{(2-\frac 14)2p}\left| \zeta\partial_{z_i}(g.\nabla \Phi_{(1-\chi_{\frac{\zeta^*}{2}})\hat w^0})+ |\zeta\partial_{z_i}h^0| \right|^{2p} \frac{dz}{\zeta^2}\right)^{\frac{1}{2p}} \\
&\leq   C \| \hat w^0\|_{\out}^{2p-1} \frac{\nu^2}{|\log \nu|} ,
\end{align*}
for $\nu$ small enough, where we used \eqref{estimationpoissonexte} and \eqref{bootstrap:out}. For the last term we integrate by parts, use H\"older and the bootstrap bounds in Definition \ref{def:bootstrap}:
\begin{align*}
& \left|\int \left(1-\chi_{\frac{\zeta^*}{4}} \right) \zeta^{(2-\frac 14)2p}(\hat w^i)^{2p-1}\zeta \partial_{z_i}N^0(\hat w^\perp)\frac{dz}{\zeta^2}\right|\\
& \quad \leq C \frac{e^{-\kappa \tau}}{\nu} \left(\left(\frac{K'\nu^2}{|\log \nu|}\right)^{2p}+\| \hat w^0\|_{\out}^{2p-1}\| \hat w^\perp\|_{\out}^{2p-1}\right)\\
&\quad  +C \frac{e^{-\kappa \tau}}{\nu} \int \left(1-\chi_{\frac{\zeta^*}{4}} \right) \zeta^{(2-\frac 14)2p} |\nabla \hat w^i|^2(\hat w^i)^{2p-2}\frac{dz}{\zeta^2}+C \frac{e^{-\kappa \tau}}{\nu} \| \hat w^0\|_{\out}^{2p-1}\| \hat w^\perp\|_{\out}^{2p-1}\\
& \quad \leq  C \frac{e^{-\kappa \tau}}{\nu} \int \left(1-\chi_{\frac{\zeta^*}{4}} \right) \zeta^{(2-\frac 14)2p} |\nabla \hat w^i|^2(\hat w^i)^{2p-2}\frac{dz}{\zeta^2}+C\left(\frac{K'\nu^2}{|\log \nu|}\right)^{2p},
\end{align*}
%\\
%& \quad = -\int (\hat w^i)^{2p-1} \left(\frac{x^*_\tau}{\mu}.\nabla \hat w^\perp\right)^0\left(\partial_{z_i} \left(1-\chi_{\frac{\zeta^*}{4}} \right) \zeta^{(2-\frac 14)2p-1}+\left(1-\chi_{\frac{\zeta^*}{4}} \right) \partial_{z_i} \zeta^{(2-\frac 14)2p-1}\right)dz\\
%& \qquad \quad -(2p-1)\int \left(1-\chi_{\frac{\zeta^*}{4}} \right) \zeta^{(2-\frac 14)2p} \partial_{z_i}\hat w^i(\hat w^i)^{2p-2}\zeta \left(\frac{x^*_\tau}{\mu}.\nabla \hat w^\perp\right)^0\frac{dz}{\zeta^2}\\
%& \quad \leq C \left|\frac{x_\tau}{\mu}\right| \| \nabla \hat w\|_{L^{2p}(\frac{\zeta^*}{4}\leq |z|\leq \frac{\zeta^*}{2})}\\
%& \qquad \quad + C\left|\frac{x_\tau}{\mu}\right| \left(\int (\hat w^i)^{2p}\left(1-\chi_{\frac{\zeta^*}{4}} \right) \zeta^{(2-\frac 14)2p-1} dz\right)^{\frac{2p-1}{2p}}\left(\int |\nabla \hat w^\perp|^{2p}\left(1-\chi_{\frac{\zeta^*}{4}} \right) \zeta^{(2-\frac 14)2p-1} dz\right)^{\frac{1}{2p}}\\
%&\qquad \quad +C\left|\frac{x_\tau}{\mu}\right| \left(\int \left(1-\chi_{\frac{\zeta^*}{4}} \right) \zeta^{(2-\frac 14)2p} |\nabla \hat w^i|^2(\hat w^i)^{2p-2}\frac{dz}{\zeta^2}\right)^{\frac 12}\left(\int \left(1-\chi_{\frac{\zeta^*}{4}} \right) \zeta^{(2-\frac 14)2p} (\hat w^i)^{2p-2} |\zeta \nabla \hat w^\perp|^2\frac{dz}{\zeta^2}\right)^{\frac 12}\\
where we used Young inequality on the last line and took $\tau_0$ large enough. The collection of above inequalities yields \eqref{id:lyapunovout2}.\\

\noindent \textbf{Step 5} \emph{(End of the proof for $\hat w^0$)}. We sum the identities \eqref{id:lyapunovout1} and \eqref{id:lyapunovout2} for $i=1,2$, concluding the proof of \eqref{bd:lyapunovext0}.\\

\noindent \textbf{Step 6} \emph{(The energy estimate for $\hat w^\perp$)}. This step is very similar to Step 3 so we shall give less details. We claim that there holds the energy estimate
\begin{align}\label{id:lyapunovout1perp}
&\frac{d}{d\tau} \left(\frac{1}{2p}\int \left(1-\chi_{\frac{\zeta^*}{4}} \right) \zeta^{(2-\frac 14)2p}(\hat w^\perp)^{2p}\frac{dz}{\zeta^2}\right)\\
\nonumber & \leq  \left(-\frac{\beta}{4}+\frac{C}{\zeta^*}\right) \int \left(1-\chi_{\frac{\zeta^*}{4}} \right) \zeta^{(2-\frac 14)2p}(\hat w^\perp)^{2p}\frac{dz}{\zeta^2} +C\frac{K^{'2p}\nu^{4p}}{|\log \nu|^{2p}}+ C\left(\int \left(1-\chi_{\frac{\zeta^*}{4}} \right) \zeta^{(2-\frac 14)2p}(\hat w^0)^{\perp} \frac{dz}{\zeta^2}\right)^{\frac{2p-1}{2p}} \frac{\nu^2}{|\log \nu|}.
\end{align}
We compute from \eqref{id:eqhatwperp},
\begin{align} \label{id:energyestimatehatwoutperp}
 &\frac{d}{d\tau} \left(\frac{1}{2p}\int \left(1-\chi_{\frac{\zeta^*}{4}} \right) \zeta^{(2-\frac 14)2p}(\hat w^\perp)^{2p}\frac{dz}{\zeta^2}\right)\\
\nonumber & =\int \left(1-\chi_{\frac{\zeta^*}{4}} \right) \zeta^{(2-\frac 14)2p}(\hat w^\perp)^{2p-1}  \left(-\beta \nabla.(z \hat w^\perp)+\Delta \hat w^\perp+e.\nabla \hat w^\perp +f\hat w^\perp+g.\nabla \Phi_{(1-\chi_{\frac{\zeta^*}{4}})\hat w^\perp} +h^\perp+ N^\perp(\hat w^\perp) \right)\frac{dz}{\zeta^2}.\nonumber
\end{align}
For the last term there holds from \eqref{est:xstar}, \eqref{pointwiseawaynonradial} and \eqref{poissonLinfty}, 
$$
\left( \int \left(1-\chi_{\frac{\zeta^*}{4}} \right) \zeta^{(2-\frac 14)2p}\left| N^\perp(\hat w^\perp) \right|^{2p} \frac{dz}{\zeta^2}\right)^{\frac{1}{2p}}\lesssim \frac{e^{-\kappa \tau}}{\nu^C} \| \hat w^\perp \|_{\out}  \leq \frac{1}{\zeta^*}\| \hat w^\perp \|_{\out},
$$
for $\tau_0$ large enough. Therefore from H\"older, the above bound and the bounds \eqref{bd:interextpoissonperp}, \eqref{bd:interexthperp}:
\begin{align*}
&\int \left(1-\chi_{\frac{\zeta^*}{4}} \right) \zeta^{(2-\frac 14)2p}(\hat w^\perp)^{2p-1}\left(g.\nabla \Phi_{(1-\chi_{\frac{\zeta^*}{4}})\hat w^\perp} +h^\perp+ N^\perp(\hat w^\perp) \right)\frac{dz}{\zeta^2}\\
& \qquad \lesssim \left(\int \left(1-\chi_{\frac{\zeta^*}{4}} \right) \zeta^{(2-\frac 14)2p}(\hat w^\perp)^{2p} \frac{dz}{\zeta^2}\right)^{\frac{2p-1}{2p}} e^{-\kappa \tau}.
\end{align*}
We inject the above identity into \eqref{id:energyestimatehatwoutperp} and use the linear estimate \eqref{id:linearenergyout} with the bounds \eqref{bootstrap:perpmid} for the boundary terms and \eqref{bd:estimationefgexterior} for the potential and vector field to obtain \eqref{id:lyapunovout1perp}.\\

\noindent \textbf{Step 7} \emph{(The energy estimate for $\nabla (\hat w^\perp)$)}. This step is very similar to Step 4. Let $\hat w^i=\zeta \partial_{z_i}\hat w^\perp$ for $i=1,2$. We claim the energy estimate for $i=1,2$:
\begin{align}\label{id:lyapunovout2perp}
 \frac{d}{d\tau} \left(\frac{1}{2p}\int \left(1-\chi_{\frac{\zeta^*}{4}} \right) \zeta^{(2-\frac 14)2p}(\hat w^i)^{2p}\frac{dz}{\zeta^2}\right) &\leq \left(-\frac{\beta}{4}+\frac{C}{\zeta^*}\right) \int \left(1-\chi_{\frac{\zeta^*}{4}} \right) \zeta^{(2-\frac 14)2p}(\hat w^i)^{2p}\frac{dz}{\zeta^2}+\frac{C}{\zeta^*}\| \hat w^0 \|_{\out}^{2p}\\
& +C\frac{K^{'2p}\nu^{4p}}{|\log \nu|^{2p}}+ C\left(\int \left(1-\chi_{\frac{\zeta^*}{4}} \right) \zeta^{(2-\frac 14)2p}(\hat w^i)^{2p} \frac{dz}{\zeta^2}\right)^{\frac{2p-1}{2p}} \frac{K\nu^2}{|\log \nu|}.\nonumber
\end{align}
The evolution equation for $\hat w^i$ is
\begin{align*}
\partial_{\tau} \hat w^i & = -\beta \nabla .(z\hat w^i)+ \Delta \hat w^i + \left( e-\frac{2z}{\zeta^2}\right).\nabla \hat w^i +\left(f+\zeta^{-2}-\frac{z.e}{\zeta^2}\right) \hat w^i  + H,
\end{align*}
where
\begin{align*}
H = \zeta \partial_{z_i}e.\nabla \hat w^0 +\zeta\partial_{z_i}f\hat w^0+\zeta\partial_{z_i}(g.\nabla \Phi_{(1-\chi_{\frac{\zeta^*}{2}})\hat w^\perp})+\zeta\partial_{z_i}\big[h^\perp + N^\perp(\hat w^\perp)\big].
\end{align*}
Giving the energy estimate with $e'= e- 2\zeta^{-2}z$ and $f'=f+\zeta^{-2}-\zeta^{-2} z.e$, we write
\begin{align*}
&\frac{d}{d\tau} \left(\frac{1}{2p}\int \left(1-\chi_{\frac{\zeta^*}{4}} \right) \zeta^{(2-\frac 14)2p}(\hat w^i)^{2p}\frac{dz}{\zeta^2}\right)\\
&  = \int \left(1-\chi_{\frac{\zeta^*}{4}} \right) \zeta^{(2-\frac 14)2p}(\hat w^i)^{2p-1}\times \left(-\beta \nabla.(z\hat w^i)+\Delta \hat w^i+e'.\nabla \hat w^i +f'\hat w^i+H \right)\frac{dz}{\zeta^2}.
\end{align*}
We apply for the linear part the estimate \eqref{id:linearenergyout} with the bounds \eqref{bootstrap:mid} and \eqref{bootstrap:perpmid} for the boundary terms and \eqref{bd:estimationefgexterior}  for the vector field $e'$ and the potential $f'$, and the bound \eqref{bd:estimationcommutatorinterext} with $\bar w =w^\perp$:
\begin{align*}
&\int \left(1-\chi_{\frac{\zeta^*}{4}} \right) \zeta^{(2-\frac 14)2p}(\hat w^i)^{2p-1} \left(-\beta \nabla.(z\hat w^i)+\Delta \hat w^i+e'.\nabla \hat w^i +f'\hat w^i+ \zeta \partial_{z_i}e.\nabla \hat w^\perp+\zeta\partial_{z_i}f\hat w^\perp \right)\frac{dz}{\zeta^2}\\
& \qquad \leq  -(2p-1)\int \left(1-\chi_{\frac{\zeta^*}{4}} \right) \zeta^{(2-\frac 14)2p}\bar w^{2p-2} |\nabla w^\perp|^2 \frac{dz}{\zeta^2} \\
& \qquad \quad + \left(-\frac{\beta}{4}+\frac{C}{\zeta^*}\right)\int \left(1-\chi_{\frac{\zeta^*}{4}} \right) \zeta^{(2-\frac 14)2p}(\hat w^i)^{2p}\frac{dz}{\zeta^2}+CK^{'2p}e^{-\kappa\tau}.
\end{align*}
From H\"older and the bounds \eqref{bd:interextpoisson0} and \eqref{bd:interexth0}:
\begin{align*}
\int \left(1-\chi_{\frac{\zeta^*}{4}} \right) \zeta^{(2-\frac 14)2p}(\hat w^i)^{2p-1} &\left(\zeta\partial_{z_i}(g.\nabla \Phi_{(1-\chi_{\frac{\zeta^*}{4}})\hat w^\perp})+\zeta\partial_{z_i}h^\perp \right)\frac{dz}{\zeta^2}  \leq C \| \hat w^\perp \|_{\out}^{2p-1} e^{-\kappa \tau}.
\end{align*}
We integrate by parts, use H\"older and the bootstrap bounds \eqref{bootstrap:mid}, \eqref{bootstrap:perpmid}, \eqref{pointwiseawaynonradial}, \eqref{poissonLinfty},
\begin{align*}
& \int \left(1-\chi_{\frac{\zeta^*}{4}} \right) \zeta^{(2-\frac 14)2p}(\hat w^i)^{2p-1}\zeta \partial_{z_i}N^\perp(\hat w^\perp) \frac{dz}{\zeta^2}\\
& \quad = -\int (\hat w^i)^{2p-1} N^\perp(\hat w^\perp) \left(\partial_{z_i} \left(1-\chi_{\frac{\zeta^*}{4}} \right) \zeta^{(2-\frac 14)2p-1}+\left(1-\chi_{\frac{\zeta^*}{4}} \right) \partial_{z_i} \zeta^{(2-\frac 14)2p-1}\right)dz\\
& \qquad -(2p-1)\int \left(1-\chi_{\frac{\zeta^*}{4}} \right) \zeta^{(2-\frac 14)2p} \partial_{z_i}\hat w^i(\hat w^i)^{2p-2}\zeta N^\perp(\hat w^\perp) \frac{dz}{\zeta^2}\\
%& \quad \leq C \left(\left|\frac{x_\tau}{\mu}\right| + e^{-\kappa\tau} \right) \left(\| \nabla \hat w \|_{L^{2p}(\frac{\zeta^*}{4}\leq |z|\leq \frac{\zeta^*}{2})}\| \nabla \hat w^\perp \|_{L^{2p}(\frac{\zeta^*}{4}\leq |z|\leq \frac{\zeta^*}{2})}^{2p-1}+\| \hat w \|_{\out}\| \hat w^\perp \|_{\out}^{2p-1}  \right)\\
%& \qquad +C\left(\left|\frac{x_\tau}{\mu}\right| + e^{-\kappa \tau} \right) \left(\int \left(1-\chi_{\frac{\zeta^*}{4}} \right) \zeta^{(2-\frac 14)2p} |\nabla \hat w^i|^2(\hat w^i)^{2p-2}\frac{dz}{\zeta^2}\right)^{\frac 12}\left(\int \left(1-\chi_{\frac{\zeta^*}{4}} \right) \zeta^{(2-\frac 14)2p} (\hat w^i)^{2p-2} |\zeta \nabla \hat w|^2\frac{dz}{\zeta^2}\right)^{\frac 12}\\
& \quad \leq C \frac{e^{-\kappa \tau}}{\nu}\left( \frac{K'\nu^2}{|\log \nu|} \left(K'e^{-\kappa \tau}\right)^{2p}+\frac{K''\nu^2}{|\log \nu|} \| \hat w \|_{\out}^{2p-1}\right)\\
& \qquad  + \int \left(1-\chi_{\frac{\zeta^*}{4}} \right) \zeta^{(2-\frac 14)2p} |\nabla \hat w^i|^2(\hat w^i)^{2p-2}\frac{dz}{\zeta^2}+C \frac{e^{-\kappa \tau}}{\nu}\| \hat w \|_{\out}^{2}\| \hat w^\perp\|_{\out}^{2p-2}\\
& \quad \leq   \int \left(1-\chi_{\frac{\zeta^*}{4}} \right) \zeta^{(2-\frac 14)2p} |\nabla \hat w^i|^2(\hat w^i)^{2p-2}\frac{dz}{\zeta^2}+\left(K'e^{-\kappa \tau}\right)^{2p},
\end{align*}
for any where we used Young inequality on the last line and took $\tau_0$ large enough. The collection of above inequalities yields \eqref{id:lyapunovout2perp}. We sum the identities \eqref{id:lyapunovout1perp} and \eqref{id:lyapunovout2perp} for $i=1,2$, concluding the proof of \eqref{bd:lyapunovextperp}.
\end{proof}

\subsection{Proof of Proposition \ref{pr:bootstrap}}\label{sec:ProofofProExist}
We give the proof of Proposition \eqref{pr:bootstrap} from which directly implies the conclusion of Theorem \ref{theo:Stab}. Assume that the solution is initially trapped in the sense of Definition \ref{def:ini}. We then define
$$
\tau^*=\sup \left\{ \tau_1\geq \tau_0, \mbox{ such that the solution is trapped on } [\tau_0,\tau_1] \right\}.
$$
We assume by contradiction that $\tau_1<\infty$ and will show that this is impossible by integrating in time the various modulation equations and Lyapunov functionals. Throughout the proof, $C$ denotes a universal constant that is independent of the bootstrap constants $\kappa,N,K,K',K''$ and the dependence on $K,K',K''$ in the various $\Oc$'s is explicitly mentioned. Recall Remark \ref{re:orderconstants} regarding the order in which the constants are chosen. \\

\noindent \textbf{Step 1:} \emph{Improved modulation for $\nu$ and $\beta$}. We inject the identity \eqref{eq:mod1} and the bootstrap bound \eqref{bootstrap:L2omeofme} in \eqref{eq:mod0}, then use the eigenvalue expansion \eqref{def:specAsb} and the compatibility condition \eqref{bootstrap:condition} to get
\begin{align}\label{bd:reintegrationinternu}
 \frac{\nu_\tau}{\nu}&=\beta-\beta \frac{a_1}{4\nu^2}(\tilde \alpha_1-1-\tilde \alpha_0)+\Oc\left(\frac{K\mathcal D(\tau)}{|\ln \nu|^2}\right)+\Oc\left(\frac{1}{|\ln \nu|^3} \right)\\
\nonumber &=\beta-\beta \left(\left(-1+\frac{1}{2\ln \nu}+\frac{\ln 2-\gamma-1-\ln \beta}{4|\ln \nu|^2}\right)\left(-1-\frac{1}{4|\ln \nu|^2}\right)\right)+\Oc\left(\frac{K\mathcal D(\tau)}{|\ln \nu|^2}\right)+\Oc\left(\frac{1}{|\ln \nu|^3} \right) \nonumber \\
&=\beta \left(\frac{1}{2\ln \nu}+\frac{\ln 2 -\gamma-2-\ln \beta}{4|\ln \nu|^2}\right)+\Oc\left(\frac{K\mathcal D(\tau)}{|\ln \nu|^2} \right)+\Oc\left(\frac{1}{|\ln \nu|^3}\right).\nonumber
\end{align}
We then inject this identity, the bootstrap bound \eqref{bootstrap:L2omeofme} and the eigenvalue expansion \eqref{def:specAsb} into \eqref{eq:mod1} to obtain
\begin{eqnarray*}
a_{1,\tau}&=&2\beta a_1\tilde \alpha_1-a_1\frac{\beta}{2|\ln \nu|^2}+\frac{a_1\beta_\tau}{\beta}+\Oc\left(\frac{1}{|\ln \nu|^3}\right)+\Oc\left(\frac{K\mathcal D(\tau)}{|\ln \nu|^2}\right)\\
&=& \beta a_1 \left(\frac{1}{\ln \nu}+\frac{\ln 2 -\gamma-2-\ln \beta}{2|\ln \nu|^2}\right)+a_1\frac{\beta_\tau}{\beta}+\Oc\left(\frac{1}{|\ln \nu|^3}\right)+\Oc\left(\frac{K\mathcal D(\tau)}{|\ln \nu|^2}\right).
\end{eqnarray*}
We differentiate the compatibility condition $-1+\frac{1}{2\ln \nu}+\frac{(\ln 2-\gamma-1-\ln \beta)}{4|\ln \nu|^2}=\frac{a_1}{4\nu^2}$ and inject the two identities above to arrive at
\begin{align*}
&\left(\frac{\beta_\tau}{\beta}+\Oc\left(\frac{1}{|\ln \nu|^3}+\frac{K\mathcal D(\tau)}{|\ln \nu|^2}\right)\right)\frac{a_1}{4\nu^2}= \left(\frac{a_{1,\tau}}{a_1}-2\frac{\nu_\tau}{\nu} \right)\frac{a_1}{4\nu^2}\\
&=\frac{\nu_\tau}{\nu}\left(-\frac{1}{2|\ln \nu|^2}-\frac{\ln 2 -\gamma-1-\ln \beta}{2|\ln \nu|^3}\right)=\Oc\left(\frac{1}{|\ln \nu|^3}+\frac{K\mathcal D(\tau)}{|\ln \nu|^4} \right).
\end{align*}
Since $\frac{a_1}{4\nu^2}=-1+\Oc(|\ln \nu|^{-1})$ and $\beta \sim 1/2$ in the bootstrap, we obtain $\beta_\tau=O\left(\frac{1}{|\ln \nu|^3}+\frac{K\mathcal D(\tau)}{|\ln \nu|^2}\right)$, from which and \eqref{bd:reintegrationinternu}, we arrive at the system
\begin{equation} \label{bd:betatau}
\left\{ \begin{array}{l l} \frac{\nu_\tau}{\nu}=\beta \left(\frac{1}{2\ln \nu}+\frac{\ln 2 -\gamma-2-\ln \beta}{4|\ln \nu|^2}\right)+\Oc\left(\frac{K}{|\ln \nu|^3}\right),\\ \beta_\tau=\Oc\left(\frac{K}{|\ln \nu|^3}\right).\end{array}\right.
\end{equation}

\noindent \textbf{Step 2:} \emph{Reintegrating the modulation equations}. We introduce the renormalised time $\tilde \tau =2\beta_0\tau_0+2\int_{\tau_0}^\tau \beta$ and write $\nu=\sqrt{2\beta^{-1}} e^{-\frac{2+\gamma}{2}}e^{-\sqrt{\tilde \tau/2}}\left(1+ \nu'\right)$. From \eqref{bd:betatau}, we have $\tilde \tau=2\beta \tau+O(\tau^{-1/2})$ and 
$$
\frac{\nu_\tau}{\nu}=\frac{\nu'_\tau}{1+\nu'}-\frac{2\sqrt 2 \beta}{\sqrt{\tilde \tau}}+\Oc(|\beta_{\tau}|)=\frac{\nu'_\tau}{1+\nu'}-\frac{ \beta }{\sqrt{ 2 \tilde \tau}}+\Oc(\tau^{-3/2}).
$$
From \eqref{bootstrap:param1} and \eqref{parametersinit}, one has $|\nu'|\lesssim (\ln |\ln \nu|)|\ln \nu|^{-1}$ and the linearisation provides
\begin{align*}
\frac{1}{2\ln \nu}+\frac{\ln 2 -\gamma-2-\ln \beta}{4|\ln \nu|^2}&=\frac{1}{2\left(\ln (\nu'+1)+\frac{\ln 2-\gamma-2-\ln \beta}{2}-\sqrt{\frac{\tilde \tau}{2}}\right)}  +\frac{\ln 2 -\gamma-2-\ln \beta}{2\tilde \tau +\Oc(\sqrt{\tau})}\\
&=-\frac{1}{\sqrt{2\tilde \tau}}-\frac{\nu'}{\tilde \tau}+\Oc(\tau^{-\frac 32}).
\end{align*}
Equation \eqref{bd:betatau} is then transformed into
$$
\nu'_\tau =-\frac{\beta \nu'}{\tilde \tau}+\Oc(K\tau^{-\frac 32}) \quad \mbox{ so that } \quad \nu'_{\tilde \tau} =-\frac{ \nu'}{2\tilde \tau}+\Oc(K\tilde \tau^{-\frac 32}) 
$$
which reintegrated in time by  using \eqref{parametersinit} gives 
\begin{equation} \label{bd:reintegrationnu}
|\nu'(\tau_1)|=|\nu'(\tilde \tau_1)|\leq \frac{\tilde \tau_0}{\tilde \tau_1}|\nu'(\tau_0)|+\frac{CK\ln \tilde \tau_1}{\sqrt{\tilde \tau_1}}\leq \frac{K'}{2}\frac{\ln |\ln \nu|}{|\ln \nu|}.
\end{equation}
As $\beta_\tau =\Oc(K|\ln \nu|^{-3})=\Oc(K\tau^{-3/2})$, we use \eqref{parametersinit} to obtain
\begin{equation} \label{bd:reintegrationbeta}
\left|\beta (\tau_1)-\frac 12\right| =\left|\beta (\tau_0)-\frac 12 +\Oc(K\tau_0^{-1/2})\right|\leq CK\tau_0^{-\frac 12}\leq \frac{K'}{2|\ln \nu_0|}.
\end{equation}
Finally, we reintegrate the modulation equations \eqref{eq:mod2} for the other parameters $a_i$ for $i=2,...,n$. Using \eqref{bootstrap:L2omeofme} and the fact that $\beta \geq 1/4$ and $|\tilde \alpha_i|\leq 1/2$ yield
\begin{align*}
\frac{d}{d\tau}(a_i^2)&=4\beta (1 - n + \tilde{\alpha}_n)a_i^2+2a_i \left( \Oc\left( \frac{\Dc(\tau)}{|\ln \nu|} \|m_\e\|_{L^2_{\frac{\omega_\nu}{\zeta}}} \right) + \Oc\left(\frac{\nu^2}{|\ln \nu|^2}\right) \right) \leq -\frac{a_i^2}{2}+K\tau^{-2}e^{-2\sqrt{\beta \tau}}.
\end{align*}
Reintegrating in time, using \eqref{bd:betatau} and \eqref{parametersinitan} yields that for some universal constant $C>0$ the estimate
\begin{align}\label{bd:intermodaiigeq2}
 a_i^2(\tau_1)& \leq e^{-\frac{\tau_1-\tau_0}{2}}a_i^2(\tau_0)+e^{-\frac{\tau_1}{2}}\int_{\tau_0}^{\tau_1}e^{\frac{\tau}{2}}K\tau^{-2}e^{-2\sqrt{\beta \tau}}d\tau \leq C \tau_1^{-2}e^{-2\sqrt{\beta \tau_1}}+CK\tau_1^{-2}e^{-2\sqrt{\beta \tau_1}} \\
 &\leq CK\frac{\nu^4}{|\log \nu|^4}\leq \frac{K^2}{2}\frac{\nu^4}{|\log \nu|^4}.\nonumber
\end{align}

\noindent \textbf{Step 3:} \emph{The Lyapunov functionals}. We inject the estimate $|\mathcal D(\tau)|\leq C |\ln \nu|^{-1}$ into \eqref{eq:mod0} and \eqref{eq:mod1} to obtain the estimate
$$
|\mbox{Mod}_0|+|\mbox{Mod}_1|\leq C\frac{\nu^2}{|\ln \nu|^{2}}.
$$
We inject this estimate and the bootstrap bound \eqref{bootstrap:L2omeofme} and the last estimate in \eqref{bd:intermodaiigeq2} into \eqref{eq:monoL2ome} to get
\begin{align*}
\frac{1}{2}\frac{d}{d\tau} \big\|\bar m_\e \big\|^2_{L^2_{\frac{\bar \omega_\nu}{\zeta}}} &\leq -2\beta\left(N - C \right)\big\|\bar m_\e \big\|^2_{L^2_{\frac{\bar \omega_\nu}{\zeta}}}+K\frac{\nu^2}{|\log \nu|^3}+CK\frac{\nu^4}{|\log \nu|^2}+C\frac{\nu^4}{|\log \nu|^2} \\
& \quad \leq -\big\|\bar m_\e \big\|^2_{L^2_{\frac{\bar \omega_\nu}{\zeta}}}+CK\frac{e^{-4\sqrt{\beta \tau}}}{\tau}
\end{align*}
for $N$ large enough. We integrate in time this identity  and use the initial condition \eqref{mepsiloninitL2omega},
\begin{align} \label{bd:reintegrationL2omega}
& \big\|\bar m_\e(\tau_1) \big\|^2_{L^2_{\frac{\bar \omega_\nu}{\zeta}}} \leq  e^{-2(\tau_1-\tau_0)}\big\|\bar m_\e(\tau_0) \big\|^2_{L^2_{\frac{\bar \omega_\nu}{\zeta}}}+CKe^{-2\tau_1}\int_{\tau_0}^{\tau_1} e^{2\tau}\frac{e^{-4\sqrt{\beta \tau}}}{\tau}d\tau  \leq CK\frac{e^{-4\sqrt{\beta \tau}}}{\tau}\leq \frac{K^2}{2}\frac{\nu^4}{|\log \nu|^2},
\end{align}
(where we used $|\beta_\tau|\lesssim K|\log \nu|^3$ to integrate by parts in time). Next, we directly have from \eqref{est:H2me} the bound
\begin{equation} \label{bd:reintegrationmid}
\| m_\e (\tau_1) \|_{H^2(\zeta_* \leq \zeta \leq \zeta^*)}\leq \frac{K'}{2} \frac{\nu^2}{|\ln \nu|}.
\end{equation}
We now turn to the inner estimates. Let $M\gg 1$ be a large constant. Then for $\zeta_*$ small enough, we obtain from \eqref{eq:1stEne} and the bootstrap bounds \eqref{bootstrap:mid}:
\begin{align*}
&\frac{d}{d\tau} \left[- e^{M\tau}\tnu^2\int \tm_v^* \As_0 \tm_v^* \frac{\omega_0}{\tr} \right] +e^{M\tau} \|\As_0 \tm_v^*\|_{L^2_\frac{\omega_0}{\tr}}   \leq  CK^{'2}e^{M\tau}\frac{\nu^4}{|\log \nu|^2}\leq CK^{'2}e^{M\tau}\frac{e^{-4\sqrt{\beta \tau}}}{\tau},
\end{align*}
from which and a reintegration in time  and the dissipation estimate from \eqref{bd:initint}, we derive
\begin{align*}
&-\tnu^2(\tau_1)\int \tm_v^*(\tau_1) \As_0 \tm_v^* (\tau_1)\frac{\omega_0}{\tr}+e^{-M\tau_1}\int_{\tau_0}^{\tau_1}e^{M\tau} \|\As_0 \tm_v^*\|_{L^2_\frac{\omega_0}{\tr}}d\tau\\
 & \quad \leq  -e^{-M(\tau_1-\tau_0)} \tnu^2(\tau_0)\int \tm_v^*(\tau_0) \As_0 \tm_v^* (\tau_0)\frac{\omega_0}{\tr}+e^{-M\tau_1}CK^{'2}\int_{\tau_0}^{\tau_1}e^{M\tau}\frac{e^{-4\sqrt{\beta \tau}}}{\tau}d\tau \\
&\quad \leq  \frac{CK^{'2}}{M} \frac{e^{-4\sqrt{\beta \tau}}}{\tau}\leq \frac{K^{''2}}{2M}\frac{\nu^4}{|\log \nu|^2}.
\end{align*}
We inject the bound \eqref{bootstrap:mid} into \eqref{eq:2ndEne} to obtain
$$
\frac{d}{d\tau}\left(e^{M\tau}\| \As_0 \tm_v^*\|_{L^2_\frac{\omega_0}{\tr}}^2\right) \leq C \left(Me^{M\tau}\| \As_0 \tm_v^*\|_{L^2_\frac{\omega_0}{\tr}}^2  + \frac{K'\nu^4}{|\ln \nu|^2}\right).
$$
Reintegrating in time and using the previous dissipation estimate for $\| \As_0 \tm_v^*\|_{L^2_\frac{\omega_0}{\tr}}^2$ and \eqref{bd:initint} yield
\begin{align}\label{bd:reintegrationint}
 & \| \As_0 \tm_v^*(\tau_1)\|_{L^2_\frac{\omega_0}{\tr}}^2\leq e^{-M(\tau_1-\tau_0)}\| \As_0 \tm_v^*(\tau_0)\|_{L^2_\frac{\omega_0}{\tr}}^2 +CMe^{-M\tau_1}\int_{\tau_0}^{\tau_1}e^{M\tau}\| \As_0 \tm_v^*\|_{L^2_\frac{\omega_0}{\tr}}^2d\tau+CK^{'2}\int_{\tau_0}^{\tau_1} e^{M\tau}\frac{e^{-4\sqrt{\beta \tau}}}{\tau}d\tau \\
 & \qquad \quad \leq  \frac{CK^{'2}e^{-4\sqrt{\beta \tau}}}{\tau}\leq \frac{K^{''2}}{2}\frac{\nu^4}{|\log \nu|^2}.\nonumber
\end{align}
We turn to the estimates for the nonradial part. Reintegrating in time \eqref{ene:ebot0} directly gives the bound
\begin{equation} \label{bd:reintegrationL2omegaperp}
\| \e^\perp \|_{0}\leq Ce^{-\kappa \tau}\leq \frac{K}{2}e^{-\kappa \tau}.
\end{equation}
From the bootstrap bound \eqref{bootstrap:perpmid}, we have
\begin{align*}
&\frac{d}{d\tau} \left[ \int \tilde{q}^\perp_* \Ms \tilde{q}_*^\perp  dy - \int \Ls_0 \tilde q^\perp_* \Ms \tilde{q}_*^\perp dy \right] \leq  -\delta_2'\left[ \frac{1}{\zeta_*^2}\int \tilde{q}^\perp_* \Ms \tilde{q}_*^\perp  dy - \int \Ls_0 \tilde q^\perp_* \Ms \tilde{q}_*^\perp dy  \right] + \frac{C(K^{2}+K^{'2})}{\nu^2}e^{-2\kappa \tau}.
\end{align*}
For $\kappa$ small enough depending on $\delta_2'$, we have from \eqref{bd:initintperp},
\begin{align} \label{bd:reintegrationintperp}
 &\left[ \int \tilde{q}^\perp_*(\tau_1) \Ms \tilde{q}_*^\perp (\tau_1)  dy - \int \Ls_0 \tilde q^\perp_*(\tau_1) \Ms \tilde{q}_*^\perp (\tau_1) dy \right] \leq C(K^2+K^{'2})\frac{e^{-2\kappa \tau}}{\nu^2}\leq \frac{K^{''2}}{2\nu^2}e^{-2\kappa \tau}.
\end{align}
Finally,\eqref{bd:lyapunovext0} and \eqref{bd:lyapunovextperp} are directly reintegrated in time for $\zeta^*$ large enough and $\kappa$ small enough, using the initial bounds \eqref{bd:initout} and \eqref{bd:initoutperp},
\begin{equation} \label{bd:reintegrationout}
\| \hat w^0\|_{\out}\leq \frac{K''\nu^2}{2|\ln \nu|}, \quad \| \hat w^\perp \|_{\out}\leq \frac{K''}{2}e^{-\kappa \tau}.
\end{equation}
\noindent \textbf{Step 4:} \emph{End of the proof of Proposition \ref{pr:bootstrap}}. In Step 1 and Step 2, all the bounds involved in the Definition \ref{def:bootstrap} have been improved by a factor $1/2$ at time $\tau_1$, from \eqref{bd:reintegrationnu}, \eqref{bd:reintegrationbeta}, \eqref{bd:intermodaiigeq2}, \eqref{bd:reintegrationL2omega}, \eqref{bd:reintegrationmid}, \eqref{bd:reintegrationint}, \eqref{bd:reintegrationL2omegaperp}, \eqref{bd:reintegrationintperp}, \eqref{bd:reintegrationout}. Hence, by a continuity argument, these bounds also hold true on some time interval $[\tau_1,\tau_1+\delta]$ for some small $\delta>0$. This contradicts the definition of $\tau_1$. Hence $\tau_1=\infty$ and the solution is trapped in the regime \ref{def:bootstrap} for $\tau \in [\tau_0, +\infty)$. Knowing the solution is global in time $\tau$, reintegrating \eqref{bd:betatau} yields
$$
\beta(\tau)=\beta_\infty+\Oc(\tau^{-1/2}), \ \ \beta_\infty=\beta(\tau_0)+\int_{\tau_0}^\infty \beta_\tau d\tau.
$$
Recall the renormalised time $\tilde \tau$ of Step 2, we obtain from the above identity that
\begin{equation} \label{reinteproxtautildetau}
\tau=\tilde \tau/2\beta_{\infty} +\Oc(\sqrt{\tilde \tau}),
\end{equation}
and since by definition $\mu_\tau/\mu=\beta$ that $\mu(\tilde \tau)=e^{-\tilde \tau/2}$. To go back to the original time variable, we integrate
$$
\frac{d\tilde \tau}{dt}=\frac{d\tau}{dt}\frac{d\tilde \tau}{d\tau}=\frac{2}{\mu^2}\beta=2e^{\tilde \tau} \beta_\infty (1+\Oc(\tilde \tau^{-1/2})).
$$
Solving the above equation, there exists $T>0$ such that:
\begin{equation} \label{reinte:exprs}
\tilde \tau=-\log \left(2\beta_\infty (T-t)\right)+\Oc\left(\frac{1}{|\log \left((T-t)\right)|^{\frac 12}}\right).
\end{equation}
Hence, we obtain the following expression for the parabolic scale $\mu$,
\begin{equation} \label{reinte:exprmu}
\mu= e^{-\frac{\tilde \tau}{2}} =\sqrt{2\beta_\infty}\sqrt{T-t}\left(1+\Oc\left(\frac{1}{|\log \left((T-t)\right)|^{\frac 12}}\right)\right).
\end{equation}
We get from $\beta=\beta_\infty+\Oc(\tau^{-1/2})=\beta_{\infty}+\Oc(|\ln(T-t)|^{-1/2})$ and \eqref{parametersinit} that
\begin{align*}
\nu &=\sqrt{\frac{2}{\beta}}e^{-\frac{2+\gamma}{2}}e^{-\sqrt{\frac{\tilde \tau}{2}}}\left(1+\Oc\left(\frac{1}{|\ln \nu|^{\frac 12}}\right)\right)=\sqrt{\frac{2}{\beta_{\infty}}}e^{-\frac{2+\gamma}{2}}e^{-\sqrt{\frac{|\ln (T-t)|}{2}}}\left(1+\Oc\left(\frac{1}{|\ln (T-t)|^{\frac{1}{4}}}\right)\right)
\end{align*}
and hence, we get the desired blowup speed
$$
\lambda =\mu \nu =2e^{-\frac{2+\gamma}{2}}\sqrt{T-t}e^{-\sqrt{\frac{|\ln (T-t)|}{2}}}\left(1+\Oc\left(\frac{1}{|\ln (T-t)|^{\frac{1}{4}}}\right)\right).
$$
From \eqref{est:xstar} (the right hand side being less than $1$) and \eqref{reinte:exprmu}, we get the rough bound:
$$
|x^*_t|=\frac{1}{\mu} \frac{|x^*_\tau |}{\mu}\lesssim \frac{1}{\sqrt{T-t}}.
$$
This implies that $x^*(t)$ converges to some $X\in \mathbb R^2$ as $t\rightarrow T$. We now turn to the proof of the continuity of the blowup time and blowup point with respect to the initial datum. Fix $u_0\in \mathcal O$ with blowup time $T(u_0)$, blowup point $x^*_u$ with limit $X(u_0)$, and renormalised time $\tau_u$, and $\delta >0$. From \eqref{reinte:exprs}, \eqref{reinteproxtautildetau} and the above inequality, for any $\tau_1$ large enough, there exists $T(u_0)-\delta/2\leq T_1<T(u_0)$ such that $\tau_u(T_1)\geq \tau_1+1$ and $|x^*_u(T)-X(u_0)|\leq \delta/3$. By continuity we then obtain that for another solution $v$, $v$ and all $v$-related parameters converge to $u$ and $u$-related parameters on $[0,T_1]$ as $v_0\rightarrow u_0$ in $\mathcal E$. In particular, $\tau_v(T_1)\geq \tau_1$ and $|x^*_v(T_1)-X(u_0)|\leq \delta/2$. By \eqref{reinte:exprs}, \eqref{reinteproxtautildetau} and the above inequality, we get that for $\tau_1$ large enough, $|X(v)-x^*_v(T_1)|\leq \delta/2$ and $|T(v_0)-T_1|\leq \delta/2$. Hence by summing we obtain the desired continuity $|X(u_0)-X(v_0)|\leq \delta$ and $|T(u_0)-T(v_0)|\leq \delta$. This concludes the proof of Proposition \ref{pr:bootstrap} as well as Theorem \ref{theo:Stab}. \hfill $\square$

\section{Unstable blowup dynamics} \label{sec:UnstabTheo2}
In this section, we sketch the idea to establish the existence of unstable blowup solutions to system \eqref{eq:KS} in the radial setting. The strategy of the proof of Theorem \ref{theo:Stab} has to be modified the following way. Fix $\ell \in \mathbb{N}$ with $\ell \geq 2$ and $N \in \mathbb{N}$ with $N \gg 1$ and consider the approximate solution of the form
\begin{equation}\label{def:mW7}
m_W(\zeta, \tau) = Q_\nu(\zeta) + a_\ell(\tau) \big(\phi_{\ell, \nu}(\zeta) - \phi_{0, \nu}(\zeta)\big) + \sum_{n = 1, n\ne \ell}^{N} a_n(\tau)\phi_{n, \nu}(\zeta) = Q_\nu + P_\nu,
\end{equation} 
where the approximate perturbation is $P_\nu = P_{\ell,\nu} + P_{+, \nu} + P_{N, \nu}$ with
$$ P_{\ell, \nu} = a_\ell(\tau) \big(\phi_{\ell, \nu}(\zeta) - \phi_{0, \nu}(\zeta)\big), \;\; P_{+,\nu} = \sum_{i = 1}^{\ell - 1} a_i(\tau)\phi_{i, \nu}(\zeta),\;\; P_{N,\nu} = \sum_{i = \ell + 1}^{N} a_i(\tau)\phi_{i, \nu}(\zeta)$$
and $\mathbf{a}(\tau) = (a_1, \cdots, a_N)(\tau)$ and $\nu(\tau)$ are unknown functions to be determined, 
and $(\phi_{n, \nu})_{0\leq n\leq N}$ are the eigenfunctions  described in Proposition \ref{prop:SpecRad}, i.e, 
\begin{equation}\label{eq:eigeninETA}
\As^\zeta \phi_{n,\nu}(\zeta) = 2\beta\Big(1 - n + \frac{1}{2 \ln \nu} + \bar{\alpha}_n(\nu) \Big)\phi_{n, \nu}(\zeta)
\end{equation} 
with $ |\bar{\alpha}(\nu)| \lesssim \frac{1}{|\ln \nu |^2}$ and $\phi_{n,\nu}$ is given by \eqref{def:hatphi}. Here, the leading order term in the approximate perturbation is $P_{\ell, \nu}$ which drives the law of blowup law. The first higher order term $P_{+, \nu}$ contains $\ell - 1$ "unstable" directions for $\ell \geq 2$, that can be controlled by tuning the initial datum in a suitable way (see Definition \ref{def:bootstrapUnell} below) via a classical topological argument. The second higher order term $P_{N,\nu}$ added to \eqref{exp:mW} is to get a large constant in the spectral gap \eqref{est:SpecGap1} which is only used for the control of the solution and does not affect the leading dynamic of blowup
 
\paragraph*{The generated error and derivation of unstable blowup rates.}  The following lemma is similar to Lemma \ref{lemm:appProf}, from which we can redesign the bootstrap regime \ref{def:bootstrap} adapted for this case.
\begin{lemma} \label{lemm:appWell} Assume that $(\beta,\nu, \mathbf{a})$ are $\Cc^1$ maps $(\beta,\nu, \mathbf{a}): [\tau_0, \tau_1) \mapsto [1/2,2]\times (0,\nu^*] \times (0, a^*]^{\ell + 1}$, for  $0 < \nu^*, a^* \ll 1$ and $1 \ll \tau_0 <\tau_1 \leq +\infty$, with a priori bounds: 
\begin{equation} \label{bd:aprioriinstable}
|\beta_\tau|\lesssim |\ln \nu|^{\frac 32}, \quad \left|\frac{\nu_\tau}{\nu} \right| \lesssim 1, \quad  |a_\ell| \lesssim \nu^2, \quad |a_n| \lesssim \frac{\nu^2}{|\ln \nu|} \;\; \textup{for}\;\; 1 \leq n \ne \ell \leq N.
\end{equation}
Then the error generated by \eqref{def:mW7} to the flow \eqref{partialmassselfsim} can be decomposed as
\begin{align}\label{exp:mEtilee}
m_E &= -\pa_\tau m_W + \pa_\zeta^2 m_W - \frac{\pa_\zeta m_W}{\zeta} + \frac{\pa_\zeta(m_W^2)}{2\zeta} - \beta \zeta\pa_\zeta m_W = \sum_{j = 0}^\ell \textup{Mod}_j  \times  \phi_{j, \nu} + \tilde{m}_E + \frac{\pa_\zeta P_\nu^2}{2\zeta},
\end{align}
where
\begin{align*}
\textup{Mod}_0 &= \left(\frac{\nu_\tau}{\nu} - \beta \right)8\nu^2 +  a_{\ell, \tau} - 2\beta a_\ell \big(1 + \frac{1}{2\ln \nu} + \bar \alpha_0 \big),\\
\textup{Mod}_j &= - a_{j, \tau}  + 2 \beta a_j \big(1 - j + \frac{1}{2\ln \nu} + \bar \alpha_j \big) \quad \textup{for} \quad 1 \leq j \leq N,
\end{align*}
 and
\begin{align} \label{bd:tildemeinstable2}
&\big \la \tm_E,\phi_{\ell,\nu} \big\ra_{L^2_\frac{\omega_\nu}{\zeta}} =  -a_\ell \| \phi_{\ell, \nu}\|^2_{L^2_\frac{\omega_\nu}{\zeta}} \left(\frac{\nu_\tau}{
\nu} \frac{1}{\ln \nu}+\frac{\beta_\tau}{\beta} \ell  \right)+O(|\ln \nu|^{\frac 12}),\\
\label{bd:tildemeinstable3} & \big \la \tm_E,\phi_{0,\nu} \big\ra_{L^2_\frac{\omega_\nu}{\zeta}} \sim  -a_\ell \frac{\nu_\tau}{8\nu}+O(|\ln \nu|^{-\frac 12}),
\end{align}
and 
\begin{equation} \label{bd:tildemeinstable}
\sum_{n = 1, n \ne \ell}^{N}\Big|\big \la \tm_E,\phi_{n,\nu} \big\ra_{L^2_\frac{\omega_\nu}{\zeta}}\Big| \lesssim |\ln \nu | \nu^2, \quad  \|\tm_E\|_{L^2_\frac{\omega_\nu}{\zeta}} \lesssim \nu^2.
\end{equation}

\end{lemma}

\begin{proof} 

Since the proof is exactly the same to that of Lemma \ref{lemm:appProf}, we only sketch it. From the a priori bounds, we focus on the leading order term of $\tilde{m}_E$ which is
\begin{equation} \label{id:instablemainterm}
\tm_E \sim -a_\ell \left( \frac{\nu_\tau}{\nu} \nu \pa_\nu (\phi_{\ell, \nu} - \phi_{0, \nu})+\frac{\beta_\tau}{\beta} \beta \pa_\beta (\phi_{\ell, \nu} - \phi_{0, \nu})\right).
\end{equation}
We introduce (the logarithmic cancellation as $r\to \infty$ is a consequence of \eqref{est:Tj}):
\begin{equation} \label{id:asymptthetainstable}
\Theta_j = (2j - 2)T_j - r\pa_r T_j =_{r\to \infty} \left\{ \begin{array}{l l} -\hat d_j r^{2j - 2}+O(r^{2j-4}\ln r) \qquad \mbox{if }j\geq 1, \\ O(r^{-4}) \qquad \qquad \qquad \quad \mbox{if }j=0,\end{array} \right.
\end{equation}
with $\hat d_1 = -\frac{1}{2}$, $\hat d_{j} = -\frac{\hat d_{j-1}}{4j(j-1)}$. This and \eqref{def:hatphi} imply the identities:
\begin{align} \label{id:instableidentity1}
&\nu\pa_\nu (\phi_{\ell, \nu} - \phi_{0, \nu}) = \sum_{j=1}^\ell c_{\ell,j}\beta^j \nu^{2j-2}  \Theta_j(\zeta/\nu) +\nu\pa_\nu \tilde \phi_{\ell, \nu} -\nu\pa_\nu \tilde \phi_{0, \nu}, \\
\label{id:instableidentity2}&\beta\pa_\beta (\phi_{\ell, \nu} - \phi_{0, \nu}) = \sum_{j=1}^\ell j c_{\ell,j}\beta^j \nu^{2j-2}  T_j(\zeta/\nu) +\beta \pa_\beta \tilde \phi_{\ell, \nu} -\beta \pa_\beta \tilde \phi_{0, \nu}.
\end{align}
From Proposition 1 in \cite{CGNNarx19a} we have the following estimates for the error in \eqref{def:hatphi}:
\begin{equation} \label{id:estimationtildephiinstable}
\| D^k\tilde \phi_{j,\nu}\|_{L^2_{\frac{\omega_\nu}{\zeta}}}+\| D^k\nu \pa_\nu \tilde \phi_{j,\nu}\|_{L^2_{\frac{\omega_\nu}{\zeta}}}+\| D^k\beta \pa_\beta \tilde \phi_{j,\nu}\|_{L^2_{\frac{\omega_\nu}{\zeta}}}\lesssim |\ln \nu|^{-\frac12}
\end{equation}
for $k=0,1,2$ and $0\leq j\leq N$. Thus, the a priori bounds \eqref{bd:aprioriinstable}, the identities \eqref{id:instableidentity1} and \eqref{id:instableidentity2}, the asymptotics \eqref{id:asymptthetainstable} and \eqref{est:Tj} and the bounds \eqref{id:estimationtildephiinstable} prove the second bound in \eqref{bd:tildemeinstable}. This bound in turns imply the first one in \eqref{bd:tildemeinstable} after applying Cauchy-Schhwarz and \eqref{est:PhinL2norm1}. We have proved \eqref{bd:tildemeinstable} for the leading term \eqref{id:instablemainterm}.

Next, we write $\la \cdot \ra=\la \cdot \ra_{L^2_{\frac{\omega_\nu}{\zeta}}}$. The identities \eqref{id:instableidentity1} and \eqref{id:instableidentity2}, the decomposition \eqref{def:hatphi} and the bound \eqref{id:estimationtildephiinstable} imply that, where we change variables  $\zeta\mapsto r=\zeta/\nu$ for the main term:
\begin{align*}
\la \nu \pa_\nu (\phi_{\ell, \nu} - \phi_{0, \nu}),\phi_{0,\nu}\ra = \sum_{j=1}^\ell c_{\ell,j}\beta^j \nu^{2j} \int_0^{\infty} \Theta_j(r) T_0(r) \frac{e^{-\beta \nu^2 r^2/2}}{rU(r)} dr+O(|\ln \nu|^{-\frac 12}),\\
\la \beta \pa_\beta (\phi_{\ell, \nu} - \phi_{0, \nu}),\phi_{0,\nu}\ra =\sum_{j=1}^\ell j c_{\ell,j}\beta^j \nu^{2j} \int_0^{\infty} T_j(r) T_0(r) \frac{e^{-\beta \nu^2 r^2/2}}{rU(r)} dr +O(|\ln \nu|^{\frac 12}).
\end{align*}
We compute the numbers appearing above. By \eqref{id:asymptthetainstable} and \eqref{est:Tj}:
\begin{align*}
 &\sum_{j=1}^\ell c_{\ell,j}\beta^j \nu^{2j} \int_0^{\infty} \Theta_j(r) T_0(r) \frac{e^{-\beta \nu^2 r^2/2}}{rU(r)} dr= -\sum_{j = 1}^{\ell}c_{\ell, j} \hat d_j 2^{j - 4} \int_{0}^{\infty} \xi^{j - 1}e^{-\xi} d\xi +O(\nu) \\
& =\frac 18 \sum_{j = 1}^\ell (-1)^{j + 1} \frac{\ell!}{(\ell - j)! j!}\ = \  \frac{1}{8}+O(\nu),
\end{align*}
using that $\hat d_j = (-1)^j \frac{2^{-2j + 1}}{j[(j -1)!]^2}$ and $c_{\ell, j} = 2^j \frac{\ell!}{(\ell - j)!}$ from \eqref{est:Tj} and \eqref{def:ProTj}. Hence
\begin{equation} \label{instableerror1}
\la \nu \pa_\nu (\phi_{\ell, \nu} - \phi_{0, \nu}),\phi_{0,\nu}\ra=\frac18+O(|\ln \nu|^{-\frac 12}).
\end{equation}
Similarly, using \eqref{est:Tj} and the identity $\int_{0}^{\infty} \xi^{j - 1}e^{-\xi} d\xi=(j-1)!$:
\begin{align*}
& \sum_{j=1}^\ell jc_{\ell,j}\beta^j \nu^{2j} \int_0^{\infty} T_j(r) T_0(r) \frac{e^{-\beta \nu^2 r^2/2}}{rU(r)} dr = |\ln \nu |\sum_{j = 1}^{\ell} c_{\ell, j} \hat d_j 2^{j - 4} j!+O(1)\\
& = \frac 18 |\ln \nu| \sum_{j = 1}^\ell (-1)^{j } \frac{\ell!}{(\ell - j)! (j-1)!}+O(1) \ = \ O(1).
\end{align*}
Hence
\begin{equation} \label{instableerror2}
\la \beta \pa_\beta (\phi_{\ell, \nu} - \phi_{0, \nu}),\phi_{0,\nu}\ra =O(|\ln \nu|^{\frac 12}).
\end{equation}
Next, we recall that $\| \phi_{\ell,\nu}\|_{L^2_{\frac{\omega_\nu}{\zeta}}}=c_\ell \ln^2 \nu$ from \eqref{est:PhinL2norm1}. Differentiating, we get:
$$
2\la \nu \pa_\nu \phi_{\ell,\nu},\phi_{\ell,\nu}\ra=2c_{\ell}\ln \nu-\la  \phi_{\ell,\nu},\phi_{\ell,\nu}\frac{\nu \pa_\nu \omega_\nu}{\omega_\nu}\ra=2c_{\ell}\ln \nu+O(|\ln \nu|^{\frac 12})
$$
where we used the cancellation $|\nu \pa_\nu \omega_\nu|\lesssim \omega_\nu \la \zeta/\nu\ra^{-2}$, \eqref{est:Tj} and \eqref{id:estimationtildephiinstable}. We recall the orthogonality $\la \phi_{\ell,\nu},\phi_{0,\nu} \ra=0$. Differentiating it, using the same cancellation, \eqref{est:Tj} and \eqref{id:estimationtildephiinstable} we get:
$$
2\la \nu \pa_\nu \phi_{0,\nu},\phi_{\ell,\nu}\ra=-\la  \phi_{0,\nu},\phi_{\ell,\nu}\frac{\nu \pa_\nu \omega_\nu}{\omega_\nu}\ra=O(|\ln \nu|^{\frac 12}).
$$
Collecting the two estimates above we have proved:
\begin{equation} \label{instableerror3}
 \la \nu \pa_\nu (\phi_{\ell, \nu} - \phi_{0, \nu}),\phi_{\ell,\nu}\ra=c_{\ell}\ln \nu+O(|\ln \nu|^{\frac 12}).
\end{equation}
Next, we write from \eqref{id:asymptthetainstable} and \eqref{def:AsAs0}, using that $\Ac_0 T_j=-T_{j-1}$:
$$
j T_j =T_j+\frac 12 r \pa_r T_j+\frac 12 \Theta_j=T_j-(2\beta \nu^2)^{-1}\Ac T_{j}-(2\beta \nu^2)^{-1}T_{j-1}+\frac 12 \Theta_j .
$$
Above, we recall that from \eqref{def:hatphi} and \eqref{id:estimationtildephiinstable}, for $j<n$ one has $\nu^{2j-2}T_j(\zeta/\nu)=\sum_0^j \tilde c_{j,i}\phi_{i,\nu}+O_{L^2_{\omega_\nu/\zeta}}(|\ln \nu|^{1/2})$ for some constants $\tilde c_{j,i}$. Using this, the above identity, the identity \eqref{def:hatphi}, the fact that $\Ac^\zeta \phi_{\ell,\nu}=\alpha_{\ell,\nu} \phi_{\ell,\nu}$ and \eqref{def:specAsb} we obtain:
$$
\beta \pa_\beta \phi_{\ell,\nu}=\ell \phi_{\ell,\nu}+\sum_{j<n}\tilde c_j \phi_{j,\nu}+O_{L^2_{\omega_\nu/\zeta}}(|\ln \nu|^{1/2})
$$
for some constants $\tilde c_j$. Therefore, from the orthogonality of the eigenfunctions:
$$
\la  \beta \pa_\beta \phi_{\ell,\nu},\phi_{\ell,\nu}\ra=\ell \| \phi_{\ell,\nu}\|_{L^2_{\omega_\nu/\zeta}}^2+O(|\ln \nu|^{3/2}).
$$
Since from \eqref{def:hatphi}, $\pa_\beta \phi_{0,\nu}=\pa_\beta \tilde \phi_{0,\nu}$, from \eqref{est:PhinL2norm1}, \eqref{id:estimationtildephiinstable} and Cauchy-Schwarz we get $\la  \beta \pa_\beta \phi_{0,\nu},\phi_{\ell,\nu}\ra=O(|\ln \nu|^{3/2})$. This and the above give
\begin{equation} \label{instableerror4}
\la  \beta \pa_\beta (\phi_{\ell,\nu}-\phi_{0,\nu}),\phi_{\ell,\nu}\ra=\ell \| \phi_{\ell,\nu}\|_{L^2_{\omega_\nu/\zeta}}^2+O(|\ln \nu|^{3/2}).
\end{equation}
Collecting \eqref{instableerror1}, \eqref{instableerror2}, \eqref{instableerror3} and \eqref{instableerror4}, using the a priori bounds \eqref{bd:aprioriinstable}, we have established \eqref{bd:tildemeinstable2} and \eqref{bd:tildemeinstable3} for the main order term \eqref{id:instablemainterm}.

\end{proof}
\begin{remark}[Unstable blowup rates] \label{rem:unBlowuprateell} From Lemma \ref{lemm:appWell}, we project \eqref{exp:mEtilee} onto $\phi_{0, \nu}$ and $\phi_{\ell, \nu}$ to obtain the following system of ODEs
\begin{equation}\label{eq:companuaellsystem}
\left\{ \begin{array}{l l} &\left(\frac{\nu_\tau}{\nu} - \beta \right)8\nu^2 +  a_{\ell, \tau} - 2\beta a_\ell \big(1 + \frac{1}{2\ln \nu} \big) + a_\ell \frac{\nu_\tau}{\nu} \frac{1}{\ln \nu} = \Oc\left(\frac{\nu^2}{|\ln \nu|^{3/2}} \right),\\
& a_{\ell, \tau} - 2\beta a_\ell \big(1 - \ell + \frac{1}{2\ln \nu}\big) + a_\ell \frac{\nu_\tau}{\nu} \frac{1}{\ln \nu} +\ell a_\ell \frac{\beta_\tau}{\beta}= \Oc\left(\frac{\nu^2}{|\ln \nu|^{3/2}} \right).
\end{array} \right.
\end{equation}
We solve this system for $0<\nu\ll 1$ and $\beta\approx 1$ under the compatibility condition
\begin{equation}\label{eq:companuaell}
\frac{a_{\ell}}{4\nu^2} = -1+\frac{1}{2 \ln \nu},
\end{equation}
which is a constraint on $\beta$. Namely one obtains from \eqref{eq:companuaellsystem} that this condition is satisfied provided that $\beta_\tau =O(|\ln \nu|^{-3/2})$. Under \eqref{eq:companuaell}, \eqref{eq:companuaellsystem} gives:
$$
\frac{\nu_\tau}{\nu} = \beta(1 - \ell)+\frac{\beta \ell}{2\ln \nu}+O(|\ln \nu|^{-\frac 32}).
$$
Solving this yields that $\beta\to \beta_{\infty}$ and $\nu(\tau)\sim e^{\beta (1-\ell)\tau}\tau^{\frac{\ell}{2(1-\ell)}}\nu_{\infty}$ for some $\beta_{\infty},\nu_{\infty}>0$. Since $\mu_\tau=\beta \mu$, $\pa_t \tau=\mu^{-2}$ and $\lambda =\mu \nu$, we get that for some blow up time $T>0$:
\begin{equation}\label{eq:lawlambdaell}
\lambda(t) \sim C(u_0) (T-t)^\frac{\ell}{2} |\ln (T-t)|^{-\frac{\ell}{2(\ell - 1)}}.
\end{equation}
\end{remark}

\begin{remark} \label{re:instable}
Note that $\int |\ln \nu|^{-\alpha}d\tau<\infty$ for all $\alpha>1$, which is not the case for the stable blow-up law \eqref{bootstrap:param1} where $\alpha>2$ is needed. This is of a simplification for the instable case: one can only perform the analysis with an accuracy of one order in $|\ln \nu|^{-1}$ less than for the stable case and still be able to close the estimates.
\end{remark}
\paragraph*{Bootstrap regime.} Lemma \ref{lemm:appWell} provides information about the size of the error and Remark \ref{rem:unBlowuprateell} formally gives us the law of $\nu$, from which we can redesign the bootstrap regime \eqref{def:bootstrap} adapted to the case $\ell \geq 2$. In particular, we control the remainder $\e$ according to the following regime. 

\begin{definition}[Bootstrap]  \label{def:bootstrapUnell}
Let $\ell \in \mathbb{N}$ with $\ell \geq 2$ and $\tau_0\gg 1$. A solution $w$ is said to be trapped on $[\tau_0,\tau^*]$ if it satisfies the initial bootstrap conditions in the sense of Definition \ref{def:ini} at time $\tau_0$ and the following conditions on $(\tau_0,\tau^*]$. There exists $\mu \in C^1([0,t^*],(0,\infty))$ and constants $K''\gg K'\gg K\gg 1$ such that the solution can be decomposed according to \eqref{decomposition}, \eqref{orthogonality} on $(\tau_0,\tau^*]$ with:
\begin{itemize}
\item[(i)] \emph{(Compatibility condition for the renormalisation rate $\beta$)}
$$\frac{a_{\ell}}{4\nu^2} = -1+\frac{1}{2 \ln \nu},$$
\item[(ii)] \emph{(Modulation parameters)}
$$
e^{\frac{(1 - \ell)\tau}{2}} \tau^{\frac{ \ell}{2(1 - \ell)}} \leq \nu(\tau) \leq 2 e^{\frac{(1 - \ell)\tau}{2}} \tau^{\frac{ \ell}{2(1 - \ell)}} \ \ \mbox{ and } \ \ \frac 12 < \beta <2,
$$
$$
|a_n|< \frac{\nu^2}{|\ln \nu|} \ \mbox{ for } 1 \leq n \ne \ell \leq N.
$$
\item[(iii)] \emph{(Remainder)}
$$
 \|m_\varepsilon(\tau) \|_{L^2_{\frac{\omega_\nu}{\zeta}}}<K \nu^2, \quad \| m_\e(\tau) \|_{H^2(\zeta_*\leq \zeta\leq \zeta^*)} < K' \nu^2,
$$
$$\quad \| \tilde m_{v} \|_{\inn}  < K''\nu^2, \quad \| \hat w^{0} \|_{\out} < K''\nu^2.$$

\end{itemize}

\end{definition}

The bootstrap definition \ref{def:bootstrapUnell} is almost the same as the one defined in Definition \ref{def:bootstrap}, except for the bounds on $m_\e$ which are of size $\nu^2$ only, see remark \ref{re:instable}. All the energy estimates as well as the derivation of the modulation equations given in Section \ref{sec:ControlStab} can be adapted to the new definition \ref{def:bootstrapUnell} without any difficulties to derive the conclusion of Theorem \ref{theo:UnStab}.

%%      ---------------------------------------------------------------------
%%      ------------------------- APPENDIX (OPTIONAL) -----------------------
%%      ---------------------------------------------------------------------
        
%%      If you have one appendix, uncomment the line \appendix and add
%%      a \section{ *** APPENDIX TITLE ***}. If you have more than
%%      one, uncomment the line \appendices and add a \section{ ***
%%      APPENDIX TITLE ***} command for each appendix title.

%\appendix
\appendices

\appendix
\section{Some useful estimates}

\noindent \textbf{Hardy-Poincar\'e type inequality:} Let us recall the following estimates in spaces with weights involving $\rho$, for functions $v$ and $u$ without radial components. The first one is a Poincar\'e-type inequality
\begin{equation} \label{bd:genpoincare}
\int_{\mathbb R^2} v^2 |z|^{2k}(1+|z|^{2})e^{-\frac{z^2}{2}}dz \lesssim \int_{\mathbb R^2} |\nabla v|^2 |z|^{2k} e^{-\frac{z^2}{2}}dz,
\end{equation}
for any $k\geq 0$. The second one is a Hardy type inequality: for $0<b<1$, there exists $C > 0$ independent of $b$ such that
\begin{equation} \label{bd:hardyL2rho}
\int_{\mathbb R^2} (1+|y|^2) u^2 e^{-\frac{b|y|^2}{2}}  \leq C \int_{\mathbb R^2} (1+|y|^4)|\nabla u|^2e^{-\frac{b|y|^2}{2}}.
\end{equation}
By the change of variables $z=\sqrt by$, the two above inequalities imply for any $1\leq \alpha \leq 3$:
\begin{equation}\label{est:HaPo1}
b^{\alpha - 1} \int_{\mathbb R^2} |q^\perp|^2  (1 + |y|^{2\alpha}) \rho dy \leq C_\alpha \int_{\mathbb R^2} \frac{|\nabla q^\perp|^2}{U}\rho dy.
\end{equation}

\bigskip

\noindent \textbf{Estimates on the Poison field:} For $u$ localized on a single spherical harmonics, the Laplace operator is  written as
$$
\Delta u (x)=\Delta^{(k)} (u^{(k,i)})(r) \phi^{(k,i)}(\theta), \ \ \Delta^{(k)}:= \pa_{rr}+\frac{1}{r}\pa_r-\frac{k^2}{r^2}.
$$
The fundamental solutions to $\Delta^{(k)}f=0$ are
$$
\left\{ \ba{l l} \log(r) \ \ \text{and} \ \ 1\ \ \text{for} \ k=0, \\r^k \ \ \text{and} \ \ r^{-k} \ \text{for} \ k\geq 1. \ea \right.
$$
and their Wronskian relation
$$
W^{(0)}=\frac{d}{dr}\log (r)=r^{-1} \ \ \text{and} \ \ W^{(k)}=\frac{d}{dr}(r^k)r^{-k}-r^k \frac{d}{dr}(r^{-k})=2kr^{-1} \ \text{for} \ k\geq1.
$$
The solution to the Laplace equation $-\Delta \Phi_u=u$ given by $\Phi_u=-(2\pi)^{-1}\log (|x|)*u$ is given on spherical harmonics by:
$$
\Phi_u^{(0,0)}(r)=-\log (r) \int_0^r u^{(0,0)}(\tilde r)\tilde r d \tilde r -\int_r^{+\infty} u^{(0,0)}(\tilde r)\log(\tilde r)\tilde r d\tilde r ,
$$
$$
\nabla \Phi_u^{(0,0)}(x)=-\frac{x}{|x|^2} \int_0^{|x|} u^{(0,0)}(\tilde r)\tilde r d \tilde r,
$$
\begin{equation} \label{id:Phiki}
\Phi_u^{(k,i)}(r)= \frac{r^k}{2k} \int_r^{+\infty} u^{(k,i)}(\tilde r)\tilde r^{1-k} d \tilde r +\frac{r^{-k}}{2k} \int_0^r u^{(k,i)}(\tilde r) \tilde r^{1+k} d\tilde r,
\end{equation}
\begin{equation} \label{id:nablaPhiki}
\pa_r \Phi_u^{(k,i)}(r)= \frac{r^{k-1}}{2} \int_r^{+\infty} u^{(k,i)}(\tilde r)\tilde r^{1-k} d \tilde r -\frac{r^{-k-1}}{2} \int_0^r u^{(k,i)}(\tilde r) \tilde r^{1+k} d\tilde r,
\end{equation}

The following lemma gives pointwise estimates of the Poison field.
\begin{lemma}
If $u^{(0,0)}=0$, there holds the estimate for any $\alpha>0$,
\begin{equation} \label{bd:poisson1}
|\Phi_{u}|^2+ |y|^2 |\nabla \Phi_{u}|^2\lesssim |y|^2(1+|y|)^{-2\alpha} \big(1 + \mathbf{1}_{|y| \leq 1} |\log |y||\big)\int_{\mathbb R^2}  |u|^2(1+|y|)^{2\alpha }dy.
\end{equation}
\end{lemma}

\begin{proof} See Lemma A.1 in \cite{CGNNarx19a}.
\end{proof}
\noindent For the control of the outer part, we need the following estimates for the Poison field in term of the outer norm. 
\begin{lemma}
For $|z|\geq 1$ one has the estimate,
\begin{equation} \label{estimationpoissonexte}
 \frac{1}{\zeta}|\Phi_{(1-\chi_{\frac{\zeta^*}{2}})\hat w}(z)| + |\nabla \Phi_{(1-\chi_{\frac{\zeta^*}{2}})\hat w}(z)|+\zeta |\nabla^2 \Phi_{(1-\chi_{\frac{\zeta^*}{2}})\hat w}(z)|\lesssim \frac{1}{\zeta^{\frac 12}} \| \hat w \|_{\out}.
\end{equation}
\end{lemma}

\begin{proof}

Notice that for $v(y)=\nu^2w(z)$, the Poisson field scales in $L^{\infty}$, that is: $\Phi_v(y)=\Phi_w(z)$. Hence \eqref{bd:poisson1} holds also in $z$ variables. We apply it with $\alpha=\frac 12$, so that for $|z|\geq 1$:
\begin{align*}
&\frac{1}{|z|}|\Phi_{(1-\chi_{\frac{\zeta^*}{2}})\hat w}|^2+ |z||\nabla \Phi_{(1-\chi_{\frac{\zeta^*}{2}})\hat w}|^2 \lesssim  \int_{\mathbb R^2}  |(1-\chi_{\frac{\zeta^*}{2}})\hat w|^2|z|dz\lesssim  \int_{\mathbb R^2}  (1-\chi_{\frac{\zeta^*}{2}}) \hat w^2|z|^{4-\frac 12} |z|^{-\frac 12}\frac{dz}{|z|^2}\\
&\quad \lesssim  \left(\int_{\mathbb R^2}  (1-\chi_{\frac{\zeta^*}{2}}) |\hat w|z|^{2-\frac 14}|^{2p}\frac{dz}{|z|^2}\right)^{\frac{1}{p}}  \left(\int_{\mathbb R^2}  (1-\chi_{\frac{\zeta^*}{2}}) |z|^{-\frac 12 p'}\frac{dz}{|z|^2}\right)^{\frac{1}{p'}} \lesssim  \| \hat w\|_{\out}^{2}
\end{align*}
where we used $(1-\chi_{\frac{\zeta^*}{2}})\leq 1$ and H\"older with $p'$ the conjugate exponent of $p$. Next, we decompose $\hat w=\hat w^{0}+\hat w^\perp$ between radial and non radial components. One has $\| \hat w^{\perp}\|_{\out}+\| \hat w^0\|_{\out}\lesssim \| \hat w \|_{\out}$. from the Sobolev embedding of $W^{3/p,2p}$ into the H\"older space $C^{1/p}$ in dimension $2$, we obtain by interpolation that for any $|z|\geq 1$:
\begin{align*}
&\| (1-\chi_{\frac{\zeta^*}{2}})\hat w^\perp \|_{C^{\frac 1p}(B(z,\frac{|z|}{2}))}  \leq  C \| \hat w^\perp \|_{W^{\frac{3}{p},2p}(B(z,\frac{|z|}{2}))}\leq C  \| \hat w^\perp \|_{L^{2p}(B(z,\frac{|z|}{2}))}^{1-\frac 3p} \| \nabla \hat w^\perp \|_{L^{2p}(B(z,\frac{|z|}{2}))}^{\frac 3p} \\
& \quad \leq  \frac{C}{\zeta^{\frac 32}}  \| \zeta^{2-\frac 12}\hat w^\perp \|_{L^{2p}(B(z,\frac{|z|}{2}))}^{1-\frac 3p} \| \zeta^{2-\frac 12} \nabla \hat w^\perp \|_{L^{2p}(B(z,\frac{|z|}{2}))}^{\frac 3p}\leq \frac{C}{\zeta^{\frac 32}} \| \hat w \|_{\out},
\end{align*}
where we used the definition of the $\| \hat w \|_{\out}$ norm and the inequality $1/2\geq 1/4+1/p$ for $p$ large enough. Therefore, in the ball $B(z,|z|/2)$, $\Phi_{(1-\chi_{\frac{\zeta^*}{2}})\hat w^\perp}$ solves $-\Delta \Phi_{(1-\chi_{\frac{\zeta^*}{2}})\hat w^\perp}=(1-\chi_{\frac{\zeta^*}{2}})\hat w^\perp$ with the estimates:
$$
\frac{1}{|z|^{\frac 12}}|\Phi_{(1-\chi_{\frac{\zeta^*}{2}})\hat w^\perp}|+|z|^{\frac 12}|\nabla \Phi_{(1-\chi_{\frac{\zeta^*}{2}})\hat w^\perp}+|z|^{\frac 32}\|(1-\chi_{\frac{\zeta^*}{2}})\hat w^\perp\|_{C^{1/p}(B(z,|z|/2)}\leq C\| \hat w \|_{\out}.
$$
By standard regularity properties of the Dirichlet problem and a rescaling argument, we obtain that $\| \nabla^2 \Phi_{(1-\chi_{\frac{\zeta^*}{2}})\hat w^\perp} \|_{B(z,\frac{|z|}{2}}\leq |z|^{-3/2}C\| \hat w \|_{\out}$. This proves \eqref{estimationpoissonexte} for the nonradial part of $\hat w$. Next, for the radial part, we have for $i=1,2$ and $\zeta \geq 1$:
\begin{eqnarray*}
\nabla \partial_{z_i} \Phi_{(1-\chi_{\frac{\zeta^*}{2}})\hat w^0}(z)  &= & -\nabla \left(\frac{z_i}{\zeta^2} \int_0^{\zeta} (1-\chi_{\frac{\zeta^*}{2}})\hat w^0(\tilde \zeta)\tilde \zeta d \tilde \zeta\right),\\
&=&-\nabla (z_i) \frac{1}{\zeta^2} \int_0^{\zeta} (1-\chi_{\frac{\zeta^*}{2}})\hat w^0(\tilde \zeta)\tilde \zeta d \tilde{\zeta} +\frac{2z_iz}{\zeta^4} \int_0^{\zeta} (1-\chi_{\frac{\zeta^*}{2}})\hat w^0(\tilde \zeta)\tilde \zeta d \tilde \zeta -\frac{2z_iz}{\zeta^2} (1-\chi_{\frac{\zeta^*}{2}})\hat w^0(\zeta).
\end{eqnarray*}
From H\"older, where $(2p)'$ is the H\"older conjugate of $2p$:
\begin{align*}
&\left| \int_0^{\zeta} (1-\chi_{\frac{\zeta^*}{2}})\hat w^0(\tilde \zeta)\tilde \zeta d \tilde \zeta \right|=\left| \int_0^{\zeta} (1-\chi_{\frac{\zeta^*}{2}})\tilde \zeta^{2-\frac 14 -\frac 1p}\hat w^0(\tilde \zeta)\tilde \zeta^{-1+\frac 14 +\frac 1p} d \tilde \zeta \right|  \lesssim  \left| \int_0^{\zeta} (1-\chi_{\frac{\zeta^*}{2}})\tilde \zeta^{2-\frac 14 -\frac{1}{2p}}\hat w^0(\tilde \zeta)\tilde \zeta^{-1+\frac 14 +\frac{1}{2p}} d \tilde \zeta \right|\\
& \quad \lesssim   \left( \int_0^{\zeta} (1-\chi_{\frac{\zeta^*}{2}})|\tilde \zeta^{2-\frac 14}\hat w^0(\tilde \zeta)|\frac{d\tilde \zeta}{\zeta} \right)^{\frac{1}{2p}} \left( \int_0^{\zeta} (1-\chi_{\frac{\zeta^*}{2}}) \tilde \zeta^{(-1+\frac{1}{4}+\frac{1}{2p})(2p)'} d\tilde \zeta\right)^{\frac{1}{(2p)'}} \lesssim  \| \hat w \|_{\out} \zeta^{\frac 14}.
\end{align*}
We recall that $|\hat w^{0}|\lesssim \zeta^{-3/2}$ for $\zeta \geq 1$ from \eqref{weightedsobo}. This, and the two above identities imply that for $\zeta \geq 1$:
$$
\left| \nabla \partial_{z_i} \Phi_{(1-\chi_{\frac{\zeta^*}{2}})\hat w^0}(z) \right|\lesssim \zeta^{-\frac 32} \| \hat w \|_{\out}.
$$
This proves \eqref{estimationpoissonexte} for the radial part of $\hat w$.

\end{proof}

\section{Coercivity of $\As_0$}
In this section, we aim at deriving the coercive estimate of $\As_0$ which is the key to establish the monotonicity formula of the inner norm \eqref{bootstrap:in}. We first claim the following.

\begin{lemma}[Coercivity of $\As_0$] \label{lemm:coerA0} Assume that $f:[0,\infty)\rightarrow \mathbb R$ satisfies
\begin{equation} \label{coerciviteA0:hp1}
\int_0^{\infty} \left(\frac{|f|^2}{1+r^2}+|\pa_r f|^2\right)\frac{\omega_0(r)}{r}dr<\infty,
\end{equation}
and 
$$\int_0^\infty f(r)T_0(r) \chi_{_M}(r) \frac{\omega_0}{r} dr=0.$$
Then, there exists a constant $\delta_0>0$ such that
$$
\int_0^\infty f \As_0 f\frac{\omega_0}{r} dr \leq -\delta_0 \int_0^\infty  \left(\frac{|f|^2}{1+r^2}+|\pa_r f|^2\right)\frac{\omega_0}{r} dr.
$$
\end{lemma}

\begin{proof}

\noindent \textbf{Step 1} We first claim the Hardy inequality for any $f$ satisfying \eqref{coerciviteA0:hp1}:
\begin{equation} \label{coerciviteA0:hardy}
\int_0^\infty |\pa_r f|^2\frac{\omega_0}{r}dr\gtrsim \int_0^\infty \frac{f^2}{1+r^2}\frac{\omega_0}{r}dr.
\end{equation}
We first study the function near the origin. If $f$ satisfies \eqref{coerciviteA0:hp1} then by standard one dimensional Sobolev embedding, $f$ is continuous on $[0,1]$. Hence the estimate \eqref{coerciviteA0:hp1} implies $f(0)=0$. From the fundamental Theorem of Calculus (which is justified for $f$ via a standard approximation procedure):
\begin{align*}
\left|f(r)\right|&=\left| f(0)+\int_0^r \pa_r (f)d\tilde r \right|=\left|\int_0^r \pa_r (f)d\tilde r \right| \leq \left( \int_0^r \frac{|\pa_r (f)|^2}{\tilde r}d\tilde r \right)^{\frac 12} \left( \int_0^r \tilde rd\tilde r \right)^{\frac 12}\lesssim r \int_0^\infty |\pa_r f|^2\frac{\omega_0}{\tilde r}d\tilde r.
\end{align*}
This proves that:
\begin{equation} \label{coerciviteA0:inter1}
\int_0^\infty |\pa_r f|^2\frac{\omega_0}{r}dr \gtrsim \int_0^1  \frac{f^2}{1+r^2}\frac{\omega_0}{r}dr.
\end{equation}
Away from the origin, integrating by parts:
$$
-\int_1^\infty \pa_r f f r^2 dr=\int_1^\infty f^2rdr+f^2(1).
$$
By Cauchy-Schwarz and Hardy:
$$
\left| \int_1^\infty \pa_r f f r^2 dr \right| \leq \frac12 \int_1^\infty |\pa_r f|^2r^3dr+\frac 12 \int_1^\infty f^2rdr.
$$
The two above identities then give:
\begin{align*}
\int_1^\infty f^2rdr&\leq -\int_1^\infty \pa_r f f r^2 dr-f^2(1)\leq -\int_1^\infty \pa_r f f r^2 dr \leq \frac12 \int_1^\infty |\pa_r f|^2r^3dr+\frac 12 \int_1^\infty f^2rdr,
\end{align*}
from what one deduces, since for $r\geq 1$, one has $cr^4\leq \omega_0(r)\leq r^4/c$ for some $c>0$:
\begin{equation} \label{coerciviteA0:inter2}
\int_1^\infty f^2rdr \leq  \int_1^\infty |\pa_r f|^2r^3dr\lesssim \int_0^\infty |\pa_r f|^2\frac{\omega_0}{r}dr.
\end{equation}
The two estimates \eqref{coerciviteA0:inter1} and \eqref{coerciviteA0:inter2} imply the desired Hardy inequality \eqref{coerciviteA0:hardy}.\\

\noindent \textbf{Step 2} \emph{Proof of the coercivity estimate}. $\As_0$ has $T_0(r)$ in its kernel, with $T_0$ a strictly positive function on $(0,\infty)$ implies that the spectrum of this self-adjoint operator is nonnegative by standard Sturm-Liouville argument. Hence for any $f$ satisfying \eqref{coerciviteA0:hp1}:
$$
\int_0^\infty f \As_0 f\frac{\omega_0}{r}dr\leq 0 .
$$
Integrating by parts:
\begin{equation} \label{coerciviteA0:inter4}
\int_0^\infty f \As_0 f\frac{\omega_0}{r} dr=-\int_0^\infty |\pa_r f|^2\frac{\omega_0}{r}dr+\int_0^\infty f^2\frac{dr}{r}.
\end{equation}
Combining this and Step 1 gives that for some $c,C>0$, for any $f$ satisfying \eqref{coerciviteA0:hp1}:
\begin{equation} \label{coerciviteA0:inter3}
\int_0^\infty f \As_0 f\frac{\omega_0}{r} dr \leq -c \int_0^\infty  \left(\frac{|f|^2}{1+r^2}+|\pa_r f|^2\right)\frac{\omega_0}{r} dr+C\int_0^\infty \frac{f^2}{r}dr.
\end{equation}
We now assume by contradiction that there exists a sequence of functions $f_n$ with
\begin{equation} \label{coerciviteA0:hpfn}
\int_0^{\infty} \left(\frac{|f_n|^2}{1+r^2}+|\pa_r f_n|^2\right)\frac{\omega_0(r)}{r}dr=1, 
\end{equation}
and 
$$ \int_0^\infty f_n \As_0 f_n\frac{\omega_0}{r} dr\rightarrow 0, \ \ \int_0^\infty f_n(r)\chi(r) T_0(r)dr=0.$$
Up to a subsequence, $f_n$ converges weakly in $H^1_{loc}$ and strongly in $L^2_{loc}$ to some function $f_\infty$. By lower-semicontinuity of the above norm, and by strong continuity in $L^2_{loc}$ it satisfies:
$$
\int_0^{\infty} \left(\frac{|f_\infty|^2}{1+r^2}+|\pa_r f_\infty|^2\right)\frac{\omega_0(r)}{r}dr\leq 1, \ \ \int_0^\infty f_{\infty}(r)\chi(r) T_0(r)dr=0.
$$
From the bound $\int_0^{\infty} |f_\infty|^2\omega_0/(r+r^3)dr\leq 1$ and the strong continuity in $L^2_{loc}$ one has:
$$
\int_0^\infty \frac{f_n^2}{r}dr\rightarrow \int_0^\infty \frac{f_\infty^2}{r}dr.
$$
The subcoercivity \eqref{coerciviteA0:inter3} and the first bound in \eqref{coerciviteA0:hpfn} imply $\int_0^\infty \frac{f_n^2}{r}dr\geq c$ for some $c>0$. From the above strong convergence this implies:
$$
\int_0^\infty \frac{f_\infty^2}{r}dr \geq c >0.
$$
From \eqref{coerciviteA0:inter4}, \eqref{coerciviteA0:hpfn}, the aforementioned strong convergence and lower semi-continuity:
\begin{align*}
0\leq -\int_0^\infty f \As_0 f\frac{\omega_0}{r}&=\int_0^\infty |\pa_r f|^2\frac{\omega_0}{r}dr-\int_0^\infty f^2\frac{dr}{r}\leq \liminf\left( \int_0^\infty |\pa_r f_n|^2\frac{\omega_0}{r}dr-\int_0^\infty f_n^2\frac{dr}{r} \right)=0.
\end{align*}
Hence $\int_0^\infty f \As_0 f\frac{\omega_0}{r}=0$. Hence $f_\infty=\lambda T_0(r)$, for some $\lambda \neq 0$, which contradicts with $\int_0^\infty f_{\infty}\chi T_0=0$. This concludes the proof of Lemma \ref{lemm:coerA0}.
\end{proof}

We also have the following coercivity estimate for $\As_0$.
\begin{lemma}[Coercivity of $\As_0$] \label{lemm:coerA02} Assume that $f:[0,\infty)\rightarrow \mathbb R$ satisfies
\begin{equation*}
\int_0^{\infty} \left(|\pa^2_r f|^2 + \frac{|\pa_r f|^2}{1 + r^2} + \frac{|f|^2}{(1+r^4)(1  +\ln ^2\la r \ra)}\right)\frac{\omega_0(r)}{r}dr<\infty,
\end{equation*}
and 
$$\int_0^\infty f(r)T_0(r) \chi_{_M}(r) \frac{\omega_0}{r} dr=0.$$
Then, there exists a constant $\delta_1>0$ such that
$$
\int_0^\infty |\As_0 f|^2\frac{\omega_0}{r} dr \geq \delta_1 \int_0^\infty  \left(|\pa^2_r f|^2 + \frac{|\pa_r f|^2}{1 + r^2} + \frac{|f|^2}{(1+r^4)(1  +\ln ^2\la r\ra )}\right)\frac{\omega_0(r)}{r} dr.
$$
\end{lemma}

\begin{proof}Since the proof follows exactly the same lines to the one of Lemma \ref{lemm:coerA0}, apart from the following Hardy inequality (in the critical case) which is used to obtain a sub-coercivity estimate,
\begin{equation}
\int_0^{+\infty} \frac{|\pa_r f|^2}{1 + r^2} \frac{\omega_0}{r}dr \geq C \int_0^{+\infty} \frac{|f|^2}{(1 + r^4)(1 + \ln^2\la r\ra )} \frac{\omega_0}{r}dr. 
\end{equation}
so we omit the proof. 
\end{proof}

%%      Type body of appendix/-ices here.

%%      ---------------------------------------------------------------------
%%      ---------------------------ACKNOWLEDGMENTS (OPTIONAL) ---------------
%%      ---------------------------------------------------------------------

%% ***** UNCOMMENT THE FOLLOWING LINE TO ADD ACKNOWLEDGMENTS.

% \ack 

%%      Type acknowledgments here.

%%      ---------------------------------------------------------------------
%%      --------------------------- BIBLIOGRAPHY ----------------------------
%%      ---------------------------------------------------------------------

\frenchspacing
\bibliographystyle{cpam}
\def\cprime{$'$}

%%      For each reference, provide the following information:

% \bibitem{ *** LABEL *** }             %% Give a reference label.
% * Name(s) of Author(s) *              %% Enter author(s) names.
% EXAMPLE:  Gray, M., Black, F., and White, A.          

%%      Use the following template for a journal article:
% * Title of article *.                 %% Example: Existence and uniqueness.
% \textit{* Abbreviated journal name *} %% Example: \textit{Comm. Pure Appl. Math.}
% \textbf{* Volume number *}            %% Example: \textbf{72}
% (* Year of publication *),            %% Example: (1993),
% * Issue number [optional],            %% Example: no. 6,
% * Page range *.                       %% Example: 675--690.
                                
%%      Use the following template for a book:
% \textit{* Title of book *}.           %% Example: \textit{Ancient Topology}.
% * Publisher *,                        %% Example: Wiley-Interscience,
% * City of publisher *,                %% Example: New York,
% * Year of publication *.              %% Example: 1993.

%%      ---------------------------------------------------------------------
%%      ------------------------ CONTACT INFORMATION ------------------------
%%      ---------------------------------------------------------------------

%      Place contact information for each author between
%      the \begin{comment} and \end{comment} commands. Include
%      preferred mailing address and e-mail addresses. Please note
%      that these will not print. We will format them for printing
%      during the editing stage.

\end{document}